\documentclass[draft,final]{istaustriathesis}

\setlicense{cc-by-nc-sa} 
\usepackage[type={CC},modifier={by-nc-sa},version={4.0},hyperxmp=false]{doclicense} 

\begingroup
\endgroup

\newcommand{\authorname}{Kamil Rychlewicz} 
\newcommand{\thesistitle}{Equivariant cohomology and rings of functions} 
\newcommand{\issn}{2663-337X} 

\usepackage[english]{babel} 
\usepackage{amsmath}    
\usepackage{amssymb}    
\usepackage{mathtools}  
\usepackage{microtype}  
\usepackage[inline]{enumitem} 
\usepackage{multirow}   
\usepackage{booktabs}   
\usepackage{subcaption} 
\usepackage[ruled,linesnumbered,algochapter]{algorithm2e} 
\usepackage[dvipsnames,table]{xcolor} 
\usepackage{nag}        
\usepackage{todonotes}  
\usepackage{morewrites} 

\usepackage[T1]{fontenc}
\usepackage[utf8]{inputenc}
\usepackage{amsmath,amsfonts,amsthm,amssymb}
\usepackage{bm}
\usepackage{mathrsfs}
\usepackage{tikz}
\usepackage{tikz-cd} 
\usepackage{graphicx}
\usepackage[]{algorithm2e}
\usepackage{wrapfig}
\usepackage{xcolor}
\usepackage{comment}
\usepackage{enumerate}
\usepackage{relsize}
\usepackage{xfrac}
\usepackage{faktor}
\usepackage{ifpdf}
\ifpdf
\fi
\usepackage[a-2b,mathxmp]{pdfx}  
\usepackage{hyperref}

\hypersetup{
	pdfpagelayout   = TwoPageRight,            
	linkbordercolor = {Melon},                 
}

\setpnumwidth{2.5em}        
\setsecnumdepth{subsection} 
\settocdepth{paragraph}

\nonzeroparskip             
\makeindex      

\makeevenfoot{plain}{}{\thepage}{} 
\makeoddfoot{plain}{}{\thepage}{} 
\makeevenfoot{Ruled}{}{\thepage}{} 
\makeoddfoot{Ruled}{}{\thepage}{} 

\settitle{\thesistitle}
\setsubtitle{} 
\setdate{2024}{06}{10} 
\setissn{\issn}

\newcommand{\istaustriaaffiliation}{ISTA, Klosterneuburg, Austria}

\setauthor{}{\authorname}{}{\istaustriaaffiliation}
\setsupervisor{}{Tam\'{a}s Hausel}{}{\istaustriaaffiliation}
\setcommitteememberI{}{Anton Mellit}{}{University of Vienna, Vienna, Austria}
\setcommitteememberII{}{Andr\'{a}s Szenes}{}{University of Geneva, Geneva, Switzerland}
\setdefensechair{}{Carl-Philipp Heisenberg}{}{\istaustriaaffiliation}

\newtheorem{theoremc}{Theorem}[chapter]

\newtheorem{theorem}{Theorem}[section]
\newtheorem{corollary}[theorem]{Corollary}
\newtheorem{lemma}[theorem]{Lemma}
\newtheorem{proposition}[theorem]{Proposition}
\newtheorem{problem}[theorem]{Problem}

\newtheorem{conjecture}[theorem]{Conjecture}

\theoremstyle{definition}
\newtheorem{definition}[theorem]{Definition}
\newtheorem{definitionc}[theoremc]{Definition}

\theoremstyle{remark}
\newtheorem{remark}[theorem]{Remark}
\newtheorem{example}[theorem]{Example}

\newcommand{\C}{\mathbb C}
\newcommand{\Z}{\mathbb Z}
\newcommand{\R}{\mathbb R}
\newcommand{\PP}{\mathbb P}

\newcommand{\Q}{\mathbb Q}
\newcommand{\F}{\mathcal{F}}
\newcommand{\OO}{\mathcal{O}}

\newcommand{\Ee}{\mathcal{E}}

\newcommand{\ZZ}{\mathcal{Z}}
\newcommand{\ttt}{\mathfrak{t}}

\newcommand{\V}{\mathcal V}
\newcommand{\w}{w}
\newcommand{\gl}{\mathfrak{gl}}
\newcommand{\ssl}{\mathfrak{sl}}

\newcommand{\he}{\mathfrak{h}}
\newcommand{\geg}{\mathfrak{g}}
\newcommand{\kek}{\mathfrak{k}}
\newcommand{\p}{\mathfrak{p}}
\newcommand{\nen}{\mathfrak{n}}
\newcommand{\lel}{\mathfrak{l}}
\newcommand{\Geg}{\mathcal{G}}
\newcommand{\ov}{\overline}
\newcommand{\W}{\mathrm W}

\newcommand{\eps}{\varepsilon}
\newcommand{\m}{\mathfrak{m}}

\newcommand{\bb}{\mathfrak b}
\newcommand{\uu}{\mathfrak u}
\newcommand{\I}{\mathcal I}
\newcommand{\Ss}{\mathcal S}
\newcommand{\into}{\hookrightarrow}
\newcommand{\Aa}{\mathcal A}
\newcommand{\Hs}{\mathrm H}
\newcommand{\Gs}{\mathrm G}
\newcommand{\Ts}{\mathrm T}
\newcommand{\Ks}{\mathrm K}
\newcommand{\Bs}{\mathrm B}
\newcommand{\Ps}{\mathrm P}
\newcommand{\Ws}{\mathrm W}
\newcommand{\Ls}{\mathrm L}
\newcommand{\Ns}{\mathrm N}
\newcommand{\Qs}{\mathrm Q}
\newcommand{\U}{\mathrm U}
\newcommand{\rk}{\operatorname{rk}}
\newcommand{\Cs}{\mathbb C^\times}
\newcommand{\Vect}{\mathrm{Vect}}
\renewcommand{\bf}{\bfseries}

\newcommand{\D}{\mathrm{D}}
\newcommand{\Tr}{\operatorname{Tr}}
\newcommand{\diag}{\operatorname{diag}}
\newcommand{\im}{\operatorname{im}}

\newcommand{\Hom}{\operatorname{Hom}}
\newcommand{\Spec}{\operatorname{Spec}}

\newcommand{\pt}{\operatorname{pt}}
\newcommand{\ch}{\operatorname{ch}}
\newcommand{\td}{\operatorname{td}}

\newcommand{\id}{\operatorname{id}}
\newcommand{\ad}{\operatorname{ad}}
\newcommand{\Ad}{\operatorname{Ad}}
\newcommand{\Lie}{\operatorname{Lie}}
\newcommand{\tr}{\operatorname{tr}}
\newcommand{\End}{\operatorname{End}}

\newcommand{\GL}{\operatorname{GL}}
\newcommand{\SL}{\operatorname{SL}}
\newcommand{\PSL}{\operatorname{PSL}}
\newcommand{\PGL}{\operatorname{PGL}}
\newcommand{\Sp}{\operatorname{Sp}}
\newcommand{\SO}{\operatorname{SO}}

\newcommand{\Span}{\operatorname{span}}

\newcommand{\coker}{\operatorname{coker}}
\newcommand{\Tor}{\operatorname{Tor}}

\newcommand{\Stab}{\operatorname{Stab}}
\newcommand{\reg}{\mathrm{reg}}
\newcommand{\tot}{\mathrm{tot}}
\newcommand{\Gr}{\operatorname{Gr}}
\newcommand{\Sym}{\operatorname{Sym}}
\newcommand{\Fix}{\operatorname{Fix}}
\newcommand{\depth}{\operatorname{depth}}
\newcommand{\red}{\operatorname{red}}

\usepackage{nomencl}
\makenomenclature

\begin{document}

\frontmatter 
\pagestyle{plain} 

\addtitlepage
\addcommitteepage
\addstatementpage

\begin{abstract}
This dissertation is the summary of the author's work, concerning the relations between cohomology rings of algebraic varieties and rings of functions on zero schemes and fixed point schemes. For most of the thesis, the focus is on smooth complex varieties with an action of a principally paired group, e.g. a parabolic subgroup of a reductive group. The fundamental theorem \ref{general} from co-authored article \cite{HR} says that if the principal nilpotent has a unique zero, then the zero scheme over the Kostant section is isomorphic to the spectrum of the equivariant cohomology ring, remembering the grading in terms of a $\Cs$ action. A similar statement is proved also for the $\Gs$-invariant functions on the total zero scheme over the whole Lie algebra. Additionally, we are able to prove an analogous result for the GKM spaces, which poses the question on a joint generalisation.

We also tackle the situation of a singular variety. As long as it is embedded in a smooth variety with regular action, we are able to study its cohomology as well by means of the zero scheme. In case of e.g. Schubert varieties this determines the cohomology ring completely. In largest generality, this allows us to see a significant part of the cohomology ring. 

We also show (Theorem \ref{nonreg}) that the cohomology ring of spherical varieties appears as the ring of functions on the zero scheme. The computational aspect is not easy, but one can hope that this can bring some concrete information about such cohomology rings. Lastly, the K-theory conjecture \ref{kthe} is studied, with some results attained for GKM spaces.

The thesis includes also an introduction to group actions on algebraic varieties. In particular, the vector fields associated to the actions are extensively studied. We also provide a version of the Kostant section for arbitrary principally paired group, which parametrises the regular orbits in the Lie algebra of an algebraic group. Before proving the main theorem, we also include a historical overview of the field. In particular we bring together the results of Akyildiz, Carrell and Lieberman on non-equivariant cohomology rings.

\end{abstract}

\begin{acknowledgements}
	
	First of all, I would like to thank my PhD supervisor Tam\'{a}s Hausel for his continuous help and invaluable assistance throughout the last years. I am grateful for the availability and for all the time he has dedicated to my PhD project. I cannot thank Tam\'{a}s enough for the rich literature suggestions and plenty of ideas for research directions. The constant support and encouragement I have received does not go unnoticed. I am also grateful for all the valuable comments and remarks on the content of this thesis.
	
	Second, I will be forever thankful to my friend, colleague and office mate Anna Sis\'{a}k, who has been an unwavering source of support, both mathematical and emotional, for the last five years. I will definitely miss all the good moments we have shared during our joint PhD journey.
	
	I would also like to express my gratitude to all the friends from ISTA 2019 PhD cohort, to everyone who has been a member of the Hausel group in the last years, and to the other members of ISTA community who made my PhD years better and more interesting. Thanks to all of you, working in ISTA and living in Austria has been a great adventure that I will keep in my memory.

	I am grateful to all my friends and family members whom I left in Poland or elsewhere, for the valuable advice, emotional support and assistance I have received in the last years. Due to the distance, staying in touch has not always been easy, which makes me appreciate all the efforts much more. In particular, I want to thank my parents, who have assisted me with many aspects of my life abroad and made it easier by always being there for me.
	
	I would additionally like to thank Michel Brion, Jim Carrell, Mischa Elkner, Jan Friedmann, Daniel Holmes, Quoc Ho, Jakub Koncki, Andreas Kretschmer, Jakub L\"{o}wit, Anton Mellit, Alexandre Minets, Leonid Monin, Anne Moreau, Shon Ngo, Rich\'{a}rd Rim\'{a}nyi, Shiyu Shen, Tanya Kaushal Srivastava, Andr\'{a}s Szenes, Gianluca Tasinato, Michael Thaddeus, Andrzej Weber and Jarosław Wiśniewski for their useful comments, discussions, mathematical support and teaching efforts. I am also grateful to the reviewers of \cite{HR} for their numerous valuable comments and remarks. Special thanks go to Mischa Elkner for a careful reading of this thesis and providing me with a long list of typos and language errors.
	
	The author was supported by the DOC Fellowship of the Austrian Academy of Sciences \emph{Topology of open smooth varieties with a torus action}.
\end{acknowledgements}

\begin{aboutauthor}

Kamil Rychlewicz completed BSc (licencjat) and MSc (magister) degrees in Mathematics at the University of Warsaw before joining ISTA in September 2019. His main research interests include group actions and equivariant cohomology theories in algebraic geometry, geometric representation theory and generalised cohomology theories. He has published his BSc thesis \cite{lic} on toric ideals in a post-conference volume and the MSc thesis \cite{mgr} on equivariant characteristic classes in Bulletin of London Mathematical Society. During the PhD, under the supervision of Tam\'{a}s Hausel, Kamil has worked on the relations between cohomology theories and rings of functions on fixed point and zero schemes. In September 2024, he will start a postdoctoral position at École Polytechnique Fédérale de Lausanne, under the mentorship of Dimitri Wyss.

Apart from the research work, he has been actively engaged in various endeavours related to mathematical olympiads. Since 2014 he has been a member of the Problem Selection Committee of Polish Juniors Mathematical Olympiad, and has been a team leader or a coordinator at Middle European Mathematical Olympiad and ``Baltic Way'' Mathematical Competition of Baltic Countries. Together with his colleagues, he has co-organised the Czech-Austrian-Polish-Slovak Mathematical Competition (CAPS) in Institute of Science and Technology Austria.

\end{aboutauthor}

\begin{publicationlist}

Significant parts of this thesis, mostly Chapter \ref{chapkos} and Chapter \ref{chapequivzero}, come from the co-authored article

\cite{HR} Tam\'{a}s Hausel and Kamil Rychlewicz, \emph{Spectrum of equivariant cohomology as a fixed point scheme}, arXiv:2212.11836,

currently submitted and in a review process.

Apart from this contribution, no other part of this thesis has been yet published nor included in a preprint.

\end{publicationlist}

\cleardoublepage
\renewcommand{\contentsname}{Table of Contents} 
\tableofcontents
\listoffigures

\begingroup 
\let\newpage\relax
\let\clearpage\relax
\let\cleardoublepage\relax

\endgroup 

\mainmatter 
\pagestyle{Ruled} 

\chapter{Introduction}
There is a long history of results which connect global topological invariants of manifold with some local data. One famous theorem of this kind is the Poincar\'{e}--Hopf theorem, which relates the Euler characteristic of the manifold to the indices of zeros of a vector field. Another important field following this idea is Morse theory, where the Betti numbers are determined from the critical points of a Morse function. Some more refined invariants can be determined this way if one restricts their attention to a particular subfamily of manifolds. For complex analytic manifolds, Bott \cite{bott} proves his residue formula, which computes the Chern numbers. If we restrict further, to algebraic varieties, Białynicki-Birula \cite{BB1,BB2} shows the decomposition into affine bundles, which also allows to find the integral homology groups of the variety, and hence also the additive structure of cohomology.

A natural question one could ask is whether in addition to such quantitative invariants, one can determine the ring structure of cohomology. This was first done by Carrell and Lieberman in their seminal paper \cite{CL}. They show that if a holomorphic vector field on a compact K\"{a}hler manifold $X$ has isolated zeros, then:
\begin{enumerate}
\item the manifold is actually algebraic;
\item its odd cohomology vanishes;
\item the coordinate ring of the zero scheme of the vector field has a filtration whose associated graded ring is $H^*(X,\C)$.
\end{enumerate}

In \cite{AC} Akyildiz and Carrell further show that if the Borel of $\SL_2$ acts on $X$ and the vector field is the one defined by
$$e = \begin{pmatrix}
0 & 1 \\
0 & 0
\end{pmatrix},$$ 
then the filtration is determined by the torus action and moreover, it actually comes from a grading, so that we see $H^*(X,\C)$ as the coordinate ring of the zero scheme. This allows one to see geometrically the cohomology ring of flag varieties or smooth Schubert varieties. Later, Akyildiz, Carrell and Lieberman managed to produce results that apply to a limited class of singular varieties, which includes in particular all Schubert varieties. This required considering a deformed zero scheme, which then in \cite{BC} was proven to be the spectrum of $\Cs$-equivariant cohomology. This is defined over the base $e + \ttt$, consisting of all the matrices
$$\begin{pmatrix}
t & 1 \\
0 & -t
\end{pmatrix}$$
in $\SL_2$. In fact, one also notices, as in Theorem \ref{etkost}, that the line $e+\ttt$ has the property analogous to the Kostant section in semisimple groups \cite{Kostsec}, i.e. every regular element of the Lie algebra $\Bs(\SL_2)$ is conjugate to a unique element of $e+\ttt$. Here an element of the Lie algebra is called regular when its centraliser is of minimal dimension, equal to the dimension of the maximal torus.

In the meantime, in the recent work \cite{hausel-hmsec} a certain infinitesimal fixed point scheme for the action of $\GL_n$ on ${\mathrm{Gr}}(k,n)$ -- the Grassmannian of $k$-planes in $\C^n$ -- was used to model the Hitchin map on a particular minuscule upward flow in the $\GL_n$-Higgs moduli space. In turn, it was noticed that this fixed point scheme is isomorphic to the spectrum of equivariant cohomology of ${\mathrm{Gr}}(k,n)$, and thus the Hitchin system on these minuscule upward flows can be modelled as the spectrum of equivariant cohomology of Grassmannians. Motivated by that, we have shown in \cite{HR} that the appearance of the spectrum of equivariant cohomology as a fixed point scheme is not a coincidence, and holds in more general situations. This provides a generalisation of the result of Brion and Carrell. The one-dimensional torus $\Cs$ is replaced by an arbitrary \emph{principally paired} group. This includes in particular all the reductive groups and their parabolic subgroups. Hence, one can work with the groups of higher rank, but also not necessarily tori.

\begin{definitionc}
A complex linear algebraic group $\Hs$ is {\em principally paired} if it contains a pair $\{e,h\}\subset \he$ in its Lie algebra, such that $[h,e]=2e$ and $e$ is a regular nilpotent, and an algebraic group homomorphism $\Bs(\SL_2)\to \Hs$ integrates the Lie subalgebra of $\he$ generated by $e$ and $h$.
\end{definitionc}

By Jacobson--Morozov theorem, reductive groups satisfy the condition, and it also stays true if we restrict to a Borel subgroup. Then it also holds for all parabolic subgroups.

We then prove in Corollary \ref{kostprincpair} that in this generality one can also construct the equivalent $\Ss$ of the Kostant section. This is an affine subspace consisting of regular elements of the Lie algebra $\he$. It has the property that every regular element of $\he$ is conjugate to exactly one element of of $\Ss$. We prove it by first tackling the solvable group case and then considering the Weyl group action for a general principally paired group.

Moreover, the restriction map $\C[\he]^\Hs \to \C[\Ss]$ is an isomorphism, and by Chevalley's restriction theorem $\C[\he]^\Hs = \C[\ttt]^\Ws$ is the $\Hs$-equivariant cohomology ring of the point. Therefore the spectrum of any $\Hs$-equivariant cohomology ring will be a scheme over $\Ss$. This suggests that in order to see $\Spec H^*_\Hs(X,\C)$ as a zero scheme, we might want to consider a zero scheme parametrised by $\Ss$. As a subscheme of $\Ss\times X$, it will be projective over $\Ss$, and hence to ensure that it is affine, we need it to be finite over $\Ss$.

\begin{definitionc}
An action of a principally paired group $\Hs$ on a smooth projective variety $X$ is \emph{regular} when a regular nilpotent element $e\in\he$ has finitely many fixed points. 
\end{definitionc}

In fact a unipotent element always has a connected fixed point set \cite{Horrocks}, so for a regular action we have $X^u=\{ o\}$ for some $o\in X$. 

Examples of $\Hs$-regular varieties include for $\Hs=\Gs$ reductive the partial flag varieties $\Gs/\Ps$ considered above (see \cite{Akyildiz}). Smooth Schubert varieties are regular when $\Hs=\Bs\subset \Gs$ is a Borel subgroup and Bott--Samelson resolutions will be examples for parabolic subgroups $\Hs=\Ps\subset\Gs$ of reductive groups.

We then construct in Section \ref{secgrpac} a vector field $V_\he$ on $\he\times X$ such that for any $y\in \he$ its restriction 
$$(V_\he)_y\in H^0(X;T_X)$$
to $\{y\}\times X$ is the infinitesimal vector field on $X$ generated by $y$. We restrict this vector field to a vector field $V_\Ss$ on $\Ss\times X$ and consider its zero scheme $\ZZ_\Ss\subset \Ss\times X$. The embedding of the Borel subgroup of $\SL_2$ provides a particular one-dimensional torus in the maximal torus $\Ts$ of $\Hs$. It will then act on $\ZZ_\Ss$ and thus make the ring of functions on $\ZZ_\Ss$ into a graded ring, by the weights of the induced action. We then prove the main result

\begin{theoremc}\label{main} Suppose a principally paired group $\Hs$ acts regularly on a smooth projective complex variety $X$. Then the zero scheme $\ZZ_\Ss\subset \Ss\times X$ of the vector field $V_\Ss$ is reduced and affine and its coordinate ring, graded by the $\Cs$-action, is isomorphic as a graded ring
$$
\begin{tikzcd}
  \C[\ZZ_{\Ss}]\arrow{r}{\cong} &
  H_\Hs^*(X;\C)   \\
   \C[\Ss]\arrow{r}{\cong} \arrow{u}{\pi^*}&
  H^*_\Hs \arrow{u}
\end{tikzcd}
$$
to the $\Hs$-equivariant cohomology of $X$, such that the structure map $H^*_\Hs\to H_\Hs^*(X;\C)$ agrees with the pullback map $H^*_\Hs\cong \C[\Ss]\to \C[\ZZ_{\Ss}]$ of the natural projection $\pi:\ZZ_\Ss \to \Ss$. In particular, $$ \begin{tikzcd}
 \ZZ_\Ss  \arrow{d}{\pi}  &
  \Spec(H_\Hs^*(X;\C)) \arrow{d} \arrow{l}{\cong} \\
   \Ss&
  \Spec(H^*_\Hs) \arrow{l}{\cong},
\end{tikzcd}$$
i.e. the spectrum of equivariant cohomology of $X$ is $\Cs$-equivariantly isomorphic to the zero scheme $\ZZ_\Ss\subset \Ss\times X$ over $\Ss\cong \Spec(H^*_\Hs)$.
\end{theoremc}

Again, we first study the case of solvable principally paired groups. Then the general case is reduced to the Borel subgroup, and we realise the Weyl group action on the zero scheme geometrically.

There is another version of Theorem~\ref{main} where we do not restrict to the Kostant section $\Ss$. Namely, if a reductive group $\Gs$ acts regularly on $X$ and we denote by $\ZZ_{\geg}\subset \geg\times X$ the zero scheme of $V_\geg$, then the $\Gs$-action on $\geg\times X$ leaves $\ZZ_{\geg}$ invariant.
We have the following:
\begin{theoremc}
\label{maintot}
Suppose a complex reductive group $\Gs$ acts regularly on a smooth projective complex variety $X$. Then the $\Gs$-invariant part of the algebra of the global functions on the total zero scheme $\ZZ_\geg$ $$ \begin{tikzcd}
  \C[\ZZ_{\geg}]^\Gs \arrow{r}{\cong} &
  H_\Gs^*(X;\C)  \\
  \C[\geg]^\Gs \arrow{r}{\cong} \arrow{u}&
  H^*_\Gs \arrow{u}
\end{tikzcd}$$ is graded isomorphic with the equivariant cohomology of $X$ over $\C[\geg]^\Gs\cong H^*_\Gs$. The gradings on $\C[\geg]^\Gs$ and $\C[\ZZ_{\geg}]^\Gs$ are induced from the weight $-2$ action of $\Cs$ on $\geg$ and the trivial action on $X$. 
 \end{theoremc}
 
Note that for partial flag varieties $X=\Gs/\Ps$ the total zero scheme $\ZZ_{\geg}\cong \tilde{\geg}_\Ps\to \geg$ is just the Grothendieck--Springer resolution. However, the total zero scheme is no longer affine. On the other hand this version also holds for GKM spaces, including toric varieties. Recall \cite{GKM} that a smooth projective variety $X$ with an action of a torus $\Ts$ is a GKM space if the number of both the zero- and one-dimensional orbits is finite. We can form the total zero scheme $\ZZ_\ttt\subset \ttt\times X$ as the zero scheme of the vector field $V_\ttt$ generated by the $\Ts$-action, as before. We prove that the ring of functions on such a scheme is again the $\Ts$-equivariant cohomology of $X$.
 
The proof is straightforward, using the explicit description of $ H_\Ts^*(X;\C)$ from \cite{GKM}. We expect this version to hold for an even larger class of group actions, including spherical varieties. However, in this thesis we concentrate on the more restrictive class of regular group actions. In that case, as in Theorem~\ref{main}, we can find an affine zero scheme $\ZZ_\Ss \subset \Ss\times X$, which is precisely the spectrum of equivariant cohomology of $X$.
 
The proofs of the main results are broadly following the approach of the proof in \cite{BC}. For this, we have to first study the structure of $\ZZ$, i.e. we prove it is a reduced Cohen--Macaulay scheme whose points can be described via the torus-fixed points. Then the isomorphism $H^*_\Gs(X)\to \C[\ZZ]$ is defined by localisation to the fixed points. Using similar methods, we are also able to prove that even under milder assumptions, we can see the cohomology ring as the ring of functions.

\begin{theoremc}\label{mainnonreg}
Assume that a principally paired group $\Hs$ acts on a smooth projective variety $X$. Let $\ZZ \subset \Ss\times X$ be the zero scheme as above. Assume that $\dim \ZZ = \dim \Ss$. Then there is an isomorphism
$$ H^*_\Hs(X,\C) \to \C[\ZZ] $$
of graded $\C[\Ss]\simeq H^*_\Hs(\pt,\C)$-algebras. Moreover, $H^i(\ZZ,\OO_\ZZ) = 0$ for any $i>0$. 
\end{theoremc}

This clearly covers the case of a regular action, although the conclusion is weaker: the scheme $\ZZ$ is not affine anymore. This in principle makes computations less feasible. However, the theorem covers many interesting examples, including all varieties with finitely many orbits. A particularly important class of such consists of \emph{spherical varieties} \cite{BrionSph}. Their cohomology is not very well understood and one could hope that this approach would allow for a better insight.

One should also mention that there are other papers in the literature which study the spectrum of the equivariant cohomology geometrically, see e.g. \cite{goresky-macpherson} and the references therein. A more recent example is \cite{hikita}, where the spectrum of equivariant cohomology of certain varieties also appears as a fixed point scheme, albeit of another -- 3D-mirror -- variety. This is also a common theme when one studies elliptic cohomology -- i.e. one sees a whole elliptic curve instead of the functions on it, and only recovers the cohomology ring by considering sections of line bundles. We would hope that our results can be translated to elliptic cohomology as well. It is however not obvious, what the correct analog of the zero scheme should be. For now, we are able to formulate the analogous conjecture \ref{kthe} in K-theory, where we replace the zero schemes of vector fields with the fixed point schemes of the group action. We also provide some evidence for the conjecture

In addition, we extend the results of \cite{BC} for singular varieties. It turns out that one can see either the whole cohomology ring, or the subalgebra generated by equivariant Chern classes, via the rings of functions on zero schemes.

This thesis is organised as follows. Chapter 2 describes the necessary background material from homological algebra, algebraic geometry and algebraic topology. It starts with the Koszul resolutions and their applications in geometry. Then we describe the theory of algebraic groups and group actions. Due to the character of the results, the emphasis is put on the Lie algebra and the vector fields defined by the Lie algebra elements. To the author's knowledge, there is no comprehensive exposition of the topic in the existing literature. We also recall the definition of equivariant cohomology and its basic properties, and the relation with linearised vector bundles.

Chapter 3 is based on Section 2 of \cite{HR}. We introduce the notion of a principally paired group and develop the theory of Kostant sections in this generality. We also prove that the regular action implies finite zero sets for any regular element of the Lie algebra. Chapter 4 is a historical interlude, based on the results of Akyildiz, Brion, Carrell, Lieberman, Sommese \cite{AC, BC, CL}. We review there the results which allow one to see the ordinary (i.e. non-equivariant) cohomology of a smooth projective variety in terms of zero schemes of vector fields.

Then, in Chapter 5 we include the main results of \cite{HR}, i.e. the Sections 3, 4, 5 ibid. In particular, we prove Theorem \ref{main}. We start from solvable groups and afterwards prove it in the generality of principally paired groups. Then the result is extended to the singular varieties, in the spirit of \cite{BC}. We also prove the version for total zero scheme, i.e. Theorem \ref{maintot} and a surprising analogous result for GKM spaces.

In Chapter 6, we then describe additional results of the author that extend the results of \cite{HR}. First, we tackle the case of wilder singularities than the results of \cite{HR} allow. Second, we relax the assumptions on regularity of the action, as in Theorem \ref{mainnonreg}, and open the field for further exploration, to determine the appropriate conditions under which the results hold. In particular, we show that the equivariant cohomology ring of spherical varieties shows up as the ring of functions on the zero scheme. Then we state the K-theory conjecture and  prove that it holds for GKM varieties.

Throughout this thesis, we adopt the convention where we try to stay in the category of algebraic schemes whenever possible, i.e. not using the topological and analytic tools if they are not necessary. Therefore many results, in particular on Lie algebras of algebraic groups and associated vector fields, might have slightly longer proofs than one would expect. However, as the literature in the field is not very extensive, we find it instructive to prove as much as possible without resorting to non-algebraic methods. This requires translation of cohomology, which is a topologically defined notion, into the algebraic world. At the heart of this lies the result of Carrell--Lieberman, i.e. Theorem \ref{clthm}, which uses non-algebraic methods.

\chapter{Background material}

Throughout this thesis, we are working with algebraic groups acting on algebraic varieties, and with cohomology theories. We therefore need to recall many results we are going to use. We first explain the role of Koszul complexes in algebraic geometry and their cohomological properties. We recall the basic facts about vector fields in algebraic setting, and then we couple that with algebraic groups and Lie algebras, and their actions on algebraic varieties. As a particular case we have the actions of an algebraic torus, which yield the Białynicki-Birula decomposition of the variety. Then we recall the definition of equivariant cohomology and the equivariant formality as defined in \cite{GKM}. We finish with an overview of equivariant vector bundles and their Chern classes. For completeness, we attach the proof of the graded Nakayama lemma, which, though very simple, is not easily found in literature.

\section{Notation}

An algebraic variety is an integral (i.e. irreducible and reduced) separated scheme of finite type over a field.

Unless otherwise specified, all the groups and varieties will be assumed to be defined over $\C$. All rings are by definition assumed to be commutative and with unity. For a scheme $X$ defined over a field $k$, by $k[X] = \OO_X(X)$ we denote the algebra of regular functions on $X$. All the cohomology groups will be understood to have complex coefficients, unless otherwise specified. For a Lie algebra $\geg$ and a subset $V\subset \geg$ we denote by $C_\geg(V)$, $N_\geg(V)$ the centraliser and normaliser of $V$ in $\geg$, respectively. If $V = \{v\}$, then we also write $C_\geg(v)$, $N_\geg(v)$. Analogously by $C_\Gs(V)$, $N_\Gs(V)$ we denote a centraliser and normaliser of a subset $V$ in a group $\Gs$. For any $\Z_{\ge 0}$-graded vector space $R = \bigoplus_{n=0}^\infty R_n$ over a field $k$, with $\dim_k(R_n)$ finite for every $n$, we denote by $P_R(t)$ its Poincar\'{e} series, i.e.
$$P_R(t) = \sum_{n=0}^\infty \dim_k(R_n) t^n.$$
For any algebra $R$ with a filtration $F_\bullet$, by $Gr_F(R)$ we denote the associated graded algebra.
Let $\diag(v_1,v_2,\dots,v_n)$ be the diagonal $n\times n$ matrix with diagonal entries $v_1$, $v_2$, \dots, $v_n$. We will denote by $I_n = \diag(1,1,\dots,1)$ the $n\times n$ identity matrix. 
Whenever we write $x\in X$ for $X$ being a variety, or $g\in \Gs$ for $\Gs$ being an algebraic group, we mean $x$ or $g$ being a closed point.

\section{Regular sequences, Cohen--Macaulay rings and Koszul complexes}
The most important sources for this section are \cite{BrHer} for Cohen--Macaulay rings, as well as \cite{24h} for Koszul complexes and their applications in local cohomology. Throughout this section, $R$ will denote a commutative ring. We recall a few definitions.

\subsection{Cohen--Macaulay and complete intersection rings} \label{seccm}

\begin{definition}
A sequence $(r_1,r_2,\dots,r_k)$ of elements of $R$ is called \emph{regular} if the following two conditions hold:
\begin{enumerate}
\item for $i=1,2,\dots,k$, the image of $r_i$ in $R/(r_1,r_2,\dots,r_{i-1})$ is not a zero divisor;
\item the ring $R/(r_1,r_2,\dots,r_k)$ is nonzero.
\end{enumerate}
\end{definition}

Geometrically, one should think of regular sequence as a sequence of functions cutting out a subscheme of codimension equal to its length. Indeed, by Krull's Hauptidealsatz \cite[Corollary 11.17]{AtMac} any component of a scheme cut out by a regular non-unit is of codimension one. Locally, if $R$ is a local Cohen--Macaulay ring (see Definition \ref{defCM}), then a sequence $(r_1,r_2,\dots,r_k)$ is regular if and only if $\dim R/(r_1,r_2,\dots,r_k) = \dim R - k$ \cite[Theorem 2.1.2(c)]{BrHer}.

Obviously, by definition, any prefix of a regular sequence is regular. However, in this generality, the regularity of a sequence depends on its order. A typical example that illustrates this phenomenon is the sequence $(x-1,xy,xz)$ in $R = \C[x,y,z]$. It is clearly regular, however its permutation $(xy,xz,x-1)$ is not regular, as $xz$ is a zero divisor in $R/(xy)$. Regularity does not depend on the order for a Noetherian local ring \cite[Proposition 1.1.6]{BrHer}. We will use below (Lemma \ref{lemcm}) that also in some graded situations the regularity does not depend on the ordering. We first note the following property of Poincar\'{e} series with respect to quotients by regular series.

\begin{proposition}\label{poinreg}
Assume that $R = \bigoplus_{i=0}^\infty R_n$ is a graded ring over a field $k$, and $\dim_k R_n < \infty$ for every $n\ge 0$. Let $r_1$, $r_2$, \dots, $r_k$ be a regular sequence of homogeneous elements of degrees $d_1$, $d_2$, \dots, $d_k$, respectively. Then
$$ P_{R/(r_1,r_2,\dots,r_k)}(t)  = \prod_{i=1}^k (1-t^{d_i}) \cdot P_R(t).$$
\end{proposition}

\begin{proof}
By induction it is enough to prove the statement for $k=1$. Assume then that $r$ is a regular element, i.e. not a zero divisor, of $R$, homogeneous of degree $d$. Then the graded $R$-module $r R$ is isomorphic to $R$ as a module, but its grading is shifted by $d$. Therefore $P_{r R}(t) = t^d P_R(t)$. Hence 
\[P_{R/rR} (t) = P_R(t) - P_{rR}(t) = (1-t^d) P_R(t). \qedhere \]
\end{proof}

\begin{definition}\label{defCM}
A Noetherian local ring $(R,\m)$ is called \emph{Cohen--Macaulay} if its depth $\depth R$, i.e. the maximal length of a regular sequence, equals its dimension $\dim R$, i.e. the maximal length of a chain of prime ideals.
\end{definition}

Note that for any Noetherian local ring $R$ we have $\depth R \le \dim R$ \cite[Proposition 1.2.12]{BrHer}. Any regular local ring is Cohen--Macaulay, but the converse is not true, e.g. the ring $\C[x]/(x^2)$ is Cohen--Macaulay, but not regular.

\begin{definition}
An arbitrary Noetherian ring $R$ is called \emph{Cohen--Macaulay} if all its localisations to maximal ideals are Cohen--Macaulay.

Similarly, a locally Noetherian scheme is called \emph{Cohen--Macaulay} if all its localisations are Cohen--Macaulay.
\end{definition}
By \cite[Theorem 2.1.3]{BrHer} all localisations of a Cohen--Macaulay ring are Cohen--Macaulay, hence both cases of above definition agree for affine schemes, i.e. $\Spec R$ is a Cohen--Macaulay scheme if and only if $R$ is a Cohen--Macaulay ring. One also easily notices that if $R$ is Cohen--Macaulay and $(r_1,r_2,\dots,r_k)$ is a regular sequence in $R$, then $R/(r_1,r_2,\dots,r_k)$ is Cohen--Macaulay.

We also mention a more general notion of a \emph{complete intersection ring}.

\begin{definition}
A Noetherian local ring $(R,\m)$ is called a \emph{complete intersection ring} if its completion $\hat{R}$ in $\m$ is a quotient of a regular local ring by an ideal generated by a regular sequence.

We call any Noetherian ring \emph{locally complete intersection} if all of its localisations are complete intersection rings. 
\end{definition}
A complete intersection ring, and more generally a locally complete intersection ring is always Cohen--Macaulay \cite[2.3]{BrHer}. We say that a finite type $k$-algebra $R$ is \emph{complete intersection} if there is a presentation $R \simeq k[x_1,x_2,\dots,x_n]/(f_1,f_2,\dots,f_r)$ where $\dim R = n-c$. Any complete intersection algebra is a local complete intersection, and hence also Cohen--Macaulay. The complete intersection condition is easier to picture geometrically -- it means that one is able to cut out the spectrum of given algebra from an affine space with the minimal possible number of equations.

\subsection{Koszul complexes in rings}
Let $R$ be a Noetherian ring and $r_1$, $r_2$, \dots, $r_n$ a sequence of elements of $R$. For $i = 1,2,\dots,n$ we consider the chain complex
$$0\to R \xrightarrow{r_i} R \to 0.$$
Note that if $r_i$ is a regular element, then this is a free resolution of the cyclic $R$-module $R/(r_i)$. Consider the tensor product of all such complexes for $i=1,2,\dots, n$. Explicitly, it is the complex

$$0\to \bigwedge^n R^n \xrightarrow{\phi_n} \bigwedge^{n-1} R^n \xrightarrow{\phi_{n-1}} \bigwedge^{n-2} R^n \xrightarrow{\phi_{n-2}} \dots \xrightarrow{\phi_3} \bigwedge^2 R^n \xrightarrow{\phi_2} \bigwedge^1 R^n \xrightarrow{\phi_1} \bigwedge^0 R^n\to 0,$$
where the map $\phi_k:\bigwedge^k R^n \to \bigwedge^{k-1} R^n$ is defined on the basis as
$$\phi_k(e_{i_1}\wedge e_{i_2}\wedge\dots\wedge e_{i_k})
 = \sum_{j=1}^k (-1)^{j+1} r_j \cdot e_{i_1}\wedge\dots\wedge \hat{e}_{i_j}\wedge\dots\wedge e_{i_k}. $$
\begin{definition}
We call this complex the \emph{Koszul complex} associated to the sequence $(r_1,r_2,\dots,r_n)$, and we denote it by $K^*_{r_1,\dots,r_n}$.
\end{definition}
When speaking of Koszul complexes, we will consider them with negative cohomological grading. In other words, we let $K^{-k}_{r_1,\dots,r_n} = \bigwedge^k R^n$ to be of the degree $-k$, so that the differential increases the grading by $1$. Note that the $0$-th cohomology of the complex is the quotient $R/(r_1,r_2,\dots,r_n)$.

\begin{proposition}
\label{koszul}
If $(r_1,r_2,\dots,r_n)$ is a regular sequence of regular elements in $R$, then the Koszul complex $K^*_{r_1,r_2,\dots,r_n}$ is a free resolution of $R/(r_1,r_2,\dots,r_n)$.
\end{proposition}

\begin{remark}
 As reordering gives an isomorphic Koszul complex, it will be a free resolution if we only assume that $(r_1,r_2,\dots,r_n)$ is a permutation of a regular sequence. Thus there is no implication in the other direction.
 
For simplicity we assume that all the elements $r_i$ are regular in $R$, as that is the only case we need. However this assumption is not necessary, see e.g. \cite[Corollary 1.6.14]{BrHer}.
\end{remark}

\begin{proof}
As noticed above, the chain complex
$$0\to R \xrightarrow{r_i} R \to 0$$
is a resolution of $R/(r_i)$ for $i=1,2,\dots,n$. We just have to prove that the cohomology of their tensor product (for $i=1,2,\dots,n$) is the tensor product of cohomology.
 
We proceed by induction and prove that for each $k$ the Koszul complex $K^*_{r_1,\dots,r_k}$ has only cohomology $R/(r_1,r_2,\dots,r_k)$ in degree $0$. It is obvious for $k=0$ or $k=1$. Let us now assume that it holds for some $k\in\{1,2,\dots,n-1\}$. As
$$K^*_{r_1,r_2,\dots,r_{k+1}} = K^*_{r_1,r_2,\dots,r_{k}} \otimes K^*_{r_{k+1}},$$
there is a K\"{u}nneth spectral sequence \cite[Theorem 10.90]{Rotman}
$$E_{pq}^2 = \bigoplus_{p_1+p_2=p} \Tor_{-q}(H^{p_1}(K^*_{r_1,r_2,\dots,r_{k}}),H^{p_2}(K^*_{r_{k+1}}))\implies H^{p+q}(K^*_{r_1,r_2,\dots,r_{k+1}}).$$
By the inductive assumption, the entries may be nonzero only for $p_1=p_2=0$ only, which gives 
 $$H^{q}(K^*_{r_1,r_2,\dots,r_{k+1}})=\Tor_{-q}(R/(r_1,\dots,r_k),R/r_{k+1}).$$
But we know the (two-term) resolution of $R/r_{k+1}$ explicitly. Therefore, as $r_{k+1}$ is regular in $R/(r_1,\dots,r_k)$, we see after tensoring with $R/(r_1,\dots,r_k)$ that the cohomology of
$$0\to R/(r_1,\dots,r_k) \xrightarrow{r_i} R/(r_1,\dots,r_k) \to 0$$
is only $R/(r_1,\dots,r_{k+1})$ in degree 0, hence the highers $\Tor$'s vanish.
\end{proof}

\subsection{Koszul resolutions on algebraic varieties}
As we have seen above, the Koszul complex of a regular sequence is a free resolution of the quotient module $R/(r_1,r_2,\dots,r_n)$. It can be seen as a complex of free sheaves over $\Spec R$. It then resolves the zero scheme of the global functions on $\Spec R$ defined by $r_1$, $r_2$, \dots, $r_n$. Equivalently, it is the zero scheme of the corresponding section of a trivial rank $n$ bundle. As acyclicity of a complex is a local property, we can generalise this to obtain resolutions of zeros of arbitrary vector bundles.

\begin{theorem} \label{thmkosz}
 Let $X$ be a Cohen--Macaulay locally Noetherian scheme and let $\Ee$ be a vector bundle on $X$ of rank $n$. Assume that $V$ is a global section of $\Ee$ whose zero scheme $\ZZ$ is of pure codimension $n$ in $X$. Then the Koszul complex
 $$0 \to \Lambda^n \Ee^* \to \Lambda^{n-1} \Ee^* \to \dots \to \Ee^* \to \OO_X \to 0$$
defined by the contraction along $V$ is a locally free resolution of $\OO_{\ZZ}$.
\end{theorem}

Note that one can define the zero scheme of a section of a vector bundle by locally trivializing the bundle and considering the zero scheme of the ideal generated by the coordinates  (as we do in Lemma \ref{lemcm}). One then needs to prove that the definition is independent of the trivialization.

Another, coordinate-free way to define the zero scheme is the following.

\begin{definition}\label{defzeros}
 Let $V$ be a global section of a vector bundle $\Ee$ of rank $n$ on a scheme $X$. Then we define the zero scheme $\ZZ$ of $V$ as the closed subscheme defined by the image of the contraction $\Ee^* \xrightarrow{\iota_V} \OO_X.$
\end{definition}

One then sees immediately that, by definition, the $0$-th cohomology of the Koszul complex is $\OO_\ZZ$, regardless of the codimension of $\ZZ$. It is only the acyclicity in other degrees which we need to check.

\begin{proof}[Proof of Theorem \ref{thmkosz}]

It is enough to prove the statement upon localisation to any closed point of $X$ by \cite[Chapter 1, Lemma 2.12]{Liu}. Let then $x\in X$ be a closed point. Let us fix an isomorphism $\Ee|_{X,x} \simeq \OO_{X,x}^n$, where by $\Ee|_{X,x}$ we mean the localisation to $\Spec \OO_{X,x}$. Then it also gives an isomorphism $\Ee^*|_{X,x}\simeq \OO_{X,x}^n$ and the contraction $\Ee^*|_{X,x} \to \OO_{X,x}$ is then a map of free $\OO_{X,x}$-modules defined by $n$ elements $r_1, r_2, \dots, r_n\in \OO_{X,x}$. The complex above, upon localisation to $\Spec \OO_{X,x}$, becomes the Koszul complex of the sequence $(r_1,r_2,\dots,r_n)$ in $\OO_{X,x}$.

We consider two cases: $x\in \ZZ$ and $x\not\in \ZZ$. Assume first that $x\in \ZZ$. This means that $\OO_{X,x}/(r_1,r_2,\dots,r_n) = \OO_{\ZZ,x}$ is of codimension $n$ in $\OO_{X,x}$. Therefore by \cite[Theorem 2.1.2(c)]{BrHer} the sequence $r_1$, $r_2$, \dots, $r_n$ is regular in $\OO_{X,x}$. This means that the Koszul complex is a resolution of the quotient.

Now assume that $x\not\in\ZZ$. This means that $(r_1,r_2,\dots,r_n) = \OO_{X,x}$ as an ideal in $\OO_{X,x}$. Then by \cite[Theorem 1.6.17]{BrHer}, the local Koszul complex is acyclic at all entries. 
\end{proof}

If the assumptions of Theorem \ref{thmkosz} are satisfied, we are therefore able to provide resolutions of $\OO_Z$, which in some cases might be a useful tool for computing derived functors. We will see below (Theorem \ref{clthm}) how this method arises in the work of Carrell and Lieberman \cite{CL0}. We later reuse it in the proof of Theorem \ref{nonreg}. In particular, the following spectral sequence will be useful.

\begin{corollary}\label{corkosz}
Under assumptions of Theorem \ref{thmkosz}, there is a (cohomological) spectral sequence with the first page
 $$
	E_1^{pq} = H^q(X,\Lambda^{-p}\Ee^*)
 $$
convergent to $H^{p+q}(X,\OO_\ZZ)$.
\end{corollary}

\begin{proof}
This is a particular case of a spectral sequence from \cite[Remark 2.67]{Huy}, here used to compute the hypercohomology $H^q = R\Gamma$ of the Koszul complex. As the Koszul complex resolves $\OO_\ZZ$, it is isomorphic to it in the derived category. Hence we de facto compute cohomologies of $\OO_\ZZ$.
\end{proof}

\section{Vector fields and derivations}

A \emph{vector field} on a smooth algebraic variety $X$ is a derivation on the sheaf of regular functions on $X$. This means that for any Zariski-open subset $U\subset X$ we are given a $\C$-linear derivation $\OO_X(U) \to \OO_X(U)$ and it is natural with respect to $U$. If we restrict to an affine neighbourhood $\Spec R$, this gives a derivation $R\to R$, which then restricted to any maximal ideal $\m\triangleleft R$ yields a $\C$-linear map $\m \to R/\m \simeq \C$ vanishing on $\m^2$. Hence if $x$ is the closed point corresponding to $\m$, it defines a tangent vector $V_x\in T_{x,X}$.

\begin{definition} Let $V$ be a vector field on a smooth variety $X$. For each open set $U\subset X$ it defines a derivation $D_V^U: \OO_X(U) \to \OO_X(U)$. Let us consider the ideal sheaf generated by the image $D_V(\OO_X) \subset \OO_X$. This is the defining ideal of the \emph{zero scheme} $\ZZ_V$ of $V$ on $X$.
\end{definition}

By the definition of the cotangent bundle, the ideal sheaf generated by $D_V(\OO_X)$ is the same as the image of the contraction map $\Omega^1_X\xrightarrow{\iota_V} \OO_X$. Hence our definition of the zero scheme agrees with the previous Definition \ref{defzeros} of the zero scheme of an arbitrary vector bundle.

\begin{remark}\label{remtan}
One can also view vector fields on smooth varieties as sections of the tangent bundle. As the tangent bundle is a locally free sheaf, we can define the zero scheme of the vector field by considering it locally as a tuple of regular functions (see Lemma \ref{lemcm}). In other words, if the tangent bundle is free over an open subset $U \subset X$, after choosing a trivialisation, its section $V$ is defined by $n$-tuple of regular functions $f_1$, $f_2$,\dots, $f_n$. Then the zero scheme of $V$ on $U$ is the zero scheme of the ideal $(f_1,f_2,\dots,f_n)\in \OO_X(U)$.
\end{remark}

For any two vector fields $V$, $W$ their Lie bracket $[V,W]$ is a vector field, as the Lie bracket of two derivations is again a derivation. We denote by $\Vect(X)$ the Lie algebra of all the global vector fields on a smooth variety $X$. The technical lemmas below will be of use when we consider the vector fields defined by elements of Lie algebras.

\begin{lemma} \label{algder}
Let $A$ be a commutative $\C$-algebra. Let $D_Y:A\to A$ be a $\C$-linear derivation and $\V$ a $\C$-vector space of $\C$-derivations $A\to A$ normalised by $D_Y$, i.e. for any $D_W\in \V$ we have $[D_W,D_Y]\in \V$. Denote by $\I = (\im D_W)_{D_W\in \V}$ the ideal generated by all the images of the derivations from $\V$. Then $D_Y (\sqrt{\I}) \subset \sqrt{\I}$.
\end{lemma}

\begin{proof}
We first prove that $D_Y(\I) \subset \I$. As $\I$ is generated by the images of derivations from $\V$ and $D_Y$ is $\C$-linear, it is enough to prove that $D_Y(f\cdot D_W(g)) \in \I$ for any $f,g\in A$ and $D_W\in \V$. By the Leibniz rule
$$D_Y(f\cdot D_W(g)) = D_Y(f) \cdot D_W(g) + f\cdot D_Y D_W(g)$$
and $D_W(g)\in \I$, hence it is enough to prove that $D_Y D_W(g) \in \I$. But by assumption we have $D_Z = D_Y D_W - D_W D_Y \in \V$. Hence
$$D_Y D_W (g) = D_Z (g) + D_W D_Y(g.)$$
Thus the images of both $D_Z$ and $D_W$ are in $\I$, hence $D_Y D_W(g)\in \I$.

Now assume that $f\in \sqrt{\I}$ so that $f^k \subset \I$. Using $D_Y(\I)\subset \I$, we get $D_Y^k(\I) \subset \I$. Hence $D_Y^k(f^k)\in \I\subset \sqrt{I}$. By the Leibniz rule $D_Y^k (f^k)$ is the sum of terms of the form
$$\prod_{i=1}^k (D_Y^{\alpha_i} f)$$
for nonnegative integers $\alpha_1,\alpha_2\dots,\alpha_k$ such that $\alpha_1+\alpha_2+\dots+\alpha_k = k$. Note that for all the terms except for $(D_Y f)^k$, at least one of $\alpha_1$, $\alpha_2$, \dots, $\alpha_k$ is zero, and all those terms belong to $\sqrt{I}$, as $f\in\sqrt{\I}$. Therefore we get $(D_Y f)^k\in \sqrt{\I}$, hence $D_Y f\in \sqrt{\I}$.
\end{proof}

As a geometric counterpart, we get the following lemma, which will prove very useful in our proofs.

\begin{lemma} \label{lemfix}
 Let $Y$ be a vector field on a smooth variety $X$. Assume that $\V$ is a $\C$-linear subspace of $\Vect(X)$. If $Y$ normalises $\V$, i.e. $[Y,\V]\subset \V$, then $Y$ is tangent to the reduction of the zero scheme of $\V$. 
 
In particular, if a subspace $\V\subset \Vect(X)$ has isolated (simultaneous) zeros, then they are fixed by the normaliser $N_{\Vect(X)}(\V)$ of $\V$ in $\Vect(X)$. 
\end{lemma}

\begin{remark} From now on, we will keep the reduction of the zero scheme simply \emph{the reduced zero scheme}.
\end{remark}

Note that even the reduced zero scheme of $\V$ may be singular. A vector from the tangent space to $X$ at $x$ is considered tangent to a reduced subscheme $Z \ni x$ if it is in the image of the tangent space to $Z$ at $x$. Equivalently, in a local affine neighbourhood it annihilates all the functions that vanish on $Z$, i.e. those from the defining ideal of $Z$. A vector field is tangent to a reduced subscheme $Z$ if it is tangent to $Z$ at all the closed points of $Z$. Equivalently, the associated derivation maps the defining ideal of $Z$ to itself.

\begin{proof}
As the statement is local, we can assume that $X = \Spec A$ is affine. The space $\V$ gives rise to a vector space of $\C$-derivations $A\to A$, and $Y$ to a single derivation $D_Y:A\to A$. The reduced zero scheme of $\V$ is defined by the ideal $\mathcal J = \sqrt{(\im D_W)_{D_W\in \V}}$. By Lemma \ref{algder} we then get $D_Y \mathcal J \subset \mathcal J$. This means that $Y$ is tangent to the scheme defined by $\mathcal J$.
\end{proof}

\begin{remark}
 There is an alternate, analytic proof, which works under the assumption of $\V$ being finite-dimensional -- which will always be the case for us. It is non-algebraic and hence also non-translatable to other fields, but one could argue it is less technically demanding. Moreover it can be applied to general differentiable manifolds, without algebraicity assumption. Hence we present it here as well. In fact, the finite dimensionality assumption can also be dropped, if we use the fact that the functions we deal with are all analytic, hence they vanish locally if all the derivatives in a point vanish -- this approach mimics the algebraic proof.

 Let $\phi = [Y,-]\big|_{\V}$ be the commutator map $\V\to \V$ induced by $Y$.
 Let $x\in X$ be fixed by $\V$ and let us consider local one-parameter subgroup $\Psi_t$ around $x$ defined by the vector field $Y$.
 For any vector field $W$ we have
 $$[Y,W]_x = \frac{d}{dt}\left( (D_x \Psi_t)^{-1} W_{\Psi_t(x)}\right)\big|_{t=0}$$
 and analogously
 $$[Y,W]_{\Psi_t(x)} = \frac{d}{du}\left( (D_x \Psi_u)^{-1} W_{\Psi_{t+u}(x)}\right)\big|_{u=0}.$$
 Composing this with the linear map $(D_x\Psi_t)^{-1}$ we get, for $W\in \V$, the following:
 $$(D_x\Psi_t)^{-1}\phi(W)_{\Psi_t(x)} = \frac{d}{du}\left( (D_x \Psi_u)^{-1} W_{\Psi_{u}(x)}\right)\big|_{u=t}.$$
 Hence if we consider the map $\tau:(-\eps,\eps)\to \Hom(\V,T_x X)$ defined as
 $$\tau(t)(W) = (D_x \Psi_t)^{-1} W_{\Psi_{t}(x)} $$
 we get
 $$\frac{d}{dt} \tau(t) = \phi^* \tau(t).$$
 
We obtain a linear differential equation, and in particular as $\tau(0)$ vanishes (because $\V$ vanishes at $x$), we get that $\tau$ vanishes also around 0, hence $\tau$ moves along zeros of $\V$.
\end{remark}

\section{Algebraic groups and their Lie algebras}
Throughout the thesis, we will be concerned with actions of linear algebraic groups on algebraic varieties. This section serves as a reminder of the most important notions and properties of algebraic groups that we use. We only discuss groups and varieties over $\C$.
The classical source on the topic is Borel's monograph \cite{Borel}. A more modern and more comprehensive overview, with a stronger emphasis on generality is \cite{Milne}. An older, but very influential source, especially on the topic of algebraic Lie algebras is the monograph by Chevalley \cite{Chev2}, \cite{Chev3}.

\subsection{General algebraic groups}
\begin{definition}
 An algebraic group $\Gs$ is a scheme of finite type over $\C$ with a structure of a group object in the category of schemes over $\C$. This means that the morphisms $\mu:\Gs \times \Gs \to \Gs$ (multiplication map), $\Gs\to \Gs$ (inverse map) and $\Spec \C \to \Gs$ (neutral element) are given, which satisfy the standard axioms of associativity, neutral element and inverse element. The morphisms of algebraic groups are algebraic morphisms that preserve those structures.
\end{definition}

Algebraic groups in positive characteristic might be non-reduced, however as we are working with groups over $\C$, all the groups are even smooth \cite[Corollary 8.39]{Milne}. Any algebraic group is separated \cite[Proposition 1.22]{Milne}, and therefore as long as it is irreducible, it is a variety. The connected component of identity is therefore always a variety. The two most basic groups for us are the multiplicative group $\Gs_m$ and the additive group $\Gs_a$. On the level of $\C$-points those are the multiplicative and additive group of $\C$, respectively. As a scheme $\Gs_m = \Spec \C[t,t^{-1}]$ and $\Gs_a = \Spec \C[x]$. The coaction on the ring level takes $t$ to $t\otimes t$ and $x$ to $x\otimes 1 + 1\otimes x$. We also denote $\Gs_m$ as $\Cs$.

As an important example, generalising $\Cs$, we also have the general linear group $\GL_n(\C)$ of invertible $n\times n$ complex matrices. It is an open subscheme of the affine space of all $n\times n$ matrices, defined by non-vanishing of one polynomial, i.e. the determinant. Thence it is also an affine scheme. In fact, in a sense it is a universal example of an affine algebraic group.

\begin{proposition}[Proposition 1.10, \cite{Borel}] An algebraic group is affine if and only if it is a closed subgroup of $\GL_n(\C)$ for some positive integer $n$.
\end{proposition}

For this reason, we refer to the affine algebraic groups as \emph{linear algebraic groups}. In this thesis, we do not consider other algebraic groups. One should however note that by Chevalley's structure theorem \cite[Theorem 8.26]{Milne}, every connected algebraic group is an extension of a complete algebraic group by a linear algebraic group. Moreover, a complete algebraic group is projective \cite[Theorem 8.45]{Milne} and abelian \cite[8.20]{Milne}. Complete algebraic groups are also called \emph{abelian varieties}.

\begin{definition}
 Any group morphism from an algebraic group $\Gs$ to $\GL_n(\C)$ is called a \emph{representation} of $\Gs$.
\end{definition}

As we saw above, any affine algebraic group admits a faithful (i.e. injective) representation. The group $\GL_n(\C)$ has its \emph{standard representation} given by the identity $\GL_n(\C)\to\GL_n(\C)$. Any algebraic group admits the adjoint representation, as explained in the next section.

\subsection{Lie algebras of algebraic groups}
As a complex algebraic group is always smooth, it can be actually viewed as a Lie group. The classical Lie theory analyses the structure of Lie groups with use of their Lie algebras, and for algebraic groups, their Lie algebras can also be defined algebraically. The Lie algebras and their actions will be central to our main results, e.g. Theorem \ref{general}.

\begin{definition}
The Lie algebra of an algebraic group $\Gs$ is the vector space $\geg = \Lie(\Gs)$ of the derivations on $\Gs$ invariant under the left multiplication by any element $g\in\Gs$. The Lie bracket of vector fields induces the Lie algebra structure on $\geg$.
\end{definition}

A left-invariant vector field on $\Gs$ is defined uniquely by its value at $1\in\Gs$, and this provides an isomorphism of vector spaces $\geg \simeq T_{1}\Gs$.

Although the definition of the Lie algebra agrees with the one known from analysis, the fundamental theorems of the Lie algebra do not hold in the algebraic category. The reason for that is that there are much fewer algebraic groups and morphisms between them than in the Lie group case. The assignment $\Lie:\Gs\mapsto \geg = \Lie(\Gs)$ satisfies the following properties.
\begin{enumerate}
\item It is a functor, i.e. a group homomorphism $\phi:\Gs\to \Hs$ induces naturally a morphism of Lie algebras $d\phi:\geg\to\he$;
\item it maps abelian groups to abelian Lie algebras. Moreover a connected algebraic group with an abelian Lie algebra is itself abelian (one needs to work in characteristic 0 here \cite[10.35]{Milne});
\item if $\Gs<\Hs$ are connected algebraic groups and the corresponding map $\Lie(G) \to \Lie(\Hs)$ is an isomorphism, then $\Gs \simeq \Hs$.
\end{enumerate}

However, the following do not agree with classical Lie group theory:
\begin{enumerate}
\item Not every complex Lie algebra is a Lie algebra of an algebraic group. We call a Lie algebra \emph{algebraic} if it is a Lie algebra of an algebraic group.
\item Moreover, if $\Hs$ is an algebraic group and $\geg\subset\he$ is an algebraic Lie algebra by itself, its inclusion in $\he$ might not come from an inclusion of a subgroup $\Gs\to \Hs$. Consider as an example an irrational subalgebra in $\Lie(\Gs_m^2)$.
\item If $\Gs$ and $\Hs$ are algebraic groups and $\Gs$ is simply connected, then a Lie algebra morphism $\geg\to\he$ does not have to lift to a group morphism $\Gs\to\Hs$. Consider for example $\Gs  = \Gs_a$ and $\Hs = \Gs_m$.
\end{enumerate}

As in the analytic case, we can talk about the adjoint representation of a Lie group and Lie algebra. Let $\Hs$ be an algebraic group. Then we are given the morphism $\Hs\times\Hs \to \Hs$ given by conjugation $(g,h)\mapsto (ghg^{-1})$. The derivative at $(g,1)$ is a linear map $\he\times\he\to\he$ which vanishes on the first factor. The map $\he\to\he$ on the second factor is called $\Ad_g$. So defined $\Ad: g\mapsto \Ad_g$ defines a representation of the group $\Hs$ \cite[3.13]{Borel} on the vector space $\he$, which we call the \emph{adjoint representation}. The image of $\Hs$ under $\Ad:\Hs\to\GL(\he)$ is a closed algebraic subgroup and it is called the \emph{adjoint group} of $\Hs$. It is the quotient of $\Hs$ by its centre.

Any representation $\rho:\Gs\to \GL(V)$ of an algebraic group defines the corresponding map $d\rho:\geg\to\gl(V) = \End(V)$ on the level of Lie algebras. It is a representation of the Lie algebra, i.e. for any $X,Y\in\geg$ we have $[d\rho(X),d\rho(Y)] = d\rho([X,Y])$, where on the left we commute operators in $\End(V)$. In particular, this yields the adjoint representation of the Lie algebra, $\ad:\geg\to\gl(\geg)$. It is simply given by $\ad_X(Y) = [X,Y]$.

\subsection{Properties of linear algebraic groups} \label{secjor}
Here we define and describe the properties of particular classes of linear groups that we will be using: solvable, unipotent, semisimple and reductive. First, we need to discuss the Jordan decomposition.

The classical Jordan decomposition for an endomorphism $A\in M_{n\times n}(\C)$ provides a basis of $\C^n$ in which the matrix is of Jordan form, i.e. $A = MJM^{-1}$ for $M\in\GL_n(\C)$ and $J$ in Jordan form. Then $J$ is in particular an upper triangular matrix and is a sum of its diagonal part $J_s$ and the nilpotent part $J_n$, consisting of the entries above the diagonal. This gives the \emph{additive Jordan decomposition} of $A$ into \emph{semisimple} (i.e. diagonalisable) and \emph{nilpotent parts} $A_s = MJ_sM^{-1}$ and $A_n = MJ_nM^{-1}$ of $A$. Those are commuting matrices and they do not depend on the choice of $M$. In fact, it is the unique pair of commuting semisimple matrix and nilpotent matrix, which sum up to $A$ \cite[Proposition 4.2]{Borel}. Now if $A$ is invertible, then $A_s$ is, and taking $A_u = I_n + A_s^{-1} A_n$ yields the \emph{multiplicative Jordan decomposition} $A = A_s A_u$. Here $A_s$ and $A_u$ commute, $A_s$ is semisimple and $A_u$ is unipotent -- the latter is called the \emph{unipotent part} of $A$.

Now let $\Gs$ be an arbitrary linear algebraic group. Then we have the following theorem-definition \cite[4.4]{Borel}.
\begin{theorem} \label{defjord}
Let $g\in \Gs$ and $X\in \geg$. Then there exist unique $g_s,g_u\in \Gs$ and $X_s,X_n\in\geg$ such that for any representation $\rho:\Gs\to \GL_n(\C)$, $\rho(g_s) \cdot \rho(g_u)$ is the multiplicative Jordan decomposition of $\rho(g)$ and $d\rho(X_s) + d\rho(X_n)$ is the additive Jordan decomposition of $d\rho(V)$. We also call $g_s\cdot g_u = g$ and $X_s+X_n = X$ respectively the multiplicative Jordan decomposition and the additive Jordan decomposition.
\end{theorem}
Directly from the properties we get that the notion of Jordan decomposition agrees with the above for $\Gs = \GL_n(\C)$. Moreover any morphism of algebraic groups $\phi:\Gs\to\Hs$ preserves both kinds of Jordan decompositions.

We say that an element of the group is semisimple/unipotent if its unipotent/semisimple part vanishes. Similarly, an element of the Lie algebra is called semisimple/nilpotent if its nilpotent/semisimple part vanishes. One should note that the Jordan decomposition in the Lie algebra is not its intrinsic property -- it depends heavily on the group. The Lie algebras of $\Gs_m^r$ and $\Gs_a^r$ are isomorphic (and abelian), but in the former each element is semisimple, and in the latter each element is nilpotent. By functoriality of the Jordan decomposition, this provides obstructions to lifting Lie algebra morphisms to Lie group morphisms. In particular, a Lie algebra of an algebraic subgroup of $\GL_n(\C)$ needs to contain the semisimple and nilpotent part of any of its elements \cite[\S 14, Proposition 3]{Chev2}. There is a five-dimensional solvable Lie algebra for which it fails, see \cite[Chapter 1, Exercise 5.6]{BouLie13}, \cite[Chapter 7, Exercise 5.1]{BouLie79}.

\subsubsection{Algebraic tori}
For any nonnegative integer $r$, we define the \emph{torus of rank $r$} to be an algebraic group isomorphic to $(\Cs)^r$. It is by definition an abelian linear group. Equivalently a torus is a connected algebraic group consisting of semisimple elements. If we choose the isomorphism with $(\Cs)^r$, then a rank $r$ torus $\Ts$ can be viewed as the subgroup of diagonal matrices within $\GL_r(\C)$. In fact, for any representation $(\Cs)^r\to \GL_n(\C)$, its image is conjugate to a subgroup of the diagonal matrices. This means that there is a basis $(v_1,\dots,v_n)$ of $\C^n$ such that each $\Span(v_i)$ is $\Ts$-invariant, and the corresponding representations are then morphisms $\Ts\to\Cs$, which we call the \emph{weights} of the representation. 

Tori are the basic groups from the point of view of representation theory, but also for group actions and equivariant cohomology. In the next paragraphs we see the important role maximal tori play in linear groups in general.

\subsubsection{Unipotent and solvable groups}

\begin{definition}
We call a linear algebraic group \emph{unipotent} if all of its elements are unipotent.
\end{definition}
A standard example of a unipotent group is the group $\U_n$ of upper-triangular matrices within $\GL_n(\C)$. In fact, every unipotent group is a closed subgroup of $\U_n$ for some $n$ \cite[4.8]{Borel}. By definition, a unipotent group does not contain $\Gs_m$. However, in general it does contain many copies of $\Gs_a$. In fact, any element of the Lie algebra integrates to an additive subgroup:

\begin{theorem}[Proposition 14.32 in \cite{Milne}, Proposition V.3.15 in \cite{Chev3}]\label{algexp}
Assume that $\U$ is a unipotent group. Then there exists an \emph{algebraic exponential} isomorphism of schemes $\exp:\uu\to\U$. It is the unique map with the property that for any representation $\rho:\U\to \GL_n(\C)$ and any $v\in \uu$, the matrix $\rho(\exp(v))$ is equal to the exponential $\exp(d\rho(v))$.
\end{theorem}

Note that the power series for $\exp(d\rho(v))$ eventually vanishes, as $\rho(v)$ is nilpotent. Moreover $\Gs_a$ is the ``building block'' of unipotent groups:

\begin{theorem}[\cite{Milne}, Proposition 14.21] \label{solvsplit}
Every unipotent group admits a central series whose quotients are isomorphic to $\Gs_a$.
\end{theorem}

Now we switch our attention to the solvable groups. An algebraic group is called solvable if the group of its $\C$-points is solvable. A typical example is the Borel subgroup $\Bs_n$ of upper-triangular matrices in $\GL_n(\C)$. In fact, by the Lie--Kolchin theorem \cite[Corollary 10.5]{Borel}, every solvable group is a closed subgroup of some $\Bs_n$.

We recall (a part of) the theorem of Borel on solvable groups (\cite[Theorem 10.6]{Borel}, see also \cite[Theorem 16.33]{Milne}) that we will often tacitly use throughout.

\begin{theorem}\label{solv}
Let $\Hs$ be a connected solvable group with the Lie algebra $\he$ and $\Hs_u$ its set of unipotent elements. Then
\begin{enumerate}
\item $\Hs_u$ is a connected normal closed, unipotent subgroup of $\Hs$ containing $[\Hs,\Hs]$. 
\item The maximal tori in $\Hs$ are all conjugate. If $\Ts$ is a maximal torus, then $\Hs = \Hs_u \rtimes \Ts$. The Lie algebra $\he_n$ of $\Hs_u$ consists of all nilpotent elements of $\he$.
\item If $\Ts$ is a maximal torus, then any semisimple element of $\Hs$ is conjugate to a unique element of $\Ts$.
\end{enumerate}
\end{theorem}

\begin{remark}
\label{nilalg}
Let $\he_n$ be the set of nilpotent elements of $\he$. It follows from above that $\he_n$ is a Lie subalgebra of $\he$. As it consists of nilpotent elements, hence acting nilpotently by the adjoint action, by Engel's theorem it is nilpotent itself. Moreover it contains $[\he,\he]$. In addition, from the second statement we get that $\he = \he_n \oplus \ttt$ for $\ttt = \Lie(\Ts)$.
\end{remark}

\subsubsection{Reductive and semisimple groups}
Of particular importance for representation theory and equivariant topology are the reductive groups, which generalise algebraic tori. As mentioned above, any torus representation splits into a direct sum of one-dimensional representations. A representation of reductive group also splits into a direct sum of \emph{irreducible representations}, i.e. representations without proper nontrivial subrepresentations. Those can however be quite complicated themselves.

\begin{definition}
A linear algebraic group is called \emph{linearly reductive} if its every finite-dimensional representation is a direct sum of irreducible representations.
\end{definition}

We can also characterise reductivity internally. To this end, we need to define the radical and the unipotent radical (\cite[6.44 and 6.66]{Milne}).

\begin{definition}
Let $\Gs$ be a connected linear algebraic group. It contains the maximal connected solvable normal closed subgroup, called the \emph{radical} of $\Gs$, as well as the maximal connected unipotent normal closed subgroup, called the \emph{unipotent radical} of $\Gs$. We call $\Gs$ \emph{semisimple} if its radical is trivial, and \emph{reductive} if its unipotent radical is trivial.
\end{definition}

Note that any unipotent group is nilpotent and hence solvable, and therefore the unipotent radical is contained in the radical. Thus semisimplicity is a stronger property than reductivity.

As we consider all the groups in characteristic $0$, a linear algebraic group is linearly reductive if and only if its identity component is reductive \cite[22.43]{Milne}. A reductive group is semisimple if and only if it has a finite centre. In general, the adjoint group of a reductive group is always semisimple, and the identity component of the centre is an algebraic torus \cite[Proposition 19.12]{Milne}. Reductive groups can be also characterised by the fixed points of their actions on algebraic varieties (see Section \ref{grpact}). The standard example of a reductive group is $\GL_n(\C)$. This is however not a semisimple group, as its centre, and at the same time radical, consists of all scalar matrices. However, its subgroup $\SL_n(\C)$ and quotient $\PGL_n(\C)$ are both semisimple. In fact, the latter is the adjoint group of both $\GL_n(\C)$ and $\SL_n(\C)$.

The semisimple groups are constructed from simple groups. Simply connected simple groups are then completely classified by their Dynkin diagrams. The classification is out of scope of this thesis, for more details we refer the reader to \cite{Knapp}. The simple groups, up to a quotient by finite subgroup, are $\SL_{n+1}$, $\SO(n)$, $\Sp(2n)$ and five exceptional groups. 

We only recall the basic notions of the theory. Let $\Gs$ be a connected reductive group and $\Ts\subset \Gs$ its maximal torus. Then any representation $\Gs$ is also a representation of $\Ts$, and such split into direct sums of one-dimensional representations. In particular, that holds for the adjoint representation of $\Gs$. As $\Ts$ is commutative, it acts trivially on $\ttt$, and in fact $\ttt = \geg^\Ts$ is the fixed point set of the adjoint action of $\Ts$ on $\ttt$. We therefore get a decomposition
$$\geg  = \ttt \oplus \bigoplus_{\alpha\in \ttt^*} \geg_\alpha,$$
where $\ttt$ acts on $\geg_\alpha$ with the weight $\alpha$. The spaces $\geg_\alpha$ are nonzero only for finitely many choices of $\alpha\in\ttt^*$, which we call the \emph{roots} of $\geg$. All of those spaces turn out to be one-dimensional, and moreover $\alpha$ is a root if and only if $-\alpha$ is a root. The roots not only lie in $\ttt^*$, but in fact in a discrete lattice $\Lambda = \Hom(\Ts,\Cs)$. One can then view them as elements of a real vector space. A sufficiently general functional from that space to $\R$ will then provide an ordering of the roots. This then splits the set of roots $\Phi$ into the \emph{positive roots} $\Phi^+$ and negative roots $\Phi^-$ such that $\Phi^+ = -\Phi^-$. The subalgebras 
$$\bb = \ttt \oplus \bigoplus_{\alpha\in \Phi^+} \geg_\alpha, \qquad \bb^- = \ttt \oplus \bigoplus_{\alpha\in \Phi^-} \geg_\alpha$$
of $\geg$ are solvable subalgebras, and in fact Borel subalgebras (cf. Section \ref{secbor}). They integrate to Borel subgroups $\Bs$ and $\Bs^-$, whose unipotent radicals $U$ and $U^-$ have the Lie algebras
$$\uu = \bigoplus_{\alpha\in \Phi^+} \geg_\alpha, \qquad \uu^- = \bigoplus_{\alpha\in \Phi^-} \geg_\alpha.$$
A positive root which is not a sum of other positive roots is called a \emph{simple root}. By definition, every positive root is a sum of the simple roots, and in fact it is so uniquely, as the simple roots turn out to be linearly independent. If $\Gs$ is semisimple, they constitute a basis of $\ttt^*$.

\subsection{Levi decomposition}
We saw above that the affine and projective groups are building blocks of all algebraic groups. Reductive and solvable groups play a similar role for linear groups. Indeed, let $\Hs$ be an arbitrary linear group. The quotient $\Hs/R_u(\Hs)$ of $\Hs$ by its unipotent radical is clearly reductive. Therefore we see that $\Hs$ is an extension of a reductive group by a unipotent group -- its unipotent radical. In case $\Hs$ is solvable, the unipotent radical is just the group of all unipotent elements and then $\Hs/R_u(\Hs)$ can be lifted to a maximal torus, see Theorem \ref{solv}. It turns out that this is not specific to solvable groups.

\begin{theorem}[\cite{Mostow}] \label{thmlevi}
 Let $\Hs$ be a linear algebraic group and $\Ns$ its unipotent radical. Then the exact sequence
 $$1\to \Ns\to \Hs\to \Hs/\Ns \to 1$$
 splits, i.e. there is a subgroup $\Ls$ in $\Hs$ which is mapped isomorphically to $\Hs/\Ns$ in the projection. Therefore $\Hs = \Ns \rtimes \Ls$ for $\Ns$ unipotent and $\Ls$ reductive. The subgroup $\Ls$ is called a \emph{Levi factor}. Its choice is not canonical, but any two Levi factors are conjugate. We can take for $\Ls$ any maximal reductive subgroup of $\Hs$.
\end{theorem}

Typical examples of groups which are neither solvable nor reductive are parabolic subgroups within reductive groups.

\subsection{Parabolic subgroups and homogeneous spaces}\label{secbor}
We will now discuss the projective homogeneous spaces of algebraic groups. Those are defined by parabolic groups.

\begin{definition}
A \emph{Borel subalgebra} of a Lie algebra $\geg$ is a maximal solvable subalgebra $\bb$. A \emph{Borel subgroup} of an algebraic group $\Gs$ is a maximal connected closed solvable subgroup $\Bs$.
\end{definition}
In $\ssl_n(\C)$, an example of a Borel subalgebra is the subalgebra $\bb_n$ of the upper triangular matrices. It integrates to the Borel subgroup $\Bs_n$ of upper triangular matrices. In general, for any reductive group the algebra generated by a maximal torus and positive roots is a Borel subalgebra.

\begin{proposition}
A Lie algebra of a Borel subgroup of $\Gs$ is a Borel subalgebra of $\geg$. Conversely, any Borel subalgebra of a Lie algebra $\geg$ of an algebraic group $\Gs$ is a Lie algebra of a Borel subgroup in $\Gs$. The quotient $\Gs/\Bs$ is projective for any Borel subgroup $\Bs < \Gs$. The maximal tori of Borel subgroups of $\Gs$ coincide with the maximal tori of $\Gs$.
\end{proposition}

\begin{definition}
A closed subgroup $\Ps$ of a connected linear group is called \emph{parabolic} if $\Gs/\Ps$ is a projective algebraic variety. 
\end{definition}

As an easy application of the Lie--Kolchin theorem we have

\begin{proposition}
For any Borel subgroup $\Bs$ and a parabolic subgroup $\Ps$ of a semisimple group $\Gs$ there exists $x\in \Gs$ such that
$$\Bs\subset x\Ps x^{-1}.$$
\end{proposition}

\begin{corollary} \label{borconj}
All Borel subgroups are conjugate. A subgroup is parabolic if and only if it contains a Borel subgroup. 
\end{corollary}

Any Borel subgroup of a group $\Hs$ clearly contains the radical $R(\Hs)$ of $\Hs$. Hence the parabolic subgroups of $\Hs$ are in fact determined by the parabolic subgroups in its Levi factor. Then in reductive groups the parabolic subgroups are determined by subsets of simple roots, see e.g. \cite[23.3]{FulHar}. For a reductive group $\Gs$, the projective homogeneous spaces $\Gs/\Ps$ for $\Ps$ parabolic are called \emph{flag varieties}. This reflects the fact that for $\Gs = \GL_n(\C)$, those parametrise (full or partial) flags in $\C^n$. The conjugacy type of $\Ps$ determines the dimensions of subspaces making up the flags. The compact homogeneous spaces will be important examples for us later.

\section{Group actions and vector fields}\label{grpact}

\subsection{Group actions} \label{secgrpac}

We say that an algebraic group $\Hs$ acts on a variety $X$ if a map $\rho: \Hs\times X\to X$ is given such that the diagram
$$
\begin{tikzcd}
\Hs\times\Hs\times X \arrow[d, "\id_\Hs\times \rho"] \arrow[r, "\mu\times \id_X"] & \Hs\times X \arrow[d, "\rho"]\\
\Hs\times X\arrow[r, "\rho"] & X
\end{tikzcd}
$$
commutes. In such a situation, we can pull back the global functions along the action. Explicitly, for $h\in \Hs$ we have the pullback map $h^*: \OO_X\to \OO_X$, where on any open $U\subset X$, this maps $\OO_X(hU) \to \OO_X(U)$ by composing the functions on $hU$ with the multiplication map $U\to hU$. As this map is contravariant, i.e. $g^*h^* = (hg)^*$, we need to invert to get a left action on $\OO_X$. Hence when we talk about the action of $\Hs$ on the global functions, we mean the action defined by $h\mapsto (h^{-1})^*$. This in particular gives a representation of $\Hs$ on $\OO_X(X)$, and in general on $\OO_X(U)$ for any $G$-invariant open subset $U\subset X$.

Whenever an algebraic group $\Hs$ acts on a variety $X$, it yields a Lie algebra homomorphism $\phi:\he\to\Vect(X)$ from $\he = \Lie(\Hs)$ to vector fields on $X$, see \cite{CoDr}. On any fixed $v\in \he$, this gives the vector field $V_v$. For any $x\in X$, its value $V_v|_x$ at $x$ can be recovered by considering the derivative at $1_{\Hs}$ of the map $\Hs\to X$ defined as $g\mapsto g\cdot x$, and evaluating it on $x$.

We will call such a homomorphism $\he\to\Vect(X)$ an \emph{action of the Lie algebra} $\he$ on $X$. We will want to define the total vector field on $\he\times X$. As it is a local problem on $X$, we can restrict to an affine open set $U$. Then 
\begin{align}
\label{ohtu}
\C[\he\times U] = \C[\he] \otimes_\C \C[U]
\end{align}
and we need to define a derivation on this $\C$-algebra. We can view $\phi|_U$ as an element of $\he^*\otimes_\C \Vect(U)$. As $\C[\he] = S^*(\he^*)$, we have a multiplication map $\he^* \otimes \C[\he] \to \C[\he]$. Additionally, $\Vect(U)$ are by definition the derivations on $\C[U]$, which gives a $\C$-bilinear $\Vect(U) \otimes \C[U] \to \C[U]$. Those two maps together with \eqref{ohtu} lead to a $\C$-bilinear map
$$(\he^* \otimes \Vect(U)) \otimes \C[\he\times U] \to \C[\he\times U]. $$
Fixing $\phi|_U\in (\he^* \otimes \Vect(U))$ gives a derivation $\C[\he\times U] \to \C[\he\times U]$.
\begin{definition}\label{totvec}
The vector field defined by this derivation will be called the \emph{total vector field} of the $\Hs$-action on $X$.
\end{definition}
Explicitly, let $\phi = \sum \psi_i \otimes D_i$ for $\psi_i\in \he^*$, $D_i\in\Vect(U)$. Then the defined derivation on $f \otimes g \in \C[\he]\otimes \C[U]$ takes value
\begin{align}
\label{derim}
    \sum (\psi_i \cdot f) \otimes D_i(g) \in \C[\he]\otimes \C[U].
\end{align}
This gives the total vector field on $\he\times X$. One can note that the vector field is tangent to $\{y\}\times X$ for any $y\in \he$, i.e. as a derivation it preserves the set of functions vanishing on $\{y\}\times X$. Indeed, locally such functions are sums of $f \otimes g \in \C[\he]\otimes \C[U]$ such that $f(y) = 0$, and in such case the image of the derivation \eqref{derim} also vanishes at $\{y\}\times X$. The vector field restricted to $\{y\} \times X$ is precisely $\phi(y)$ and for any $y\in\he$ with $\Hs$ acting on $X$ we will denote this vector field by $V_y$. Later we will consider restrictions of the total zero schemes to bigger subsets of $\he$.

As the vector field vanishes in the $\he$ direction, it is not only a section of the tangent bundle, but a section of the pullback $\pi_2^* T_X$ of the tangent bundle of $X$ via the projection $\pi_2:\he\times X\to X$.

One can also think of vector fields via the total spaces of tangent bundles. If we denote the total vector bundle of a variety $Y$ by $TY$, then the action $\rho:\Hs\times X\to X$ defines the map $T\rho:T\Hs\times TX\to TX$ and its restriction to $T_{1}\Hs\times X \simeq \he\times X$ gives a section $\he\times X \to TX$ of the vertical tangent bundle, identical with the vector field on $\he\times X$ constructed above.

\begin{lemma} \label{lemad}
 Let an algebraic group $\Hs$ act on a variety $X$. Then for any $g\in \Hs$, $y\in\he = \Lie(\Hs)$ and $x\in X$ we have
 $${V_{\Ad_g(y)}}|_{gx} = \D g({V_y}|_{x}).$$
\end{lemma}

\begin{proof}
Let $\mu:\Hs\times \Hs\to \Hs$ denote the multiplication map and $\rho:\Hs\times X\to X$ denote the action of $\Hs$ on $X$. Consider the following commutative diagram.
$$
\begin{tikzcd}
& \Hs\times \Hs\times \Hs\times X \arrow[ld, "\mu\times\id_\Hs\times\id_X"] \arrow[rd, "\id_\Hs\times\id_\Hs\times\rho"]
\\
\Hs\times \Hs\times X\arrow[dd, "\mu\times\id_X"] & & \Hs\times \Hs\times X\arrow[dd, "\id_\Hs\times\rho"]
\\ \\
\Hs\times X\arrow[dr,"\rho"] && \Hs\times X\arrow[dl,"\rho"]
\\
& X
\end{tikzcd}.
$$
If we fix a point on the top, it yields an analogous commutative diagram of differential maps. Take $(g,1,g^{-1},gx) \in \Hs\times \Hs\times \Hs\times X$ and $(0,y,0,0)$ in its tangent space. Going through the left branch, it is mapped to ${V_{\Ad_g(y)}}|_{gx}$ and going through the right one, it is mapped to $\D g({V_y}|_{x})$.
\end{proof}

The next Lemma will be used to show that zeros of generalised Jordan matrices are zeros of the torus.
\begin{lemma} \label{lemzer}
Let a Lie algebra $\he$ act on a smooth variety $X$. Let $d,n \in \he$ commute and assume that the Lie subalgebra generated by $[\he,\he]$ and $n$ is nilpotent. Let $x\in X$ be an isolated zero of the vector field $V_j$ associated to $j=d+n$. Then $x$ is also a simultaneous zero of $C_{\he}(d)$. In particular, $x$ is a zero of any abelian subalgebra of $\he$ containing $d$.
\end{lemma}
\begin{proof}
Let $\kek$ be the Lie subalgebra generated by $[\he,\he]$ and $n$. By Lemma \ref{lemfix} we first get that $x$ is a zero of $d$ and $n$, as they commute with $j$.

We will first prove that $x$ is a zero of $C'(d) = C_{\he}(d)\cap \kek$. As $\kek$ is nilpotent by assumption, its subalgebra $C'(d)$ is nilpotent as well.

By definition $d$ is in the centre of $C_{\he}(d)$, in particular it commutes with $C'(d)$. Hence from Lemma \ref{lemfix} -- for $\mathcal V$ spanned by $d$ and $n$ -- we have that $x$ is a zero of $N_{C'(d)}(\C\cdot n)$. It is therefore an isolated simultaneous zero of $d$ and $N_{C'(d)}(\C\cdot n)$ and we can apply the same argument repeatedly to get that for $i=1,2,\dots$ it is a zero of $N_{C'(d)}^i(\C\cdot n)$.

The sequence $\left(N_{C'(d)}^i(\C\cdot n)\right)_{n=1}^\infty$ has to stabilise at a Lie subalgebra of $C'(d)$ which is its own normaliser in $C'(d)$. As $C'(d)$ is nilpotent, it then has to be equal to whole $C'(d)$, see \cite[Proposition 3 in Chapter 1, \S 4.1]{BouLie13}. Therefore $d$ and $C'(d)$ vanish at $x$. But $[C_{\he}(d),C_{\he}(d)] \subset C_{\he}(d)\cap [\he,\he] \subset C_{\he}(d)\cap \kek = C'(d)$, hence $C'(d)$ is normalised by whole $C_{\he}(d)$. Therefore by Lemma \ref{lemfix} the whole $C_{\he}(d)$ vanishes at $x$.
\end{proof}

From Remark \ref{nilalg} the assumptions about $d$ and $n$ hold whenever $\he$ is solvable, $[d,n]=0$ and $n\in \he_n$ (as $\he_n$ is nilpotent and contains $[\he,\he]$ as well as $n$).

\subsection{Fixed point schemes of group actions}

Actions of algebraic groups and their fixed points can be quite difficult to study in full generality. There are however some known results in particular cases. First, the Lie--Kolchin theorem together with linearisation of the action (Section \ref{linsec}) imply the following.

\begin{theorem}[Borel fixed point theorem]\label{borelfix}
If a solvable group $\Bs$ acts on a complete nonzero variety $X$, then its fixed-point scheme $X^\Bs$ is nonempty.
\end{theorem}
We talk about the fixed point scheme, which one defines e.g. by its functor of points \cite[7.B]{Milne}. However, in the complex case this just means that there is a closed point in $X$ which is fixed by the whole group $\Bs$.

\begin{theorem}[Horrocks, \cite{Horrocks}]\label{horrfix}
If a unipotent group $\U$ acts on a variety $X$, then its fixed point scheme $X^\U$ is connected.
\end{theorem}

\begin{theorem}\label{fixred}
Assume that a group $\Gs$ is reductive and acts on a smooth variety $X$. Then the fixed-point scheme $X^\Gs$ is smooth, in particular reduced.
\end{theorem}

\begin{remark}
For proof see \cite[13.1]{Milne}. This property is in fact an equivalent characterization of the reductive groups among all the complex algebraic groups \cite{Fog}.
\end{remark}

In fact, when the group is an algebraic torus, the fixed point schemes provide a lot of information about the topology of the space. This comes from the Białynicki-Birula decomposition.

\subsection{Białynicki-Birula decomposition} \label{secbb}
The content of this section is based on the two influential papers of Białynicki-Birula \cite{BB1,BB2}. He proves that the topology, in particular the additive structure of the cohomology groups can be read from the local data around the fixed point scheme.

Let the one-dimensional torus $T = \Cs$ act on a smooth complete variety $X$. An element $t\in\Cs$ maps $x\in X$ to $t\cdot x\in X$. As we know from Theorem \ref{fixred}, all the connected components of $X^\Ts$ are smooth varieties. Let us denote them by $X^\Ts_1$, $X^\Ts_2$, \dots, $X^\Ts_k$. Then for $i=1,2,\dots,k$ we define the corresponding \emph{plus--} and \emph{minus--cell}:

$$W_i^+ = \{x\in X: \lim_{t\to 0} t\cdot x \in X^\Ts_i \}, \qquad W_i^- = \{x\in X: \lim_{t\to \infty} t\cdot x \in X^\Ts_i  \}.$$

Each of those sets comes with the limit map $W_i^+\to X^\Ts_i$, $W_i^-\to X^\Ts_i$. As the variety is projective, for any $x$ there exist limits of $t\cdot x$ in both directions, hence
$$X = \bigcup_{i=1}^k W_i^+ = \bigcup_{i=1}^k W_i^-.$$
We call those the \emph{plus-} and \emph{minus--decomposition}. Now for $i=1,2,\dots,k$, let $N_i$ be the normal bundle of $X^\Ts_i$ in $X$. The action of $\Cs$ on $X$ descends to an action on $N_i$, as in Section \ref{linsec}. As the action on $X^\Ts_i$ is trivial, every fiber of $N_i$ is a representation of $\Ts$. As $X^\Ts_i$ is the whole fixed point component, $\Cs$ does not fix any nontrivial vector in a fiber of $N_i$. Thence $N_i$ splits into the positive-weight bundle $N_i^+$ and the negative-weight bundle $N_i^-$. 

\begin{theorem}
Both $W_i^+$ and $W_i^+$ are locally closed in $X$. They are affine fiber bundles over $X^\Ts_i$ and 
$$T(W_i^+)|_{X^\Ts_i} \simeq T(X^\Ts_i) \oplus N_i^+, \qquad T(W_i^-)|_{X^\Ts_i} \simeq T(X^\Ts_i) \oplus N_i^-,$$
hence $N_i^+$ and $N_i^-$ are the normal bundles of $X^\Ts_i$ in $W_i^+$ and $W_i^-$, respectively.
\end{theorem}

In particular, when the fixed points of $\Ts$ are isolated, the cells are affine spaces, hence this provides a paving of $X$ into affine spaces. This in particular means that the variety has a structure of a CW-complex whose cells are all even-dimensional, which then means that their closures form an additive basis of the homology of $X$. This is an instance of a more general statement.

\begin{theorem}\label{cohobb}
Let $\nu_i^+$ and $\nu_i^-$ be the ranks of $N_i^+$ and $N_i^-$, respectively. Then there are isomorphisms of graded abelian groups:
$$H_*(X;\Z) = \bigoplus_{i=1}^k H_*(X^\Ts_i;\Z)[-\nu_i^+] = \bigoplus_{i=1}^k H_*(X^\Ts_i;\Z)[-\nu_i^-].$$
For a graded abelian group $A$, by $A[n]$ we mean the abelian group with $(A[n])_i = A_{n+i}$.
\end{theorem}

\section{Equivariant cohomology and equivariant formality}
The main topological invariant that we study in this thesis is the equivariant cohomology ring, introduced by Borel in \cite{Borsem}. We recall the most important definitions and properties here. For a more detailed introduction, we refer the reader to the classical sources \cite{AB,BrionEq,equiv}.

\subsection{Equivariant cohomology} \label{seceqcoh}
Let us a fix a commutative ring $A$ of cohomology coefficients -- most of the time we will consider $A=\C$ and we will mean that when we do not write the coefficients. For any topological space $X$ we can construct the singular cohomology ring $H^*(X,A)$. Now let $\Gs$ be a topological group, for example a complex algebraic group, acting on a topological space $X$. We can then construct the \emph{equivariant} cohomology theory. If $\Gs$ acts on $X$ freely, we would like the equivariant cohomology to be simply the cohomology $H^*(X/G,A)$ of the quotient space. In general, it will give us rich information about fixed points of the action.

\begin{definition}
 If $X$ is a $\Gs$ space we define the \emph{$\Gs$-equivariant cohomology} of $X$ to be
 $$H^*_\Gs(X,A) = H^*(E\Gs\times^\Gs X,A).$$
\end{definition}
Here $E\Gs\to B\Gs$ is the universal $\Gs$ bundle over the classifying space $B\Gs$. The space $X_\Gs = E\Gs\times^\Gs X$ is the mixed quotient, i.e. the quotient of $E\Gs\times^\Gs X$ by the diagonal action of $\Gs$. The precise convention is usually as follows: $\Gs$ acts on $E\Gs$ on the right, and for any $(x,y) \in E\Gs\times X$ we identify $(xt,y)$ with $(x,ty)$. As the action of $\Gs$ on $E\Gs$ is free, this diagonal action is free as well. The space $E\Gs\times^\Gs X$ is a fiber bundle over $B\Gs$, with $X$ as the fiber. Inclusion of the fiber over the basepoint provides a natural map $H^*_\Gs(X,A)\to H^*(X,A)$. This construction is contravariantly functorial under $\Gs$-equivariant maps of $\Gs$-spaces. In particular, the map $X\to \pt$ provides the bundle map $E\Gs\times^\Gs X\to B\Gs$. On the equivariant cohomology, it induces a ring homomorphism $H^*_\Gs(\pt) \to H^*_\Gs(X)$. This means that the equivariant cohomology ring actually has a natural structure of an algebra over $H^*_\Gs(\pt)$. Note also that whenever we have a map of groups $\Gs\to\Hs$, this gives a natural morphism of cohomology rings $H^*_\Hs(X)\to H^*_\Gs(X)$.

It is then essential to know the ring $H^*_\Gs(\pt)$ for algebraic groups $\Gs$. In general, determining it is a hard problem for integral coefficients \cite[2.5]{AndFul}. However, it turns out to have an easy description if $\Gs$ is a connected linear group and $A$ is a ring containing $\Q$. First, assume that $\Gs = \Ts \simeq (\Cs)^r$ is a rank $r$ torus. Then one possible model of the universal bundle $E\Gs\to B\Gs$ is the quotient $(S^\infty)^r \to (\C\PP^\infty)^r$, where each factor of $(\Cs)^r$ acts on the corresponding factor of $(S^\infty)^r$. As $H^*(\C\PP^\infty,A) = A[t]$, we get 
$$H^*_\Ts(\pt,A) = H^*_\Ts((\C\PP^\infty)^r,A) = A[t_1,t_2,\dots,t_r].$$
This identification depends on the choice of the isomorphism $\Ts\simeq (\Cs)^r$. However, one can naturally identify
$$H^*_\Ts(\pt,\C) = \C[\ttt].$$

Then assume that $\Gs$ is an arbitrary linear group. Let $\Ts$ be its maximal torus and $\Ws = N(\Ts)/\Ts$ be the \emph{Weyl group} of $\Gs$. The normaliser $N(\Ts)$ acts on $\ttt$ by adjoint action and by commutativity of $\Ts$ its normal subgroup $\Ts$ acts trivially, hence $\Ws$ has a well defined action on $\ttt$. Then $H^*_\Gs(\pt) \simeq \C[\ttt]^\Ws$ and the canonical map $H^*_\Gs(\pt) \to H^*_\Ts(\pt)$ coming from the inclusion $\Ts\to \Gs$ is simply the inclusion $\C[\ttt]\to \C[\ttt]^\Ws$.

\begin{example}
 Let $\Gs = \GL_n(\C)$. Then there is a maximal torus $\Ts \subset \Gs$ consisting of all the diagonal matrices. It is of rank $n$ and hence $H^*_\Ts(\pt) = \C[\eta_1,\eta_2,\dots,\eta_n]$, where $\eta_i$ is the functional which maps a diagonal matrix to its $(i,i)$ entry. The normaliser $N(\Ts)$ equals $\Ts\cdot \Ps$, where $\Ps$ is the group of permutation matrices. Hence the action of $\Ws$ on $\ttt^*$ permutes $\eta_1$, $\eta_2$, $\dots$, $\eta_n$. Therefore $H^*_\Gs(\pt) = \C[\ttt]^{\Sigma_n} = \C[e_1,e_2,\dots,e_n]$, where $e_i$ is the $i$-th elementary symmetric polynomial in the variables $\eta_1$, $\eta_2$, \dots, $\eta_n$.
 
We will see in Section \ref{kostsecsec} that $H^*_\Gs(\pt)$ is actually isomorphic to a polynomial ring, for any connected linear group $\Gs$. By Chevalley's restriction theorem for $\Gs$ semisimple \cite[Theorem 3.1.38]{ChGi} this ring is also equal to the ring of $\Gs$-invariant functions $\C[\geg]^\Gs$ on $\geg$. Below, in Lemma \ref{genrest} we prove that this is the case for a wider class of linear groups.
\end{example}

In fact, a similar statement is true for any space $X$ \cite[Chapter III, Proposition 1]{equiv}.

\begin{theorem} \label{cohoweyl}
Let $\Gs$ be a fixed algebraic group, $\Ts$ a maximal torus and $\Ws$ the corresponding Weyl group. Then $\Ws$ acts naturally on $H^*_\Ts(X)$ for all topological spaces $X$ and
 $$ H^*_\Gs(X) = H^*_{N(\Ts)}(X) = H^*_\Ts(X)^\Ws. $$
\end{theorem}

Throughout the thesis we will also use the notation $H^*_\Gs$ for $H^*_\Gs(\pt)$.

\subsection{Equivariant formality} \label{seceqform}
As we mentioned before, if we fix an algebraic group $\Gs$, there is a natural map $H^*_\Gs(X)\to H^*(X)$. However, in general there is no way to determine one of those rings from the other. We are however provided with the Serre spectral sequence \cite[Section 2.2]{Kochman}. As $X_\Gs$ is a fiber bundle over $B\Gs$ with a fiber $X$, it has the form

$$E_2^{p,q} = H^p(B\Gs, H^q(X)) \implies H^{p+q}_\Gs(X).$$

Let $\Gs = \Ts$ be an algebraic torus. Then we call a space $X$ with an action of $\Ts$ \emph{equivariantly formal} \cite{GKM} if this spectral sequence collapses on $E^2$. Equivariant formality is in a sense the opposite of the action being free. For a free action, we have $H^*_\Ts(X) = H^*(X/\Ts)$ and if $X$ is a sufficiently nice, finite dimensional space, this is a finite-dimensional vector space over $\C$. However, if the action is equivariantly formal and $X$ is nonempty, then its equivariant cohomology is infinite-dimensional. Let $\I\subset H^*_\Ts$ be the maximal ideal generated by the elements of positive degree -- so that $H^*_\Ts/\I \simeq \C$.

\begin{theorem}[\cite{GKM}]\label{thmfor}
Any space whose odd cohomology vanishes is equivariantly formal. If $X$ is equivariantly formal, then
\begin{align}\label{formal}H_{\Ts}^*(X) \cong H_{\Ts}^*(\pt)\otimes H^*(X)\end{align} as $H_{\Ts}^*(\pt)$-modules and $H^*(X) \cong H_{\Ts}^*(X)/\I H_{\Ts}^*(X)$ as $\C$-algebras.
\end{theorem}

In fact, any smooth complex complete variety under the action of a linear algebraic group is equivariantly formal \cite[Theorem 1]{Webform}. The above theorem means that the equivariant cohomology ring of an equivariantly formal space $X$ carries within it the information about the non-equivariant cohomology.

\subsection{Localisation and GKM spaces}
A part of the importance of equivariant cohomology comes from the power of localisation to the fixed points. This was pioneered in the influential paper of Atiyah and Bott \cite{AB}. For simplicity, let us consider an algebraic (or compact) torus $\Ts$ acting on a topological space $X$. The inclusion of the fixed-point set $X^\Ts\to X$ induces the \emph{localisation} map on equivariant cohomology
$$H^*_\Ts(X) \to H^*_\Ts(X^\Ts) = H^*(X) \otimes_\C H^*_\Ts.$$
In \cite{AB} it is proved that if $X$ is a compact manifold, then one can recover the class in $H^*_\Ts(X)$ from its localisation to $H^*_\Ts$, and in particular compute its integral in $H^*_\Ts$ easily. This is of a particular use when $H^*_\Ts$ is a discrete set, as then $H^*_\Ts(X^\Ts)$ is just a product of polynomial rings.

In fact, in \cite[Theorem 1.6.2]{GKM} a more general statement is proven.

\begin{theorem}
If $X$ is an equivariantly formal space with an action of a torus $\Ts$, then the restriction map $H^*_\Ts(X)\to H^*_\Ts(X^\Ts)$ is injective.
\end{theorem}

The cokernel of the localisation map is also identified in terms of $0$- and $1$-dimensional orbits of $X$. In general, it is hard to compute, but in a particular case of \emph{GKM spaces} there is an easy combinatorial description.

\begin{definition}
A projective algebraic variety with an action of an algebraic torus $\Ts$ is called a \emph{GKM space} if the action has finitely many fixed points and finitely many 1-dimensional orbits.
\end{definition}

\begin{theorem}[Goresky--Kottwitz--MacPherson, \cite{GKM}, 1.2.2]\label{gkmres}
Assume that a torus $\Ts$ acts on a smooth GKM space $X$. Denote the $\Ts$-fixed points by $\zeta_1$, $\zeta_2$, \dots, $\zeta_s$ and the one-dimensional orbits by $E_1$, $E_2$, \dots, $E_{\ell}$. The closure of any $E_i$ is an embedding of $\PP^1$ and contains two fixed points $\zeta_{i_0}$ and $\zeta_{i_\infty}$, which for any $x\in E_i$ are equal to the limits $\lim_{t\to 0} t x$ and $\lim_{t\to\infty} tx$. The action of $\Ts$ on $E_i$ has a kernel of codimension $1$, which is uniquely determined by its Lie algebra $\kek_i$.

Then the restriction $H_\Ts^*(X,\C)\to H_\Ts^*(X^T,\C) \cong \C[\ttt]^s$ is injective and its image is
$$
H = \left\{
(f_1,f_2,\dots,f_s)\in \C[\ttt]^s  \, \bigg| \, f_{i_0}|_{\kek_i} = f_{i_\infty}|_{\kek_i} \text{ for } i=1,2,\dots,\ell
\right\}.
$$
\end{theorem}

The assumptions are satisfied for many classical varieties one encounters in algebraic geometry.

\begin{example}
 Let $X$ be a smooth projective toric variety \cite{Fultor}. Then an algebraic torus $\Ts$ acts on $X$ with a dense open orbit, and there is finitely many orbits of $\Ts$ in $X$. In particular, this means that $X$ is a GKM space. The variety is fully described by its fan, which partitions into cones the real vector space $M_\R = \Hom(\Ts,\Cs)\otimes_\Z \R$ spanned by the characters of $\Ts$. The top-dimensional cones correspond to fixed points, and the 1-dimensional orbits correspond to codimension 1 fans. From the description above it follows that $H_\Ts^*(X,\C)$ can be identified with the (complex) piecewise polynomial functions on the fan; for details see \cite[section 2.2]{Briontor}.
\end{example}

\begin{example}\label{exflagv}
For any linear algebraic group $\Gs$, its projective homogeneous spaces $\Gs/\Ps$ are all GKM spaces. A description of the GKM data can be found in \cite{Guill}. The equivariant cohomology of the homogeneous spaces can be however also computed differently. Indeed, we have
$$(\Gs/\Ps)_\Gs = E\Gs\times^\Gs (\Gs/\Ps) = (E\Gs\times \Gs)/(\Gs\times \Ps)$$
where $\Gs$ acts diagonally with the left translation action on $\Gs$, and $\Ps$ acts on $\Gs$ on the right. We can quotient first by $\Gs$ and then we are left with $E\Gs$ with the right action of $\Ps$. As $E\Gs$ can be viewed as a model for $E\Ps$, we get $(\Gs/\Ps)_\Gs \simeq B\Ps$, hence
$$H^*_\Gs(\Gs/\Ps) = H^*_\Ps.$$
The structure of $H^*_\Gs$-algebra on the right comes simply from the quotient map $B\Ps\simeq E\Gs/\Ps \to E\Gs/\Gs \simeq B\Gs$, hence on the algebra level it is an inclusion $\C[\ttt]^{\Ws}\to \C[\ttt]^{\Ws_\Ps}$. Therefore abstractly the equivariant cohomology is just an affine space, but it has a nontrivial structure of an algebra over another polynomial ring, the equivariant cohomology of the point.

If $\Gs = \GL_n$, there is a very elegant description of the equivariant cohomology of $\Gs/\Ps$ in terms of generators and relations \cite[Corollary 5.4]{AndFul}. Such a flag variety, depending on the choice of $\Ps$, parametrises the flags $0 = V_{0} \subset V_{1} \subset V_{2} \subset\dots \subset V_{k} \subset V_{k+1} \subset \C^n$, where $\dim V_j = i_j$ for some fixed subsequence $(i_1,i_2,\dots,i_k)$ of $(1,2,\dots,n-1)$. The algebra $H^*_\Gs(\Gs/\Ps)$ over $\C[e_1,e_2,\dots,e_n]$ is generated by the Chern classes of the quotient bundles with fibers $V_{j}/V_{j-1}$ (see Example \ref{exflag2}). If for simplicity we denote $d_j = i_j - i_{j-1}$, then this gives Chern classes $c_1^j, \dots, c_{d_j}^j$ for $j=1,2,\dots,k+1$. The relations are the ones coming from comparing both sides of the polynomial equation
$$ \prod_{j=1}^{k+1}  (1 + c_1^j x + c_2^j x^2 + \dots + c_{d_j}^j x^{d_j}) = 1+e_1 x + e_2 x^2 + \dots e_n x^n.$$  
\end{example}

\begin{remark}
It is worth mentioning another, modern approach to equivariant cohomology, which involves quotient stacks, i.e. we take $H^*_\Gs(X) = H^*[X/\Gs]$. We do not use this approach in this thesis at all. An interested reader is referred to \cite{Behrend} for details.

Another classical definition of equivariant cohomology, if $X$ is a $\Gs$-manifold, is the equivariant de Rham complex. This is analogous to de Rham construction of ordinary cohomology. A detailed construction can be found in the classical monograph \cite{deRham}.
\end{remark}

\section{Linearised vector bundles and linearisation of the action} \label{linsec}

\subsection{Equivariant sheaves and bundles}

Once an algebraic group $\Gs$ acts on an algebraic variety $X$, we can define a notion of a $\Gs$-linearised, or equivariant, vector bundle on $X$, and more generally of a $\Gs$-equivariant coherent sheaf  \cite[2.2.1]{Merkur}. Recall that a vector bundle on $X$ is a locally free sheaf of $\OO_X$-modules.

\begin{definition}
 Let $\Ee$ be a coherent sheaf on $X$. Consider two maps $\rho,\pi_2:\Gs\times X\to X$ being the action of $\Gs$ and the projection on the second factor, respectively.
 
 We say that $\Ee$ is a \emph{$\Gs$-equivariant sheaf}, or shortly a \emph{$\Gs$-sheaf}, if we are given an isomorphism $\phi$ of $\OO_{\Gs\times X}$-modules
 $$\phi: \pi_2^* \Ee \to \rho^* \Ee,$$
 for which the following agreement of maps of $\OO_{\Gs\times\Gs\times X}$-modules holds:
 $$(\mu \times \id_X)^* \phi = (\id_\Gs\times \rho)^*\phi \circ \pi_{23}^*\phi.$$ 
 Here $\mu:\Gs\times \Gs\to \Gs$ is the multiplication, $\pi_{23}:\Gs\times\Gs\times X\to \Gs\times X$ is the projection to the product of the second and the third factor.
\end{definition}

\begin{remark}
In case $\Ee$ is a $\Gs$-equivariant vector bundle, we also call it \emph{$\Gs$-linearised}. For a vector bundle, one can equivalently work with the total space $E$ and require that $\Gs$ acts on $E$ so that the projection map $E\to X$ is equivariant, and for any $g\in \Gs$ and $x\in X$ the corresponding map on the fibers $E_x\to E_{gx}$ is linear \cite[3.2]{Brionlin}.
 
If $\Ee$ is a vector bundle, one should think of $\phi$ in the following way. The sheaf $\pi_2^*\Ee$ is a locally free sheaf on $\Gs\times X$, whose fiber over $(g,x)$ is $\Ee_{x}$. The sheaf $\rho^*\Ee$ has $\Ee_{gx}$ as a fiber over $(g,x)$. The isomorphism $\phi$ gives then an isomorphism $\Ee_{x}\to \Ee_{gx}$ on the fibers, depending in an algebraic way on $g$ and $x$. The condition from the definition ensures that this gives an action of $\Gs$ on the total space, i.e. for any $(g,h,x)\in\Gs\times\Gs\times X$ the action of $gh$ on $\Ee_{x}$ yields the same map as the composite of the action $h$ on $\Ee_{x}$ and of $g$ on $\Ee_{hx}$.
\end{remark}

Note that if $X$ is a point, then the $\Gs$-linearised bundles over $X$ are precisely the representations of $X$. For any smooth $\Gs$-variety, its tangent bundle has the canonical structure of a $\Gs$-bundle. The map $\phi$ on the fibers is defined via the differential of the action of $\phi$. A direct sum, a tensor product of $\Gs$-bundles and an exterior or symmetric power of a $\Gs$-bundle is a $\Gs$-bundle, and so is the quotient of a $\Gs$-bundle by its subbundle which is also a $\Gs$-bundle. This means that if $Y\subset X$ is an inclusion of smooth $\Gs$-varieties, then the normal bundle of $Y$ is a $\Gs$-vector bundle on $Y$.

\subsection{Linearisation of the action}

Assume that a $\Gs$-linearised vector bundle $\Ee$ on a $\Gs$-variety $X$ is very ample. On the space of global sections $V = \Gamma(X,\Ee)$ we have the representation $\rho:\Gs\to \GL(V)$. Then $\Ee$ being very ample implies that the map $X\to \PP(V^*)$ is a $\Gs$-invariant embedding. Therefore we are able to embed $X$ in a projective space with a linear action of $\Gs$. It turns out one can always find a suitable vector bundle under mild assumptions \cite[Theorem 7.3]{Dolgachev}.

\begin{theorem}
Assume that $X$ is a quasi-projective normal algebraic variety and a connected algebraic group $\Gs$ acts on $X$. Then there exists a $\Gs$-equivariant embedding $X\hookrightarrow \PP^n$ for some integer $n$, where $\Gs$ acts on $\PP^n$ by a linear representation $\Gs\to \GL_{n+1}$.
\end{theorem}

\subsection{Equivariant K-theory} \label{seceqk}
Let $\Gs$ be a linear group and let $X$ be a $\Gs$-variety. One can then define the equivariant algebraic K-theory groups of $X$, i.e. $K_i^\Gs(X)$, defined by the $\Gs$-sheaves, and $K^i_\Gs(X)$, defined by the $\Gs$-linearised vector bundles. In this thesis we are only concerned with $K^0$ and $K_0$. A reader interested in higher K-theory might want to consult sources like \cite{Merkur}, \cite{Thomason}. 

\begin{definition}
We define the equivariant K-theory of coherent sheaves $K_0^\Gs(X)$ as the Grothendieck group of the abelian category of coherent $\Gs$-sheaves on $X$. Similarly, we define the equivariant K-theory of vector bundles $K^0_\Gs(X)$ as the Grothendieck group of the exact category of $\Gs$-linearised vector bundles on $X$.
\end{definition}
We recall that the Grothendieck group $K(\mathcal C)$ of an exact category $\mathcal C$ is the free abelian group generated by the objects of $\mathcal C$, divided by the subgroup generated by the expressions $[A] - [B] + [C]$ for all short exact sequences $0\to A\to B\to C\to 0$.

Note that $K^0_\Gs(X)$ is a contravariant functor on the category on $\Gs$-varieties. Any $\Gs$-equivariant morphism $X\to Y$ gives rise to the pullback map between the categories of $\Gs$-linearised vector bundles. This gives a ring morphism $K^0_\Gs(Y) \to K^0_\Gs(X)$, where the ring structure is given by the tensor product of the vector bundles. In particular, for any $\Gs$-variety $X$, the ring $K^0_\Gs(X)$ is an algebra over $K^0_\Gs(\pt)$, which is the representation ring $R(\Gs)$ of the group $\Gs$ \cite[23.2]{FulHar}. Note also that if $X$ is an $\Hs$-variety and we are given a morphism of algebraic groups $\Gs\to \Hs$, this makes $X$ into a $\Gs$-variety and yields a ring morphism $K^0_\Hs(X)\to K^0_\Gs(X)$. In particular, for $\Gs$ being trivial, it gives the restriction map $K^0_\Hs(X)\to K^0(X)$ to ordinary K-theory.

\begin{example}
If we take $\Gs = \Cs$, the representation ring is $\Z[t,t^{-1}]$, where $t$ corresponds to one-dimensional representation of the weight 1. Now consider the standard $\Cs$ action on $\PP^1$, i.e. $t\cdot[x,y] = [tx,y]$. Then $K_0^{\Cs}(\PP^1) = K_{\Cs}^0(\PP^1) = \Z[t,t^{-1}][x]/(x-1)(x-t)$, where $x$ is the class of the tautological bundle. This is a special case of the projective bundle theorem for equivariant K-theory, \cite[Theorem 3.1]{Thomason}.
\end{example}

\begin{example}
For any reductive group $\Gs$ and a parabolic subgroup $\Ps$ we have 
$$K_0^\Gs(\Gs/\Ps) = K^0_\Gs(\Gs/\Ps) = R(\Ps) = R(\Ls),$$
where $\Ls$ is a Levi subgroup of $R(\Ls)$.
\end{example}
In fact, the equality of $K_0$ and $K^0$ is not an accident in those cases, and this follows from a general theorem \cite[Corollary 5.8(5)]{Thomason}.

\begin{theorem}\label{ksmcoh}
 Let $\Gs$ be a linear group acting on a smooth variety $X$. Then the inclusion of the category of vector bundles in the category of coherent sheaves yields the isomorphism $K^0_\Gs(X)\to K_0^\Gs(X)$.
\end{theorem}

Moreover, if $X$ is a smooth projective variety, then the torus-equivariant K-theory turns out to contain almost all the information about potential equivariant K-theory rings.

\begin{theorem}\cite[Corollary 6.7]{HLS}
Assume that $\Gs$ is a connected reductive group, $\Ts \subset \Gs$ is its maximal torus and let $X$ be a smooth projective $\Gs$-variety. Then the Weyl group $\Ws$ of $\Gs$ acts on $K^0_\Ts(X)$, and the canonical map $K^0_\Gs(X)\to K^0_\Ts(X)$ yields the isomorphism $K^0_\Gs(X) \simeq K^0_\Ts(X)^\Ws$.
\end{theorem}

\begin{theorem} \label{thmlevik}
Assume that a complex linear algebraic group $\Hs$ is acting on a complex variety $X$. Let $\Ls \subset \Hs$ be a Levi subgroup of $\Hs$. Then the restriction map
$$K_\Hs^0(X) \to K_\Ls^0(X)$$
is an isomorphism.
\end{theorem}

\begin{proof}
By \cite[6.2]{Thomason}, the restriction along $X\to \Hs\times^\Ls X$ induces an isomorphism
$$K_\Hs^0(\Hs\times^\Ls X)\to K_\Ls^0(X).$$
Now $\Hs\times^\Ls X$ maps $\Hs$-equivariantly to $X$ (simply by $[(h,x)] \mapsto hx$) and we will show that this map induces an isomorphism on $K_\Hs^0$. Let $\Ns$ be the unipotent radical of $\Hs$, so that $\Hs = \Ns \rtimes \Ls$. Then we have the $\Hs$-equivariant isomorphism
$$\Hs\times^\Ls X \simeq \Ns \times X,$$
where $\Hs$ acts on $\Ns\times X$ diagonally by conjugation and action.

Indeed, every element of $\Hs$ is uniquely decomposed as $ul$ for $u\in \Ns$, $l\in \Ls$. This means that $\Hs\times^\Ls X \simeq \Ns\times X$. Now we need to see how the $\Hs$ acts on this product. Note that in $\Hs\times^\Ls X$ we have
$$h\cdot[(u,x)] = [(hu,x)] = [(huh^{-1},hx)],$$
and as $huh^{-1}\in \Ns$, upon identification with $\Ns\times X$ we have $h\cdot(u,x) = (huh^{-1},hx)$.

We want to prove that the map $\Hs\times^\Ls X\to X$ induces an isomorphism on $K_\Hs^0$. But we have proved that in fact it is a map $\Ns\times X\to X$. Note that it is not the projection, but the action of $N$ on $X$. However, we can split it into the isomorphism $N\times X\to N\times X$ given by $(u,x)\mapsto (u,ux)$, and the projection. Note that this isomorphism is in fact $\Hs$-invariant, as 
$$h\cdot (u,ux) = (huh^{-1}, hux) = (huh^{-1}, huh^{-1} hx).$$
Therefore we have to show that the projection $\Ns\times X \to X$ yields an isomorphism on $K_\Hs^0$.

Now note that by Theorem \ref{algexp} the exponential map $\exp:\nen\to \Ns$ is an isomorphism of schemes, so in fact $\Ns\times X \simeq \nen\times X$ has a structure of a (trivial) vector bundle over $X$. Note that $\Hs$ acts on it linearly. Indeed, we have $h\exp(v) h^{-1} = \exp(h v h^{-1})$ and the adjoint representation of $\Hs$ on $\nen$ is linear. Then by \cite[4.1]{Thomason} the projection $\Ns\times X\to X$ gives an isomorphism on $K_\Hs^0$.
\footnote{The author would like to thank Andrzej Weber for the idea of this proof.}
\end{proof}

\subsection{Equivariant Chern classes}\label{seceqchern}
We assume the reader is familiar with the classical theory of characteristic classes, in particular Chern classes, as described e.g. in the classical book \cite{CharCl}. We want to sketch the equivariant part of the story here. We only work with algebraic equivariant vector bundles, as introduced earlier in this section -- however the Chern classes in cohomology can be equally well defined for topological bundles. This means that in fact we study the composition
$$ K^0_\Hs(X)^{alg} \to K^0_\Hs(X)^{top} \to H^*_\Hs(X).$$

One can consider the Chern classes purely algebraically and see them in the Chow ring of the variety $X$. However, here we only discuss the Chern classes in equivariant cohomology. Let $\Ee$ be a $\Gs$-linearised rank $n$ vector bundle on a $\Gs$-variety $X$ with the total space $E$. As $\Gs$ acts on $E$ and the projection $E\to X$ is $\Gs$-equivariant, this gives a continuous map $E_\Gs \to X_\Gs$, which in fact is also a rank $n$ (complex) vector bundle. Then we can consider its Chern classes in $H^*(X_\Gs) = H^*_\Gs(X)$, we use the notation $c_i^\Gs(\Ee) = c_i(E_\Gs)$ and we call these classes the \emph{equivariant Chern classes} of the bundle $\Ee$. Their sum $c^\Gs(\Ee) = \sum_{i=0}^\infty c_i^\Gs(\Ee)$ is the \emph{total equivariant Chern class} of $\Ee$.

Straight from the definition we see that the equivariant Chern classes satisfy the conditions analogous to those for ordinary Chern classes:
\begin{enumerate}
\item If $\Ee$ is of rank $n$, then $c_i^\Gs(\Ee)$ vanishes for $i>n$. Moreover $c_0^\Gs(\Ee) = 1$ for any $\Gs$-linearised bundle $\Ee$. For any $i\ge 0$ we have $c_i^\Gs(\Ee) \in H^{2i}_\Gs(\Ee)$;
\item If $f:X\to Y$ is a $\Gs$-equivariant map of $\Gs$-varieties and $\Ee$ is a $\Gs$-linearised vector bundle on $Y$, then 
$$ c^\Gs(f^*\Ee) = f^*c^\Gs(\Ee).$$
On the left hand-side $f^*\Ee$ means the pullback of $\Ee$ to $Y$ along $f$, and on the right $f^*$ is the pullback map on equivariant cohomology;
\item If
$$0\to \Ee \to \F \to \Geg \to 0 $$
is an exact sequence of $\Gs$-linearised vector bundles on a $\Gs$-variety $X$, then 
$$c^\Gs(\F) = c^\Gs(\Ee) \cdot c^\Gs(\Geg).$$
\end{enumerate}

Moreover, for any $\Gs$-linearised vector bundle $\Ee$ on $X$, its ordinary Chern classes $c_i\in H^*(X,\C)$ are the restrictions of the equivariant ones under the canonical map $H^*_\Gs(X,\C)\to H^*(X,\C)$.

\begin{example}\label{exflag2}
Let $\Gs = \GL_n$. Consider $X = \Gs/\Ps$, a flag variety of flags of type $(i_1,i_2,\dots,i_k)$ in $\C^n$. The trivial rank $n$ bundle $\F \simeq X\times \C^n$ on $X$, with the standard representation of $\GL_n$ on $\C^n$, has a sequence of $\Gs$-equivariant subbundles $\Ee_0 \hookrightarrow \Ee_1 \hookrightarrow \Ee_2 \hookrightarrow \dots \hookrightarrow \Ee_k \hookrightarrow \Ee_{k+1}$, with $\Ee_j$ of rank $i_j$ (we let $i_0 = 0$). A fiber of $\Ee_j$ over the flag $0 = V_0 \subset V_{1} \subset V_{2} \subset\dots\subset V_{k}\subset V_{k+1} = \C^n$ is the vector space $V_j$. 

As in Example \ref{exflagv}, we denote by $c_1^j, \dots, c_{d_j}^j$ the Chern classes of $\Ee_j/\Ee_{j-1}$. By the properties described above, we have 
$$ \prod_{j=1}^{k+1}  (1 + c_1^j + c_2^j  + \dots + c_{d_j}^j) = c(\F).$$
One then checks that $c(\F) = 1 + e_1 + e_2 + \dots e_n $ and this means that the Chern classes satisfy the relation from Example \ref{exflagv}.  
\end{example}

Now that we have defined equivariant Chern classes, we can also define the Chern characters in the same way as for ordinary Chern classes \cite[\S 10]{Hirzebruch}. In other words, the Chern character of a $\Gs$-linearised vector bundle $\Ee$ of rank $n$ is equal to the formal sum $\ch^\Gs(\Ee) = \sum_{i=1}^n \exp(\gamma_i)$, where $\gamma_i\in H^2_\Gs(X)$ are the Chern roots of $\Ee$ so that $c(\Ee) = \prod_{i=1}^n (1+\gamma_i)$. For each $i$, the $2i$-th degree component $\ch^\Gs_i(\Ee)$ is in fact a polynomial in $c_1^\Gs(\Ee)$, $c_2^\Gs(\Ee)$, \dots, $c_i^\Gs(\Ee)$, dependent only on $i$ and $n$ (not on $\Ee$). One has to be careful, as the equivariant cohomology is usually nontrivial in arbitrarily high degrees; therefore $\ch^\Gs(\Ee)$ does not lie in $H^*_\Gs(X) = \bigoplus_{i=0}^\infty H^i_\Gs(X)$, but in the completion $\hat{H}^*_\Gs(X) = \prod_{i=0}^\infty H^i_\Gs(X)$. Note also that in this step it is crucial that the cohomology coefficients contain $\Q$.

So defined equivariant Chern character satisfies the following conditions:
\begin{itemize}
\item If $0\to \Ee \to \F \to \Geg \to 0$ is an exact sequence of $\Gs$-linearised vector bundles, then $\ch^\Gs(\F) = \ch^\Gs(\Ee) + \ch^\Gs(\Geg)$;
\item For any two $\Gs$-linearised vector bundles $\Ee$, $\F$, we have $\ch^\Gs(\Ee\otimes \Geg) = \ch^\Gs(\Ee)\cdot \ch^\Gs(\Geg)$.
\end{itemize}

Therefore, the Chern character is in fact a ring homomorphism
$$\ch^\Gs: K_\Gs^0(X) \to \hat{H}^*_\Gs(X).$$
If $X$ is smooth, then by Theorem \ref{ksmcoh} $K_\Gs^0(X) \simeq K^\Gs_0(X)$, which means we have a well-defined Chern character for any coherent $\Gs$-sheaf. Going backwards, we can produce this way the Chern classes of any sheaf. To make sure that they terminate, one can resolve the sheaf by vector bundles, and then its Chern classes must come from an alternating product of the Chern classes for the resolving bundles. By Theorem \ref{ksmcoh}, the result will not depend on the choice of the resolution.

\section{Graded Nakayama lemma}
For the sake of completeness we provide here the proof of the version of the graded Nakayama Lemma that we will need (see also \cite[Corollary 4.8, Exercise 4.6]{Eis}).

\medskip
\noindent Let $R$ be a $\Z_{\ge 0}$-graded ring $R = \bigoplus_{n\ge 0} R_n$ and $I = \bigoplus_{n > 0} R_n$ the ideal generated by elements of positive degree.

\begin{lemma}
 If a $\Z_{\ge 0}$-graded $R$-module $M$ satisfies $M = IM$, then $M = 0$.
\end{lemma}
\begin{proof}
 Suppose on contrary that $a\in M$ is a nonzero homogeneous element of minimal degree $d\in\Z_{\ge 0}$ present in $M$. By assumption $M = IM$ we have that
 $$a = \sum_{i=1}^k r_i a_i$$
 for some $r_i\in I$, $a_i\in M$. But as $r_i\in I$, the minimal degree present in $r_i$ is at least $1$. As $a_i\in M$, the minimal degree present in $a_i$ is at least $d$. Therefore the elements $r_ia_i$ have zero part in degrees less than $d+1$. In particular, we cannot get $a$ as a sum of them, as it has nonzero part in degree $d$.
\end{proof}

\begin{corollary}\label{cornak}
Let $M$ be a $\Z_{\ge 0}$-graded $R$-module $M$. Suppose that elements $(a_j)_{j\in J}$ of $M$ generate the $R/I$-module $M/IM$. Then they generate $M$ as $R$-module.
\end{corollary}

\begin{proof}
Let us consider the map of $R$-modules $\phi:R^J\to M$ defined by the elements $a_j$. We have the exact sequence
$$R^J\xrightarrow{\phi} M \to \coker \phi \to 0.$$
As tensor product is right exact, by tensoring with $R/I$ we get an exact sequence of $R/I$-modules:
$$(R/I)^J\to M/IM \to (\coker \phi)\otimes_R R/I \to 0.$$
By assumption the first map is an epimorphism, hence $(\coker \phi)\otimes_R R/I =0$. In other words, $\coker \phi$ satisfies the conditions of lemma. Therefore $\coker\phi = 0$, hence $\phi$ is surjective.
\end{proof}

\chapter{Regular elements and Kostant sections}\label{chapkos}
As we mentioned in Section \ref{seceqcoh}, for a connected linear group $\Gs$ all the equivariant cohomology rings $H_\Gs^*(X)$ are algebras over the coordinate ring $H_\Gs^* \simeq \C[\ttt]^\Ws$, where $\ttt$ is the Lie algebra of a maximal torus $\Ts\subset \Gs$ and $\Ws = N(\Ts)/\Ts$ is the corresponding Weyl group. We also refer to Chevalley's restriction theorem which claims that the restriction $\C[\geg]^\Gs \to \C[\ttt]^\Ws$ is an isomorphism. The ring $H_\Gs^*$ turns out to always be a polynomial ring. Moreover, Kostant shows in \cite{Kostsec} that one can explicitly define an affine space $\Ss\subset \geg$ such that the restriction $\C[\geg]^\Gs \to \C[\Ss]$ is an isomorphism. The space $\Ss$ is referred to as the \emph{Kostant section}. On the level of spectra, this means that the map $\Ss \to \geg/\!\!/\Gs$ is an isomorphism. In particular, every conjugacy orbit in $\geg$ corresponds to some element in $\Ss$. However, in the GIT quotient $\geg/\!\!/\Gs$ some of the conjugacy orbits get identified, and one cannot expect an arbitrary orbit to contain an element of $\Ss$. It is however true if we restrict to \emph{regular} orbits, i.e. to elements with the minimal possible centraliser. In case of $\Gs = \GL_{n+1}$, this means that all the eigenspaces are 1-dimensional, and hence equivalently the corresponding vector field on $\PP^n$ has finitely many zeros.

In this section, we recall the notion of regular elements and the results of Kostant. We generalise them to a wider class of \emph{principally paired} groups which contains all the parabolic subgroups of reductive groups. We conclude the chapter by defining \emph{regular actions} and showing how the regular elements behave in such a setting.

\section{Regular elements} \label{secreg}
Let $\Hs$ be an algebraic group and $\Ts\subset \Hs$ be a maximal torus, of dimension $r$. We will call an element $v\in \he = \Lie(\Hs)$ \emph{regular} if $\dim C_\he(v) = r$. 
This is stronger than the usual notion of a regular element in the literature (see e.g. \cite{Chev3}) -- an element whose centraliser has minimal possible dimension. All the centralisers have dimension not smaller than $r$, but it is possible that no regular element exists. For example for $\Hs = \Gs_m \times \Gs_a$ -- the product of the multiplicative and the additive group -- all centralisers are $2$-dimensional, but the maximal torus is of dimension $1$.

Nilpotent regular elements are also sometimes refered to as \emph{principal nilpotents}.

\begin{example}
 For $\Hs = \GL_n(\C)$ or $\Hs = \SL_n(\C)$, a regular element of $\he$ is a matrix with all eigenspaces of dimension $1$. For example, among the following matrices in $\gl_4(\C)$, the first two are regular, the third is not:
 $$
 \begin{pmatrix}
 2 & 0 & 0 & 0 \\
 0 & 1 & 0 & 0 \\
 0 & 0 & 3 & 0 \\
 0 & 0 & 0 & 7
 \end{pmatrix}, \qquad
 \begin{pmatrix}
 0 & 0 & 0 & 0 \\
 0 & 2 & 0 & 0 \\
 0 & 0 & 1 & 0 \\
 0 & 0 & 1 & 1 
 \end{pmatrix}, \qquad
 \begin{pmatrix}
 0 & 1 & 0 & 0 \\
 0 & 0 & 2 & 0 \\
 0 & 0 & 0 & 0 \\
 0 & 0 & 0 & 0 
 \end{pmatrix}.
 $$
\end{example}

\begin{example}\label{exregnilp}
More generally any reductive group $\Gs$ contains regular elements in its Lie algebra, in particular a regular nilpotent element. Indeed, once we choose a maximal torus $\Ts\subset \Gs$ and positive roots, we can take $e = x_1 + x_2 + \dots + x_s$, where $x_1$, $x_2$, \dots, $x_s$ are the root vectors of $\geg$ corresponding to the positive simple roots ($s=r - \dim Z(\Gs)$). Then $e$ is a regular nilpotent in $\Gs$ (see \cite[Section 4, Theorem 4]{Kostsec}).
\end{example}

 The condition $\dim C_{\he}(w) > r$ is a Zariski-closed condition on $w$ -- as it means that $[w,-]$ has sufficiently small rank, which amounts to vanishing of some minors of a matrix. Therefore, if $\Hs$ admits a regular element in its Lie algebra, the subset of regular elements $\he^\reg\subset \he$ is open and dense.

Note that if $\Hs$ is solvable, then by Theorem \ref{solv} we have $[\he,\he]\subset \he_n$. This means that for any $v\in\he$ we have $[v,\he]\subset \he_n$. As the codimension of $\he_n$ is exactly $r = \dim \Ts$, the dimension of maximal torus, $v$ being regular is equivalent to $[v,\he] = \he_n$.

Note also that if $\Hs'\subset \Hs$ is a subgroup which contains a maximal torus $\Ts$ of $\Hs$, then any regular $v\in \he$ contained in $\he'$ is also regular in $\he'$. This means in particular that the centraliser $C_\he(v)$ is contained in $\he'$.

\section{\texorpdfstring{$\ssl_2$-triples and $\bb(\ssl_2)$-pairs}{sl2-triples and b(sl2)-pairs}}
\label{sl2p}
The classical version of Carrell--Lieberman theorem \cite[Main Theorem and Remark 2.7]{CL} deals with an arbitrary vector field $V$ on a smooth projective variety $X$, which vanishes in a discrete, nonempty set. They prove the following
 \begin{theorem}
  Let $X$ be a smooth projective complex variety and $V$ a vector field with finitely many zeros and denote its zero scheme by $Z$. Then there exists an increasing filtration $F_\bullet$ on $\C[Z]$ such that
  $$H^*(X) \simeq Gr_F(\C[Z]).$$
  The degree on the left is multiplied by two, in particular $X$ only has even cohomology.
 \end{theorem}
 The theorem therefore gives some information on cohomology, but this depends on determining the filtration $F_\bullet$. This can be hard in general. Only if $V$ comes with a $\Cs$-action which satisfies $t_*(V) = t^k V$ for some nonzero integer $k$, we get $H^*(X) \cong \C[Z(V)]$ (\cite{ACLS}, \cite[Theorem 1.1]{AC}). We sketch the details in Section \ref{seccarli}. In our situation, we will consider the vector fields as coming from an action of a Lie group. Hence the following definition.

\begin{definition}
 For any complex Lie algebra $\he$, by \emph{$\bb(\ssl_2)$-pair} in $\he$ we mean a pair $(e,h)$ of elements of $\he$ that satisfy the condition $[h,e] = 2e$. By \emph{$\ssl_2$-triple} in $\he$ we mean a triple $(e,f,h)$ of elements of $\he$ such that $[h,e]=2e$, $[h,f] = -2f$, $[e,f] = h$.
\end{definition}

If $\Gs$ is a semisimple group, then by the Jacobson--Morozov theorem (see e.g. \cite[Theorem 3.7.1]{ChGi}) for any nilpotent element $e\in\geg$ there exists an $\ssl_2$-triple $(e,f,h)$ in $\geg$ such that $f$ is nilpotent and $h$ is semisimple. The same is then true for any reductive Lie group $\Gs$, as a reductive Lie algebra is a direct sum of its centre and a semisimple ideal (\cite[Theorem II.11]{Jac}).

Let us consider the connected subgroup $\Ks\subset \Gs$ whose Lie algebra $\mathfrak k$ is the smallest one which contains $e$, $f$, $h$ (see \cite[7.1]{Borel}). Then the Lie algebra of $[\Ks,\Ks]$ is equal to $[\mathfrak k,\mathfrak k]$ (see \cite[Proposition 7.8]{Borel}). However, by \cite[Corollary 7.9]{Borel} we have $[\mathfrak k,\mathfrak k] = [\Span(e,f,h),\Span(e,f,h)] = \Span(e,f,h)$. Hence we get an algebraic subgroup $[\Ks,\Ks]$ (contained in $\Ks$, hence equal to $\Ks$) of $\Gs$ whose Lie algebra is $\Span(e,f,h)$. As its Lie algebra is semisimple, the group itself is semisimple. By \cite[Theorem 20.33]{Milne}, if it is nontrivial, it has to be either $\SL_2(\C)$ or $\PSL_2(\C)$. In either case, there is a covering map \begin{align}\label{phi}\phi:\SL_2(\C)\to \Ks.\end{align} As any automorphism of $\ssl_2(\C)$ lifts to an automorphism of $\SL_2(\C)$, we can assume that the canonical basis $e_0$, $f_0$, $f_0$ of $\ssl_2$ maps to $e$, $f$, $h$, respectively. Hence we get the following.

\begin{proposition}\label{reductive}
 For any nilpotent element $e$ in the Lie algebra $\geg$ of an algebraic reductive group, there exists an $\ssl_2$-triple $(e,f,h)$ within $\geg$ with $f$ nilpotent and $h$ semisimple. If $e\neq 0$, the element $h$ integrates to a map $\Cs\to \Gs$ with discrete kernel, whose differential is $h$.
\end{proposition}

\begin{remark}\label{princesl2}
As we saw in Example \ref{exregnilp}, if $\Gs$ is reductive, then there exists a principal nilpotent $e\in \geg$. By the proposition, this means that there is a an $\ssl_2$ triple $(e,f,h)$ with $e$ principal nilpotent, $f$ nilpotent and $h$ semisimple. By the general theory of representations of $\ssl_2$, the ranks of the operators $[e,-]$ and $[f,-]$ are equal, hence $f$ is also regular. This motivates the following definition.
\end{remark}

\begin{definition}
 An $\ssl_2$-triple $(e,f,h)$ will be called \emph{principal} if $e$ and $f$ are regular nilpotents.
\end{definition}

\begin{definition}
 For a linear algebraic group $\Hs$, an \emph{integrable $\bb(\ssl_2)$-pair} in $\he = \Lie(\Hs)$ is a $\bb(\ssl_2)$-pair $(e,h)$ in $\he$ which consists of a nilpotent element $e$ and a semisimple element $h$ which is tangent to some one-parameter subgroup $H:\Cs\to \Hs$, i.e. $h = dH(1)$. This means that $(e,h)$ comes from an algebraic group morphism $\Bs_2 = \Bs(\SL_2)\to \Hs$. We call an integrable $\bb(\ssl_2)$-pair \emph{principal} if $e$ is a regular element of $\he$.
\end{definition}

\begin{remark}
 Note that, unlike an $\ssl_2$-triple, a $\bb(\ssl_2)$-pair does not have to be integrable. As an easy counterexample, we may take
 $$h = 
 \begin{pmatrix}
  \pi & 0 & 0\\
  0 & \pi-2 & 0 \\
  0 & 0 & 2-2\pi
 \end{pmatrix},
 \qquad
 e = 
 \begin{pmatrix}
  0 & 1 & 0\\
  0 & 0 & 0 \\
  0 & 0 & 0
 \end{pmatrix}
 $$
 for $\Hs = \SL_3(\C)$. Then $[h,e] = 2e$, but $h$ is not tangent to a one-dimensional torus (we can replace $\pi$ with any irrational number).
\end{remark}
\begin{definition} We call a connected linear algebraic group $H$ {\em principally paired} if it contains a principal integrable $\bb(\ssl_2)$-pair. 
\end{definition}

For example a reductive group is principally paired because of Proposition~\ref{reductive}. More generally we have the following

\begin{lemma} \label{parabolic} Let $\Gs$ be a reductive group. Then any parabolic subgroup $\Ps\subset \Gs$ is principally paired.
 \end{lemma}

\begin{proof}
Because there is a Borel subgroup $\Bs\subset \Ps$, it is enough to prove the result for $\Bs=\Ps$. Note that if $\Bs = \Bs_2$ is the Borel subgroup of $\SL_2(\C)$, then the image $\phi(\Bs_2)$ of $\Bs_2$ in the map \eqref{phi} is a solvable connected subgroup of $\Gs$, hence it is contained in a Borel subgroup of $\Gs$. All Borel subgroups of $\Gs$ are conjugate by Corollary \ref{borconj}, hence they are all principally paired.
\end{proof}

\section{Kostant section and generalisations}\label{kostsecsec}
The seminal work of Kostant shows the following theorem (\cite[Theorem 0.10]{Kostsec}).

\begin{theorem}\label{kostant2}
Assume that $\Gs$ is a reductive group and $(e,f,h)$ is a principal $\ssl_2$-triple. 
Then every regular element of $\geg = \Lie(\Gs)$ is conjugate to exactly one element of $\Ss = e + C_\geg(f)$. Moreover, the restriction $\C[\geg]^\Gs \to \C[\Ss]$ is an isomorphism.
\end{theorem}
The affine plane $\Ss$ is called the \emph{Kostant section}. We will provide in Theorems \ref{kostarb} and \ref{restkos} a version that works for arbitrary principally paired groups.

\subsection{Solvable groups}
Assume first that $\Hs$ is a solvable group. Let $\Ts$ be its maximal torus and $\he_n$ be the nilpotent part of $\he = \Lie(\Hs)$. Assume that $e\in\he_n, h\in\ttt$ are such that $(e,h)$ is a principal integrable $\bb(\ssl_2)$-pair. Let $\{H^t\}_{t\in\Cs}$ be the one-parameter subgroup in $\Hs$ to which $h\in\he$ integrates. Recall from Remark \ref{nilalg} that $\he = \ttt\oplus \he_n$.

\begin{lemma}\label{etreg}
 All elements of $e+\ttt$ are regular and not conjugate to one another.
\end{lemma}
\begin{proof}
 Assume that for some $v\in \ttt$ the element $e+v$ is not regular. This means that $\dim C_\he(v) \ge r+1$. As $\Ad_{H^t}(e+v) = t^2 e + v$, for any $t\in \Cs$ we have
 $$\dim C_\he\left(e + v/t^2\right) = \dim C_\he\left(t^2e + v\right) = \dim C_\he (e+v) \ge r+1.$$
 As the set of nonregular elements is closed in $\he$, we get $\dim C_\he(e) \ge r+1$. This contradicts the regularity assumption.
 
 For any $x \in \he$ and $M\in \Hs$ we have $\Ad_M(x) - x \in [\he,\he] \subset \he_n$ by \cite[Propositions 3.17, 7.8]{Borel}. Therefore no two distinct elements from $e+\ttt$ can be conjugate to one another, as they differ on the $\ttt$ component. 
\end{proof}

\begin{lemma}\label{etkost}
\footnote{This is based on an argument provided  by Anne Moreau.}
 Every regular element of $\he$ is conjugate to a unique element of $e+\ttt$.
\end{lemma}
\begin{proof}
 We know that $\he = \ttt \oplus \he_n$. Assume that $x = v + n$, where $v\in\ttt$ and $n\in\he_n$, is regular. This means that $[x,\he] = \he_n$ (see Section \ref{secreg}). Let us consider the map \begin{align} \label{ad} \Ad_{-}(x):\Hs \to \he.\end{align} As in the proof of the previous lemma, we see that the image is actually contained in $v + \he_n$.
 
 Note that the image of the derivative of \eqref{ad} at $1$ is $[x,\he] = \he_n = T_x(v+\he_n)$. Therefore by \cite[Theorem 4.3.6]{Springer} the morphism $\Ad_{-}(x):\Hs \to v+\he_n$ is dominant. Analogously, the morphism $\Ad_{-}(e+v):\Hs \to v+\he_n$ is dominant, as $e+v$ is regular from the previous lemma. Therefore the images of $\Ad_{-}(x)$ and $\Ad_{-}(e+v)$ are both dense in $v+\he_n$. By \cite[Theorem 1.9.5]{Springer} they both contain open dense subsets of $v+\he_n$ and hence they intersect, which means that $x$ and $e+v$ are conjugate.
 
 Uniqueness follows from the previous lemma.
\end{proof}

Now we will also provide an equivalent of the classical Jordan form, for arbitrary solvable groups. Recall that by Remark \ref{nilalg}
every $x\in\he$ is of the form $x = \w + n$, where $\w\in\ttt$ and $n\in\he_n$.
\begin{theorem}\label{jordan}
 For any $x = \w + n\in\he$ with $\w\in\ttt$, $n\in\he_n$, there exists $M\in \Hs$ such that $x = \Ad_M(\w+n')$ with $[\w,n']=0$ and $n'\in\he_n$.
\end{theorem}

\begin{proof}
We have the Jordan decomposition (see Section Remark \ref{nilalg}) $x = x_s + x_n$, where $x_s$ is semisimple, $x_n$ is nilpotent and $[x_s,x_n]=0$. Then by Theorem \ref{solv} the element $x_s$ is conjugate to an element of $\ttt$. Hence there exists $M\in \Hs$ such that $\Ad_{M^{-1}}(x_s) \in \ttt$. Note that \begin{align*}\Ad_{M^{-1}}(x_s) - x_s\in [\he,\he]\end{align*} as in the proof of Lemma \ref{etreg}. Moreover \begin{align*} x_s-\w = (x-x_n) - (x-n) = n-x_n\in \he_n.\end{align*} As $[\he,\he]\subset \he_n$ by Theorem \ref{solv}, we get $\Ad_M^{-1}(x_s) - \w \in \he_n$. As both $\Ad_{M^{-1}}(x_s)$ and $\w$ lie in $\ttt$, we get that they are equal. Therefore putting $n' = \Ad_{M^{-1}}x_n$ we get
$$x = x_s + x_n = \Ad_M(\w) + \Ad_M(n') = \Ad_M(\w + n')$$
and the conditions are satisfied.
\end{proof}

\noindent
Note that if $\w\in\ttt^\reg:=\ttt\cap \he^\reg$ is a regular element in $\ttt$, then the only nilpotent $n'$ commuting with $\w$ is $0$. Therefore we get

\begin{corollary}\label{corre}
For every $\w\in\ttt^\reg$ and $n\in\he_n$, the elements $\w$ and $n+\w$ are conjugate.
\end{corollary}

\begin{example}\label{exjor} Let us see two examples for $\Hs = \Bs_3$, the Borel subgroup (of upper triangular matrices) of $\SL_3(\C)$. Let the principal nilpotent element $e$ be of the form
$$e = \begin{pmatrix}
 0 & 1 & 0\\
 0 & 0 & 1\\
 0 & 0 & 0
\end{pmatrix}.$$
\begin{enumerate}
\item Let $\w\in\ttt$ be of the form $\w = \diag(0,v_1,v_2) - \frac{v_1+v_2}{3}I_3$ with $v_1\neq 0$, $v_2\neq 0$, $v_1\neq v_2$. Then note that the matrix $e+\w$ is diagonalisable in the basis defined by the matrix
$$
M_{\w} =
\begin{pmatrix}
 1 & \frac{1}{v_1} & \frac{1}{v_2(v_2-v_1)} \\
 0 & 1 & \frac{1}{v_2-v_1} \\
 0 & 0 & 1 
\end{pmatrix},
$$
i.e.
$
 e + \w = 
M_{\w} \w M_{\w}^{-1}
$.

\item Consider the matrix $e+\w$, where $\w\in\ttt$ is of the form $\w = \diag(0,v_1,0) - \frac{v_1}{3}I_3$ with $v_1\neq 0$.
If we take
$$M_{\w} =
\begin{pmatrix}
 1 & \frac{1}{v_1} & 0 \\
 0 & 1 & 1 \\
 0 & 0 & -v_1 
\end{pmatrix},$$
then 
$$
\begin{pmatrix}
 0 & 1 & 0 \\
 0 & v_1 & 1 \\
 0 & 0 & 0
\end{pmatrix} = 
M_{\w}
\begin{pmatrix}
 0 & 0 & 1 \\
 0 & v_1 & 0 \\
 0 & 0 & 0
\end{pmatrix}
M_{\w}^{-1}.
$$
Therefore for $e+\w = \begin{pmatrix}
 0 & 1 & 0 \\
 0 & v_1 & 1 \\
 0 & 0 & 0
\end{pmatrix} - \frac{v_1}{3}I_3$ we get 
$$(e+\w) = M_{\w}
\begin{pmatrix}
 -v_1/3 & 0 & 1 \\
 0 & 2v_1/3 & 0 \\
 0 & 0 & -v_1/3
\end{pmatrix} M_{\w}^{-1}.$$
The matrix $M_{\w}$ used here does not have determinant one. We can however multiply it by any cubic root of $v_1^{-1}$ to get a matrix from $\Bs_3$.
\end{enumerate}
\end{example}

\begin{remark}
 Even for $\Hs = \Bs_m$, the Borel subgroup of $\SL_m$, we cannot require $\w+n'$ from Theorem \ref{jordan} to be the classical Jordan form, under no additional assumption on $x$. Even for $\w = 0$, there is an infinite number of nilpotent orbits of adjoint action of $\Bs_m$ on $\bb_m$ for $m\ge 6$, see \cite{DjoMal}.
 One can prove that if $x$ is a regular matrix, then we can actually find $n'$ which is a nilpotent Jordan matrix.
\end{remark}

\subsection{Reductive groups}
Assume that $\Gs$ is a reductive group. Let $\Ts$ be its maximal torus, $\Bs$ a Borel subgroup containing $\Ts$, $\Bs^-$ the opposite Borel, $\U$ and $\U^-$ the respective unipotent subgroups. Let $\geg$, $\ttt$, $\bb$, $\bb^-$, $\uu$, $\uu^-$ be the corresponding Lie algebras.  Let $(e,f,h)$ be a principal $\ssl_2$-triple in $\geg$, such that $e\in \uu$, $f\in\uu^-$, $h\in \ttt$. Let $$\Ss = e + C_\geg(f)$$ be the Kostant section.
 
\begin{lemma}\label{lembal}
 Under the assumptions above 
 $$ \Ad_{-}(-): \U^- \times \Ss \to e+\bb^-$$
 is an isomorphism.
\end{lemma}

\begin{proof}
If $\Gs$ is semisimple, then the map
 $$ \Ad_{-}(-): \U^- \times \Ss \to e+\bb^-$$
is an isomorphism (\cite[Theorem 1.2]{Kost}, see also another proof in \cite[Theorem 7.5]{Ginz}).

Now if $\Gs$ is an arbitrary reductive group, let $\Gs^{\ad}$ be its adjoint group and let $\pi:\Gs\to \Gs^{\ad}$ be the quotient map. From \cite[Proposition 17.20]{Milne} we have that $\pi(\Bs)$ and $\pi(\U^-)$ are Borel and maximal unipotent in $\Gs^{\ad}$, respectively. Note that $\ker \pi = Z(\Gs)$ and the identity component of $Z(\Gs)$ is a torus (\cite[Proposition 19.12]{Milne}). As a torus contains no nontrivial unipotent elements, we have $\ker \pi\cap U^- = \{1\}$. Therefore $\pi|_{\U^-}$ is an isomorphism $\U^-\cong \pi(\U^-)$. We then know from above that
$$\Ad_{-}(-): \pi(\U^-) \times \Ss_{\Gs^{\ad}} \to e+\bb^-_{\Gs^{\ad}}$$
is an isomorphism. From \cite[Theorem II.11]{Jac} we can identify $\geg^{\ad}$ with an ideal inside $\geg$ such that $\geg = Z(\geg) \oplus \geg^{\ad}$. Then we have 
$$\pi(\U^-) \times \Ss_{\Gs} \cong (\pi(\U^-) \times \Ss_{G^{\ad}}) \times Z(\geg)$$
and 
$$\bb^-_{\Gs} = \bb^-_{\Gs^{\ad}}\times Z(\geg).$$
As the adjoint representation is trivial on the centre of a Lie algebra, we have the following diagram, where the middle column is the product of the left and right and the horizontal arrows are the projections.

$$
\begin{tikzcd}
\pi(\U^-) \times \Ss_{\Gs^{\ad}} \arrow[dd, "\Ad_{-}(-)"]
& \arrow[l, twoheadrightarrow] \U^- \times \Ss_{\Gs} \arrow[dd, "\Ad_{-}(-)"] \arrow[r, twoheadrightarrow] &
Z(\geg) \arrow[dd, "="] 
\\ \\
e+\bb^-_{\Gs^{\ad}} &
\arrow[l, twoheadrightarrow] e+\bb^-_{\Gs} \arrow[r, twoheadrightarrow] 
&
Z(\geg)
\end{tikzcd}
$$
As the peripheral vertical arrows are isomorphisms, we get that also for $\Gs$ the map
$$ \Ad_{-}(-): \U^- \times \Ss \to e+\bb^-$$
is an isomorphism.
\end{proof}
\noindent
Let us now consider the preimage of $e+\ttt$ and for any $\w\in\ttt$ denote by $A(\w)\in \U^-$, $\chi(\w)\in \Ss$ the elements such that
 \begin{equation} \label{adchi}
 \Ad_{A(\w)}(e+\w) = \chi(\w).
 \end{equation}
 Note that we have two inclusions of affine spaces $\Ss\into \geg$ and $e+\ttt \into \geg$. The former induces the isomorphism $\Ss\cong \geg/\!\!/\Gs$, i.e. $\C[\geg]^\Gs\xrightarrow{\cong} \C[\Ss]$ by \cite[Section 4.7, Theorem 7]{Kostsec}. The latter induces a map $\C[\geg]^\Gs\to \C[e+\ttt]$. However, a regular element $\w\in\ttt$ is conjugate to $\w+e$ (see Corollary \ref{corre}). Let us then consider the composition $\C[\geg]^\Gs\to \C[e+\ttt]\to \C[\ttt]$, where the last map comes from translation by $e$. It is equal to the map $\C[\geg]^\Gs\to \C[\ttt]$ coming from inclusion $\ttt\to\geg$ -- as the dual maps of schemes agree on a dense subset of $\ttt$, i.e. on $\ttt^\reg$. 
 
 Note that if we compose $\chi^*: \C[\Ss]\to \C[\ttt]$ with the isomorphism $\C[\geg]^\Gs\to \C[\Ss]$ described above, then we get the composite map above $\C[\geg]^\Gs\to \C[\ttt]$, which now we know is induced by inclusion $\ttt\to\geg$. By Chevalley's restriction theorem (cf. \cite[Theorem 3.1.38]{ChGi}) this map is an inclusion whose image is $\C[\ttt]^\W$.
 
\begin{remark}\label{chevrem}
Chevalley's theorem is originally formulated for semisimple groups. However, if we again consider $\geg^{\ad}$ as an ideal of $\geg$ such that $\geg = \geg^{\ad}\oplus Z(\geg)$, we have
 $$ \C[\geg]^\Gs = \C[\geg^{\ad}\oplus Z(\geg)]^{\Gs^{\ad}} = \C[\geg^{\ad}]^{\Gs^{\ad}} \oplus Z(\geg) = \C[\ttt\cap\geg^{\ad}]^\Ws \oplus Z(\geg) = \C[\ttt]^\Ws,$$
 where the third equality follows from original Chevalley's theorem for $\Gs^{\ad}$.
\end{remark}

Therefore we get
 \begin{proposition}\label{isoquot}
 \label{kostiso}
  The map $\chi:\ttt\to\Ss$ defined by property \eqref{adchi} induces an isomorphism $\ttt/\!\!/\W\to \Ss$.
 \end{proposition}

\subsection{Principally paired groups}\label{kostgensec}
Let now $\Hs$ be any principally paired group. Let $\Ns$ be the unipotent radical of $\Hs$. Then $\Ns$ is a normal subgroup of $\Hs$ and $\Hs/\Ns$ is reductive. Let $\Ls\subset \Hs$ be any Levi subgroup, i.e. a section of $\Hs\to \Hs/\Ns$, cf. Theorem \ref{thmlevi}. We have $\Hs = \Ns\rtimes \Ls$ and hence $\he = \nen \oplus \lel$, where $\he$, $\nen$, $\lel$ are the Lie algebras of $\Hs$, $\Ns$, $\Ls$, respectively. Let $r$ be the dimension of any maximal torus in $\Hs$.

Assume that $(e,h)$ is an integrable principal $\bb(\ssl_2)$-pair within $\he$ and let $\{H^t\}_{t\in\Cs}$ be the embedding of $\Cs$ to which $h$ integrates. We can choose $\Ls$ such that $h\in\lel$, hence we will assume this inclusion from now on. We then have $e = e_n + e_l$, where $e_n\in\nen$, $e_l\in\lel$. Let us consider, by the Jacobson--Morozov Theorem (cf. Section \ref{sl2p}), the $\ssl_2$-triple $(e_l,f_l,h_l)$ within $\lel$. 

\begin{lemma}\label{redpartreg}
 For $\Hs$ and $(e,h)$ as above, $e_l$ is a regular element of $\lel$.
\end{lemma}

\begin{proof}
We know that $e$ is a regular element of $\he$. This means that $[e,\he]$ is of codimension $r$ in $\he$. But note that $[e,\he] \subset \nen \oplus [e_l,\lel]$ as $\nen$ is an ideal. Therefore $[e_l,\lel]$ is of codimension at most $r$ in $\lel$. Therefore $\dim C_\lel(e_l) \le r$, hence actually $\dim C_\lel(e_l) = r$ and $e_l$ is regular in $\lel$.
\end{proof}

Now, let $\Bs_l$ be a Borel subgroup of $\Ls$ whose Lie algebra contains $e_l$ and $h$ and inside it let $\Ts$ be a torus whose Lie algebra contains $h$. Let $\Bs = \Ns\rtimes \Bs_l \subset \Hs$ -- it is easy to see that $\Bs$ is then a Borel subgroup of $\Hs$. Let $\U$ be its subgroup of unipotent elements. Given $\Bs_l$ and $\Ts$, let $\Bs_l^-$ be the opposite Borel subgroup of $\Ls$ and $\U_l$, $\U_l^-$ the unipotent subgroups of $\Bs_l$ and $\Bs_l^-$. By $\bb$, $\bb_l$, $\ttt$, $\bb_l^-$, $\uu$, $\uu_l$, $\uu_l^-$ we denote the corresponding Lie algebras. Let $\Ws$ be the Weyl group of $\Hs$ (equal to the Weyl group of $\Ls$).

\begin{lemma}\label{posint}
 All the weights of $\{H^t\}$-action on $\uu$ are positive even integers.
\end{lemma}

\begin{proof}
As $e$ is regular in $\Hs$, it has to be regular in $\Bs$ as well. Therefore $[e,\bb] = \uu$ (cf. Section \ref{secreg}).

We can choose a basis of $\bb$ which consists of eigenvectors of $[h,-]$. We then choose from it a subset $\{v_1,v_2,\dots,v_k\}$ such that $\{[e,v_i]\}_{i=1}^k$ forms a basis of $\uu$. Then $[e,-]$ is an isomorphism $\Span(v_1,\dots,v_k) \to \uu$. Let $\phi$ denote this restricted commutator operator $[e,-]$. For any $v\in\bb$ we have
$$[h,[e,v]] = [[h,e],v] + [e,[h,v]] = 2[e,v] + [e,[h,v]],$$
hence if $[h,v] = \lambda v$, we get $[h,[e,v]] = (\lambda+2)[e,v]$. Therefore for an $h$-weight vector $v$, $\phi$ satisfies the condition
$$[h,v] = \lambda v \iff [h,\phi(v)] = (\lambda+2)\phi(v).$$
Let us consider a weight vector $w\in \uu$ such that $[h,w] = \lambda w$ and assume that $\lambda$ is not a positive even integer. We now know that $w = \phi(w_1)$ for some $w_1\in\bb$ with $[h,w_1] = (\lambda-2) w_1$. As $\lambda-2 \neq 0$, we have $w_1\in \uu$ (as $\ttt$ has only zero weights of $H^t$-action). Then analogously $w_1 = \phi(w_2)$ for $w_2\in\bb$ of weight $\lambda-4$. As again $\lambda-4\neq 0$, we get $w_2 = \phi(w_3)$, and we continue this procedure to get an infinite sequence $w=w_0$, $w_1$, $w_2$, \dots, such that $w_i$ is a weight vector of weight $w_i - 2i$. However, $\bb$ is finite-dimensional, so we get a contradiction.
\end{proof}

\noindent
For our principally paired $\Hs$ the role of the Kostant section will be played by 
\begin{align}\label{ppS} \Ss := e + C_{\lel}(f_l)\subset \he.\end{align}

\begin{lemma}
\label{congen}
The conjugation map
$$
\Ad_{-}(-): \U_l^- \times \Ss \to e + \bb_l^-
$$
is an isomorphism.
\end{lemma}

\begin{proof}
With Lemma \ref{redpartreg} we know that the conjugation map
\begin{equation}
\label{ident}
\Ad_{-}(-): \U_l^- \times (e_l + C_\lel(f_l)) \to e_l + \bb_l^-
\end{equation}
is an isomorphism. But note that the weights of the $\Ts$-action on $\uu_l^-$ are exactly the negatives of the weights on $\uu_l$. Hence by Lemma \ref{posint}, evaluated on $h$ they are all negative even integers. As $\nen$ is an ideal in $\he$, we have
$$[\uu_l^-,e_n] \subset \nen.$$
However, we know (again from Lemma \ref{posint}) that the $h$-weight of $e_n$ (equal to 2) is the lowest possible among the weights in $\nen$. 
However, all the $h$-weights in $[\uu_l^-,e_n]$ would be lower, as the weights on $\uu_l^-$ are negative. Therefore in fact $[\uu_l^-,e_n] = 0$. Hence $\U_l^-$ commutes with $e_n$.

Then we get the conclusion simply by adding $e_n$ to both sides of \eqref{ident}.
\end{proof}
Now note that we are given two one-parameter subgroups: $H^t$ and $H_l^t$ generated by $h$ and $h_l$, respectively. We show that they are actually equal up to the centre of $\Ls$.
\begin{lemma}\label{semcen}
 Let $\Gs$ be a reductive group and $e$ a regular nilpotent element in $\geg = \Lie(\Gs)$. Then the only semisimple elements in its centraliser $C_\geg(e)$ are the ones in the centre $Z(\geg)$.
\end{lemma}
\begin{proof}
 Assume that $v\in\geg$ is a semisimple element such that $[v,e] = 0$. Choose a Borel subgroup $\Bs\subset \Gs$ whose Borel subalgebra $\bb\subset \geg$ contains $e$ and $v$ and let $\nen$ be the nilpotent part of $\bb$. We can choose a maximal torus $\Ts$ within $\Bs$ whose Lie algebra contains $v$. Let $r = \dim \Ts = \dim C_{\geg}(e)$.
 
 As $e$ is regular in $\Hs$, it is also regular in $\bb$ and $\nen = [\bb, e]$ (cf. Section \ref{secreg}). However, $\bb = \ttt \oplus \nen$, so by iterating we easily see that as a Lie algebra $\bb$ is generated by $e$ and $\ttt$. Then as $[v,e]=0$, this easily leads to $[v,\bb] = 0$. As $\bb$ was a Borel subalgebra and $v$ is semisimple, from this $[v,\geg] = 0$ follows.
\end{proof}
From this lemma, as $[h,e] = [h_l,e] = 2e$, we infer $h-h_l\in Z(\lel)$.
In the map $\Ad_{-}(-)$ from Lemma \ref{congen}, let us consider the preimage of $e+\ttt$ and for any $w\in\ttt$ denote by $A(\w)\in \U^-$, $\chi(\w)\in \Ss$ the elements such that
 \begin{equation} 
 \Ad_{A(\w)}(e+\w) = \chi(\w).
 \end{equation}

We will now want to generalise Kostant's Theorem~\ref{kostant2}. First, we find the contracting $\Cs$-action on $\Ss = e + C_{\lel}(f_l)$. Note that as $e_l$ is regular in $\Ls$, also $f_l$ is regular in $\Ls$ (see Remark \ref{princesl2}). Moreover, as all the weights of the ${H^t}$-action on $\uu_l$ are positive integers, on $\uu_l^-$ they are all negative integers. As the weight of the action on $f_l$ is $-2$ (note that we use Lemma \ref{semcen} to switch between the actions of $h_l$ and $h$), $f_l$ must lie in $\uu_l^-$. In particular $f_l \in \bb_l^-$, and as $\bb_l^-$ contains the Lie algebra of the maximal torus of $\Ls$, we have that $f_l$ is regular in $\bb_l^-$. This means that $C_{\he}(f_l)\subset \bb_l^-$ (cf. Section \ref{secreg}). In particular, all the weights of the ${H^t}$-action on $C_{\he}(f_l)$ are nonpositive integers. Therefore, for any $x\in C_{\he}(f_l)$, we have
$$\Ad_{H^t}(x+e) = \Ad_{H^t}(x) + t^2 e = t^2 \left( \Ad_{H^t}(x)/t^2 + e\right)$$
and
$$\lim_{t\to\infty} \Ad_{H^t}(x)/t^2 = 0.$$

Therefore if we define the action of $\Cs$ on $\Hs$ by
$$t \cdot v = t^{-2} \Ad_{H^t}(v),$$
then it preserves $\Ss$ and for any $v\in \Ss$ we have 
$$\lim_{t\to\infty} t\cdot v = e.$$
\begin{theorem}\label{kostarb}
 Every element of $\Ss$ is regular in $\he$. Moreover, every regular orbit of the adjoint action of $\Hs$ on $\he$ meets $\Ss$.
\end{theorem}
\begin{proof}
 For the first part, we proceed as in the proof of Lemma \ref{etreg}. Assuming that for some $x\in C_{\he}(f_l)$ the element $x+e$ is not regular, we get that $\Ad_{H^t}(x)/t^2+e$ is not regular for any $t$ and from continuity ($t\to \infty$) we get that $e$ is not regular.
 
 Now assume that some $y\in \he$ is regular. It lies in a Borel subalgebra and as all Borel subalgebras are conjugate, we can assume $y\in \bb$. As $\Bs$ contains a maximal torus of $\Hs$, we have that $y$ is regular in $\bb$ as well. Therefore by Lemma \ref{etkost} it is conjugate to an element of the form $e + v$ for $v\in\ttt$. It is then conjugate to $\chi(v) \in \Ss$.
\end{proof}

To finish the proof of $\C[\he]^\Hs = \C[\Ss]$ we need the following lemma, known for reductive groups already.
\begin{lemma}\label{genrest}
 $\C[\he]^\Hs = \C[\lel]^\Ls = \C[\ttt]^\Ws$.
\end{lemma}

\begin{proof}
 The latter equation is just Chevalley's restriction theorem (\cite[Theorem 3.1.38]{ChGi} and Remark \ref{chevrem}). We need to prove that the restriction map $\C[\he]^\Hs\to \C[\lel]^\Ls$ is an isomorphism.
 
 Let us first prove that it is surjective. We have the projection map $\pi:\Hs\to \Ls \simeq \Hs/\Ns$ and we can use it to pull back any $\Ls$-invariant function on $\lel$. If $f$ is such a function, its pullback is $f\circ \pi_*$ and for any $g\in \Hs$ and $v\in\he$ we have
 $$(f\circ\pi_*)(\Ad_g(v)) = f(\Ad_{\pi(g)}(\pi_*(v))) = f(\pi_*(v)) = (f\circ\pi_*)(v),$$
 hence $f\circ \pi_*$ is $\Hs$-invariant. It also obviously restricts to $f$ on $\lel$, hence we proved that the restriction $\C[\he]^\Hs \to \C[\lel]^\Ls$ is surjective.
 
 Now we prove injectivity. As every element of $\Hs$ is contained in a Lie algebra of a Borel subgroup, and they are all conjugate, a function from $\C[\he]^\Hs$ is fully determined by its values on $\bb$. We know that $\bb = \ttt \oplus \uu$ and the weights of the $H^t$-action on $\ttt$ are all 0, and on $\uu$ they are all positive. 
 
 Therefore any polynomial on $\bb$ which is invariant under this action, can only contain the $\ttt$-variables. Hence it is uniquely determined by its values on $\ttt$.
 \end{proof}
 
 From the proof of Lemma \ref{congen} and from Proposition \ref{isoquot} the map $\chi$ defines an isomorphism $\C[\Ss]\to \C[\ttt]^\Ws$ and when composed with the restriction from $\C[\he]^\Hs$, it clearly gives the restriction $\C[\he]^\Hs \to \C[\ttt]^\Ws$ (note that $x$ and $\chi(x)$ are always conjugate). Then from Lemma \ref{genrest} we get
 
 \begin{theorem}\label{restkos}
  The restriction map $\C[\he]^\Hs\to \C[\Ss]$ is an isomorphism.
 \end{theorem}

 In particular, this means that no elements of $\Ss$ are conjugate to each other. Together with Theorem \ref{kostarb} this gives
 \begin{corollary} \label{kostprincpair}
  Every regular orbit of the adjoint action of $\Hs$ on $\he$ meets $\Ss$ exactly once.
 \end{corollary}

\section{Regular actions and fixed point sets}
\label{secregact}

\begin{definition}\label{defreg}
Assume we are given a principally paired $\Hs$ with $(e,h)$ being the integrable principal $\bb(\ssl_2)$-pair in $\he$. If $\Hs$ acts on a smooth projective variety $X$, we say that it acts \emph{regularly} if $e$ has a unique zero $o\in X$. 
\end{definition}

\begin{remark}
The choice of integrable principal pair $(e,h)$ in $\he$ is not unique. However, we will see below in  Lemma~\ref{isoreg} that the property of the action being regular does not depend on the choice. 

Note that as $e$ is nilpotent, it generates an additive subgroup of $\Hs$ (see Theorem \ref{algexp}) and hence by Theorem \ref{horrfix} the zero scheme $X^e$ of $V_e$ is connected. It is therefore enough to assume that the zeros of $e$ are isolated. We will in fact prove in Lemma~\ref{isoreg} that all the regular elements of $\he$ have isolated zeros on $X$.
\end{remark}

\begin{example}\label{exgr2} This example is from the PhD thesis of Ersan Akyildiz \cite{Akyphd}, see also \cite{Akyildiz}.
Consider a complex reductive group $\Gs$, with the choice of $e$ as in Example \ref{exregnilp}. By the discussion in Section \ref{sl2p} there exists $h\in\geg$ which makes $\Gs$ principally paired. Let $X= \Gs/\Bs$ be the full flag variety of $\Gs$. Then for any $x=g\Bs\in X$, from Lemma \ref{lemad}
 $$V_e|_{x} = \D g(V_{\Ad_{g^{-1}}(e)}|_{[1]}).$$
Therefore $V_e$ vanishes at $x$ if and only if $\Ad_{g^{-1}}(e)$ vanishes at $[1] = \Bs$. This means that $\Ad_{g^{-1}}(e) \in \bb = \Lie(\Bs)$, or in other words, $e\in\Lie(g\Bs g^{-1})$. The subgroup $g\Bs g^{-1}$ is of course a Borel subgroup of $\Gs$. By \cite[Proposition 3.2.14]{ChGi} the group $\Bs$ is the unique Borel subgroup of $\Gs$ whose Lie algebra contains $e$. Therefore $e\in\Lie(g\Bs g^{-1})$ only if $g\Bs g^{-1} = \Bs$. By \cite[11.16]{Borel} this is true only for $g \in \Bs$, i.e. $x = [1]$. Therefore $\Gs$ acts on the full flag variety $\Gs/\Bs$ regularly.
 
 Hence it also acts regularly on all the partial flag varieties $\Gs/\Ps$. Indeed, assume that $x\in \Gs/\Ps$ is a zero of the vector field given by $e$. If we denote by $\pi_\Ps$ the projection $\pi_\Ps:\Gs/\Bs\to \Gs/\Ps$, then $\pi_\Ps^{-1}(x)$ is a closed subvariety of $\Gs/\Bs$, closed under the action of $\Gs_a$ generated by $e$. Hence by the Borel fixed point theorem (Theorem \ref{borelfix}), it contains a fixed point of $\Gs_a$, which is unique on $\Gs/\Bs$. Therefore $x$ must be the image of the only fixed point on $\Gs/\Bs$. 
\end{example}

\begin{example}[see {\cite[Section 6]{BC}}]\label{exsl2}
 Let $\Hs = \SL_2(\C)$ and consider the irreducible representation $W$ of $\SL_2(\C)$ of dimension $n+1$. In particular, the regular nilpotent
 $$e = \begin{pmatrix} 0 & 1 \\ 0 & 0\end{pmatrix}$$ acts on $W$ with the matrix
 $$\begin{pmatrix}
       0 & 1 & 0 & 0 & \dots & 0\\
       0 & 0 & 1 & 0 & \dots & 0\\
       0 & 0 & 0 & 1 & \dots & 0\\
       \vdots & \vdots & \vdots & \vdots & \ddots & \vdots \\
       0 & 0 & 0 & 0 & \dots & 1 \\
       0 & 0 & 0 & 0 & \dots & 0
      \end{pmatrix}.
 $$
 If we consider $X = \PP(W)$, the action is clearly regular and the only fixed point of $e$ corresponds to the vector of highest weight in $V$.
\end{example}

\subsection{Solvable groups}
\begin{lemma}\label{regzer}
Let $\Hs$ be a solvable group. Let $\Ts$ be its maximal torus and $\he_n$ be the nilpotent part of $\he = \Lie(\Hs)$. Assume that $e\in\he_n, h\in\ttt$ are such that $(e,h)$ is an integrable $\bb(\ssl_2)$-pair and that $\Hs$ acts regularly on a smooth projective variety $X$. Then any element of $e+\ttt$ has isolated zeros on $X$.
\end{lemma}

\begin{proof}
We will denote by $\{H^t\}_{t\in\Cs}$ the one-parameter subgroup to which $h$ integrates.
Define $\ZZ \in \ttt \times X$ as the zero scheme of the total vector field restricted to $e+\ttt \cong \ttt$. In other words, for any $\w\in\ttt$, that vector field restricted to $\{\w\} \times X$ equals $V_{e+\w}$ (cf. Definition \ref{totvec}).
Consider also an action of $\Cs$ on $\ttt \times X$ which is defined on $\ttt$ by multiplication by $t^{-2}$  and on $X$ by the action of $H^t$. As $\Ad_{H^t}(e) = t^2 e$, for any $v\in\ttt$ we have 
$$\Ad_{H^t}(e+v) = t^2 e + v = t^2 (e + v^2/t).$$
Therefore by Lemma \ref{lemad} the action on $\ttt \times X$ preserves $\ZZ$.
 
Consider the map $\pi:\ZZ\to \ttt$ defined as the projection onto the first factor of $\ttt \times X$. As it is a morphism of schemes locally of finite type, by Chevalley’s semicontinuity theorem \cite[13.1.3]{EGA43}, the set 
 $$D = \{ (\w,x)\in \ZZ : \dim \pi_{\w} \ge 1 \}$$
 is closed. Here $$\pi_\w:=\pi^{-1}(\w)\subset \ZZ$$ denotes the fibre.  Suppose $D$ is nonempty. Hence we have some $\w\in\ttt$ such that $\dim \{x\in\ZZ : (\w+e)|_{x} = 0\} \ge 1$. Note that for any $t\in\Cs$ we have
$$t^2\w + e = t^2(\w + t^{-2}e) = t^2 \Ad_{H^t}^{-1}(\w + e).$$
Therefore the zero set of $t^2\w + e$ is the same as the zero set of $\Ad_{H^t}^{-1}(\w + e)$, which by Lemma \ref{lemad} is isomorphic -- via the action of $H^t$ -- to the zero set of $\w + e$. Hence for each $t\neq 0$ we have $(t\w,o)\in D$, where $o\in X$ is the unique zero of $e$. By closedness of $D$ we therefore get $(0,o)\in D$. Hence $\dim \pi_0 \ge 1$, which is impossible, as $\pi_0 = \{(0,o)\}$ by our regularity assumption.
\end{proof}

\begin{theorem}\label{isolsolv}
 Assume $\Hs$ and $X$ are as in Lemma \ref{regzer} and $e$ is a principal nilpotent. Then any regular element of $\he$ has isolated zeros on $X$.
\end{theorem}

\begin{proof}
It now follows directly from Lemma \ref{regzer} and Lemma \ref{etkost}. 
\end{proof}
In particular, regular semisimple elements have isolated zeros on $X$. Therefore we get
\begin{corollary}\label{fintor}
There are finitely many $\Ts$-fixed points on $X$.
\end{corollary}

\subsection{General principally paired groups}
With the use of the results of Section \ref{kostgensec} we can also provide a version of Theorem \ref{isolsolv} for arbitrary principally paired groups.

\begin{lemma}\label{isoreg}
Let a principally paired group $\Hs$ act regularly on a smooth projective variety $X$. Then all the regular elements of $\he$ have isolated zeros on $X$.
\end{lemma}

\begin{proof}
 We know from Lemma \ref{kostarb} that every regular element of $\he$ is conjugate to an element of $\Ss$ from \eqref{ppS}. Therefore it is enough to prove the statement for the elements of $\Ss$. The argument is the same as in the proof of Lemma \ref{regzer}, using the contracting action from Section \ref{kostgensec}. Note that if $p\in X$ is a zero of $x+e$, then $H^t p$ is a zero of $\Ad_{H^t}(x)/t^2 + e$. Therefore if $(x+e,p) \in D$, we have $(\Ad_{H^t}(x)/t^2 + e,H^tp)\in D$ for any $t\in\Cs$ and then $\left(e,\lim_{t\to \infty} H^t p\right) \in D$.
\end{proof}

\chapter{Zero schemes and ordinary cohomology}\label{ordcoh}
In this chapter we review already existing results that relate topological invariants, in particular the cohomology ring, to local invariants of vector fields. A good review, with the focus on torus actions, can be also found in \cite{carrell}.

\section{Quantitative invariants}
The idea of extracting global topological data of manifolds from local data of a vector field has a long history. The first notable example is of course the Poincar\'{e}--Hopf theorem.

\begin{theorem}[Poincar\'{e}--Hopf]
 Let $X$ be a smooth compact manifold. Let $V$ be a vector field on $X$ with isolated zeros. Then the sum of indices of the zeros is equal to the Euler characteristic $\chi(X)$.
\end{theorem}

A reader interested in the proof should check one of the classical books on differential geometry, e.g. \cite{Milnor}. We will be interested in results that refine the Poincar\'{e}--Hopf theorem. This means that we would like to be able to determine more complicated topological invariants from the vector fields. One particular such result is Bott's residue formula \cite{bott}. First, for any endomorphism $f: W\to W$ of a vector space we denote $c_i(f) = \tr \Lambda^i f$. It is the coefficient of the characteristic polynomial, i.e.
$$\det (f - \lambda \id_W) = \sum_{i=0}^{\dim W} (-1)^i c_{\dim W-i}(f) \lambda^i. $$
Every vector field $V$ on $X$ defines the Lie derivative on the vector fields, which maps a vector field $Y$ to $[V,Y]$. If $V$ vanishes at a point $x\in X$, then for any vector field $Y$, $[V,Y]_x$ only depends on $Y_x$. This means that $V$ defines an endomorphism $L_x:T_{x,X}\to T_{x,X}$. Bott proves that we can compute the Chern numbers of complex manifolds with use of those operators.

\begin{theorem}[Bott, 1967]
 Assume that $X$ is a compact connected complex manifold of complex dimension $m$. Let $V$ be a \emph{nondegenerate} holomorphic vector field, i.e. a vector field with isolated zeros such that $c_m(L_x)$ is nonzero for any zero $x$ of $V$. Consider a polynomial $\Phi \in \Q[c_1,c_2,\dots,c_m]$. We consider $c_i$ to be of weight $i$ and assume that $\Phi$ is of weight at most $m$. We denote the corresponding Chern number by $\Phi(X)$. Then
 $$\sum_{V_x = 0} \Phi(L_x)/c_m(L_x) = \Phi(X).$$
\end{theorem}

Note that if $\Phi = c_m$, then $\Phi(X) = \chi(X)$. The requirement $c_m(L_x)\neq 0$ implies $\operatorname{ind}_x(V) = 1$ and therefore the formula agrees with Poincar\'{e}--Hopf theorem. However, both theorems only allow us to compute quantitative invariants, i.e. numbers associated with the topology of the manifold. Note that Morse theory, or its algebraic counterpart, i.e. Białynicki-Birula decomposition, also allows us to infer the Betti numbers of the manifold/variety by considering the isolated singularities of a vector field. However, this does not give a way of recovering the ring structure on cohomology. The first result of this kind, which allows the computation of the cohomology ring, comes from two papers of Carrell and Lieberman from the 1970s \cite{CL0, CL}.

\section{Carrell--Lieberman theorem for smooth varieties}
\label{seccarli}

In \cite[Main Theorem and Remark 2.7]{CL} the following theorem is proved. We give its formulation and sketch the proof. The paper talks about complex manifold, but in fact the assumptions imply \cite{CL} that the manifold is algebraic. As we are in general interested in algebraic situation, we will work under algebraicity assumption.

\begin{theorem}\label{clthm}
 Let $X$ be a smooth projective complex algebraic variety of dimension $n$. Consider an algebraic vector field $V\in\Vect(X)$ and assume that its zero set is isolated, but nonempty. Let $Z$ denote the zero scheme of $V$. Then the coordinate ring $\C[Z]$ admits an increasing filtration 
 $$F_\bullet: 0 = F_{-1} \subset F_0 \subset F_1 \subset \dots \subset F_n = \C[Z]$$
 such that
 $$H^*(X,\C) \simeq \Gr_F(\C[Z]).$$
The degree on the left is twice the degree on the right, and in particular $H^i(X,\C)$ vanishes for $i$ odd.
\end{theorem}

\begin{remark}
 One might obviously be tempted to ask what the relation between the cohomology ring and the ring of functions on the zero scheme is, if the zeros are not isolated. We will address this question in more detail in Chapter \ref{chapfur}. For now, let us note that in \cite{CL0} the authors prove that if the zero set $Z$ is nonempty, then
 $$H^q(X,\Omega^p) = 0$$
whenever $|p-q| > \dim_\C Z$. In particular, if $\dim Z \le 1$, then the cohomology of $X$ is Tate, i.e. $H^q(X,\Omega^p) = 0$ if $p\neq q$.
\end{remark}

\begin{proof}[Sketch of proof of Theorem \ref{clthm}]
 The vector field $V$ defines a section of the tangent bundle $TX$ of $X$. Consider the Koszul complex defined by that section:
 $$ 0 \to \Omega^n_X \xrightarrow{\iota_V} \Omega^{n-1}_X  \xrightarrow{\iota_V} \dots  \xrightarrow{\iota_V} \Omega^{1}_X  \xrightarrow{\iota_V} \OO_X \to 0. $$
 By Theorem \ref{thmkosz} it is a resolution of $\OO_Z$ and by Corollary \ref{corkosz} we then get the spectral sequence
 $$
	E_1^{pq} = H^q(X,\Omega^{-p})
 $$
convergent to $H^{p+q}(X,\OO_Z)$. We will prove that it degenerates. We inductively prove that all the differentials in the $E_k$ page are trivial for $k=1,2,\dots$. First note that the Koszul complex comes with a pairing $ \wedge: \Omega^i \otimes \Omega^j \to \Omega^{i+j}$. This in turn induces naturally the product structure on the spectral sequence, with the products
 $$ \wedge: E_r^{p_1,q_1} \otimes E_r^{p_2,q_2} \to E_r^{p_1+p_2,q_1+q_2}.$$
 They stay compatible with the differentials, so that
 $$d_r(\alpha\wedge\beta) = d_r(\alpha) \wedge \beta + (-1)^{p_1+q_1} \alpha\wedge d_r(\beta)$$
 whenever $\alpha\in E_r^{p_1,q_1}$, $\beta\in E_r^{p_2,q_2}$.
 
Let $\omega\in H^1(X,\Omega^1)$ be the class of the K\"{a}hler form. We prove first that $d_1(\omega)=0$ in $H^1(X,\OO_X)$. By Serre duality it is enough to prove that it wedges trivially with any element of $H^{n-1}(X,\Omega^n)$. An element of $H^{n-1}(X,\Omega^n)$ is always of the form $\omega^{n-1} \wedge \eta$ for $\eta\in H^0(X,\Omega^1)$. We have
$$d_1(\omega) \wedge \omega^{n-1} \wedge \eta = \iota_V(\omega) \wedge \omega^{n-1}\wedge\eta
=\frac{1}{n} \iota_V(\omega^{n}) \wedge \eta.
$$
Now by the Leibniz identity we have
$$\iota_V(\omega^n \wedge \eta) = \iota_V(\omega^n) \wedge \eta + \omega^n \wedge \iota_V(\eta).$$
The left-hand side vanishes as $\omega^n \wedge \eta=0$ for dimensional reasons. Moreover, $\iota_V(\eta)$ is a function in $H^0(X,\Omega^0)$ which vanishes in zeros of $V$. As the zero set is nonempty, the function is zero. Hence, we get $\iota_V(\omega^n) \wedge \eta = 0$, and therefore $d_1(\omega) = 0$. The higher differentials $d_r$ for $r>1$ then also vanish on the image of $\omega$ for dimensional reasons -- the target is $H^{-r}(X,\Omega^{-1+r}) = 0$.

Now we proceed with induction. For a given $r \ge 1$ we assume that the $r$-th page is
$$ E_r^{pq} = H^q(X,\Omega^{-p}).$$
We know that any element of $H^q(X,\Omega^{-p})$ admits the decomposition
$$\sum_i \omega^i \alpha_i,$$
where $\alpha_i \in H^{q-i}(X,\Omega^{-p-i})$ is primitive. As we have proved $d_r(\omega) = 0$, it remains to prove that $d_r(\alpha) = 0$ if $\alpha$ is primitive. This means that $\alpha\in H^q(X,\Omega^p)$ with $p+q\le n$ and $\omega^{n-p-q+1} \alpha = 0$. We have $d_r(\alpha) \in H^{q-r+1}(X,\Omega^{p-r})$ and hence that $d_r(\alpha)$ vanishes if and only if $\omega^{n-p-q+2r-1}d_r(\alpha)$ does. However, as $d_r(\omega) = 0$, we have
$$\omega^{n-p-q+2r-1}d_r(\alpha) =  \omega^{2r-2} d_r(\omega^{n-p-q+1}\alpha) = 0,$$
where the last equality is the primitivity of $\alpha$.

We proved that the spectral sequence degenerates at $E_1$. Moreover, it converges to the cohomology of $\OO_Z$. However, as $\dim Z = 0$, we have 
$$H^i(X,\OO_Z) =
\begin{cases}
\C[Z] \text{ for } i=0; \\
0 \text{ otherwise.}
\end{cases}
$$
Therefore the convergence means that there is an increasing filtration $F_i$ on the ring $\C[Z]$ such that $F_i/F_{i-1} = H^i(X,\Omega^i)$, and moreover $H^q(X,\Omega^p) = 0$ for $p\neq q$. This proves the claim.
\end{proof}

\begin{remark}
 One can notice that this refines the Poincar\'{e}--Hopf theorem. Indeed, $\dim \C[Z]$ is equal to the sum of indices of zeros of $V$. As the vector field is holomorphic, all the indices are positive. If the vector field $V$ is defined by a derivative of a torus action, then the scheme $Z$ is reduced (Theorem \ref{fixred}) and hence all the indices are equal to 1. The Euler characteristic is equal to their number. In that case, we actually see the additive structure of cohomology by Białynicki-Birula decomposition (Theorem \ref{cohobb}).
\end{remark}

The theorem provides some relation between the ring of functions on $Z$ and the cohomology ring of $X$. However, the filtration is usually not easy to determine. See for example \cite{Kaveh}, \cite{CarKavPup}, where the filtration is worked out e.g. for toric varieties.

There is however a situation, where the filtration is very easy. This is due to Akyildiz and Carrell \cite{AC} (announced in \cite{ACLS}) and requires an auxiliary torus action. We include here the result and its proof, as we will use it throughout.

\begin{theorem}[\cite{AC}, Theorem 1.1] \label{thmgrad}
 Assume that $X$ is a smooth projective variety of dimension $n$. Assume that an algebraic vector field $V\in\Vect(X)$ has isolated, but nonempty zero set. In addition, assume that we are given an action of $\Cs$ on $X$, $(t,x)\mapsto tx$, which satisfies the condition
 $$t_* V = t^k V $$
for some nonzero $k\in\Z$. In other words, $V$ is scaled under the pushforward action of $\Cs$, and the weight of the scaling is nonzero. Then $\Cs$ preserves the zero scheme $Z$ of $V$ and all the weights of its action on $\C[Z]$ are nonnegative multiples of $k$. This gives a grading $\C[Z] = \bigoplus_{i=0}^n \C[Z]_i$, where $\C[Z]_i$ is the part of $\C[Z]$ of weight $-ki$. Moreover, the filtration $F_\bullet$ in Theorem \ref{clthm} satisfies 
$$F_i(\C[Z]) = \bigoplus_{j<i} \C[Z]_j.$$
Hence $\Gr_F(\C[Z]) \simeq \C[Z]$ and therefore
 $$H^*(X,\C) \simeq \C[Z].$$
 as graded rings.
\end{theorem}

\begin{proof}
First of all, notice that for any $t\in \Cs$, the zero scheme of $V$ is the same as the zero scheme of $t^k V$, hence $\Cs$ preserves $Z$. Now we will lift the action of $\Cs$ on $X$ to an action on the Koszul resolution from Theorem \ref{clthm}. First, for $t\in \Cs$, let $t_p:\Omega^p_X\to \Omega^p_X$ be the pullback of forms along the map $t:X\to X$. Consider the following.

$$
\begin{tikzcd}
0\arrow[r] &
\Omega^n_X \arrow[rr, "\iota_V"] \arrow[d, "t^{kn} t^{-1}_n"]&&
\Omega^{n-1}_X \arrow[rr, "\iota_V"] \arrow[d, "t^{k(n-1)} t^{-1}_{n-1}"] &&
\dots \arrow[r, "\iota_V"] \arrow[d]&
\Omega^{1}_X \arrow[rr, "\iota_V"] \arrow[d, "t^{k} t^{-1}_1"]  &&
\OO_X \arrow[r] \arrow[d,"(t^{-1})^*"] & 0\\
0\arrow[r] &
\Omega^n_X \arrow[rr, "\iota_V"] &&
\Omega^{n-1}_X \arrow[rr, "\iota_V"] &&
\dots \arrow[r, "\iota_V"] &
\Omega^{1}_X \arrow[rr, "\iota_V"] &&
\OO_X \arrow[r] & 0
\end{tikzcd}
$$

The vertical maps commute with the horizontal ones due to the condition
$$t_* V = t^k V.$$
Hence we get a group of automorphisms of the Koszul complex, parametrised by $\Cs$. The spectral sequences for hypercohomology are functorial and we use that property for the sequence 
$$
	E_1^{pq} = H^q(X,\Omega^{-p}) \Rightarrow H^{p+q}(X,\OO_Z)
$$
from the proof of Theorem \ref{clthm}. As the action on $\OO_X$, the zeroth term of the Koszul complex, is defined by the action of $\Cs$, on the right-hand side we see the action of $\Cs$ on $\OO_Z$ which descends from the action on $X$. We need to determine the action on the left-hand side. By definition, $t$ acts on $\Omega^p$ by $t^{kp} t^{-1}_p$. The map $t^{-1}_p$ is the pullback of $p$-forms via the action of $\Cs$ on $X$. As $\Cs$ is connected, for any $t$ the corresponding map is homotopic to the identity, and hence descends to the identity on the level of $H^q(X,\Omega^p)$. Therefore the action on the left-hand side is simply the multiplication by $t^{kp}$. Hence the $\Cs$-invariant filtration $F_\bullet$ on $\C[Z]$ satisfies the property that the action of $\Cs$ on $F_i/F_{i-1}$ is of pure weight $ki$. But there is only one such filtration, and it is the one described in the theorem, coming from the grading by weights.
\end{proof}

However unnatural the assumptions of Theorem \ref{thmgrad} might seem at the first glance, there is in fact a plenty of situations, where they are satisfied. Assume that $\SL_2$ acts on a smooth projective variety $X$. In fact, we only need to assume that its Borel $\Bs(\SL_2)$, acts on $X$. Remember that within $\SL_2$ we have the diagonal maximal torus, i.e. 
$$\left\{ H^t = \left. \begin{pmatrix}
t & 0 \\
0 & t^{-1}
\end{pmatrix} \right| t\in \Cs
\right\}.
$$

In the Lie algebra $\ssl_2$, we have the three distinguished elements $e$, $f$, $h$. Recall
$$
e = 
 \begin{pmatrix}
0 & 1 \\
0 & 0
\end{pmatrix}.
$$
Then $H^t e H^{-t} = t^2 e$, i.e. $\Ad_{H^t}(e) = t^2 e$. Now if we take $V = V_e$ to be the vector field defined by the infinitesimal action of $e$ on $X$, then by Lemma \ref{lemad} we get
$$t_* V = t^2 e.$$
Then if the action is regular, i.e. $e$ has a single zero on $X$, the assumptions of Theorem \ref{thmgrad} are satisfied with $k=2$.

\begin{example}
 This is a continuation of Example \ref{exsl2}. Recall that we consider the action of $\SL_2$ on $\PP(W)$, where $W$ is the $n+1$-dimensional irreducible representation of $\SL_2$. The element $e$ acts with the matrix
 $$\begin{pmatrix}
       0 & 1 & 0 & 0 & \dots & 0\\
       0 & 0 & 1 & 0 & \dots & 0\\
       0 & 0 & 0 & 1 & \dots & 0\\
       \vdots & \vdots & \vdots & \vdots & \ddots & \vdots \\
       0 & 0 & 0 & 0 & \dots & 1 \\
       0 & 0 & 0 & 0 & \dots & 0
      \end{pmatrix}.
 $$
 As we noted, the action is regular, hence the assumptions of Theorem \ref{thmgrad} are satisfied. Now let us determine the scheme structure on the zero scheme $Z$ of $e$ and the grading defined by the $\Cs$ action. Note that $h$ acts on $W$ with the matrix
 $$
 \begin{pmatrix}
       n & 0 & 0 & \dots & 0\\
       0 & n-2 & 0 & \dots & 0\\
       0 & 0 & n-4 & \dots & 0\\
       \vdots & \vdots & \vdots & \ddots & \vdots \\
       0 & 0 & 0 & \dots & -n
      \end{pmatrix},
 $$
 i.e. $H^t\in \SL_2$ acts with the matrix
 $$
 \begin{pmatrix}
       t^{n} & 0 & 0 & \dots & 0\\
       0 & t^{n-2} & 0 & \dots & 0\\
       0 & 0 & t^{n-4} & \dots & 0\\
       \vdots & \vdots & \vdots & \ddots & \vdots \\
       0 & 0 & 0 & \dots & t^{-n}
      \end{pmatrix}.
 $$
 The only zero of $e$ has coordinates $[1:0:0\dots:0]$, hence it lies in the $\Cs$-invariant affine patch $X_o$ consisting of elements $[1:x_1:x_2:\dots:x_n]$. Then we have
 $$
 \begin{pmatrix}
       t^{n} & 0 & 0 & \dots & 0\\
       0 & t^{n-2} & 0 & \dots & 0\\
       0 & 0 & t^{n-4} & \dots & 0\\
       \vdots & \vdots & \vdots & \ddots & \vdots \\
       0 & 0 & 0 & \dots & t^{-n}
 \end{pmatrix}
 \begin{pmatrix}
 1 \\
 x_1 \\
 x_2 \\
 \vdots \\
 x_n
 \end{pmatrix}
 =
 \begin{pmatrix}
 t^n \\
 t^{n-2} x_1 \\
 t^{n-4} x_2 \\
 \vdots \\
 t^{-n} x_n
 \end{pmatrix}
 $$
and $[t^n:t^{n-2}x_1:t^{n-4} x_2:\dots:t^{-n}x_n] = [1:t^{-2} x_1 : t^{-4} x_2:\dots:t^{-2n} x_n]$. Hence we see that the action of $H^t$ on $X_o$ scales the $x_i$ coordinate by $t^{-2i}$. Then when we consider the action on the ring of functions $\C[X_o] = \C[x_1,x_2,\dots,x_n]$, each $x_i$ variable is scaled by $t^{2i}$. This is because we act on the functions by $(t^{-1})^*$. This defines the grading on $\C[X_o]$. Now we need to find the ideal that cuts out $Z$ from $X_o$. Note that

$$\begin{pmatrix}
       0 & 1 & 0 & 0 & \dots & 0\\
       0 & 0 & 1 & 0 & \dots & 0\\
       0 & 0 & 0 & 1 & \dots & 0\\
       \vdots & \vdots & \vdots & \vdots & \ddots & \vdots \\
       0 & 0 & 0 & 0 & \dots & 1 \\
       0 & 0 & 0 & 0 & \dots & 0
\end{pmatrix}
\begin{pmatrix}
 1 \\
 x_1 \\
 x_2 \\
 \vdots \\
 x_{n-1} \\
 x_n
 \end{pmatrix}
 = \begin{pmatrix}
 x_1 \\
 x_2 \\
 x_3 \\
 \vdots \\
 x_{n} \\
 0
 \end{pmatrix}
$$
Now we have 
\begin{multline*}
(x_1,x_2,x_3,\dots,x_n,0) =
\\
 x_1 \cdot (1,x_1,x_2,\dots,x_{n-1},x_n) + (0,x_2-x_1^2,x_3-x_1x_2,\dots,x_n-x_1x_{n-1},-x_1x_n).
\end{multline*}
This is an equality on vectors tangent to $W$ in the point $(1,x_1,x_2,\dots,x_n)$. The first summand is mapped to 0 when we pass to the quotient $\PP(W)$. The second is tangent to the affine subspace with first coordinate 1. Therefore it maps to the vector $(x_2-x_1^2,x_3-x_1x_2,\dots,x_n-x_1x_{n-1},-x_1x_n)$ tangent to $X_o$. 

The ideal of $Z$ in $\C[X_o]$ is then generated by the coordinates $x_2-x_1^2$, $x_3-x_1x_2$, \dots, $x_n-x_1x_{n-1}$, $-x_1x_n$. This means that in $\C[Z]$ all the generators are determined by $x_1$. Explicitly $x_2 = x_1^2$, $x_3 = x_1^3$, \dots, $x_n = x_1^n$. The last equation then reduces to $x_1^{n+1} = 0$. Hence
$$\C[Z] = \C[x_1]/(x_1^{n+1})$$
and $x_1$ has degree $1$ in the grading defined by the action of $\Cs$, as in Theorem \ref{thmgrad}. Hence it corresponds to a degree $2$ element in $H^*(\PP^n,\C)$. We recover the standard equality
$H^*(\PP^n,\C) = \C[x_1]/(x_1^{n+1})$
with $x_1\in H^2(\PP^n,\C)$.
\end{example}

\section{Singular varieties}
So far, the results above have been formulated only for smooth varieties. A natural question arises, whether something similar can be said in a singular case. The situation is much trickier, as in general the tangent vectors do not form a bundle, and moreover the cohomology cannot be expressed as $\bigoplus H^q(X,\Omega^p)$. 

However, if we can embed the singular variety with a regular action of $\Bs_2$ in another, now smooth variety also with a regular action of $\Bs_2$, we can use the ambient variety to talk about the vector fields. We sketch in short the story outlined in more detail in \cite[\S 6.2]{carrell}. Note first the following easy observation.

\begin{proposition}
 Assume that $\Bs_2$ acts regularly on a smooth projective variety $X$. Suppose $Y\subset X$ is a smooth $\Bs_2$-invariant closed subvariety. Let $Z \into X$ be the zero scheme of $V_e$, the vector field defined by the infinitesimal action of $e$. Then the scheme-theoretic intersection $Z\cap Y$ is the zero scheme of the vector field defined by the infinitesimal action of $e$ on $Y$. The restriction $\C[Z]\to \C[Z\cap Y]$ fits into the commuting diagram
\begin{equation}
\begin{tikzcd}
\C[Z] \arrow[r] \arrow[dd]
& H^*(X,\C) \arrow[dd]
\\ \\
\C[Z\cap Y] \arrow[r]
& H^*(Y,\C),
\end{tikzcd}
\end{equation}
where the horizontal maps are the ones from Theorem \ref{thmgrad}.
\end{proposition}

Even if $Y$ is not smooth, we can define the intersection $Z\cap Y$. One would then hope that $\C[Z\cap Y]$ will be still isomorphic to $H^*(Y,\C)$, with the isomorphism fitting in the commuting diagram as above. Note that the map $\C[Z]\to \C[Z\cap Y]$ is the restriction of global functions from an affine scheme to a subscheme. Therefore it is surjective, and so if we want a singular $\Bs_2$-invariant $Y\subset X$ to satisfy the Proposition, the restriction map $H^*(X,\C) \to H^*(Y,\C)$ must also be surjective. Akyildiz--Carrell--Lieberman in \cite{ACL} prove the following.

\begin{theorem}
Let $\Bs_2$ act regularly on a smooth projective variety $X$. Assume that $Y\subset X$ is a $\Bs_2$-invariant closed subvariety such that $H^*(X,\C)\to H^*(Y,\C)$ is a surjection. Then there is a surjective homomorphism of graded $\C$-algebras $\C[Z\cap Y] \to H^*(Y,\C)$ such that the following diagram commutes:
 \begin{equation}
\begin{tikzcd}
\C[Z] \arrow[r] \arrow[dd]
& H^*(X,\C) \arrow[dd]
\\ \\
\C[Z\cap Y] \arrow[r]
& H^*(Y,\C).
\end{tikzcd}
\end{equation}
The homomorphism $\C[Z\cap Y]\to H^*(Y,\C)$ is an isomorphism if $\dim H^*(Y,\C) = \dim \C[Z\cap Y]$.
\end{theorem}

The last line of the theorem follows immediately from surjectivity. It turns out that the condition $\dim H^*(Y,\C) = \dim \C[Z\cap Y]$ is not void, although the examples are not easy to produce. A typical situation, where the conditions of the theorem are satisfied, is $X = \Gs/\Ps$ being some generalised flag variety for a reductive group $\Gs$ and $Y\subset X$ being a Schubert subvariety, see \cite{ACL}, \cite{CarBru}. As proved in \cite{2Aky} and \cite{ALP}, the dimension equality is always satisfied if $\Gs = \SL_n$. However, for any simple $\Gs$ different from $\SL_n$, there exists a codimension one Schubert variety of $\Gs/\Ps$, for which $Z = Z\cap Y$. Hence, the theorem does not hold in that case.

This led Carrell in \cite{CarDef} to define the extended zero scheme. Remember in Section \ref{kostgensec} we defined the Kostant section, which for a solvable principally paired group $\Hs$ is equal to $e+\ttt$, where $\ttt$ is a maximal torus of $\Hs$ containing $h$ such that $(e,h)$ is a principal $\bb(\ssl_2)$-pair. In case $\Hs = \Bs_2$, this is just $\Ss = e + \C\cdot h = e + \ttt$, where $h = \diag(1,-1)$ and $\Ts$ is the standard diagonal torus of $\ssl_2$. Consider now the total vector field (Definition \ref{totvec}) on $\bb(\ssl_2)\times X$, restricted to $\Ss \times X$. In other words, this is a family of vector fields on $X$ parametrised by $t\in\C$ such that for each $t$, the vector field equals $V_{e+th}$. Let $\ZZ$ be the zero scheme of this vector field. Clearly, $\ZZ \cap (\{e\}\times X) = Z$ is equal to $\Spec H^*(X,\C)$ by Theorem \ref{thmgrad}. However, $Z$ was a thick point, but the deformation $\ZZ$ turns out to be reduced! Moreover, the reduced intersections turn out to provide the information on cohomology of subvarieties.

\begin{theorem}[\cite{CarDef}, Theorem 1] \label{thmsingcar}
Let $X$ be a smooth projective variety with regular action and let the zero scheme $\ZZ$ be defined as above. Then $\ZZ$ is reduced. Let $Y\subset X$ be a closed $\Bs_2$-invariant subvariety of $X$ such that the restriction on cohomology $H^*(X,\C)\to H^*(Y,\C)$ is a surjection. Define $\ZZ_Y$ to be the \emph{reduced} intersection $\ZZ\cap (\Ss\times Y)$. Then the scheme-theoretic intersection $Z_Y = \ZZ_Y\cap (\{e\}\times Y)$ is isomorphic to $\Spec H^*(Y,\C)$. Moreover, the isomorphism $\C[Z_Y]\to H^*(Y,\C)$ fits into the commutative diagram
 \begin{equation}
\begin{tikzcd}
\C[\ZZ] \arrow[dd] \arrow[r]&  \C[Z] \arrow[r, "\simeq"] \arrow[dd]
& H^*(X,\C) \arrow[dd]
\\ \\
\C[\ZZ_Y] \arrow[r] & \C[Z_Y] \arrow[r, "\simeq"]
& H^*(Y,\C).
\end{tikzcd}
\end{equation}
\end{theorem}

The right column contains topological information, and the middle and left columns contain algebraic information. The theorem proves that the middle column corresponds to a topological object. One could ask whether this can be also said about the left column. This was answered by Brion and Carrell in \cite{BC}. The scheme $\ZZ$ turns out to be the spectrum of $\Cs$-equivariant cohomology of $X$. We recall and generalise this theorem in the next section. The generalisation is the main result of \cite{HR}, a joint work of the author and Tam\'{a}s Hausel.

\begin{remark}
One can also ask what we could say about singular subvarieties which do not satisfy the condition of surjective restriction on cohomology. We address this question later in Theorem \ref{singful}.
\end{remark}

\chapter{Zero scheme as the spectrum of equivariant cohomology} \label{chapequivzero}

In \cite{BC} the following theorem is proven.

\begin{theoremc} \label{thmbc}
 Assume that $\Bs_2$, the Borel subgroup upper triangular matrices in $\SL_2$, acts regularly on a smooth projective variety $X$. Let $\Ss = e + \ttt$ be the Kostant section in $\Bs_2$. Let $\ZZ$ be the zero scheme of the vector field defined by the Lie algebra action on $\Ss\times X$. Then $\ZZ \simeq \Spec H^*_{\Cs}(X)$, where $H^*_{\Cs}(X)$ is the cohomology ring of $X$ equivariant with respect to one-dimensional diagonal torus in $\Bs_2$. Moreover, the structure of algebra over $H^*_{\Cs}(\pt)\simeq \C[\ttt]$ is given by the projection map $\ZZ\to \Ss\simeq \ttt$.
\end{theoremc}

In \cite{HR} we prove a vast generalisation of this result. This section, based on \cite{HR}, presents this result, proven with use of the techniques presented in Chapter \ref{chapkos}. 

\section{Main theorem for solvable groups}
\label{solvsec}

We first consider a solvable group $\Hs$ acting on a variety $X$. We will prove that if the action is regular, then for maximal torus $\Ts\subset \Hs$ we can find $\Spec H_\Ts^*(X)$ as a particular subscheme of $\ttt\times X$. This will generalise the above-mentioned result of \cite{BC}, i.e. Theorem \ref{thmbc}. The goal of this section is to find necessary assumptions on $\Hs$ and construct the scheme $\ZZ = \Spec H_\Ts^*(X)$ inside $\ttt\times X$.

\subsection{Principally paired solvable groups} \label{secsolvreg}
Throughout this section we assume that $\Hs$ is a principally paired solvable group and $(e,h)$ the principal integrable $\bb(\ssl_2)$-pair within $\Hs$. By $\{H^t\}_{t\in\Cs}$ we denote the one-parameter subgroup to which $h$ integrates. Let $\Ts\subset \Hs$ be the maximal torus which contains it. From Theorem \ref{solv} we have $\Hs = \Ts \ltimes \Hs_u$, where $\Hs_u\subset \Hs$ is the subgroup of unipotent elements. We denote by $r$ the dimension of $\Ts$ (or $\ttt$), equal to the rank of $\Hs$. The torus $\Ts$ acts on the Lie algebra $\he$ by the adjoint action $\Ad$. It splits into two representations $\he = \ttt\oplus \he_n$, where $\he_n = \Lie(\Hs_u)$. The first one is trivial and the weights of the other, $\alpha_1, \alpha_2, \dots, \alpha_k \in\ttt^*$, will be called the \emph{roots} of $\Hs$, in analogy with the roots of a reductive group.
This means that if $v_1$, $v_2$, \dots, $v_k$ are the root vectors, then for any map $\phi:\Cs\to\Ts$ we have $$\Ad_{\phi(t)} (v_i) = t^{\alpha_i(\D \phi|_1(1))} v_i.$$
Recall from Section \ref{secreg} that an element $v\in\he$ is called \emph{regular} if its centraliser has dimension $r$. The (open and dense) set of regular elements in $\he$ is denoted $\he^\reg$. We denote by $\ttt^\reg = \ttt\cap\he^\reg$ the subset of $\ttt$ consisting of regular elements. As any element of $\ttt$ commutes with the whole $\ttt$, the condition of $v\in\ttt$ being regular means $C_\he(v) = \ttt$. This means that $[v,-]$ does not have zeros on $\he_n$, i.e. $\alpha_1(v)$, $\alpha_2(v)$, \dots, $\alpha_k(v)$ are all nonzero. Hence we see that the elements of $\ttt^\reg$ are those in $\ttt$ that are not annihilated by any root of $\Hs$. As $h\in\ttt$ is regular, all the roots are nonzero on $h$ -- by Lemma \ref{posint} they are even positive integers when evaluated on $h$ -- hence non-zero. Therefore $\ttt^\reg$ is a complement of a union of hyperplanes.

In our applications $\Hs$ will mostly be the Borel subgroup of some reductive group $\Gs$. Let us see an example below.

\begin{example}\label{exgr}
 A simple case of the above is $\Hs = \Bs_m =\Bs(\SL_m)$, the Borel subgroup of $\SL_m$ consisting of upper triangular matrices. Let $\bb_m$ be its Lie algebra. We have the torus $\Ts\subset \Bs_m$ consisting of diagonal matrices of determinant $1$ and its Lie algebra $\ttt\subset\bb_m$ consisting of traceless diagonal matrices.
 
For the sake of computations, in the examples we will identify $\ttt$ with $\C^{m-1}$ via the isomorphism
 $$(v_1,v_2,\dots,v_{m-1}) \mapsto 
  \diag(0,v_1,v_2,\dots,v_{m-1}) - \frac{v_1+v_2+\dots+v_{m-1}}{m} I_m,
 $$
  i.e. $(v_1,v_2,\dots,v_{m-1})$ corresponds to the unique matrix $A$ in $\ttt$ with $a_{ii}-a_{11} = v_{i-1}$ for $i=1,2,\dots,m-1$.
 Then we can take e.g.
 $$e = \begin{pmatrix}
       0 & 1 & 0 & 0 & \dots & 0\\
       0 & 0 & 1 & 0 & \dots & 0\\
       0 & 0 & 0 & 1 & \dots & 0\\
       \vdots & \vdots & \vdots & \vdots & \ddots & \vdots \\
       0 & 0 & 0 & 0 & \dots & 1 \\
       0 & 0 & 0 & 0 & \dots & 0
      \end{pmatrix}\in\geg
 $$
 and 
 $$
 h = \begin{pmatrix}
       m-1 & 0 & 0 & \dots & 0\\
       0 & m-3 & 0 & \dots & 0\\
       0 & 0 & m-5 & \dots & 0\\
       \vdots & \vdots & \vdots & \ddots & \vdots \\
       0 & 0 & 0 & \dots & 1-m
      \end{pmatrix},
 $$
 or equivalently $h = (-2,-4,\dots,2-2m)\in \C^{m-1}$. Such $e$ and $h$ form a $\bb(\ssl_2)$-pair. Then
 $$
 H^t = \diag(t^{m-1}, t^{m-3},t^{m-5},\dots, t^{3-m}, t^{1-m}).
 $$
 The regular elements of $\ttt$ are the diagonal traceless matrices with pairwise distinct diagonal entries.
 
 We can generalise this example by taking $\Hs$ to be a Borel subgroup of any reductive group $\Gs$. This choice defines the choice of positive roots (as those whose root vectors lie in $\he$). We can therefore take $e = x_1 + x_2 + \dots + x_s$, where $x_1$, $x_2$, \dots, $x_s$ are the root vectors of $\geg$ corresponding to the positive simple roots ($s=r - \dim Z(\Gs)$). Then $e$ is a regular nilpotent in $\geg$ and $\he$ (see Example \ref{exregnilp}). From the discussion in Section \ref{sl2p} we see that there exists $h$ that satisfies the conditions.
\end{example}

\subsection{Uniform diagonalisations}
We saw in Corollary \ref{corre} that $e+\w$ is always conjugate to $\w$ if $\w\in\ttt^\reg$. In the first case in Example \ref{exjor}, we have a closed formula for the conjugating matrix in case of $\Bs_3$. We generalise this observation here.

\begin{theorem}\label{unif}
There exists a morphism $M:\ttt^\reg\to \Hs$ denoted by $\w\mapsto M_\w$ that satisfies the equality
$$\Ad_{M_\w}(\w) = e+\w$$
for any $\w\in\ttt^\reg$.
\end{theorem}

\begin{proof}
From Corollary \ref{corre} we know that for each $\w\in\ttt^\reg$ and $n\in\he_n$ there exists $A\in\Hs$ such that \begin{align}\label{conjugate}\Ad_{A}(\w) = n+\w.\end{align} We have to prove that for $n = e$ we can choose such matrices in a way that varies regularly when $\w$ varies.
 
We know by Theorem \ref{solv} that for any $A \in \Hs$ there exists $V\in\Ts$ such that $AV\in \Hs_u$. Any element of $\Ts$ clearly centralises $\w$. Hence if $A$ is chosen such that $\Ad_{A}(\w) = n+\w$, then also $\Ad_{AV}(\w) = n+\w.$  Hence we can assume that $A\in \Hs_u$. We first show that $A\in \Hs_u$ is unique with respect to \eqref{conjugate}. Indeed, assume on the contrary that $A$, $A'$ are both unipotent and $\Ad_{A}(\w) = \Ad_{A'}(\w) = n + \w$. Then
 $$\Ad_{A^{-1}A'}(\w) = \Ad_{A^{-1}}(n+\w) = \w.$$
 Thus $A^{-1}A'$ centralises $\w$. Hence it centralises $\Aa(\w)$, the smallest closed subgroup of $\Hs$ whose Lie algebra contains $\w$. The group $\Aa(\w)$ is contained in the torus $\Ts$, therefore by \cite[19.4]{Humph} its centraliser $C_{\Hs}(\Aa(\w))$ is connected. But $\Lie(C_{\Hs}(\Aa(\w)))$ has to commute with $\Lie(\Aa(\w))$, which contains $\w$. By the regularity assumption $C_\he(\w) = \ttt$, thus from connectivity we get $C_{\Hs}(\Aa(\w))\subset \Ts$. Therefore $A^{-1}A'\in\Ts$, but as $A^{-1}A'$ is unipotent we get $A^{-1}A' = 1$, hence $A = A'$.

Now consider the map
$$\phi: \Hs_u \times \ttt^\reg \to \he_n\oplus\ttt^\reg$$
$$\phi(A, \w) = \Ad_{A}(\w).$$
We have just proved that $\phi$ is a bijection. Now by Grothendieck's version of Zariski's main theorem (\cite[Theorem 4.4.3]{EGA31}) it can be factored as $\phi = \tilde{\phi}\circ\iota$, where $\iota:\Hs_u\times \ttt^\reg \to Y$ is an open embedding and $\tilde{\phi}:Y\to \he_n\oplus\ttt^\reg$ is finite. By restricting $Y$ to the closure of $\im\iota$, we can assume that $\im\iota$ is dense in $Y$. The map $\phi$ is clearly dominant, and its source is irreducible, hence by \cite[Proposition 7.16]{Harr} it is birational. Therefore $\tilde{\phi}$ is birational as well, but it is finite and its target is normal, hence $\tilde{\phi}$ is an isomorphism. Therefore $\phi$ is an open embedding, which has to be an isomorphism, as it is surjective.

Hence we get the desired map $M:\ttt^\reg\to \Hs_u$ by considering the first coordinate of $\phi^{-1}|_{\{e\}\times\ttt^\reg}$.
\end{proof}

\subsection{Regular actions}
From now on we will assume that our principally paired solvable group $\Hs$ acts on a smooth projective variety $X$ regularly (Definition \ref{defreg}). By Lemma \ref{lemzer} the unique zero $o\in X$ of $e$ is a zero of the whole $\he$. Let $n = \dim X$.

\begin{example}
In Example \ref{exgr2} we see regular actions of reductive group $\Gs$ on flag varieties. In Example \ref{exsl2} we constructed a regular action of $\SL_2$ on $\PP^n$. In both cases, when we restrict to a Borel subgroup, we get a solvable principally paired group (Example \ref{exgr}) acting regularly on smooth projective varieties.
\end{example}

By Corollary \ref{fintor} there are finitely many fixed points of the torus $\Ts$ acting on $X$. We will call them $\zeta_0 = o$, $\zeta_1$, $\dots$, $\zeta_s$. Moreover, combining Lemma \ref{regzer} with Lemma \ref{lemfix} we get that for any $\w\in\ttt^\reg$ the only zeros of $V_\w$ on $X$ are $\zeta_0$, $\zeta_1$, \dots, $\zeta_s$.

Now, following the idea of \cite{BC}, we define the scheme whose coordinate ring will turn out to be the $\Hs$-equivariant cohomology of $X$. This will be an analog of Theorem \ref{thmbc}. As $\Hs$ is homotopically equivalent to its maximal torus $\Ts$, this is the same as the $\Ts$-equivariant cohomology.

\begin{definition}\label{defz}
Let $\ZZ \subset \ttt\times X$ be defined as the zero scheme of the total vector field (Definition \ref{totvec}) restricted to $e+\ttt \cong \ttt$. We will denote that restricted vector field by $V_{e+\ttt}$. In other words, for any $\w\in\ttt$, the vector field $V_{e+\ttt}$ restricted to $\{\w\}\times X$ equals $V_{e+\w}$.
\end{definition}
We will also consider an action of $\Cs$ on $\ttt\times X$ which is defined on $\ttt$ by multiplication by $t^{-2}$ and on $X$ by the action of $H^t$. As $\Ad_{H^t}(e) = t^2 e$, for any $v\in \ttt$ we have $\Ad_{H^t}(e + v) = t^2(e+v/t^2)$. Hence from Lemma \ref{lemad} this $\Cs$-action preserves $\ZZ$.

Our goal will be to prove the following theorem.

\begin{theorem}\label{finsolv}
Let $\Hs$ be a principally paired solvable group acting regularly on a smooth complex projective variety $X$. Then there is a homomorphism
$$\rho: H_\Ts^*(X)\to \C[\ZZ]$$ to be defined in \eqref{rho},
which is an isomorphism of graded $\C[\ttt]$-algebras. Moreover, the zero scheme $\ZZ$ is affine, so that we have the following diagram with horizontal isomorphisms:
$$ \begin{tikzcd}
 \ZZ  \arrow{d}{\pi} \arrow{r}{\rho^*}  &
  \Spec H_\Ts^*(X;\C) \arrow{d} \\
   \ttt \arrow{r}{\cong}&
  \Spec H^*_\Ts.
\end{tikzcd}$$
\end{theorem}

We will first study the structure of $\ZZ$ with connection to the torus-fixed points $\zeta_0$, \dots, $\zeta_s$. We will also prove that $\ZZ$ is reduced. This will allow us to define a map $\rho: H^*_\Ts(X)\to \C[\ZZ]$ by specifying $\rho(c)$ by its values. To show that $\rho(c)$ is a regular function on $\ZZ$, we will prove that $H^*_\Ts(X)$ is generated by Chern classes of $\Hs$-equivariant vector bundles.

\subsection{Equivariant cohomology and Białynicki-Birula decomposition}
We know that the $\Ts$-equivariant cohomology $H_{\Ts}^*(\pt) = \C[\ttt]$ of the point is the ring of polynomials on $\ttt$. As in Section \ref{seceqform}, by $\I$ we will denote the ideal of polynomials vanishing at $0$, equivalently $\I = \bigoplus_{n>0} H_{\Ts}^n(\pt)$. The multiplicative group $\Cs$ acts on $X$ by the means of the morphism $H:\Cs\to \Hs$, $t\mapsto H^t$. As noted above, this action has finitely many fixed points $\zeta_0$, $\zeta_1$, $\dots$, $\zeta_s$. As outlined in Section \ref{secbb}, we can consider its Białynicki-Birula plus- and minus- decomposition i.e.
$$W_i^+ = \{x\in X: \lim_{t\to 0} H^t\cdot x = \zeta_i \}, \qquad W_i^- = \{x\in X: \lim_{t\to \infty} H^t\cdot x = \zeta_i \}.$$

All those sets are locally closed varieties, isomorphic to affine spaces. By Theorem \ref{cohobb}, the space $X$ has cohomology only in even degrees. Therefore by Theorem \ref{thmfor}, the $\Ts$-space $X$ is equivariantly formal. In particular
\begin{align} \label{formal2}
H_{\Ts}^*(X) \cong H_{\Ts}^*(\pt)\otimes H^*(X)
\end{align} as $H_{\Ts}^*(\pt)$-modules and 
$$H^*(X) \cong H_{\Ts}^*(X)/\I H_{\Ts}^*(X)$$ as $\C$-algebras.

\begin{theorem}\label{ABBst}
Białynicki-Birula plus-decomposition $X = \bigcup_{i=0}^s W_i^+$ is $\Hs$-stable.
\end{theorem}

\begin{proof}
Assume that $x\in W_i^+$, i.e. $\lim_{t\to 0} H^t\cdot x = \zeta_i$. Let $M\in \Hs$ and $x' = Mx$ and let $\zeta_j = \lim_{t\to 0} H^t\cdot x'$. Then
\begin{align}\label{bbeq1}
    H^t x' = H^tMx = (H^t M (H^t)^{-1}) H^t x.
\end{align}
Let $M = D \cdot U$, where $D\in\Ts$ and $U\in \Hs_u$. As $H^t\in \Ts$, it commutes with $D$, therefore
\begin{align}\label{bbeq2}
H^t M (H^t)^{-1} = D H^t U (H^t)^{-1}.
\end{align}
Now as $U \in \Hs_u$, we have $U = \exp(u)$ for some $u\in\he_n$. Here $\exp$ should be understood as the algebraic exponential for unipotent groups, as in Theorem \ref{algexp}. We then have
$$H^t U (H^t)^{-1} = H^t \exp(u) (H^t)^{-1} = \exp(\Ad_{H^t}(u)).$$
By Lemma \ref{posint}, the weights of the $H^t$-action on $\he_n$ are positive. Therefore 
$$\lim_{t\to 0} \Ad_{H^t}(u) = 0,$$
hence $\lim_{t\to 0} H^t U (H^t)^{-1} = 1.$ Combining \eqref{bbeq1} and \eqref{bbeq2} gives
$$H^t x' = D H^t U (H^t)^{-1} H^t x.$$
Passing to limit $t\to 0$ then yields
$$\zeta_j = D \zeta_i.$$
As $\zeta_i$ is fixed by $D\in \Ts$, we get $i=j$, hence $x'\in W_i^+$ as desired.
\end{proof}

\subsection{Structure of \texorpdfstring{$\ZZ$}{Z}}
\label{structure}

In order to prove $H^*_\Ts(X)\cong \C[\ZZ]$ we study the structure of $\ZZ$ and construct a map $H^*_\Ts(X)\to \C[\ZZ]$. Let $(\w,x)\in\ZZ$. This means that $e+\w$ vanishes on $x$ and by Lemma \ref{regzer} it is an isolated zero. From Theorem \ref{jordan}, there exists $M\in \Hs$ such that $e+\w = \Ad_M(\w + n')$, where $[\w,n'] = 0$ and $n'\in [\he,\he]$. Then by Lemma \ref{lemad} we have that $M^{-1} x$ is a zero of $\w+n'$ and from Lemma \ref{lemzer} it is a zero of $\ttt$. Hence we get $x = M\zeta_i$ for some $i\in\{0,1,\dots,s\}$. Moreover, not only is $\zeta_i$ a zero of $\ttt$, but also of $n'$.

\begin{example}
We continue Example \ref{exjor} and use the notation from Example \ref{exgr} for the elements of $\ttt$.
\begin{enumerate}
\item Let $\w\in\ttt\cong \C^2$ be of the form $\w = (v_1,v_2)$ with $v\neq 0$, $w\neq 0$, $v\neq w$. We know that $
 e+\w = 
M_{\w} \w M_{\w}^{-1}
$ and therefore any zero of $e+\w$ is of the form $x = M_\w \zeta_i$ and conversely, for any $i$, the point $M_\w \zeta_i$ is fixed by $\w+e$.
\item If $\w = (v_1,0)$ with $v_1\neq 0$, then we have a matrix $M_\w\in\Bs_3$ such that
$$(e+\w) = M_{\w}
\begin{pmatrix}
 -v_1/3 & 0 & 1 \\
 0 & 2v_1/3 & 0 \\
 0 & 0 & -v_1/3
\end{pmatrix} M_{\w}^{-1}.$$
Therefore every zero of $e+\w$ is of the form $x = M_\w \zeta_i$ for $i$ such that $\zeta_i$ is also a zero of $$E_{13}= \begin{pmatrix}
 0 & 0 & 1 \\
 0 & 0 & 0 \\
 0 & 0 & 0
\end{pmatrix}.$$ But conversely, if $\zeta_i$ is additionally a zero of $E_{13}$, then $M_\w\zeta_i$ is a zero of $e+\w$.
\end{enumerate}
\end{example}

\begin{remark}\label{remun}
By Theorem \ref{ABBst}, if $x = M\zeta_i$, then $\zeta_i$ is in the same plus-cell as $x$. But $\zeta_i$ itself is a torus-fixed point, hence $\zeta_i = \lim_{t\to 0} H^t\cdot x$. In particular, this means that regardless of the potential choice of $M$ we might make, we always get the same torus-fixed point, i.e. if $x = M_1\zeta_{i_1} = M_2\zeta_{i_2}$, then $i_1 = i_2$. The elements $M$ and $n'$ are however not unique.

Note that for $i=0,1,\dots,s$ and $\w\in\ttt$, there is at most one zero of $e+\w$ in the plus-cell of $\zeta_i$. Indeed, assume that there are two such points. By above, if we choose any $M$ such that $e+\w = \Ad_M(\w+n')$, then they are of the forms $x_1 = M\zeta_{i_1}$, $x_2 = M\zeta_{i_2}$. But as in the last paragraph, in fact we have $i_2 = i_1 = i$. Therefore $x_1 = x_2$.
\end{remark}

If $M^{-1}x$ is a fixed point of the torus and $w$ is not regular, then in general it does not follow that $e+w$ vanishes at $x$, it does so only for particular torus-fixed points. Assume that we are given $\w\in\ttt$ and $M_\w\in \Hs$, $n'\in\he_n$ such that $e+\w = \Ad_{M_\w}(\w+n')$ and $[w,n']=0$. In this case if $\zeta_i$ is a zero of $n'$, then $M_\w \zeta_i$ is a zero of $e+\w$. However, for given $\w$, the corresponding vector field $V_{n'}$ in general does not vanish in all the torus-fixed points.

\begin{example}
Let us consider the standard action of $\Bs_3$ on $\PP^2$, i.e. we define
$$
\begin{pmatrix}
a & b & c \\
0 & d & e \\
0 & 0 & f
\end{pmatrix}
\cdot [v_0:v_1:v_2]
= [u_0:u_1:u_2]
$$
for $u_0$, $u_1$, $u_2$ such that
$$\begin{pmatrix}
a & b & c \\
0 & d & e \\
0 & 0 & f
\end{pmatrix}
\begin{pmatrix}
v_0 \\ v_1 \\ v_2
\end{pmatrix}
= 
\begin{pmatrix}
u_0 \\ u_1 \\ u_2
\end{pmatrix}.
$$

We have three torus-fixed points $\zeta_1=o=[1:0:0]$, $\zeta_2 = [0:1:0]$, $\zeta_3= [0:0:1]$. For $\w = (v_1,v_2)\in\C^2\cong \ttt$ regular there exists a matrix $M_\w$ such that $e+\w = M_\w \w M_\w^{-1}$. Then $M_\w \zeta_i$ is a fixed point of $e+\w$ for $i=1,2,3$.

However, if $\w = (v_1,0)$ with $v_1\neq 0$, then there exists a matrix $M_\w$ such that $e+\w = M_\w (\w + e_{13}) M_\w^{-1}$. The vector field $V_{e_{13}}$ corresponding to $e_{13}$ vanishes at $\zeta_1$ and $\zeta_2$ (but not at $\zeta_3$), therefore the zeros of $e+\w$ are exactly of the forms $M_\w \zeta_1$ and $M_\w \zeta_2$.

Specializing even more, if we consider $\w = (0,0)$, then $e+\w = e$ is already a Jordan matrix (we can take $M_\w = I_3$). Its only zero is $\zeta_1=o$, so the only fixed point of $e+\w$ is $o$.
\end{example}

We will define a map $H_\Ts^*(X)\to \C[\ZZ]$ by constructing, for each element of $H_\Ts^*(X)$, a function in $\C[\ZZ]$ by its values. To ensure that it is well defined, we first show that $\ZZ$ is reduced.

Remember that we defined a $\Cs$-action on $X$ and $\ttt$ -- see the comment below Definition \ref{defz}. It turns out (\cite[Proposition 1]{CarDef}) that if we consider the Białynicki-Birula minus-decomposition on $X$, then the minus-cell $X_o:=W_0^-$ corresponding to $o$ is open. In other words, all the weights of the action around $o$ are negative. Remember from Section \ref{secgrpac} that $t\in \Cs$ acts on $\C[X_o]$ via $(t^{-1})^*$. Therefore we can choose on $X_o$ the affine coordinates $x_1$, $x_2$, \dots, $x_n$ that are weight vectors of $\Ts$ and the values of weights on $h$ are positive integers $a_1$, $a_2$, \dots, $a_n$. Using these coordinates we model $X_o$ as a vector space, thus we can identify the tangent spaces to its points with $X_o$ itself.

We also have the grading on $\C[\ttt]$ defined by the action of $\Cs$ on $\ttt$, which by definition is of weight $-2$. Therefore choosing coordinates $v_1,\dots,v_r$ on $\ttt$ we have
$$\C[\ttt\times X_o] = \C[v_1,v_2,\dots,v_{r},x_1,x_2,\dots,x_n]$$
with $\deg v_i = 2$ (for $i=1,2,\dots,r$), $\deg x_i = a_i$ (for $i=1,2,\dots,n$). The tangent bundle of $X_o$, as an affine space, is trivial, and the coordinates on $X_o$ define its trivialisation, hence we can speak of coordinates of $V_{e+\ttt}$, cf. Remark \ref{remtan}. The vector field $V_{e+\ttt}$ is vertical, i.e. only in the direction of $X$, hence it has $n$ coordinates. We now prove the following lemma, which for $\Hs = \Bs_2$ was proven in \cite[Theorem 4]{CarDef}.

\begin{lemma}\label{lemcm}
The scheme $\ZZ$ is complete intersection and reduced and contained in $\ttt\times X_o$, hence affine. The ideal of $\ZZ$ in $\C[\ttt\times X_o] = \C[v_1,v_2,\dots,v_{r},x_1,x_2,\dots,x_n]$ is then generated by the vertical coordinates of the vector field $V_{e+\ttt}$:
$$\left(V_{e+\ttt}\right)_1,\left(V_{e+\ttt}\right)_2,\dots, \left(V_{e+\ttt}\right)_n.$$
The degree of each $\left(V_{e+\ttt}\right)_i$ is equal to $a_i+2$ and together with $v_1$, $v_2$, \dots, $v_{r}$ they form a homogeneous regular sequence in $\C[v_1,v_2,\dots,v_{r},x_1,x_2,\dots,x_n]$.
\end{lemma}

\begin{proof}
First, let us see that $\ZZ$ is contained in $\ttt\times X_o$. Let $(\w,x)\in \ZZ$. We then know that $x$ is a zero of the vector field $V_{e+\w}$. For any $t\in\Cs$ by Lemma \ref{lemad} we have that $H^t\cdot x$ is a zero of $V_{\Ad_{H^t}(e+\w)}$. As $\Ad_{H^t}(e+\w) = t^2e + \w$, this means that $H^t \cdot x$ is a zero of $e+t^{-2}\w$. When we take $t\to\infty$, this converges to $e$. Therefore $\lim_{t\to\infty} H^t\cdot x$ is a zero of $e$, hence equal to $o$. This means that $x\in X_o$.

Now we will prove that $\left(V_{e+\ttt}\right)_i$ is homogeneous of degree $a_i+2$. We have
$$\left(V_{e+\ttt}\right)_i |_{t\cdot(x,\w)} = 
\left(V_{e+\w/t^2}|_{H^t\cdot x}\right)_i =
\left(H^t_* (V_{\Ad_{H^{t^{-1}}}(e+\w/t^2)}|_x)\right)_i =
\left(H^t_* (V_{e/t^2+\w/t^2}|_x)\right)_i
$$
and $H^t$ acts on $i$-th coordinate of tangent space by multiplying it by $t^{-a_i}$, therefore
$$\left(H^t_* (V_{e/t^2+\w/t^2}|_x)\right)_i = t^{-a_i} \left(V_{e/t^2+\w/t^2}|_x\right)_i
=t^{-a_i-2} \left(V_{e+\ttt}\right)_i |_{(x,\w)},$$
hence the homogeneity follows. Since $v_1$, $v_2$, \dots, $v_{r}$ have degree $2$, we have that the sequence
$$\left(V_{e+\ttt}\right)_1,\left(V_{e+\ttt}\right)_2,\dots, \left(V_{e+\ttt}\right)_n, v_1, v_2, \dots, v_{r}$$
consists of homogeneous functions on the $(r+n)$-dimensional affine space $\ttt\times X_o$. There are $r+n$ of them and they have only one common zero. Therefore by \cite[Proposition 4.3.4]{Benson} they form a regular sequence. In particular, $\ZZ$ is the zero scheme of a regular sequence $\left(V_{e+\ttt}\right)_1$, $\left(V_{e+\ttt}\right)_2$, \dots, $\left(V_{e+\ttt}\right)_n$, therefore it is complete intersection and hence also Cohen--Macaulay, see Section \ref{seccm}.

Now we have to prove that $\ZZ$ is reduced. Let $\pi:\ZZ\to\ttt$ be the projection to the first factor of $\ttt\times X$. By Theorem \ref{unif} we get an isomorphism $\pi^{-1}(\ttt^\reg) \cong \ttt^\reg\times X^\Ts$. The first factor, as an open subscheme of affine space, is reduced. The fixed points of the torus are also reduced by Theorem \ref{fixred}, therefore $\pi^{-1}(\ttt^\reg)$ is reduced.

Now note that $\pi^{-1}(\ttt^\reg)$ is an open dense subset in $\ZZ$. It is open because $\ttt^\reg$ is open in $\ttt$. To prove that it is dense, assume on the contrary that there exists $x\in \ZZ\setminus\overline{\pi^{-1}(\ttt^\reg)}$. Let $Y$ be its irreducible component in $\ZZ$. As $\ZZ = \overline{\pi^{-1}(\ttt^\reg)} \cup \pi^{-1}(\ttt\setminus\ttt^\reg)$ and both sets are closed, by irreducibility $Y$ has to be contained in one of them. As $x$ is not contained in the former, $Y$ has to be contained in the latter, so that $\pi(Y)\subset \ttt\setminus\ttt^\reg$. As $\ttt\setminus\ttt^\reg$ is a union of hyperplanes in $\ttt$, the same argument shows that $\pi(Y)$ lies within one of them (of dimension $r-1$). Considering $\pi|_Y$ as mapping to $\overline{\pi(Y)}$ and reducing if needed, we get a dominant map between integral schemes. Note that as $\ZZ$ is Cohen--Macaulay, it is equidimensional by \cite[Theorem 17.6 and Theorem 6.5]{Matsu}. As $\ttt\times\{o\}$ is closed in $\ZZ$ and of dimension $r$, the dimension of $\ZZ$ is at least $r$. Therefore by the fiber dimension theorem (see \cite[Ex. 3.22(b)]{Hart}) the fibers of $\pi|_Y$ are at least one-dimensional. But they are finite by Lemma \ref{regzer}, so we get a contradiction.

Now as $\pi^{-1}(\ttt^\reg)$ is an open dense subset in $\ZZ$, it contains its generic points, hence $\ZZ$ is generically reduced. Using that $\ZZ$ is Cohen--Macaulay, by \cite[Proposition 14.124]{GW} we get that $\ZZ$ is reduced.
\end{proof}

\subsection{The homomorphism \texorpdfstring{$\rho$}{ρ}}\label{sectionmap}
Let $c \in H_\Ts^*(X)$. In Section~\ref{structure} we show that every element $(\w,x)$ of $\ZZ$ satisfies $x = M_\w \zeta_i$. Here $M_\w$ is some element of $\Hs$ depending on $\w$ and $\zeta_i$ is a uniquely determined fixed point of $\Ts$-action, which in fact can be determined as a limit of $\Cs$ action on $x$. The localisation $c|_{\zeta_i}$ of $c$ to the torus-fixed point can be now seen as a polynomial on $\ttt$, because $H_\Ts^*(\pt) = \C[\ttt]$.
We then define
\begin{align} \label{rho} \rho(c)(\w,x) = c|_{\zeta_i}(\w).\end{align}
This follows the idea of \cite{BC}, where $\rho$ is defined this way for $\Bs_2$.
For any $c\in H_\Ts^*(X)$ this defines a function $\rho(c)$ on the set of closed points $\ZZ$. This clearly gives a $\C[\ttt]$-homomorphism between $H_\Ts^*(X)$ and the algebra of all $\C$-valued functions on $\ZZ$. We have to first prove that for any $c\in H_\Ts^*(X)$ the image $\rho(c)$ defines a regular function, which will be unique by reducibility, proved in Lemma \ref{lemcm}. Thus we will get a $\C[\ttt]$-homomorphism
$$\rho: H_\Ts^*(X)\to \C[\ZZ].$$
In general, assume that we are given an algebraic group $\Hs$ and an $\Hs$-variety $A$. For any $\Hs$-linearised bundle $\Ee$ on $A$ we may consider its equivariant Chern classes $c^\Hs_k(\Ee)\in H^{2k}_\Hs(A)$. Let $p\in A$ be a fixed point of $\Hs$. From naturality of Chern classes we get $c^\Hs_k(\Ee)|_p = c^\Hs_k(\Ee_p)$, where $\Ee_p$ is the fiber of $\Ee$ over $p$. This belongs to $H_\Hs^*(\pt)\subset \C[\he]$ and for any $y\in \he$ we get
\begin{align}
\label{cherntr}
c^\Hs_k(\Ee)|_p(y) = \Tr_{\Lambda^k \Ee_p}(\Lambda^k y_p).
\end{align}
Here $y_p$ is the infinitesimal action of $y\in\he$ on $\Ee_p$, which is a representation of $\Hs$.

\begin{lemma}\label{nicefun}
 Let $\Ee$ be an $\Hs$-linearised vector bundle on $X$ and let $k$ be a nonnegative integer. Then for any $(\w,x)\in\ZZ$ we have
 $$\rho(c_k^\Ts(\Ee))(\w,x) = \Tr_{\Lambda^k \Ee_x}(\Lambda^k (e+\w)_x).$$
 In particular, $\rho(c_k^\Ts(\Ee))\in \C[\ZZ]$.
\end{lemma}

\begin{proof}
We have $x = M_\w \zeta_i$ for some $\zeta_i\in X^\Ts$ and $M_\w\in\Hs$. Moreover,
$$e+\w = \Ad_{M_\w}(\w+e')$$
for some $e'\in \he_n$ that vanishes at $\zeta_i$ and commutes with $\w$. Note that, as $\Ee$ is $\Hs$-linearised, 
$$\Tr_{\Lambda^k \Ee_x}(\Lambda^k (e+\w)_x) = \Tr_{\Lambda^k \Ee_{M_\w^{-1}x}}(\Lambda^k \left( \Ad_{M_\w^{-1}}(e+\w) \right)_{M_\w^{-1}x})
= \Tr_{\Lambda^k \Ee_{\zeta_i}}(\Lambda^k(\w+e')_{\zeta_i}).
$$
From \eqref{rho} and \eqref{cherntr} we have
$$\rho(c_k^\Ts(\Ee))(\w,x) = c_k^\Ts(\Ee)|_{\zeta_i}(\w) = \Tr_{\Lambda^k \Ee_{\zeta_i}}(\Lambda^k w_{\zeta_i}).$$
Thus we have to prove that 
$$ \Tr_{\Lambda^k \Ee_{\zeta_i}}(\Lambda^k(\w+e')_{\zeta_i}) = \Tr_{\Lambda^k \Ee_{\zeta_i}}(\Lambda^k w_{\zeta_i}).$$
But by the assumptions that $[\w,e']=0$, $\w$ is semisimple and $e'$ is nilpotent, we get that the sum $\w + e'$ is the Jordan decomposition of $\Ad_{M_\w^{-1}}(e+\w)$ in the sense of Theorem \ref{defjord}. Then by the naturality of the Jordan decomposition the derivative of the representation $\Stab_\Hs(\zeta_i)\to \GL(\Ee_{\zeta_i})$ preserves it. Therefore $\w_{\zeta_i}$ seen as an element of $\gl(\Ee_{\zeta_i})$ is the semisimple part of $(\w+e')_{\zeta_i}$ seen as an element of $\gl(\Ee_{\zeta_i})$.

But for Jordan decomposition in the general linear group, the eigenvalues of the semisimple part are the same as the eigenvalues of the decomposed element. Because traces of external powers are polynomials in eigenvalues, this concludes the proof.
\end{proof}
{The following lemma is based on \cite[Proposition 3]{CarDef}, which proves it for $\Bs_2$.}
\begin{lemma}\label{chern1}
The cohomology ring $H^*(X)$ is generated, as a $\C$-algebra, by Chern classes of $\Hs$-linearised vector bundles on $X$.
\end{lemma}

\begin{proof}
We know that the fundamental classes of the plus--cells form a basis of $H_*(X)$, hence their Poincar\'{e} duals form a basis of $H^*(X)$. Now we use Baum--Fulton--MacPherson's Grothendieck--Riemann--Roch theorem (see \cite[Theorem 18.3, (5)]{Ful}). We get that for any plus-cell $W_i\in X$ the homology class $(\ch(W_i) \td(X_i))\cap [X]$ is equal to the sum of $[W_i]$ and lower-degree terms. Therefore $\ch(W_i)$ is equal to the sum of the dual class of $[W_i]$ and higher-degree terms. Therefore Chern characters of the structure sheaves of plus--cells generate $H^*(X)$. 

As the plus--cells are $\Hs$-stable by Theorem \ref{ABBst}, we get that $\ch$ is surjective when restricted to the Grothendieck group of $\Hs$-equivariant coherent sheaves. By \cite[Corollary 5.8]{Thomason} it is generated by the classes of $\Hs$-equivariant vector bundles and the conclusion follows.
\end{proof}

\begin{remark} \label{remarkchern}
We did not use the regularity of the action in the proof. In fact, it was enough to know that the fixed points of $\Ts$ are isolated. One could also argue the following in such generality. By Theorem \ref{solvsplit} a linear solvable group over $\C$ is split, i.e. admits a composition series with quotients equal to $\Gs_a$. Then the restriction $K^0_\Hs(X) \to K^0_\Ts(X)$ is an isomorphism \cite[Corollary 2.16]{Merk} and the restriction $K^0_\Ts(X)\to K^0(X)$ is a surjection \cite[Proposition 3.1]{Merk}. The Chern character is an isomorphism from $K^0(X)\otimes \C$ to $A^*(X)\otimes \C$ \cite[Theorem 18.3]{Ful} and the cycle class map $A^*(X)\to H^*(X,\Z)$ is an isomorphism due to the paving given by Białynicki-Birula decomposition \cite[Example 19.1.11]{Ful}. Therefore the (non-equivariant) Chern character gives a surjection $K^0_\Hs(X)\to H^*(X,\C)$.
\end{remark}

\begin{lemma}\label{chern2}
The equivariant cohomology $H_\Ts^*(X)$ is generated, as a $\C[\ttt]$-algebra, by $\Ts$-equivariant Chern classes of $\Hs$-equivariant vector bundles on $X$.
\end{lemma}

\begin{proof}
Recall that $\I$ denotes the maximal ideal of $\C[\ttt]$ cutting out the zero point. Since $X$ is equivariantly formal, we have an exact sequence
$$0\to \I H_\Ts^*(X)\to H_\Ts^*(X)\to H^*(X)\to 0.$$
By Lemma \ref{chern1} we get that the $\C$-algebra $H^*(X)$ is generated by Chern classes of $\Hs$-linearised vector bundles on $X$. Then from graded Nakayama lemma (see Corollary \ref{cornak}) the $\C[\ttt]$-algebra $H^*_\Ts(X)$ is generated by their equivariant Chern classes.
\end{proof}

\noindent
This together with Lemma \ref{nicefun} gives
\begin{corollary}
The map $\rho$ is a homomorphism of $\C[\ttt]$-algebras $H_\Ts^*(X)\to\C[\ZZ]$.
\end{corollary}

\subsection{Proof of isomorphism}

\begin{proof}[Proof of Theorem~\ref{finsolv}]
Clearly $\rho$ preserves the grading. For injectivity, note that for any $c\in H_\Ts^*(X)$, we can extract from $\rho(c)$ the localisations $c|_{\zeta_i}$ for all $i$ -- as on the regular locus the function $\rho(c)$ is defined by all those localisations. Recall that $X$ is equivariantly formal \eqref{formal2}. Therefore we get injectivity of $\rho$ by injectivity of localisation on equivariantly formal spaces \cite[Theorem 1.6.2]{GKM}.

Hence to prove that the map is an isomorphism, it suffices to check that the Poincar\'{e} series of the two sides coincide. Since $X$ is equivariantly formal, $H_\Ts^*(X)$ is a free $\C[\ttt]$-module and
$$H_\Ts^*(X)/\I H_\Ts^*(X) \cong H^*(X).$$
Therefore
\begin{align}
\label{PHT}
    P_{H^*(X)}(t) = P_{H_\Ts^*(X)}(t)(1-t^2)^{r}.
\end{align}

On the other hand, from Lemma \ref{lemcm} we know that the generating set of $\I$ is a regular sequence in $\C[\ZZ]$, hence 

\begin{align}
\label{PCZ}
P_{\C[\ZZ]/\I\C[\ZZ]}(t) = P_{\C[\ZZ]}(t)(1-t^2)^{r}.
\end{align}

Now $\C[\ZZ]/\I\C[\ZZ]$ is the zero scheme of the vector field given by $e$. In addition, the action of the torus $H^t$ satisfies $\Ad_{H^t}(e) = t^2 e$. Therefore by Theorem \ref{thmgrad} we have $\C[\ZZ]/\I\C[\ZZ]\cong H^*(X)$ and in particular
$$P_{\C[\ZZ]/\I\C[\ZZ]}(t) = P_{H^*(X)}(t).$$
Therefore, from \eqref{PHT} and \eqref{PCZ} we get
$$P_{\C[\ZZ]}(t) = P_{H_\Ts^*(X)}(t).$$
\end{proof}

\begin{remark}
From Theorem~\ref{finsolv} we get that $\C[\ZZ]$ is a finitely generated free module over $\C[\ttt]$. Therefore the map $\pi:\ZZ\to\ttt$ is finite flat.
\end{remark}

\begin{remark}
The theorem can in fact be proved for a slightly larger class of solvable groups. We need $\Hs$ to be a connected linear algebraic solvable group, and as before $(e,h)$ to be an integrable $\bb(\ssl_2)$-pair, but it does not necessarily have to be principal. For the proof of Theorem \ref{ABBst} we need to assume $\alpha(h) > 0$ for any root $\alpha$ of $\Hs$. However, even this assumption can be made unnecessary as we can consider the subgroup $\Hs'$ generated by $\Ts$ and the additive group generated by $e$. By \cite[Theorem 7.6]{Borel} it is algebraic and its Lie algebra is generated by $\ttt$ and $e$. As Lie bracket of $h$-weight vectors adds the weights, we clearly see that all the weights on $\Hs'$ are nonnegative multiples of $2$.

Even if we assume that $\Hs$ is generated by $\Ts$ and the additive group generated by $e$, it does not follow that $e$ is regular. Take for example
 $$
 \Hs = \left\{\left.
 \begin{pmatrix}
 t/u^2 & * & * & * \\
 0 & t & * & * \\
 0 & 0 & u & * \\
 0 & 0 & 0 & u/t^2
 \end{pmatrix}\right| t,u\in \Cs
 \right\},
 $$
 where the asterisks are understood to stand for any complex numbers. We choose
 $$
 h = 
 \begin{pmatrix}
 3 & 0 & 0 & 0 \\
 0 & 1 & 0 & 0 \\
 0 & 0 & -1 & 0 \\
 0 & 0 & 0 & -3
 \end{pmatrix}, \qquad
 e = 
 \begin{pmatrix}
 0 & 1 & 0 & 0 \\
 0 & 0 & 1 & 0 \\
 0 & 0 & 0 & 1 \\
 0 & 0 & 0 & 0
 \end{pmatrix}.
 $$
 The maximal torus is two-dimensional and the centraliser of $e$ is three-dimensional, hence $e$ is not regular. However together with the diagonal matrices it generates $\he$ as a Lie algebra.
 
 In all our examples of regular actions, we only consider principally paired groups and this extension seems to only include very exotic cases. Therefore we formulate our results in terms of principally paired groups.
\end{remark}

\subsection{Functoriality}

We prove now that Theorem \ref{finsolv} is actually functorial, with respect to both the group and the variety. We prove the latter first.

\begin{proposition} \label{funcprop}
 Assume that $X$ and $Y$ are two $\Hs$-regular varieties and $\phi:X\to Y$ is an $\Hs$-equivariant morphism between them. Let $\ZZ_X \cong \Spec H^*_\Ts(X)$ and $\ZZ_Y \cong \Spec H^*_\Ts(Y)$ be the schemes constructed above for $X$ and $Y$, respectively. The map $(\id,\phi):\ttt\times X\to \ttt\times Y$ induces a morphism $\ZZ_X\to \ZZ_Y$ and the following diagram commutes:
$$
\begin{tikzcd}
H^*_\Ts(Y) \arrow[r, "\phi^*"] \arrow[dd, "\rho_Y"]
& H^*_\Ts(X) \arrow[dd, "\rho_X"]
\\ \\
\C[\ZZ_Y] \arrow[r, "{(\id,\phi)^*}"]
& \C[\ZZ_X]
\end{tikzcd}.
$$
In other words, $\rho$ is a natural isomorphism between the functors $H^*_\Ts$ and $\C[\ZZ]$ on the category of $\Hs$-regular varieties.
\end{proposition}

\begin{proof}
Consider a class $c\in H^*_\Ts(Y)$. We want to show that for any $(\w,x)\in\ZZ_X$ the functions $\rho_X(\phi^*(c))$ and $(\id,\phi)^*(\rho_Y(c))$ take the same value on $(\w,x)$. We know from Section \ref{structure} that $x = M_\w \zeta$, where $M_\w$ is some element of $\Hs$ depending on $\w$, and $\zeta$ is one of $\Ts$-fixed points of $X$. Obviously then $\phi(\zeta)$ is a $\Ts$-fixed point in $Y$ and $\phi(x) = M_\w \phi(\zeta)$. We have then
$$(\id,\phi)^*(\rho_Y(c))(\w,x) = \rho_Y(c)(\w,\phi(x)) = c|_{\phi(\zeta)}(\w).$$
On the other hand
$$\rho_X(\phi^*(c))(\w,x) = \phi^*(c)|_{\zeta}(\w).$$
Now the equality of the above follows from functoriality of $H^*_\Ts$ and commutativity of
$$
\begin{tikzcd}
\{\zeta\} \arrow[dd, "\iota_\zeta"] \arrow[r, "\phi"]
& \{\phi(\zeta)\} \arrow[dd, "\iota_{\phi(\zeta)}"]
\\ \\
X \arrow[r, "\phi"] 
& Y
\end{tikzcd}.
$$
\end{proof}

\begin{proposition}\label{funcgrp}
 Assume that $\Hs_1$, $\Hs_2$ are solvable principally paired groups. Let $\Ts_i\subset \Hs_i$ be the corresponding maximal tori and $e_i\in(\he_i)_n$ the corresponding nilpotent elements in their Lie algebras. Let $\psi:\Hs_1\to\Hs_2$ be a homomorphism of algebraic groups satisfying
 $$\psi(\Ts_1)\subset \Ts_2,\qquad \psi_*(e_1) = e_2.$$
 Assume that $\Hs_2$ acts regularly on a smooth projective variety $X$. Then the map $\psi$ together with the $\Hs_2$-action induce an action of $\Hs_1$ on $X$, which is also regular. In turn, the map $(\psi_*,\id)$ induces a morphism $\ZZ_{\Hs_1} \to \ZZ_{\Hs_2}$ and the following diagram commutes:
$$
\begin{tikzcd}
H^*_{\Ts_2}(X) \arrow[r, "\psi^*"] \arrow[dd, "\rho_{\Hs_2}"]
& H^*_{\Ts_1}(X) \arrow[dd, "\rho_{\Hs_1}"]
\\ \\
\C[\ZZ_{\Hs_2}] \arrow[r, "{(\psi_*,\id)^*}"]
& \C[\ZZ_{\Hs_1}]
\end{tikzcd}.
$$
\end{proposition}

\begin{proof}
 As $\psi_*(e_1) = e_2$, the group $\Hs_1$ clearly acts on $X$ regularly. Obviously if $(\w,x)\in\ZZ_{\Hs_1}$, then $e_1+\w$ vanishes at $x$, and therefore $\psi_*(e_1+\w) = e_2 + \psi(\w)$ vanishes at $x$, hence $(\psi_*,\id)$ maps $\ZZ_{\Hs_1}$ to $\ZZ_{\Hs_2}$.
 \\
 Now let $c\in H^*_{\Ts_2}(X)$ and $(\w,x)\in \ZZ_{\Hs_1}$. We want to prove that
 $$(\psi_*,\id)^*(\rho_{\Hs_2}(c))(\w,x) = \rho_{\Hs_1}(\psi(c))(\w,x).$$
 We know that $x = M_\w \zeta$ for some $M_\w\in \Hs_1$ depending on $\w$ and an isolated $\Ts_1$-fixed point $\zeta$. Then by Lemma \ref{lemfix} the point $\zeta$ is fixed by $\Ts_2$. Therefore (cf. Remark \ref{remun}) we have 
 $$(\psi_*,\id)^*(\rho_{\Hs_2}(c))(\w,x) = \rho_{\Hs_2}(c)(\psi_*(\w),x) = c|_{\zeta}(\psi_*(\w))$$
 and
 $$\rho_{\Hs_1}(\psi(c))(\w,x) = \psi(c)|_{\zeta}(\w).$$
 Now the equality follows from commutativity of
$$
\begin{tikzcd}
H^*_{\Ts_2}(\pt) \arrow[r, "\psi^*"] \arrow[dd, "\cong"]
& H^*_{\Ts_1}(\pt) \arrow[dd, "\cong"]
\\ \\
\C[\ttt_2] \arrow[r, "(\psi_*)^*"]
& \C[\ttt_1]
\end{tikzcd}.
$$
\end{proof}

\subsection{Examples and comments}
We illustrate Theorem \ref{finsolv} with a few examples.

\begin{figure}[ht!]
\begin{center}
 \includegraphics[width=10cm]{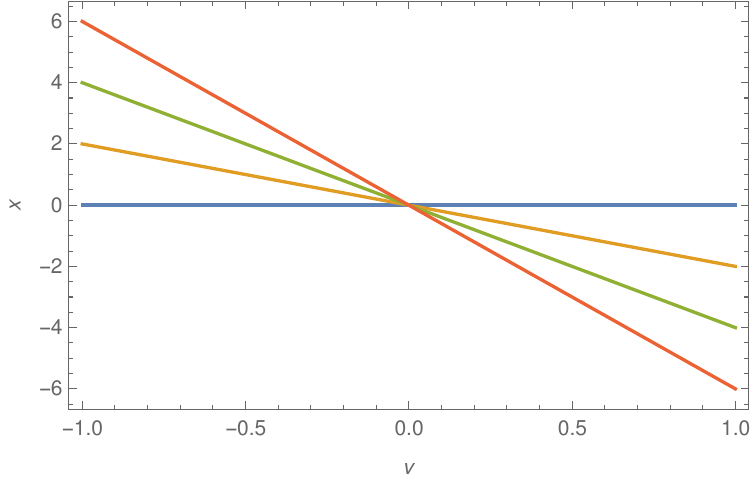}
\end{center}
\caption[Spectrum of $\Cs$-equivariant cohomology of $\PP^3$]{$\Spec H_{\Cs}^*(\PP^3)$.}
\label{hcp4}
\end{figure}

\begin{example}\label{exsl22}
 We continue Example \ref{exsl2} which already appears in \cite{BC}. The point $o = [1:0:\dots:0]$ is the unique zero of $e$. 
 If $[z_0:z_1:\dots:z_n]$ are the homogeneous coordinates of $\PP^n$, then the scheme $\ZZ$ lies completely in the affine chart $X_o$ of $o$, with affine coordinates $x_i = z_i/z_0$, for $i=1,2,\dots,n$. We have
 $$V_h|_{x_1,\dots,x_n} = (-2x_1,-4x_2,\dots,-2nx_n)$$
 and
 $$V_e|_{x_1,\dots,x_n} = (x_2-x_1x_1, x_3 - x_1x_2, x_4 - x_1x_3,\dots,x_n-x_1x_{n-1},-x_1x_n).$$
 Then
 \begin{multline*}
 V_{e+vh}|_{x_1,\dots,x_n} =\\
  (x_2-x_1(x_1+2v), x_3 - x_2(x_1+4v), \dots, x_n-x_{n-1}(x_1+2(n-1)v), -x_n(x_1+2n)).
 \end{multline*}
 If we consider the zero scheme $\ZZ$ of $e+vh$ within $\ttt\times X_o$, then the coordinates $x_2,\dots, x_n$ are clearly determined by $x_1$ and $v$ and we can identify $\ZZ$ with the subscheme of $\Spec \C[v,x_1]$ cut out by the equation
 $$x_1(x_1+2v)(x_1+4v)\dots(x_1+2nv) = 0.$$
 In other words, $H_{\Cs}^*(\PP^n) = \C[v,x]/\big(x(x+2v)(x+4v)\dots(x+2nv)\big)$ with $\deg v = \deg x = 2$. See Figure \ref{hcp4}.
\end{example}

\begin{remark}
Clearly a product $X\times Y$ of two varieties with a regular $\Hs$-action is also regular and its equivariant cohomology scheme can be represented as a fiber product, i.e. $H_{\Ts}^*(X,Y) = H_{\Ts}^*(X) \otimes_{H_{\Ts}^*} H_{\Ts}^*(Y)$.

In particular the product $\PP^1 \times \PP^1$ is regular under the action of $\SL_2$, hence also of $\Bs_2$. It embeds in $\PP^3$ via the Segre embedding. The action of $\SL_2$ on $\PP^3$ from Example \ref{exsl2} is also regular. However the Segre embedding cannot be $\SL_2$- or even $\Bs_2$-equivariant with respect to those two actions. In fact, using Theorem \ref{finsolv} we can prove a more general statement:
\end{remark}

\begin{corollary} \label{surj}
Let a principally paired solvable group $\Hs$ act regularly on a smooth projective variety $X$. Assume that $Z$ is its closed, smooth, $\Hs$-invariant subvariety. Then the induced map on cohomology rings
$$f^*: H^*(X,\C) \to H^*(Z,\C)$$
is surjective.
\end{corollary}

\begin{proof}
Clearly $Z$ is also an $\Hs$-regular variety. From Theorem \ref{finsolv} we have $H^*_\Ts(X) = \C[\ZZ_X]$ and $H^*_\Ts(Z) = \C[\ZZ_Z]$ where $\ZZ_X$ and $\ZZ_Z$ are the zero schemes constructed for $X$ and $Z$ according to Definition \ref{defz}. But clearly from the definition we see that $\ZZ_Z$ is the (reduced) intersection $\ZZ_X\cap Z$, hence a closed subvariety of $\ZZ_X$. This means that the induced map $\C[\ZZ_X]\to \C[\ZZ_Z]$ is surjective. By Proposition \ref{funcprop} this is the same as the map induced on equivariant cohomology. By equivariant formality we get the non-equivariant cohomology by tensoring with $\C$ over $H^*_\Ts$, and this operation is right-exact, hence it preserves surjectivity.
\end{proof}

In particular, as $h^2(\PP^1\times\PP^1) = 2$, the product $\PP^1\times\PP^1$ cannot be embedded $\Bs_2$-equivariantly in any $\PP^m$ with regular action.

\begin{figure}[ht!]
\begin{center}
\subfloat{
  \includegraphics[width=7cm]{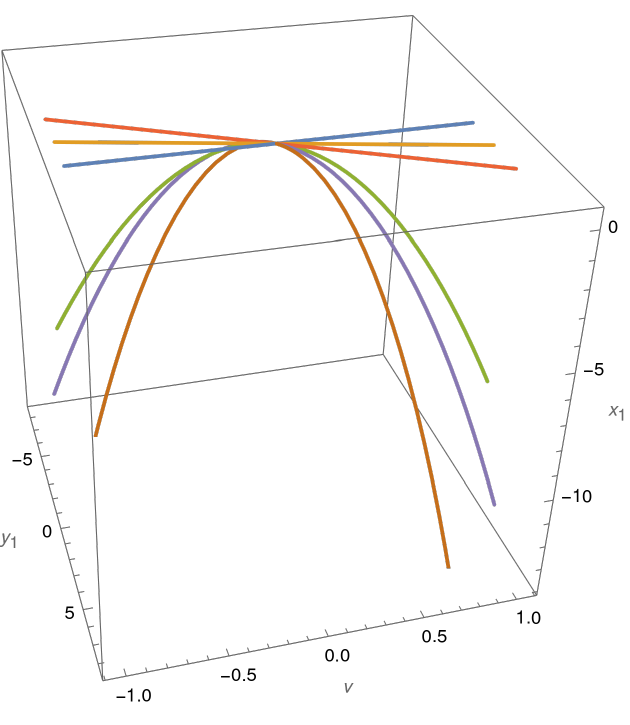}}
  \hfill
\subfloat{
  \includegraphics[width=8.2cm]{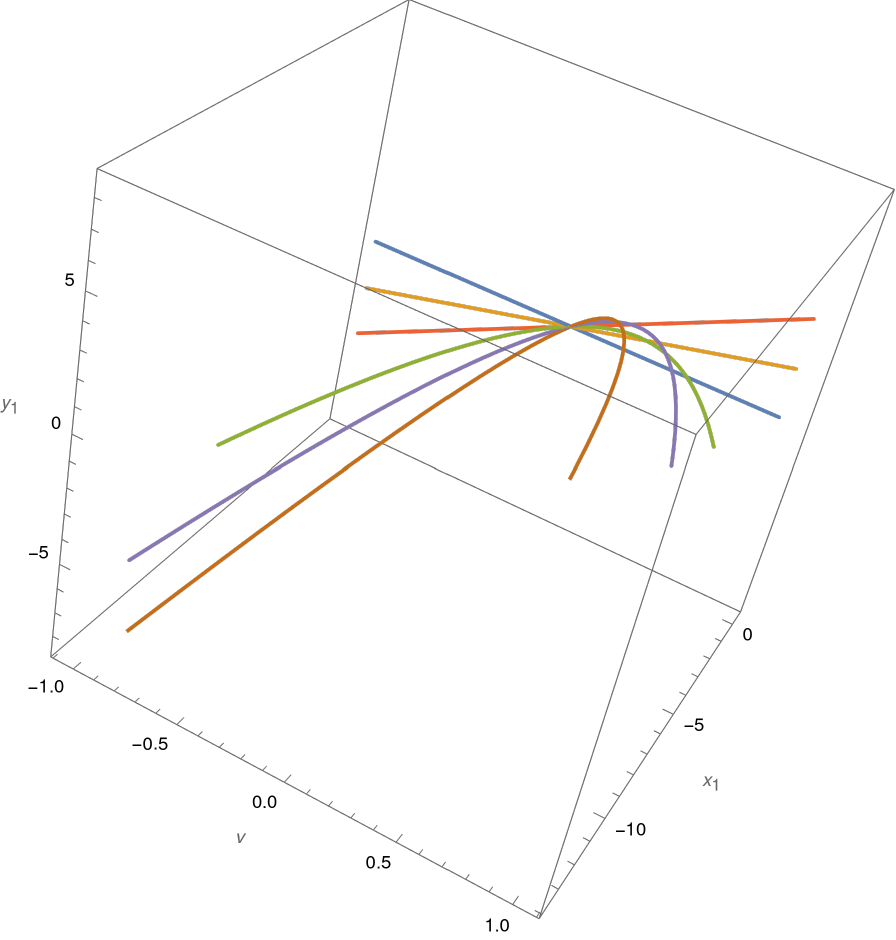}}
\end{center}
\caption[Spectrum of $\Cs$-equivariant cohomology of $\Gr(2,4)$]{Two different views of $\Spec H_{\Cs}^*(\Gr(2,4))$. Note that all the components project bijectively to the $v$ axis.}
\label{hcgr}
\end{figure}

\begin{example}\label{csgr}
In Example \ref{exsl2} we considered an action of $\SL_2(\C)$ on any $\C^n$. We can also use it to define actions on partial or full flag varieties. Let us consider the action of the upper Borel subgroup of $\SL_2$ on $\C^4$ and the induced action on the Grassmannian $\Gr(2,4)$ of two-planes in $\C^2$. We can identify it with $\SL_4(\C)/\Ps$, where $\Ps$ is the parabolic group of matrices of the form
$$\begin{pmatrix}
    * & * & * & * \\
    * & * & * & * \\
    0 & 0 & * & * \\
    0 & 0 & * & * \\
   \end{pmatrix}.
$$
The only zero of $e$ is $o=\Span(e_1,e_2)$ and in the representation above $X_o$ can be thought of as the set of classes of matrices of the form
$$\begin{pmatrix}
    1 & 0 & * & * \\
    0 & 1 & * & * \\
    x_1 & y_1 & * & * \\
    x_2 & y_2 & * & * \\
   \end{pmatrix}.
$$
Then if we write down the coordinates $x_1$, $y_1$, $x_2$, $y_2$ in this order, one checks that
$$V_e|_{x_1,y_1,x_2,y_2} = (x_2-x_1y_1,-x_1-y_1^2+y_2,-x_1y_2,-x_2-y_1y_2)$$
and
$$V_h|_{x_1,y_1,x_2,y_2} = (4x_1,2y_1,6x_2,4y_2).$$
Therefore the equations of $\ZZ$ in $\C[v,x_1,y_1,x_2,y_2]$ are
$$4vx_1 + x_2-x_1y_1 = 0,\quad
2vy_1-x_1-y_1^2+y_2 = 0, \quad
6vx_2-x_1y_2 = 0,\quad
4vy_2-x_2-y_1y_2 = 0.$$
We can determine $x_2$ and $y_2$ from the first two equations and plugging in to the other two, we get
$$
x_1 (x_1 + 24 v^2 - 8 v y_1 + y_1^2) = 0,\quad
(y_1-4v) (2 x1 - 2 vy_1 + y_1^2) = 0.
$$
This gives six one-parameter families of solutions (one for each torus-fixed point):
\begin{align*}
&(x_1 = 0, y_1 = 0);&\quad 
&(x_1 = 0, y_1 = 2v);&\quad 
&(x_1 = -8v^2, y_1 = 4v);\\
&(x_1 = 0, y_1 = 4v);&\quad 
&(x_1 = -12v^2, y_1 = 6v);&\quad
&(x_1 = -24v^2, y_1 = 8v);
\end{align*}
see Figure \ref{hcgr}.
\end{example}

\begin{figure}[ht!]
\begin{center}
 \includegraphics[width=7cm]{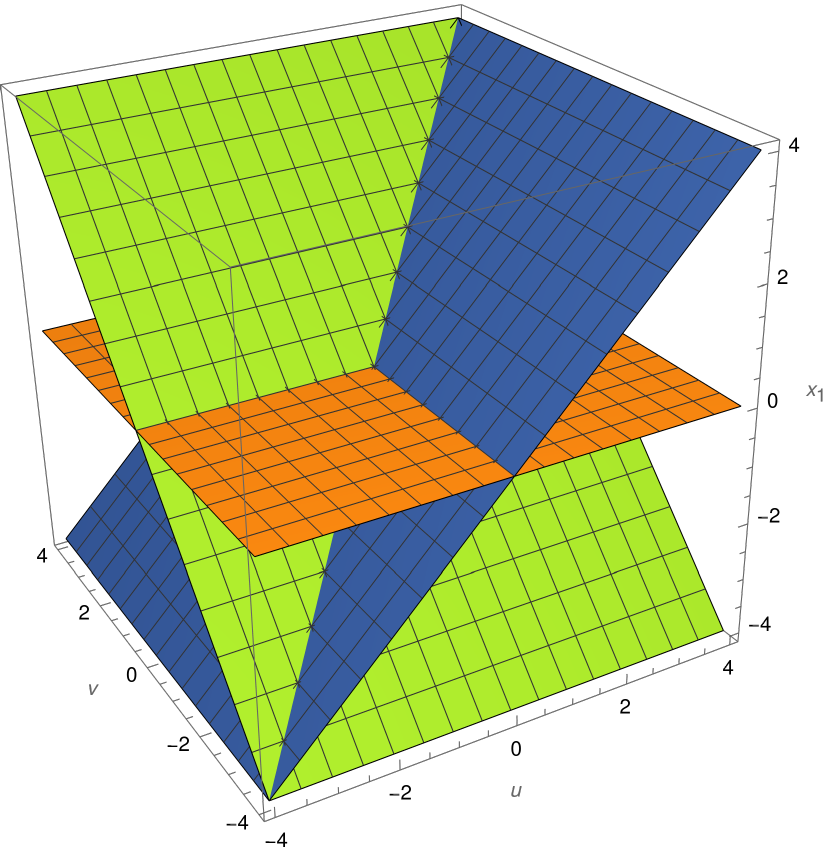}
\end{center}
\caption[Spectrum of torus-equivariant cohomology of $\Gr(2,4)$]{$\Spec H_{\Ts}^*(\PP^2)$.}
\label{htpn}
\end{figure}

\begin{example}\label{exsl3p2}
 Let us now switch to groups of higher rank. As in Example \ref{exgr2}, we can consider the regular nilpotent
 $$e = \begin{pmatrix}
       0 & 1 & 0 & 0 & \dots & 0\\
       0 & 0 & 1 & 0 & \dots & 0\\
       0 & 0 & 0 & 1 & \dots & 0\\
       \vdots & \vdots & \vdots & \vdots & \ddots & \vdots \\
       0 & 0 & 0 & 0 & \dots & 1 \\
       0 & 0 & 0 & 0 & \dots & 0
      \end{pmatrix}
 $$
 in $\SL_{n+1}$. We have the regular action of $\SL_{n+1}$ on $\PP^n$, which in particular restricts to a regular action of its upper Borel subgroup. We continue using notation from Example \ref{exgr} for the elements of $\ttt$. As in Example \ref{exsl22}, we have 
 $$V_e|_{x_1,\dots,x_n} = (x_2-x_1x_1, x_3 - x_1x_2, x_4 - x_1x_3,\dots,x_n-x_1x_{n-1},-x_1x_n).$$
 For the element $(v_1,v_2,\dots,v_{n})\in\C^n$, which corresponds to the diagonal matrix
 $$\diag(0,v_1,v_2,\dots,v_{n}) - \frac{v_1+v_2+\dots+v_{n}}{n+1} I_{n+1},$$ the associated vector field at $(x_1,x_2,\dots,x_n)$ has coordinates equal to
 $(v_1x_1,v_2x_2,\dots,v_nx_n)$. Hence 
 \begin{multline*}
 V_{e+(v_1,v_2,\dots,v_n)}|_{x_1,\dots,x_n} = \\
 (x_2-x_1(x_1-v_1), x_3 - x_2(x_1-v_2), \dots, x_n-x_{n-1}(x_1-v_{n-1}), -x_n(x_1-v_n)).
 \end{multline*}
 Thus we can determine $x_2$, $x_3$, \dots, $x_n$ from $x_1$ and $v_1$, $v_2$, \dots, $v_n$. The scheme $\ZZ$ can be then realised within $\Spec\C[v_1,v_2,\dots,v_n,x_1]$ and cut out by one equation
 $$x_1(x_1-v_1)(x_1-v_2)\dots(x_1-v_n) = 0.$$
 This scheme consists of $n+1$ hyperplanes. Their intersections, when projected on the $(v_1,\dots,v_n)$-plane, form the toric fan of $\PP^{n}$. The functions on the scheme consist of $n+1$ polynomials, one for each component, that agree on the intersections. This agrees with the classical description of equivariant cohomology of toric variety as piecewise polynomials on the fan, as e.g. in \cite[section 2.2]{Briontor}.
 For $n=2$ the scheme is depicted in Figure \ref{htpn}.
\end{example}

 \begin{example}
 \label{exflag}
  We can extend the previous example to full flag varieties. Take for example the variety $F_3 = \SL_3/\Bs$ of full flags in $\C^3$. The only zero of $e$ is the flag $\Span(e_1)\subset\Span(e_1,e_2)$ and the cell $X_o$ in this case consists of the flags represented by matrices of the form
 $$
 \begin{pmatrix}
  1 & 0 & * \\
  a & 1 & * \\
  b & c & *
 \end{pmatrix} \in \SL_3(\C).
 $$
 One finds that
 $$V_e|_{a,b,c} = (-a^2+b,-ab,-b+ac-c^2).$$
 If, as before, we consider a pair $\w = (v_1,v_2)\in\C^2$ as an element of $\ttt$, then we have
 $$V_\w|_{a,b,c} = 
 (v_1a,v_2b,(v_2-v_1)c).$$
 Hence the equations for $V_{e+\w} = 0$ are
 $$-a^2+b+v_1a = 0,\qquad -ab+v_2b,\qquad -b+ac-c^2+(v_2-v_1)c.$$
 Plugging $b$ from the first one into the others yields two equations
 $$a(a-v_1)(a-v_2) = 0,\qquad -a^2+av_1+ac-c^2-cv_1+cv_2 =0.$$
 By splitting the first equation into cases, we easily get the six families of solutions (one for each coordinate flag):
 \begin{align*}
 &(a = 0, c = 0);&\quad 
 &(a = v_1, c = 0);&\quad 
 &(a = v_1, c = v_2);\\
 &(a = v_2, c = v_2);&\quad 
 &(a = 0, c = -v_1+v_2);&\quad
 &(a = v_2, c = -v_1+v_2).
 \end{align*}
 \end{example}
 
\begin{remark}
One can note in the above examples that the scheme $\ZZ$ always consists of finitely many copies of the base $\ttt$. The essential information is then contained not in the shape of the components, but in the way they intersect. In fact, we can prove the general statement:
\end{remark}

\begin{lemma} \label{solvcomps}
Assume that a solvable principally paired group $\Hs$ acts regularly on a smooth projective variety $X$. Then every irreducible component of $\ZZ$ is mapped isomorphically to $\ttt$. Moreover, the irreducible components are indexed by the $\Ts$-fixed points in $X$.
\end{lemma}

\begin{proof}
We proved in Lemma \ref{lemcm} that $\ZZ\cap \pi^{-1}(\ttt^\reg) \simeq \ttt^\reg\times X^\Ts$ and it is dense in $\ZZ$. Therefore for each $\zeta_i\in \Ts$ we have the subset
$$A_i = \{ (w,M_w\zeta_i) | w\in\ttt^\reg  \}  \subset \ZZ\cap \pi^{-1}(\ttt^\reg)$$
and $\ZZ = \bigcup \overline{A_i}$. Here $M_w$ are the elements defined by Theorem \ref{unif}.

We will prove that each $\overline{A_i}$ maps isomorphically to $\ttt$. Note that it is a closed subscheme of $\ZZ$, which is projective over $\ttt$. Therefore its image in $\ttt$ is closed. As it contains $\ttt^\reg$, it has to be the whole $\ttt$. The projection $\overline{A_i} \to \ttt$ is a birational morphism with finite fibers. Hence by Zariski's main theorem it is an isomorphism to an open subscheme of $\ttt$. But we have seen that its image is the whole $\ttt$, hence $\overline{A_i} \to \ttt$ is an isomorphism.

We showed that $\ZZ$ is a finite union of irreducible subspaces $\overline{A_i}$ isomorphic to $\ttt$  -- therefore they are its irreducible components.  
\end{proof}

\begin{example}
Another natural family of examples are the Bott--Samelson resolutions of Schubert varieties (\cite{BS},\cite{Hansen},\cite{Demazure}). We first recall their construction here. Let $\Gs$ be a semisimple group of rank $r$, with simple roots $\alpha_1$, $\alpha_2$, \dots, $\alpha_r$. The reflections $s_1$, $s_2$, \dots, $s_r$ in the simple roots generate the Weyl group $\W$ of $\Gs$. Let $(e_i,f_i,h_i)$ be an $\ssl_2$-triple corresponding to $\alpha_i$. For any sequence $\underline{\omega}=(\alpha_{i_1},\alpha_{i_2},\dots,\alpha_{i_l})$ of simple roots we can construct the \emph{Bott--Samelson variety} as follows:

$$X_{\underline{\omega}}
=\Ps_{i_1}\times^{\Bs}\Ps_{i_2}\times^{\Bs}\dots \times^{\Bs} \Ps_{i_l}/\Bs,
$$
where $\Bs$ is the Borel subgroup of $\Gs$ and $\Ps_i$ is the minimal (non-Borel) parabolic subgroup corresponding to the root $\alpha_i$. Here $\Bs$ acts on $\Ps_i$ both on left and right, hence we can define the mixed quotients as above, and the last quotient is by the right $\Bs$-action on $\Ps_{i_l}$. The variety admits the multiplication map $X_{\underline{\omega}}\to \Gs/\Bs$. If $\underline{\omega}$ is a reduced word representing an element $\omega\in\W$, then this map is a resolution of the Schubert variety $X_\omega = \ov{\Bs\omega \Bs/\Bs}$. The Borel subgroup $\Bs$ acts on the Bott--Samelson variety on the left.

\begin{lemma}\label{botsamreg}
 The Bott--Samelson resolutions are regular $\Bs$-varieties.
\end{lemma}

\begin{proof}
Assume that an element $x\in X_{\underline{\omega}}$ represented by $(g_1,g_2,\dots,g_l)$ is a zero of the vector field defined by the regular nilpotent $e$. As $e$ generates an additive subgroup $\exp(te)$ inside $\Bs$, every zero of $e$ is fixed by this subgroup, and in particular by $b_1 = u = \exp(e)$. This means that in the Bott--Samelson variety
\begin{align}\label{bs1}
[(b_1g_1,g_2,g_3,\dots,g_l)] = [(g_1,g_2,g_3,\dots,g_l)]
\end{align}
First, this means that $b_1 g_1 = g_1 b_2$ for some $b_2\in \Bs$, hence $b_1\in g_1 \Bs g_1^{-1}$. As $b_1$ is a regular unipotent element, there is only one Borel subgroup, namely $\Bs$, which contains $b_1$ (see the discussion in Example \ref{exgr2}). As $N_{\Gs}(\Bs) = \Bs$, we have $g_1\in \Bs$. From $b_1 g_1 = g_1 b_2$ we have that $b_2$ is conjugate to $b_1$, hence it is also a regular unipotent in $\Bs$. From \eqref{bs1} we have
$$[(b_2g_2,g_3,\dots,g_l)] = [(g_2,g_3,\dots,g_l)]$$
in the Bott--Samelson variety corresponding to the sequence $(\alpha_{i_2},\dots,\alpha_{i_l})$. Applying the same reasoning, we get inductively that $g_1,g_2,\dots,g_l\in\Bs$, hence $[(g_1,g_2,g_3,\dots,g_l)] = [(1,1,\dots,1)]$ in $X_{\underline{\omega}}$.\footnote{The author is grateful to Jakub L\"{o}wit for this argument.}
\end{proof}

\noindent This means that using Theorem \ref{finsolv} above we can determine $H_{\Ts}^*(X_{\underline{\omega}})$, where $\Ts$ is the maximal torus inside $\Bs$. The open Białynicki-Birula cell $X_o$ consists of the classes
$$[(\exp(x_1f_{i_1}),\exp(x_2f_{i_2}),\dots,\exp(x_lf_{i_l}))]_{x_1,x_2,\dots,x_l\in\C}$$
and we would like to find the scheme $\ZZ$ inside $X_o\times \ttt$. We need to determine the infinitesimal action of $\Bs$ on that cell. We will proceed coordinate by coordinate. Note that for $i\in\{1,2,\dots,l\}$ the group $\Ps_{i}$ contains $\{\exp(t \cdot f_i)| t\in\C\} \cdot \Bs$ as an open dense subset. Therefore for any $x\in\C$ there exists an open neighbourhood $U\subset \Bs$ of $1_B$ such that for all $g\in U$ we have
$$g \cdot \exp(x f_i) = \exp(b(g)f_i) h(g)$$
for some maps $b:U\to \C$ and $h:U\to \Bs$ with $b(1) = x$ and $h(1) = 1$. The two sides of the equality are functions of $g$. Let us differentiate them at $g=1$ in the direction of $y\in \bb$. We get
$$y \cdot \exp(x f_i) = \exp(x f_i) \cdot (Db|_1(y) f_i + Dh|_1(y)),$$
where on the left hand side the dot denotes the right translation by $\exp(x f_i)$ and on the right hand side it analogously denotes the left translation. Therefore
$$Db|_1(y) f_i + Dh|_1(y) = \Ad_{\exp(-x f_i)}(y).$$
Now let $y = e + w$, where $w\in\Ts$. Then
\begin{multline*}
\Ad_{\exp(-x f_i)}(y) = \Ad_{\exp(-x f_i)}(e) + \Ad_{\exp(-x f_i)}(w) = (e+xh_i-x^2f_i) + (w-\alpha_i(w)x f_i) \\
= (-\alpha_i(w)x-x^2)f_i + (e + w + xh_i).
\end{multline*}
Thus we get $Db|_1(y) = -\alpha_i(w)x-x^2$ and $Dh|_1(y) = e + w + xh_i$. Hence the infinitesimal action on $X_{\underline{\omega}}$ in direction $e+w$ yields the vector of first coordinate $-\alpha_{i_1}(w)x_1-x_1^2$ and induces the infinitesimal action of $e + w + x_1h_{i_1}$ on the second coordinate. We can apply this procedure inductively and get that the $j$-th coordinate is acted upon by
$$e+w+\sum_{k=1}^{j-1} x_k h_{i_k}$$
and the corresponding coordinate of the vector field $V_{e+w}$ is
$$-\sum_{k=1}^{j-1} \alpha_{i_j}(h_{i_k})x_k x_j -\alpha_{i_j}(w)x_j-x_j^2.$$
Therefore if we define the numbers $b_{jk} = \alpha_{i_j}(h_{i_k})$, we obtain the following presentation of the equivariant cohomology ring:
$$H_{\Ts}^*(X_{\underline{\omega}}) = 
\C[\ttt][x_1,x_2,\dots,x_l]/\left( x_j^2 = -\sum_{k<j} b_{jk} x_kx_j -\alpha_{i_j}(w)x_j \right),
$$
where $w$ denotes the $\ttt$ coordinate. Note that e.g. for $\alpha_1, \dots, \alpha_r$ being the standard simple roots of $\SL_n$, those numbers vanish whenever $|i_j-i_k|>1$.

The variety has $2^l$ torus-fixed points and hence the equivariant cohomology ring is a free module over $\C[\ttt]$ of rank $2^l$. An additive basis consists of all the square-free monomials in $x_1,x_2,\dots,x_l$. We recover then the results obtained e.g. in \cite[Proposition 4.2]{BS} or \cite[Proposition 3.7]{Willems}. 
\end{example}

\section{Reductive and arbitrary principally paired algebraic groups}
\label{redarbsec}

\subsection{Reductive groups}
In this section, we will make a transition from solvable groups to reductive groups. We do that by restricting to Borel subgroups and utilizing Theorem \ref{finsolv}.

Let then $\Gs$ be a complex reductive algebraic group of rank $r$. We assume that $e\in\geg = \Lie(\Gs)$ is a regular nilpotent element. Let $f, h\in\geg$ denote the remaining elements of an $\ssl_2$-triple $(e,f,h)$ (see the discussion in Section \ref{sl2p}). In fact, all the regular nilpotents are conjugate by \cite[Section 3, Theorem 1]{Kostsec}. Hence, we can actually assume $e = x_1 + x_2 + \dots + x_s$, as in Example \ref{exgr}. In particular, $h$ is semisimple and contained in the unique Borel subalgebra $\bb$ of $\geg$ containing $e$. It integrates to a map $H^t:\Cs\to \Gs$ with finite kernel. We denote by $\Ss = e+C_\geg(f)$ the corresponding Kostant section (cf. Theorem \ref{kostant2}). We have $\C[\Ss] = \C[\geg]^\Gs = \C[\ttt]^\Ws = H^*_\Gs(\pt)$. The goal will be to prove the following result.

\begin{theorem}
\label{semisimp}
Let $\Gs$ be as above and assume that $\Gs$ acts on a connected smooth projective variety $X$ regularly. Let $\ZZ_\Gs$ be the closed subscheme of $\Ss\times X$ defined as the zero set of the total vector field (Definition \ref{totvec}) restricted to $\Ss\times X$.
Then $\ZZ_\Gs$ is an affine, reduced scheme and $H^*_\Gs(X) \cong \C[\ZZ_\Gs]$ as graded $\C[\Ss]$-algebras, where the grading on right-hand side is defined by the action of $\Cs$ on $\Ss$ via $\frac{1}{t^2}\Ad_{H^t}$ and on $X$ via $H^t$. In other words, $\ZZ_\Gs = \Spec H^*_\Gs(X)$, $\Ss = \Spec H^*_\Gs$ and the pullback of functions along the projection $\ZZ_\Gs\to \Ss$ yields the structure map $H^*_\Gs\to H^*_\Gs(X)$, so we have the following diagram.

$$ \begin{tikzcd}
 \ZZ_\Gs  \arrow{d}{\pi} \arrow{r}{\cong}  &
  \Spec H^*_\Gs(X;\C) \arrow{d} \\
   \Ss \arrow{r}{\cong}&
  \Spec H^*_\Gs.
\end{tikzcd}$$

Moreover, the isomorphism $H^*_\Gs(X) \cong \C[\ZZ_\Gs]$ of graded $\C[S]$-algebras is functorial both in $X$ and $\Gs$. The admissible morphisms are those that map a $\Gs_1$-regular variety $X$ to a $\Gs_2$-regular variety $Y$ in a $\Gs_1$-equivariant way, where $\Gs_1\to\Gs_2$ is a homomorphism between two reductive algebraic groups which maps the fixed principal $\ssl_2$-triple to the other fixed principal $\ssl_2$-triple.

$$ \begin{tikzcd}
 \ZZ_{\Gs_1}(X)  \arrow{d}{} \arrow{r}{\cong}  &
  \Spec H^*_{\Gs_1}(X;\C) \arrow{d} \\
 \ZZ_{\Gs_2}(Y)  \arrow{r}{\cong}&
  \Spec H^*_{\Gs_2}(Y;\C).
\end{tikzcd}$$

\end{theorem}

Note that from Theorem \ref{cohoweyl} $H^*_\Gs(X) = H^*_\Ts(X)^\W$, where $\Ts$ is the maximal torus and $\W = N_\Gs(\Ts)/\Ts$ is the Weyl group of $\Gs$. Therefore, we will be able to make use of the result for solvable groups, i.e. Theorem \ref{finsolv}.

\subsection{Motivating example: \texorpdfstring{$\Gs = \SL_2(\C)$}{G=SL2(C)}}

For $\Gs = \SL_2(\C)$ we can choose the canonical $e$, $f$, $h$:

$$e = \begin{pmatrix}
0 & 1 \\
0 & 0
\end{pmatrix}, \qquad
f = \begin{pmatrix}
0 & 0 \\
1 & 0
\end{pmatrix}, \qquad
h = \begin{pmatrix}
1 & 0 \\
0 & -1
\end{pmatrix}.
$$
Then we get $\Ss = \{e+vf| v\in\C \}$. Again, let us adapt convention from Example \ref{exgr} for the basis of $\ttt$, i.e. a number $v\in\C$ will denote $-vh/2$. We know that $H_\Ts^*(X) = \C[\ZZ_{\Bs_2}]$, where $\ZZ_{\Bs_2}$ is defined as in Definition \ref{defz} for solvable (Borel) subgroup $\Bs_2$ of $\SL_2(\C)$ consisting of upper triangular matrices. 
Let us now see how the Weyl group -- in this case $\Sigma_2 = \{1, \epsilon\}$ -- acts on $H^*_\Ts(X)$. For any $\zeta_i\in X^\Ts$ we have the following commutative diagram
$$
\begin{tikzcd}
H^*_\Ts(X) \arrow[r, "\epsilon^*"] \arrow[dd, "\iota_{\epsilon\zeta_i}^*"]
& H^*_\Ts(X) \arrow[dd, "\iota_{\zeta_i}^*"]
\\ \\
H^*_\Ts(\epsilon\zeta_i) \arrow[r, "\epsilon^*"]
& H^*_\Ts(\zeta_i)
\end{tikzcd}.
$$
Note that in the bottom line we have the (contravariant) action of $\W$ on $H^*_\Ts(\pt) \cong \C[\ttt]$, which is defined by the (covariant) adjoint action of $\W$ on $\ttt$. In the case of $\SL_2$ the element $\epsilon$ acts on $\ttt$ by $v\mapsto -v$.

\noindent
Therefore we get that for any $c\in H^*_\Ts(X)$ and any $\Ts$-fixed point $\zeta_i$ we have
$$(\epsilon^*c)|_{\zeta_i} = (c|_{\epsilon\zeta_i}) \circ \epsilon,$$
where $\epsilon$ is here seen as a map $\ttt\to\ttt$. This determines $\epsilon^*c$ completely, as the restriction $H^*_\Ts(X)\to \bigoplus H^*_\Ts(\zeta_i)$ is injective.
Hence when we apply the isomorphism $\rho:H^*_\Ts(X)\to \C[\ZZ_{\Bs_2}]$, we will get
$$\rho(\epsilon^* c)(\w,M_{\w}\zeta_i) = \rho(c)(\epsilon \w,M_{\epsilon\w}\epsilon\zeta_i).$$
We get an algebra homomorphism $\C[\ZZ_{\Bs_2}]\to\C[\ZZ_{\Bs_2}]$, which has to come from a morphism $\ZZ_{\Bs_2}\to\ZZ_{\Bs_2}$. This morphism sends $(\w,M_{\w}\zeta_i)$ to $(\epsilon \w,M_{\epsilon\w}\epsilon\zeta_i)$.
\\
We will now look at the adjoint action of elements of the form
$$\exp(sf) = \begin{pmatrix}
1 & 0 \\
s & 1
\end{pmatrix}
\in \SL_2.
$$
We have
\begin{align}
\label{conquo}
\Ad_{\exp(tf)}(e+th) &= e + t^2 f \\
\label{conmin}
\Ad_{\exp(2tf)}(e+th) &= e - th.
\end{align}
From \eqref{conmin} and Lemma \ref{lemad} we infer that the map
$$\psi_\epsilon:(v,x) \mapsto (-v,\exp(-vf) x)$$
is an isomorphism of $\ZZ_{\Bs_2}$ (note that in our choice of basis the number $v$ denotes $-vh/2$). We claim that it is equal to the above (i.e. it is dual to $\rho\circ \epsilon^*\circ \rho^{-1}$). Clearly the action on the first factor agrees. Now we have
$$\exp(-vf)(M_v \zeta_i) = 
\begin{pmatrix}
1 & 0 \\
-v & 1
\end{pmatrix}
\begin{pmatrix}
1 & 1/v \\
0 & 1
\end{pmatrix}
\zeta_i
$$
and we get
$$M_{-v}^{-1}\exp(-vf)(M_v \zeta_i)=
\begin{pmatrix}
1 & 1/v \\
0 & 1
\end{pmatrix}
\begin{pmatrix}
1 & 0 \\
-v & 1
\end{pmatrix}
\begin{pmatrix}
1 & 1/v \\
0 & 1
\end{pmatrix}
\zeta_i
=
\begin{pmatrix}
0 & 1/v \\
-v & 0
\end{pmatrix}
\zeta_i = \epsilon\zeta_i.
$$
Therefore
$$\psi_\epsilon(v,M_v\zeta_i)= (-v,\exp(-vf)(M_v \zeta_i)) = (-v,M_{-v}\epsilon\zeta_i)$$
and indeed
$$\rho(\epsilon^*c)(v,x) = \rho(c)(\psi_\epsilon(v,x)).$$
Thus $\Spec H^*_{\SL_2(\C)}(X)$ is the GIT quotient of $\ZZ_{\Bs_2}$ over this action.

\noindent
Now from \eqref{conquo} we get that the map $\phi: (v,x)\mapsto (v,\exp(-vf/2)x)$ is an isomorphism between $\ZZ_{\Bs_2}$ and $\ZZ' = \{(v,x)\in \C\times X: (V_{e+v^2/4 f})|_{x} = 0\}$. Therefore we might as well look for the GIT quotient of $\ZZ'$ by $\phi \circ \psi_\epsilon\circ\phi^{-1}$.
We get
$$\phi \circ \psi_\epsilon\circ\phi^{-1}(v,x) = \phi\circ\psi_\epsilon(v,\exp(vf/2)x) = \phi(-v,\exp(-vf/2)x) = (-v,x).$$
The GIT quotient of $\ZZ' = \{(v,x): (V_{e+v^2/4 f})|_{x} = 0\}$ by this action is clearly isomorphic to $\ZZ_\Gs = \{(t,x)\in \C\times X: (V_{e+tf})|_{x} = 0\}$.

\subsection{General case}
\label{secred}
We will want to mimic the proof for $\SL_2$ in the general reductive case. Let $\ZZ_\Bs$ be the scheme defined in Definition \ref{defz}, for the Borel subgroup $\Bs$ of $\Gs$. We need the following:
\begin{itemize}
 \item Regular maps $A:\ttt\to \Gs$ and $\chi:\ttt\to \Ss$ that satisfy
 $$\Ad_{A(\w)}(e+\w) = \chi(\w),$$
 so that $(\id_{\ttt},A(\w))$ maps $\ZZ_\Bs$ to $\ZZ'$, where
 $$ \ZZ' = \{(\w,x)\in \ttt\times X: V_{\chi(\w)}|_{x} = 0\}.$$
 \item Moreover we want $\chi$ to be $\W$-invariant and induce an isomorphism $\ttt/\!\!/\W\to \Ss$, so that we can construct 
 $\ZZ_{\Gs}$ as a quotient of $\ZZ'$.
 \item We want to realise the Weyl group action on $\ZZ_{\Bs}$ by action on the second factor, i.e. for each $\eta\in\W$ we want to define a map $B_{\eta}:\ttt\to \Gs$ such that
 $$(\w,x) \mapsto (\eta(\w), B_\eta(\w) \cdot x)$$
 is the action of the Weyl group.
 \item If we conjugate above with the isomorphism $\ZZ_{\Bs}\to\ZZ'$, we want to get a map that fixes the $X$-coordinate. In other words,
 $$A(\eta\w) B_\eta(\w) A^{-1}(\w) = 1,$$
 i.e. $B_\eta(\w) = A(\eta\w)^{-1} A(\w)$.
\end{itemize}
We will now formalise those ideas. First, let $\Bs$ be the unique Borel subgroup of $\Gs$ containing the regular nilpotent $e$ (cf. Section \ref{sl2p}). We denote by $\U$ the corresponding maximal unipotent subgroup and by $\Bs^-$, $\U^-$ the opposite Borel and unipotent subgroup. Let $\bb$, $\uu$, $\bb^-$, $\uu^-$ denote the corresponding Lie algebras. As above, by $\ZZ_\Bs\subset \ttt\times X$ we denote the zero scheme defined by the action of $\Bs$, which by Theorem~\ref{finsolv} is isomorphic to $\Spec H^*_\Ts(X)$. First, we get the elements $A(\w)$ and $\chi(w)$ as in \eqref{adchi}. In other words, from Lemma \ref{lembal} we have that the map
$$ \Ad_{-}(-): \U^- \times \Ss \to e+\bb^-$$
is an isomorphism. We consider the preimage of $e+\ttt$ and denote by $A(\w)\in \U^-$, $\chi(\w)\in \Ss$ the elements such that
 \begin{equation} \label{adchi2}
 \Ad_{A(\w)}(e+\w) = \chi(\w).
 \end{equation}
 We know then from Lemma \ref{lemad} and \eqref{adchi2} that the map $\phi$ defined as
 $$\phi(\w,x) = (\w,A(\w)x)$$
 is an isomorphism from
 $$\ZZ_\Bs = \{(\w,x)\in \ttt\times X: V_{e+\w}|_{x} = 0\}.$$
 to 
 $$\ZZ' = \{(\w,x)\in \ttt\times X: V_{\chi(\w)}|_{x} = 0\}.$$
 Moreover, let 
 $$B_\eta(\w) = A(\eta\w)^{-1} A(\w)$$
 for any $\eta\in\W$, $\w\in \Ts$. Then by Proposition \ref{isoquot} the map $\psi_\eta$ defined as
 $$\psi_\eta = \phi^{-1} \circ (\eta,\id) \circ\phi,$$
 i.e. $\psi_\eta(\w,x) = (\eta\w, B_\eta(\w)x)$, is an automorphism of $\ZZ_\Bs$. Here $\eta$ is seen as a map $\ttt\to\ttt$.

\begin{lemma}
\label{lemweyl}
 The map $\psi_\eta$ defines the action of $\W$ on $\ZZ_\Bs$. In other words, $\W$ acts on the right on $H^*_\Ts(X)$ and the dual left action on $\ZZ_\Bs = \Spec H^*_\Ts(X)$ is defined by $\psi$.
\end{lemma}
\begin{proof}
 For any $\eta\in\W$ we have the commutative diagram
 $$
\begin{tikzcd}
H^*_\Ts(X) \arrow[r, "\eta^*"] \arrow[dd, "\iota_{\eta\zeta_i}^*"]
& H^*_\Ts(X) \arrow[dd, "\iota_{\zeta_i}^*"]
\\ \\
H^*_\Ts(\eta\zeta_i) \arrow[r, "\eta^*"]
& H^*_\Ts(\zeta_i)
\end{tikzcd}.
$$
 In the bottom row both entries are isomorphic to $\C[\ttt]$ and the map is precomposition with $\eta:\ttt\to\ttt$.
 Therefore for any $c\in H^*_\Ts(X)$ and any $\Ts$-fixed point $\zeta_i$ we get
 $$(\eta^*c)|_{\zeta_i} = (c|_{\eta\zeta_i}) \circ \eta.$$
 This determines $\eta^*c$ completely, as the restriction $H^*_\Ts(X)\to \bigoplus H^*_\Ts(\zeta_i)$ is injective. We want to determine what this action of $\W$ induces on $\ZZ_\Bs$. The action of $\Ws$ on $\ZZ_\Bs$, which we will denote by $\eta\mapsto \eta_*$, has to satisfy the equality
 $$\rho(c)(\eta_* (\w,x)) = \rho(\eta^*(c))(\w,x).$$
 From the proof of Lemma \ref{lemcm} we know that $\ZZ_\Bs\cap (\ttt^\reg\times X)$ is dense in $\ZZ_\Bs$. Therefore to determine $\eta^*$, it is enough to determine its values $\eta_*(\w,x)$ for $\w$ regular. In this case if $(\w,x)\in\ZZ_\Bs$, then, by Section \ref{structure}, we have that $\w = M_\w \zeta_i$, where $M_\w\in \Bs$ is such that $\Ad_{M_\w}(\w) = e+\w$. Then $\rho(c)(\w,x) = c|_{\zeta_i}(\w)$. Hence 
 $$\rho(\eta^*(c))(\w,x) = \eta^*(c)|_{\zeta_i}(\w) = (c|_{\eta\zeta_i}) (\eta\w)
 = \rho(c)(\eta\w,M_{\eta\w}\eta\zeta_i).$$
 Thus 
 $$\eta_*(\w,M_{\w}\zeta_i) = (\eta\w,M_{\eta\w}\eta\zeta_i).$$
 We claim that $\eta_* = \psi_\eta$, i.e. $B_\eta(\w) M_{\w}\zeta_i = M_{\eta\w}\eta\zeta_i$. We have to prove that $C_{\eta,\w} = M_{\eta\w}^{-1} B_\eta(\w) M_{\w}$ sends $\zeta_i$ to $\eta\zeta_i$. 

  Note that
 \begin{multline*}
 \Ad_{C_{\eta,\w}}(\w) = \Ad_{M_{\eta\w}^{-1} B_\eta(\w) M_{\w}}(\w) = 
 \Ad_{M_{\eta\w}^{-1} A_{\eta\w}^{-1}A_{\w}M_{\w}}(\w) = 
 \Ad_{M_{\eta\w}^{-1} A_{\eta\w}^{-1}A_{\w}}(e+\w)= \\
 =
 \Ad_{M_{\eta\w}^{-1} A_{\eta\w}^{-1}}(\chi(\w))= 
 \Ad_{M_{\eta\w}^{-1} A_{\eta\w}^{-1}}(\chi(\eta\w))= 
 \Ad_{M_{\eta\w}^{-1}} (e + \eta\w)= \eta\w.
 \end{multline*}
 Therefore for any representative $\tilde{\eta} \in N_{\Gs}(\Ts)$ of $\eta$ we have
 $$\tilde{\eta}^{-1}C_{\eta,\w} \in C_{\Hs}(\w).$$
 As $\w$ is regular, its centraliser within $\he$ is just $\ttt$. It is the Lie algebra of $C_{\Gs}(\w)$, which is connected by \cite[Corollary 3.11]{Steinberg}, hence equal to $\Ts$. Therefore $\tilde{\eta}^{-1}C_{\eta,\w}\in \Ts$, hence $C_{\eta,\w}$ represents the class of $\eta$ in $N_\Gs(\Ts)/\Ts$. Thus $C_{\eta,\w}$ sends $\zeta_i$ to $\eta\zeta_i$, as we wanted to prove.
\end{proof}

\begin{proof}[Proof of Theorem \ref{semisimp}]
 We saw above that the map $\phi$ defined as $\phi(\w,x) = (\w,A(\w)x)$ is an isomorphism from $\ZZ_\Bs$ to $\ZZ' = \{(\w,x)\in \ttt\times X: V_{\chi(\w)}|_{x} = 0\}$. Then we can conjugate the maps $\psi_\eta$ with this isomorphism, hence getting maps $\phi\circ\psi_\eta\circ\phi^{-1}:\ZZ'\to \ZZ'$. We have
 \begin{multline}
 \label{actw}
 \phi\circ\psi_\eta\circ\phi^{-1}(\w,x) = \phi\circ\psi_\eta(\w,A(\w)^{-1}x)
 = \phi(\eta\w,B_\eta(\w)A(\w)^{-1}x) \\ 
 = (\eta\w,A(\eta\w)B_\eta(\w)A(\w)^{-1}x) = (\eta\w,x).
 \end{multline}
The last equality follows from the definition $B_\eta(\w) = A(\eta\w)^{-1}A(\w)$. By Lemma \ref{lemweyl} the map $\phi\circ\psi_\eta\circ\phi^{-1}$ gives the action of $\W$ on $\ZZ'\cong \ZZ_\Bs \cong \Spec H^*_\Ts(X)$.
 
 We have $H^*_\Gs(X) = H^*_\Ts(X)^\W$ and therefore $\Spec H^*_\Gs(X) = \Spec H^*_\Ts(X)/\!\!/\W =\ZZ'/\!\!/\W$. But we know from \eqref{actw} that $\W$ acts only on the $\ttt$-coordinate of $\ZZ'$ and moreover from Proposition \ref{kostiso} the map $\chi$ induces an isomorphism $\ttt/\!\!/\W\to\Ss$. Therefore
 \begin{multline*}\Spec H^*_\Gs(X) = \ZZ'/\!\!/W = 
 \{(\w,x)\in \ttt\times X: V_{\chi(\w)}|_{x} = 0\}/\!\!/\W \\
 = \{(v,x)\in \Ss\times X: V_{v}|_{x} = 0\} = \ZZ_\Gs.
 \end{multline*}
 The zero scheme $\ZZ_\Gs$ is reduced because $\ZZ' \cong \ZZ_{\Bs}$ is reduced from Lemma~\ref{lemcm}. The agreement of $\C[\Ss]$-algebra structures follows from commutativity of the diagram
 $$
\begin{tikzcd}
\Spec H^*_\Ts(X) = \ZZ_{\Bs} \arrow[r, "\cong"]  \arrow[dd, "\pi_{\Bs}"] & \ZZ' \arrow[r, "/\!\!/\W"]
& \Spec H^*_\Gs(X) = \ZZ_\Gs \arrow[dd, "\pi_\Gs"]
\\ \\
\Spec H^*_\Ts = \ttt \arrow[rr, "/\!\!/\W"]
& & \Spec H^*_\Gs = \ttt/\!\!/\W
\end{tikzcd}
$$
and the analogous statement for $\Bs$ in Theorem~\ref{finsolv}.

It remains to show that the grading agrees on the two sides. We know from Theorem~\ref{finsolv} that the grading in the solvable case is defined by the weights of the torus acting on $\ttt\times X$ by $\left(\frac{1}{t^2},H^t\right)$. We have to prove that it descends to the action by $\left(\frac{1}{t^2}\Ad_{H^t},H^t\right)$. But we have
$$\Ad_{A(\w)}(e+\w) = \chi(\w)$$
and thus
$$\Ad_{H^tA(\w)H^{-t}}(\Ad_{H^t}(e+\w)) = \Ad_{H^t}(\chi(\w)) $$
and dividing both sides by $t^2$ gives
$$\Ad_{H^tA(\w)H^{-t}}\left(e+\frac{\w}{t^2}\right) = \frac{1}{t^2}\Ad_{H^t}(\chi(\w)).$$
However,
$$H^tA(\w)H^{-t}\in \U^-,\qquad \frac{1}{t^2}\Ad_{H^t}(\chi(\w))\in \Ss.$$ 
The latter follows from $\frac{1}{t^2}\Ad_{H^t}(e) = e$ and $\Ad_{H^t}$ preserving the centraliser of $f$, as $\Ad_{H^t}(f) = \frac{1}{t^2}f$. Therefore by uniqueness we have
$$H^tA(\w)H^{-t} = A\left(\frac{\w}{t^2}\right), \qquad \frac{1}{t^2}\Ad_{H^t}(\chi(\w)) = \chi\left(\frac{\w}{t^2}\right).$$
The quotient map $\ZZ_\Bs\to \ZZ_\Gs$ sends $(\w,x)$ to $(\chi(\w),A(\w)x)$. And by the above, it sends $t\cdot(\w,x) = \left(\frac{\w}{t^2},H^t x\right)$ to 
\begin{multline*}
\left(\chi\left(\frac{\w}{t^2}\right),A\left(\frac{\w}{t^2}\right)H^t x\right) = \left(\frac{1}{t^2}\Ad_{H^t}\left(\chi(\w)\right),H^tA(\w)H^{-t} H^t x\right) \\
 = \left(\frac{1}{t^2}\Ad_{H^t}\left(\chi(\w)\right),H^t A(\w) x\right),
\end{multline*}
which proves that the action of $\Cs$ on $\ZZ_{\Bs}$ descends to the action by $\left(\frac{1}{t^2}\Ad_{H^t},H^t\right)$ on $\ZZ_{\Gs}$.

The functoriality follows immediately from functoriality for $\Bs$ (cf. Propositions \ref{funcprop} and \ref{funcgrp}).
\end{proof}

\begin{remark}\label{affred}
 We know $C_\geg(f)\subset \bb^-$ (Section \ref{secreg}) and all the weights of the $\Hs^t$ action on $\bb^-$ are nonpositive (Lemma \ref{posint}). Therefore the argument as in Lemma \ref{lemcm} shows that $\ZZ_\Gs$ lies in $\Ss\times X_o$. This means that for any computations we have to consider only an affine part $X_o$ of $X$.
\end{remark}

\begin{remark}\label{chernred}
 In the spirit of Lemma \ref{nicefun} we can determine in the reductive case too what functions on $\ZZ_\Gs$ the particular Chern classes are mapped to. Assume that $\Ee$ is a $\Gs$-linearised vector bundle on $X$. Let $k$ be a nonnegative integer and consider $c_k^\Gs(\Ee)\in H^*_\Gs(X) = H^*_\Ts(X)^\Ws$. If we first consider the map $\rho:H^*_\Ts(X)\to \C[\ZZ_\Bs]$ from Section \ref{sectionmap}, then from Lemma \ref{nicefun} we know for any $(\w,x)\in \ZZ_\Bs$ that 
 $$\rho(c_k^\Ts(\Ee))(\w,x) = \Tr_{\Lambda^k \Ee_x}(\Lambda^k (e+\w)_x).$$
 The map $\phi$ defined as
 $$\phi(\w,x) = (\w,A(\w)x)$$
 maps $\ZZ_\Bs$ isomorphically to $\ZZ'$. Then $c_k^\Ts(\Ee)$ defines on $\ZZ'$ the function $\rho(c_k^\Ts(\Ee))\circ\phi^{-1}$ which satisfies
 $$\rho(c_k^\Ts(\Ee))\circ\phi^{-1}(\w,y) = \rho(c_k^\Ts(\Ee))(\w,A(\w)^{-1}y)
 = \Tr_{\Lambda^k \Ee_{A(\w)^{-1}y}}(\Lambda^k (e+\w)_{A(\w)^{-1}y}).
 $$
 As $\Ee$ is $\Gs$-invariant, this is equal to
 $$\Tr_{\Lambda^k \Ee_{y}}(\Lambda^k \Ad_{A(\w)}(e+\w)_{y}) = 
 \Tr_{\Lambda^k \Ee_{y}}(\Lambda^k \chi(\w)_{y})
 $$
 This means that on the quotient $\ZZ_\Gs$ the function $\rho_\Gs(c_k^\Gs(\Ee))$ corresponding to $c_k^\Gs(\Ee)$ satisfies
 $$\rho_\Gs(c_k^\Gs(\Ee))(v,x) = \Tr_{\Lambda^k \Ee_{x}}(\Lambda^k v_{x}).$$
\end{remark}
\noindent Let us also note that the $\Gs$-equivariant Chern classes generate the equivariant cohomology ring.

\begin{lemma}\label{genred}
 In the setting above, the $\Gs$-equivariant cohomology $H^*_\Gs(X)$ is generated as a $\C[\ttt]^\Ws$-algebra by equivariant Chern classes of $\Gs$-equivariant vector bundles.
\end{lemma}

\begin{proof}
 By the Nakayama lemma it is enough to prove (see the proof of Lemma \ref{chern2}) that the non-equivariant cohomology $H^*(X)$ is generated by Chern classes of $\Gs$-equivariant vector bundles.
 
 We know from the proof of Lemma \ref{chern1} that $H^*(X)$ is generated by Chern characters of $\Ts$-equivariant coherent sheaves. For any such sheaf $\F$, we can consider the ``averaged'' sheaf $\F_\Ws = \frac{1}{|\Ws|} \bigoplus_{\eta\in \Ws} \eta_* \Ws$. As the group $\Gs$ is connected, for any $g\in \Gs$ we have $\ch(g_* \F) = \ch(\F)$, hence $\ch(\F_\Ws) = \ch(\F)$. Therefore $H^*(X)$ is generated by Chern characters of $N_\Gs(\Ts)$-equivariant coherent sheaves. Then again by \cite[Corollary 5.8]{Thomason} it is generated by Chern characters of $N_\Gs(\Ts)$-equivariant vector bundles. Every $N_\Gs(\Ts)$-equivariant vector bundle is a $\Ws$-invariant element of $K_\Ts(X)$. However we know by \cite[Corollary 6.7]{HLS} that $K_\Ts(X)^\Ws = K_\Gs(X)$, hence $H^*(X)$ is generated by Chern classes of $\Gs$-equivariant vector bundles.
\end{proof}

\subsection{Examples} \label{examples}
We finish this section by providing examples for Theorem \ref{semisimp}. These are extensions of the examples above for Theorem \ref{finsolv}.

\begin{figure}[ht!]
\begin{center}
 \subfloat{
  \includegraphics[width=7cm]{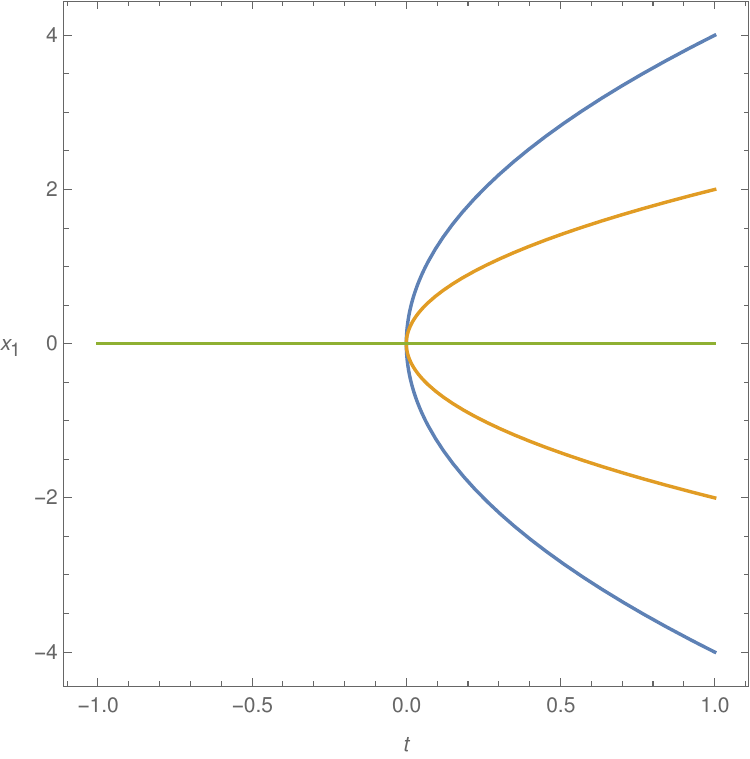}}
  \hfill
\subfloat{
  \includegraphics[width=7cm]{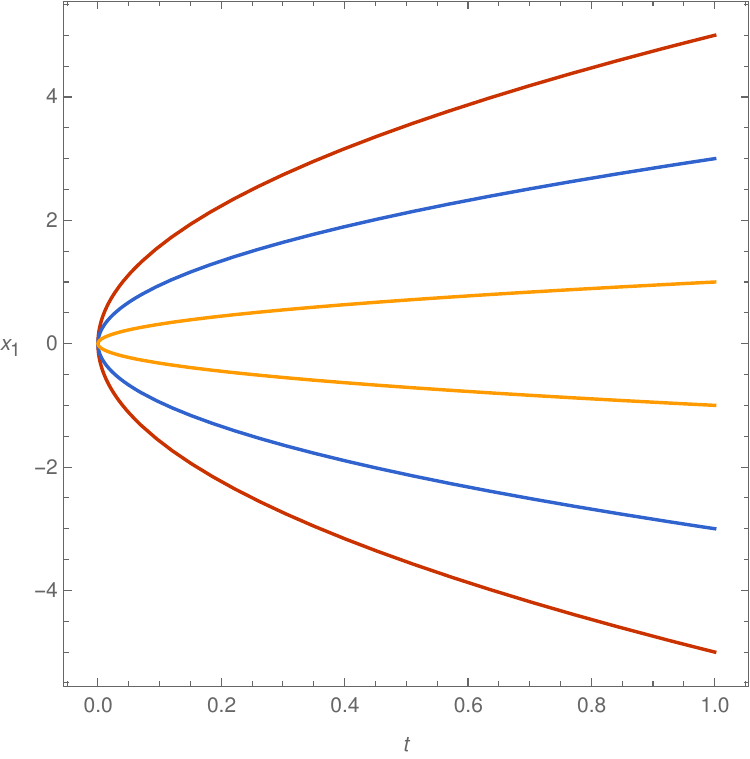}}
\end{center}
\caption[Spectrum of $\SL_2$-equivariant cohomology of $\PP^4$ and $\PP^5$]{$\Spec H_{\SL_2(\C)}^*(\PP^4)$ and $\Spec H_{\SL_2(\C)}^*(\PP^5)$.}
\label{sl2pn}
\end{figure}

\begin{example}\label{exslpn}
 We continue Example \ref{exsl22}. There, we found the $\Cs$-equivariant cohomology of $\PP^n$. Now, using the tools above, we can also find $\Spec H_{\SL_2(\C)}(\PP^n)$. We know that the map $(v,x)\mapsto (v,(I+vf)x)$ maps the zeros of $V_{e+vh}$ isomorphically to the zeros of $V_{e+v^2 f}$. The former form the subscheme cut out by $x_1(x_1+2v)(x_1+4v)\dots(x_1+2nv) = 0$ in the $(v,x_1)$-plane. Note that 
 $$f = \begin{pmatrix}
       0 & 0 & 0 & 0 & \dots & 0\\
       1\cdot n & 0 & 0 & 0 & \dots & 0\\
       0 & 2\cdot (n-1) & 0 & 0 & \dots & 0\\
       \vdots & \vdots & \vdots & \vdots & \ddots & \vdots \\
       0 & 0 & (n-1)\cdot 2 & 0 & \dots & 0 \\
       0 & 0 & 0 & n\cdot 1 & \dots & 0
      \end{pmatrix},
 $$
 hence the map $I+vf$ acts on the $x_1$ coordinate by adding $nv$. This means that the zeros of $V_{e+v^2 f}$ are defined by 
 $$(x_1-nv)(x_1-(n-2)v)(x_1-(n-4)v)\dots(x_1+(n-2)v)(x_1+nv) = 0.$$
 Bringing the symmetric factors together, we get
 $$
 \begin{cases}
  (x_1^2-n^2v^2)(x_1^2-(n-2)^2v^2)\dots(x_1^2-4v^2)x_1 = 0 
  &\text{ for $n$ even;}\\
  (x_1^2-n^2v^2)(x_1^2-(n-2)^2v^2)\dots(x_1^2-9v^2)(x_1^2-v^2)= 0 
  &\text{ for $n$ odd.}\\
 \end{cases}
 $$
 Therefore
 $$
 H_{\SL_2(\C)}^*(\PP^n)=
 \begin{cases}
  \C[t,x_1]/\big((x_1^2-n^2t)(x_1^2-(n-2)^2t)\dots(x_1^2-4t)x_1\big) &\text{ for $n$ even;}\\
  \C[t,x_1]/\big((x_1^2-n^2t)(x_1^2-(n-2)^2t)\dots(x_1^2-9t)(x_1^2-t)\big) &\text{ for $n$ odd.}
 \end{cases}
 $$
 The scheme has $\lceil\frac{n+1}{2}\rceil$ components, one for each orbit of the action of $\W=\Z/2\Z$ on $(\PP^n)^{\Cs}$. The parabolas correspond to two-element orbits and the line (for even $n$) corresponds to unique fixed point of $\Cs$ fixed by $\W$. It is equal to $[\underbrace{0:0:\dots:0}_{n/2}:1:\underbrace{0:\dots:0:0}_{n/2}]$. Examples of the scheme for $n=4$ and $n=5$ are depicted in Figure \ref{sl2pn}.
\end{example}

\begin{remark}
Note that, contrary to Lemma \ref{solvcomps}, if the group is not solvable, then the components of $\ZZ$ will not, in general, project isomorphically to $\Ss$. In fact, they are quotients of unions of copies of $\ttt$ by the action of $\Ws$. The components then correspond to the orbits of $\Ws$ on $X^\Ts$. A component will project isomorphically to $\Ss$ if and only if it corresponds to a one-element orbit, i.e. a fixed point of $\Ns(\Ts)$ on $X$. This is the case for example for the ``middle $\Ts$-fixed point'' in $\PP^{2n}$ under the $\SL_n$ action. One then sees a single line in the left part of Figure \ref{sl2pn}. 
\end{remark}

\begin{figure}[ht!]
\begin{center}
\subfloat{
  \includegraphics[width=7cm]{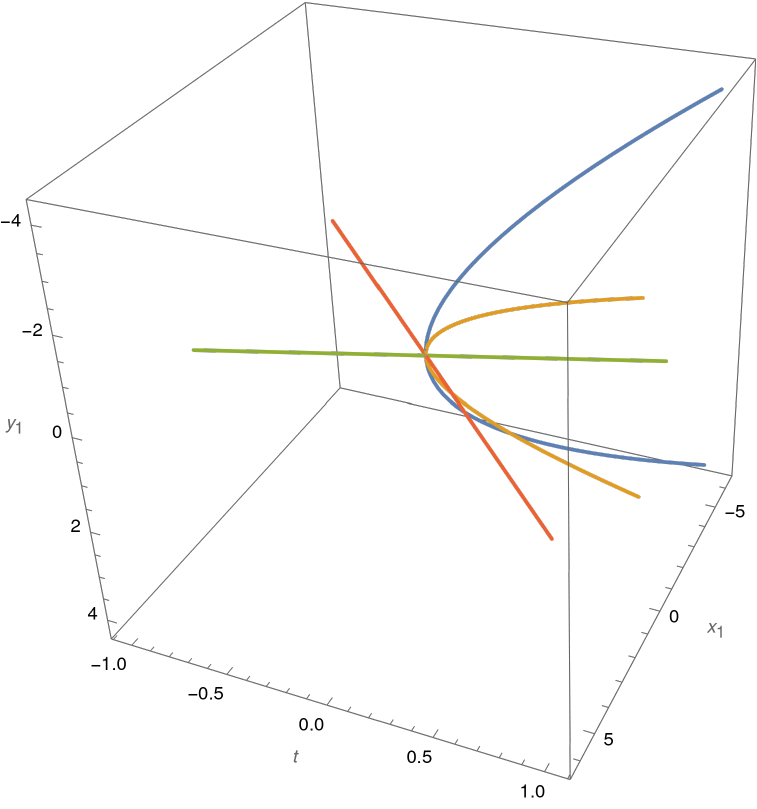}}
  \hfill
\subfloat{
  \includegraphics[width=7cm]{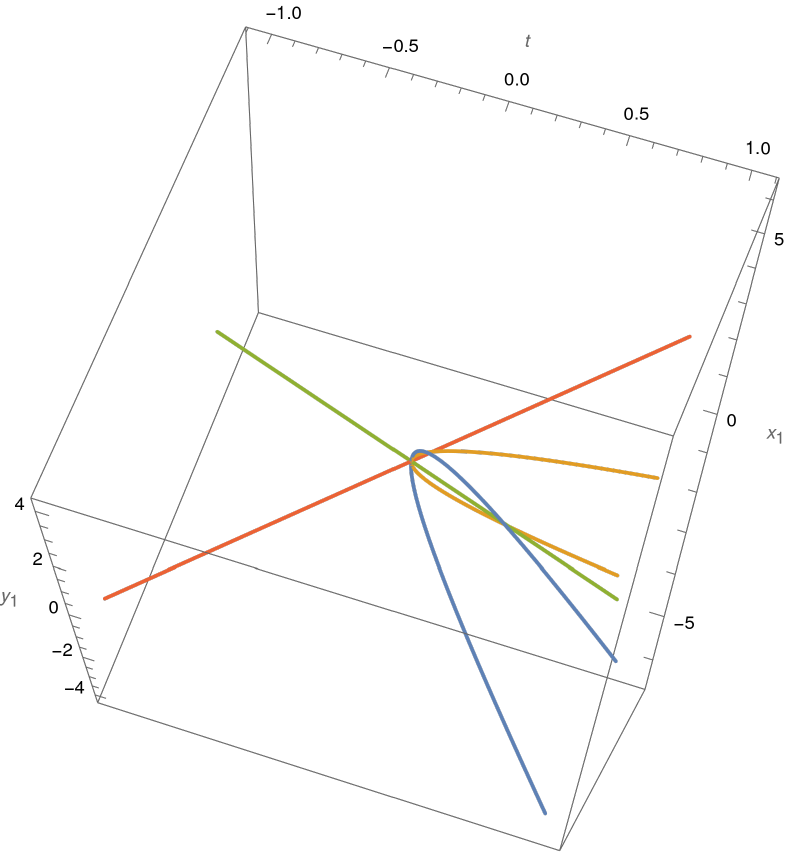}}
\end{center}
\caption[$\SL_2$-equivariant cohomology of $\Gr(2,4)$]{Two different views of $\Spec H_{SL_2(\C)}^*(\Gr(2,4))$.}
\label{sl2gr}
\end{figure}

\begin{example}
 We continue Example \ref{csgr}. The principal $\SL_2(\C)\subset \SL_4(\C)$ subgroup acts on $\Gr(2,4)$. One can check that
 $$V_f|_{x_1,y_1,x_2,y_2}
 = (-3y_1,4,3x_1-3y_2,3y_1).$$
 Then
 $$V_{e+tf}|_{x_1,y_1,x_2,y_2}
 =(x_2-x_1y_1 -3ty_1,
 -x_1-y_1^2+y_2 +4t,
 -x_1y_2 +3tx_1 -3ty_2,
 -x_2-y_1y_2 +3ty_1).
 $$
 As before, from the first two equations of $V_{e+tf} = 0$, we can determine $x_2$ and $y_2$, so $\Spec H_{\SL_2(\C)}^*(\Gr(2,4))$ can be embedded in $\C[t,x_1,y_1]$. Its equations are
 $$12t^2+4tx_1-x_1^2-3ty_1^2-x_1y_1^2 = 0,\quad
 y_1(4t-2x_1-y_1^2) = 0.$$
 By considering two possibilities in the latter, one easily arrives at four possibilities:
 $$
 (x_1 = -2t, y_1 = 0),
 \quad
 (x_1 = 6t, y_1 = 0),
 \quad
 (x_1 = -6t, y_1^2 = 16t),
 \quad
 (x_1 = 0, y_1^2 = 4t).
 $$
 As in the previous example, the components correspond to orbits of $\W$ acting on $\Gr(2,4)^{\Cs}$. The former two correspond to one-element orbits, i.e. $\{\Span(e_2,e_3)\}$ and $\{\Span(e_1,e_4)\}$, and the latter come from two two-element orbits. The scheme embedded in $t,x_1,y_1$-space is presented in Figure \ref{sl2gr}.
\end{example}

\begin{example}
 We consider an example for a group of higher rank, $\SL_3(\C)$, that we can still draw. Let it act on $\PP^2$ in the standard way. In Example \ref{exsl3p2} we calculated the equivariant cohomology of $\PP^2$ with respect to (two-dimensional) torus. Now we will compute the $\SL_3$-equivariant cohomology. The Kostant section is
 $$\Ss = \left\{
 \begin{pmatrix}
 0 & 1 & 0 \\
 c_2 & 0 & 1 \\
 c_3 & c_2 & 0
 \end{pmatrix}: c_2,c_3\in\C.
 \right\}
 $$
The coordinates $c_2,c_3\in \C[\Ss]\cong H^*(B\SL_3(\C))$ are (up to scalar multiples) the universal Chern classes of principal $\SL_3(\C)$-bundles, or equivalently, of rank 3 vector bundles with trivial determinant.
 We have already computed that $V_e|_{x_1,x_2} = (x_2-x_1^2,-x_1x_2)$. Then it is easy to see that for 
 $$M = \begin{pmatrix}
 0 & 1 & 0 \\
 c_2 & 0 & 1 \\
 c_3 & c_2 & 0
 \end{pmatrix}$$
 we have
 $V_M|_{x_1,x_2} = (x_2-x_1^2+c_2,-x_1x_2+c_2x_1+c_3)$. As before, we can eliminate $x_2$ by substituting from the first equation and we get the equation $x_1^3-2c_2x_1-c_3 = 0$. The corresponding scheme $\Spec H^*_{\SL_3(\C)}(\PP^2)$ in coordinates $c_2$, $c_3$, $x_1$ is illustrated in Figure \ref{hsl3p2}. It is irreducible, as all three torus-fixed points lie in one orbit of the Weyl group. The projection to the $(c_2,c_3)$-plane is generically a $3-1$ map.
 
 On the right hand side of Figure \ref{hsl3p2} the slice $c_3 = 0$ is marked in red. The elements of $\Ss$ that satisfy $c_3 = 0$ form the Kostant section of the principal $\SL_2$ subgroup -- which acts as in Example \ref{exslpn}. Therefore the red scheme is equal to $\Spec H^*_{\SL_2}(\PP^2)$. Additionally, the functoriality of Theorem \ref{semisimp} implies that restriction to $c_3 = 0$ yields the base restriction map
 $$H^*_{\SL_3}(\PP^2)\to H^*_{\SL_2}(\PP^2).$$
\end{example}

\begin{figure}[ht!]
\begin{center}
\subfloat{
  \includegraphics[width=7.5cm]{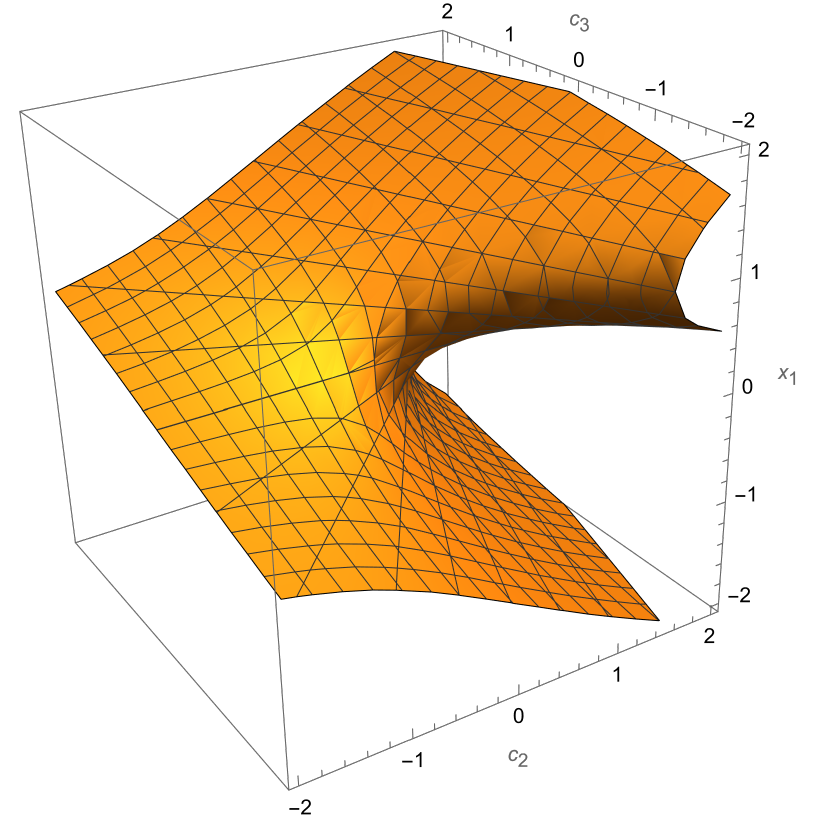}}
  \hfill
\subfloat{
  \includegraphics[width=7.5cm]{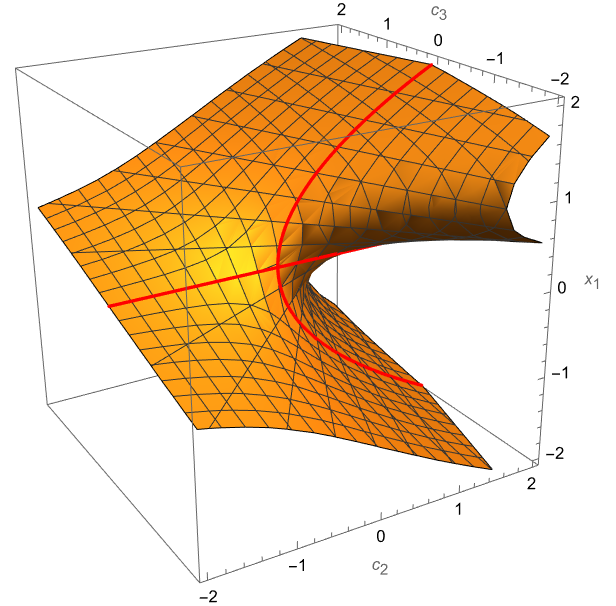}}
\end{center}
\caption[$\SL_2$- and $\SL_3$-equivariant cohomology of $\PP^2$]{$\Spec H_{\SL_3(\C)}^*(\PP^2)$. On the right the subscheme $\Spec H_{\SL_2(\C)}^*(\PP^2)$ is marked. Compare with Figure \ref{sl2pn}.}
\label{hsl3p2}
\end{figure}

\begin{example}
 As in Example \ref{exflag}, we now consider the action of $\SL_3(\C)$ on the variety $F_3$ of full flags in $\C^3$. We determined $V_e$ in Example \ref{exflag}. We can analogously determine the vector fields corresponding to lower triangular matrices. Then for 
 $$M = \begin{pmatrix}
 0 & 1 & 0 \\
 c_2 & 0 & 1 \\
 c_3 & c_2 & 0
 \end{pmatrix}$$ we easily get
 $$V_M|_{a,b,c} =
 (-a^2+b+c_2,-ab+ac_2+c_3,-b+ac-c^2+c_2).$$
 Plugging in $b$ from the first equation, we obtain
 $$a^3-2c_2a+c_3 = 0,\qquad
 a^2-ac+c^2 = 2c_2.$$
 The first equation for $a$ clearly coincides with the equation for $x_1$ from the previous example. One can easily see that the equations mean that $a$ and $-c$ are two of the three roots of the polynomial $x^3-2c_2x+c_3=0$. The map to the $(c_2,c_3)$-plane is generically $6-1$. As all the torus-fixed points, i.e. coordinate flags, lie in one orbit of the Weyl group, in the GIT quotient of $\Spec(H_\Ts^*(F_3))$ they are joined together and the scheme is irreducible.
\end{example}

\subsection{Principally paired algebraic groups}\label{secarb}
In fact, we can prove the equivalent of Theorem \ref{semisimp} for a principally paired, but not necessarily reductive algebraic group. This version will yield a common generalisation to Theorem \ref{semisimp} and Theorem \ref{finsolv}. Note that for any principally paired group $\Hs$ with maximal torus $\Ts$ and Weyl group $\Ws$ we have $H_\Hs^* = \C[\ttt]^\Ws = \C[\Ss]$ -- see the comment above Theorem \ref{restkos}. We will prove the following.

\begin{theorem}\label{general}
 Assume that $\Hs$ is a principally paired algebraic group which acts on a smooth projective variety $X$ regularly. Let $\ZZ_\Hs$ be the closed subscheme of $\Ss\times X$, defined as the zero set of the total vector field (Definition \ref{totvec}) restricted to $\Ss\times X$.
 
 Then $\ZZ_\Hs$ is an affine reduced scheme and $H_{\Hs}^*(X)\cong \C[\ZZ_\Hs]$ as graded $\C[\Ss]$-algebras, where the grading on the right-hand side is defined on $\Ss$ via $\frac{1}{t^2}\Ad_{H^t}$ and on $X$ by the action of $\Cs$ via $H^t$. In other words, $\ZZ_\Hs = \Spec H_\Hs^*(X)$, $\Ss = \Spec H_\Hs^*$ and the projection $\ZZ_\Hs\to \Ss$ yields the structure map $H_\Hs^*\to H_\Hs^*(X)$. This isomorphism is functorial as in Theorem \ref{semisimp}.
$$ \begin{tikzcd}
 \ZZ_\Hs  \arrow{d}{\pi} \arrow{r}{\cong}  &
  \Spec H^*_\Hs(X;\C) \arrow{d} \\
   \Ss \arrow{r}{\cong}&
  \Spec H^*_\Hs.
\end{tikzcd}$$
\end{theorem}

\begin{remark}
 As $\Ns$ is contractible, the Levi subgroup $\Ls\subset \Hs$ is a homotopy retract of $\Hs$, and for any $\Hs$-space $X$ we have $H_\Hs^*(X) = H_\Ls^*(X)$. In particular, if $\Hs$ is solvable, we have $H_\Hs^*(X) = H_\Ts^*(X)$, where $\Ts$ is a maximal torus within $\Hs$. This explains how the theorem above generalises Theorem \ref{finsolv}.
\end{remark}

\begin{proof}[Proof of Theorem~\ref{general}.]
We will proceed analogously to the proof in Section \ref{secred}. We follow the notation from Section \ref{kostgensec}. In particular, $\Bs$ is the Borel subgroup of $\Hs$ such that its Lie algebra $\bb$ contains $e$. We first consider the scheme $\ZZ_\Bs\subset \ttt\times X$, defined as in Section \ref{solvsec}, i.e. the zero scheme of the total vector field on $\geg\times X$, restricted to $(\ttt+e)\times X$.
Then from Lemma \ref{congen} we get morphisms $A:\ttt\to\U^-$, $\chi:\ttt\to\Ss$ such that
$$\Ad_{A(\w)}(e+\w) = \chi(\w),$$
so that $(\id,A(\w))$ maps $\ZZ_\Bs$ to $\ZZ'$, where
$$\ZZ' = \{(\w,x)\in \ttt\times X: V_{\chi(\w)}|_{x} = 0\}.$$
In fact, $A$ and $\chi$ are exactly the same as in Section \ref{secred} (see the proof of Lemma \ref{congen}). In particular, $\chi$ induces the isomorphism $\ttt/\!\!/\Ws\to\Ss$.
For any regular $\w\in\ttt$ we have the element $M_w\in \Bs$ such that $\Ad_{M_w}(w) = e+w$. Just like in Lemma \ref{lemweyl}, for any $\eta\in \Ws$ we get that for any regular $\w$ the element
$C_{\eta,\w} = M_{\eta(\w)}^{-1} A(\eta\w)^{-1}A(\w) M_{\w}$ is in the class of $\eta$ in $N_\Ls(\Ts)/\Ts$. Note that here $M_{\eta(\w)}\in \Bs_l$.

This then proves, similarly as in Section \ref{secred}, that the Weyl group action on $\ZZ_\Bs$, when transported to $\ZZ'$, is defined by $\eta\mapsto (\eta,\id)$. And then as $\chi:\ttt/\!\!/\Ws \to \Ss$ is an isomorphism, we get that $$\ZZ_\Hs \cong \ZZ_{\Bs}/\!\!/\Ws = \Spec H^*_\Ts(X)/\!\!/\Ws = \Spec H^*_\Hs(X).$$ 

We have to prove that the grading on $\C[\ZZ_\Hs]$ defined by the grading on $H^*_\Hs(X)$ agrees with the one described in the theorem. We know that in the solvable case the grading is defined by the action of $\Cs$ on $\ZZ_\Bs$ via $\left(\frac{1}{t^2},H^t\right)$ (Definition \ref{defz}). Just like in the reductive case, we need to prove that under quotient by $\Ws$ it descends to the action by $\left(\frac{1}{t^2}\Ad_{H^t},H^t\right)$. The argument for reductive groups does not translate exactly, as a priori we do not know whether $H^t$ preserves the centraliser of $f$. However we know that $H_l^t$, the one-parameter subgroup generated by $h_l$, does.

On the other hand, as $[h,e] = [h_l,e] = 2e$, from Lemma \ref{semcen} we infer $h-h_l\in Z(\lel)$. As in the proof of Theorem \ref{semisimp}, we have
$$\Ad_{H^tA(\w)H^{-t}}\left(e+\frac{\w}{t^2}\right) = \frac{1}{t^2}\Ad_{H^t}(\chi(\w))$$
and
$$H^tA(\w)H^{-t}\in \U^-,\qquad \frac{1}{t^2}\Ad_{H^t}(\chi(\w))\in \Ss,$$ 
where now the latter follows from $\frac{1}{t^2}\Ad_{H^t}(e) = e$ and $\Ad_{H^t} = \Ad_{H_l^t}$ preserving the centraliser of $f$, as $\Ad_{H_l^t}(f) = \frac{1}{t^2}f$. Therefore we have
$$H^tA(\w)H^{-t} = A\left(\frac{\w}{t^2}\right), \qquad \frac{1}{t^2}\Ad_{H^t}(\chi(\w)) = \chi\left(\frac{\w}{t^2}\right).$$
and the same reasoning follows. This proves Theorem \ref{general}.
\end{proof}

\begin{example}
Basic examples of non-reductive, non-solvable linear groups are parabolic subgroups of reductive groups. Let us consider such a group $\Ps\subset \Gs$, where $\Gs$ is reductive and assume that $\Bs\subset \Ps$ is a Borel subgroup of $\Gs$ contained in $\Ps$. Then we can consider a principal $\bb(\ssl_2)$-triple $(e,f,h)$ in $\geg$ such that $e,h\in \bb$. This makes $\Ps$ into a principally paired group and we can make use of Theorem \ref{general}.

Suppose that $X$ is a Schubert variety in some partial flag variety $\Gs/\Qs$. Its stabiliser $\Ps$ in $\Gs$ contains $\Bs$, hence it is a parabolic subgroup. In general it is larger than $\Bs$ (see more in \cite[Section 2]{SanVan}). Remember that $\Bs$ acts on $\Gs/\Qs$ regularly (Example \ref{exgr2}). Therefore if $X$ is smooth, Theorem \ref{general} gives the $\Ps$-equivariant cohomology of $X$.
\end{example}

\begin{example}
As in the previous example, assume that $X$ is a Schubert variety -- possibly non-smooth -- in $\Gs/\Qs$ fixed by $\Ps$. One can then construct a Bott--Samelson resolution of $X$ \cite[Section 2, p. 446]{SanVan} which is $\Ps$-equivariant. As in Lemma \ref{botsamreg}, such a resolution will be a smooth regular $\Ps$-variety. Hence we can use Theorem \ref{general} to compute its $\Ps$-equivariant cohomology.
\end{example}

We also extend Lemma \ref{genred} to principally paired groups.

\begin{lemma}\label{genprinc}
 Assume that a principally paired group $\Hs$ acts regularly on a smooth projective variety $X$. Then the $\Hs$-equivariant cohomology $H^*_\Hs(X)$ is generated as a $\C[\ttt]^\Ws$-algebra by equivariant Chern classes of $\Hs$-equivariant vector bundles.
\end{lemma}

\begin{proof}
As the Levi subgroup is a deformation retract of $\Hs$, the restriction $H^*_\Hs(X)\to H^*_\Ls(X)$ is an isomorphism. From Lemma \ref{chern1}, $H^*(X)$ is generated by the Chern characters of vector bundles equivariant with respect to the Borel subgroup of $\Ls$. Then the proof analogous to Lemma \ref{genred} shows that $H^*(X)$ is generated by the Chern characters of $\Ls$-equivariant vector bundles. Then by Theorem \ref{thmlevik} it is generated by the Chern characters of $\Hs$-equivariant vector bundles. Then by graded Nakayama lemma the same is true for $H^*_\Hs(X)$.
\end{proof}

\section{Extensions: singular varieties and total zero schemes}

In this section we discuss two directions to extend our results. 
First we discuss generalisations to singular varieties.

\subsection{Singular varieties}
\label{secsing}

Our main Theorem~\ref{general} may be generalised to singular varieties, in the spirit of \cite[Section 7]{BC}. There the singular case is considered for $\Bs_2$-equivariant cohomology, extending Theorem \ref{thmsingcar}. A sufficient condition will be an embedding in a smooth regular variety such that the corresponding map on ordinary cohomology is surjective. Compare this with Corollary \ref{surj}.

\begin{proposition} \label{sing}
 Assume that $\Hs$ is a principally paired algebraic group and let $\Ss$ be the Kostant section within $\Hs$, as defined in Section \ref{secarb}. Let $\Bs$ be a Borel subgroup of $\Hs$. Assume that $\Hs$ acts regularly on a smooth projective variety $X$ and let $\ZZ_\Hs^X$ be the zero scheme defined in Theorem \ref{general} for the $\Hs$-action on $X$.
 
 Assume $Y$ is a closed $\Hs$-invariant subvariety whose cohomology is generated by Chern classes of $\Bs$-linearised vector bundles. Then analogously to Section \ref{secarb} we can define an isomorphism of graded $\C[\Ss]$-algebras $H_\Hs^*(Y)\to\C[\ZZ_\Hs^Y]$, where $\ZZ_\Hs^Y$ is the reduced intersection $\ZZ_\Hs^Y = \ZZ_\Hs^X\cap (\Ss\times Y)$. The isomorphism makes the diagram
\begin{equation}\label{diagsing}
\begin{tikzcd}
H^*_\Hs(X) \arrow[r] \arrow[dd]
& H^*_\Hs(Y) \arrow[dd]
\\ \\
\C[\ZZ_\Hs^X] \arrow[r]
& \C[\ZZ_\Hs^Y]
\end{tikzcd}
\end{equation}
 commutative.
 The assumption on the cohomology of $Y$ holds in particular if the inclusion $Y\to X$ induces a surjective map $H^*(X)\to H^*(Y)$ on ordinary cohomology.
\end{proposition}
\begin{proof}
The proof is essentially the same as in \cite[Section 7]{BC}. We only sketch it here. Assume first that $\Hs$ is solvable. Every point of the variety $\ZZ_\Hs^Y$ is of the form $(\w,M_\w \zeta)$, where $M_\w\in\Hs$ depends on $\w$ and $\zeta$ is a $\Ts$-fixed point contained in $Y$. Therefore, for any $c\in H^*_\Ts(Y)$, we can define $\rho_Y(c)$ (we only localise to points in $Y$). The condition on cohomology of $Y$ allows us to use Lemma \ref{nicefun} to show that $\rho_Y$ actually maps $H^*_\Ts(Y)$ to $\C[\ZZ_\Hs^Y]$. The injectivity of $\rho_Y$ follows again from injectivity of localisation on equivariantly formal spaces (\cite[Theorem 1.2.2]{GKM}). The diagram is obviously commutative and surjectivity of $\rho_Y$ follows then from the surjectivity of restriction $\C[\ZZ_\Hs^X]\to\C[\ZZ_\Hs^Y]$ to closed subvariety.

Now assume that $\Hs$ is arbitrary principally paired group. Let $\Bs$ be its Borel subgroup and by $\ZZ_\Bs^Y$ denote the appropriate zero scheme defined for $\Bs$ acting on $Y$. As $Y$ is $\Hs$-invariant, the arguments from the proof of Theorem \ref{general} show that $\C[\ZZ_\Hs^Y] = \C[\ZZ_\Bs^Y]^{\Ws}$ and the conclusion follows.
The last line of the proposition is implied by Lemma \ref{chern2}.
\end{proof}

\begin{example}
  Let $\Hs = \Bs$, the Borel subgroup of a reductive group $\Gs$. Natural examples of singular regular $\Bs$-varieties are Schubert varieties in flag variety $\Gs/\Bs$ or any other subvarieties that are unions of Bruhat cells, see \cite[Theorem 5 with remarks]{ACL}.
  In general, Schubert varieties are stabilised by parabolic subgroups (see in \cite[Section 2]{SanVan}). Those are therefore singular $\Ps$-regular varieties for parabolic groups $\Ps$.
\end{example}

\begin{example}
Assume that $X = \Gs/\Bs$ is the flag variety of type A, i.e. $\Gs = \SL_m(\C)$. Then if $Y$ is any Springer fiber within $X$, the restriction on cohomology $H^*(X)\to H^*(Y)$ is surjective \cite{KumPro}, hence Proposition \ref{sing} also holds in that case.
\end{example}
However, there exist $\Gs$-invariant subvarieties for which the restriction map on cohomology is not surjective.

\begin{remark}\label{exdiscri}
 The assumption on surjectivity on cohomology of $Y$ is necessary in the proposition above. Consider the following. Let $\SL_2$ act on $\PP^3$ as in Example \ref{exslpn}. It comes from a representation $\Sym^3 V$, where $V$ is the fundamental representation of $\SL_2$. It has two extreme (highest and lowest) weights and two ``middle'' weights. The point $o$ of $\PP^3$ which represents the highest weight space is fixed by the Borel subgroup of upper triangular matrices and hence one sees that its orbit is isomorphic to the full flag variety $\SL_2/\Bs_2\cong \PP^1$. However, if we consider a point $p\in \PP^3$ representing a non-highest weight space, its stabiliser is a torus, i.e. $\Stab_{\SL_2}(p)\cong\Ts$. Hence its $\SL_2$-orbit is not closed. We denote its closure by $Y:=\overline{\SL_2\cdot p}$. We claim that $Y$ is not smooth. We can see this directly, by noticing that it is the projectivised variety of polynomials $a_0 x^3 + a_1 x^2y + a_2 xy^2 + y^3$ with at least two roots (vanishing lines) equal, and writing down the discriminant equation. We can also see this using our results. If $Y$ were smooth, by Corollary \ref{surj}, the map $H^*(\PP^3)\to H^*(Y)$ would be surjective, but both varieties admit an action of $\Ts$, with the same set of fixed points, therefore the map would have to be an isomorphism. It is impossible for dimensional reasons ($H^6(Y) = 0$).
 
 Moreover, not only is $Y$ singular, but in any case the map $H^*(\PP^3)\to H^*(Y)$ cannot be surjective. Otherwise, this would mean that Proposition \ref{sing} applies. However, as all the $\Ts$-fixed points are already in $Y$, one sees immediately that $\ZZ$ is already in $Y$. Then again, we would have $H^*(Y) = H^*(\PP^3)$, which is impossible for the same reason as above. Thus $H^*(\PP^3)\to H^*(Y)$ is not surjective, and moreover $H^*(\PP^3)$ is not generated by Chern classes of $\Bs_2$-equivariant bundles. This shows that the assumption is necessary in the proposition.
\end{remark}

\begin{remark}
 Assume we are given an $\Hs$-invariant subvariety $Y$ of a regular smooth $\Hs$-variety $X$. By Proposition \ref{sing} and Corollary \ref{surj} the surjectivity of the restriction $H^*(X)\to H^*(Y)$ is necessary and sufficient for the existence of an isomorphism $H^*_\Hs(Y)\to \C[\ZZ_\Hs^Y]$ which makes \eqref{diagsing} commutative. Carrell and Kaveh prove in \cite[Theorem 2]{CarKav}, for the case of $\Hs = \Bs_2$, that it is equivalent to $H_\Ts^*(Y)$ being generated by Chern classes of $\Bs_2$-equivariant bundles.
\end{remark}

\subsection{Total zero scheme}
\label{sectot}

Assume that $\Gs$ is a principally paired algebraic group, e.g. $\Gs$ reductive. We showed in Theorem \ref{general} how to see geometrically the spectrum of $\Gs$-equivariant cohomology of $X$ for $\Gs$ acting regularly on a projective variety $X$. However, this needed a choice -- of a concrete $\bb(\ssl_2)$-pair $(e,h)$ and the associated Kostant section. A natural challenge would be to try to find equivariant cohomology as global functions on a scheme that does not depend on choices. 

\begin{definition}\label{totzer}
Let an algebraic group $\Gs$ act on a smooth projective variety $X$. Consider the total vector field on $\geg\times X$ (Definition \ref{totvec}). We call its reduced zero scheme
$$\ZZ_{\tot}\subset \geg\times X$$
the {\em total zero scheme}.
\end{definition}
Now we are ready to show the following.

\begin{theorem}\label{wholeg}
 Assume that $\Gs$ is principally paired. Let it act on a smooth projective variety regularly. Consider the action of $\Cs$ on the total zero scheme $\ZZ_{\tot}$ by $t\cdot(v,x) = \left(\frac{1}{t^2} v,x\right)$ and the action of $\Gs$ by $g\cdot (v,x) = (\Ad_g(v),g\cdot x)$. Then the ring $\C[\ZZ_{\tot}]^\Gs$ of $\Gs$-invariant functions on $\ZZ_{\tot}$ is a graded algebra over $\C[\geg]^\Gs\cong H^*_{\Gs}(\pt)$ isomorphic to $H^*_{\Gs}(X)$, where the grading comes from the weights of the $\Cs$-action on $\C[\ZZ_{\tot}]^\Gs$:
 $$ \begin{tikzcd}
  \C[\ZZ_{\geg}]^\Gs \arrow{r}{\cong} &
  H_\Gs^*(X;\C)  \\
  \C[\geg]^\Gs \arrow{r}{\cong} \arrow{u}&
  H^*_\Gs \arrow{u}.
\end{tikzcd}$$
 \end{theorem}
 
Following the notation from Theorem \ref{general}, we show that the restriction $\C[\ZZ_\tot]^\Gs \to \C[\ZZ_\Gs]$ is an isomorphism, so that we get the following diagram
$$
\begin{tikzcd}
\C[\ZZ_\tot]^\Gs \arrow[r] 
& \C[\ZZ_\Gs] \arrow[r]
& H^*_{\Gs}(X,\C) 
\\ \\
\C[\geg]^\Gs \arrow[r]\arrow[uu]
& \C[\ttt]^\Ws \arrow[r]\arrow[uu]
& H^*_{\Gs}(\pt,\C)\arrow[uu]
\end{tikzcd}
$$
with all horizontal arrows being isomorphisms. The statement for the bottom line follows from Lemma \ref{genrest}. First we prove that the restriction $\C[\ZZ_\tot]^\Gs \to \C[\ZZ_\Gs]$ is an epimorphism.

\begin{lemma}\label{surjtot}
 Under the assumptions of Theorem \ref{wholeg}, the restriction $\C[\ZZ_\tot]^\Gs \to \C[\ZZ_\Gs]$ is surjective.
\end{lemma}

\begin{proof}
By Lemma \ref{genprinc} we know
that $\C[\ZZ_\Gs]$ is generated over $\C[\ttt]^\Ws \cong \C[\geg]^\Gs$ by functions $\rho_\Gs(c_k^\Gs(\Ee))$ for positive integers $k$ and $\Gs$-equivariant vector bundles $\Ee$. Those functions are defined by
$$\rho_\Gs(c_k^\Gs(\Ee))(v,x) = \Tr_{\Lambda^k \Ee_{x}}(\Lambda^k v_{x}),$$
see Remark \ref{chernred}. For each such function, we can consider the regular function $f_{k,\Ee}$ defined on $\ZZ_\tot$ by its values:
$$f_{k,\Ee}(v,x) = \Tr_{\Lambda^k \Ee_{x}}(\Lambda^k v_{x}).$$
It is clearly $\Gs$-invariant and restricts to $\rho_\Gs(c_k^\Gs(\Ee))$ on $\ZZ_\Gs$. As $\C[\ZZ_\Gs]$ is generated by such functions, the conclusion follows.
\end{proof}

For the injectivity, let us start with an easy intermediate step. Let $\ZZ_\reg$ be the open subscheme of $\ZZ_\tot$ consisting of the part over $\geg^\reg\subset \geg$ (hence a closed subscheme of $\geg^\reg\times X$). Then we have

\begin{lemma}\label{inj1}
 Let $\Gs$ be a principally paired algebraic group. Assume it acts on a connected smooth projective variety, not necessarily regularly. The restriction $\C[\ZZ_\reg]^\Gs\to \C[\ZZ_\Gs]$ is injective, where $\ZZ_\reg$ and $\ZZ_\Gs$ are defined as above, as zero schemes over $\geg^\reg$ and $\Ss$.
\end{lemma}

\begin{proof}
As $\ZZ_\reg$ is reduced, a function is determined by its values on closed points. 
By Lemma \ref{kostarb} every $\Gs$-orbit in $\geg^\reg$ intersects $\Ss$, thus the $\Gs$-orbit of any closed point in $\ZZ_\reg$ intersects $\ZZ_\Gs$. It is therefore enough to specify a $\Gs$-invariant function on $\ZZ_\reg$ on closed points of $\ZZ_\Gs$. The result follows.
\end{proof}

To finish the proof, we are only left with the proof of injectivity of the restriction $\C[\ZZ_\tot]^\Gs\to \C[\ZZ_\reg]^\Gs$. We will utilise the following Lemma to prove that the restriction $\C[\ZZ_\tot]\to\C[\ZZ_\reg]$ is injective.

\begin{lemma}\label{projinj}
 Let $Y$ be a reduced scheme over a field $k$. Assume $Z$ is a closed subvariety and every closed point $p\in Y$ is contained in a projective closed subvariety that intersects $Z$. Then the restriction map on regular functions $k[Y]\to k[Z]$ is injective.
\end{lemma}

\begin{proof}
Let us assume that $f\in k[Y]$ vanishes on $Z$. Consider any closed point $p\in Y$. Let $A_p$ be a projective closed subvariety that contains $p$ and intersects $Z$ in a closed point $q$. Then $f|_{A_p}$ is a regular function on a projective variety over $k$, hence it is has constant value on all closed points of $A_p$. As $f(q) = 0$, this means that it takes the value 0. Therefore $f(p)=0$. Hence $f$ vanishes on every closed point.

As $Y$ is reduced and of finite type over a field, we know that regular functions are uniquely determined by their values on closed points. Therefore $f = 0$.
\end{proof}
To arrive at the lemma's assumptions, we first prove slightly stronger versions of Lemmas \ref{lemfix} and \ref{lemzer}, under the condition that the action of the Lie algebra is integrable.

\begin{lemma}\label{grpfix}
 Assume that a solvable algebraic group $\Hs$ acts on a smooth complex variety $X$. Let $P\subset X$ be a projective irreducible component of the reduced zero scheme of a linear subspace $\V\subset \he$. Then $P$ contains a simultaneous zero of $N_\he(\V)$.
\end{lemma}

\begin{proof}
By \cite[Lemma 7.4]{Borel} we have $N_{\he}(\V) = \Lie(N_\Hs(\V))$. Let $N_\Hs(\V)^{o}$ be the connected component of the identity within $N_\Hs(\V)$. We know from Lemma \ref{lemad} that $N_\Hs(\V)$ preserves the zero set of $\V$. Thus $N_\Hs(\V)^{o}$ preserves its irreducible components, in particular $P$. By the Borel fixed point theorem \cite[Corollary 17.3]{Milne}, $N_\Hs(\V)^{o}$ it must have a fixed point $p\in P$. Then its Lie algebra $ \Lie(N_\Hs(\V)^{o}) = \Lie(N_\Hs(\V)) = N_{\he}(\V)$ vanishes on $p$.
\end{proof}

\begin{lemma}\label{grpzer}
 Assume that an algebraic group $\Hs$ acts on a smooth variety $X$. Let $d,n \in \he$ commute and assume that the Lie subalgebra generated by $[\he,\he]$ and $n$ is nilpotent. Let $P$ be a projective irreducible component of the reduced zero scheme of $j = d+n$. Then $P$ contains a simultaneous zero of $C_{\he}(d)$, in particular, a zero of any abelian subalgebra of $\he$ containing $d$.
\end{lemma}

\begin{proof}
 By restricting to the connected component of the identity, we can assume that $\Hs$ is connected. As $[\he,\he]$ is nilpotent, $\he$ must be solvable, hence $\Hs$ is solvable too.
 
 Let $\kek$ be the Lie subalgebra generated by $[\he,\he]$ and $n$. By Lemma \ref{grpfix} we get that inside $P$ there is a zero $p$ of $N_{\he}(\C\cdot j)$, which in particular contains $d$ and $n$. As $P$ is irreducible, any irreducible component of the simultaneous zero set of $d$ and $n$ which contains $p$ is completely contained in $P$. Let $P_1\subset P$ be one such irreducible component. As it is closed inside $P$, it also has a structure of a projective scheme.
 
 We will first show that $P_1$ contains a simultaneous zero of $C'(d) = C_{\he}(d)\cap \kek$. As $\kek$ is nilpotent, $C'(d)$ is as well.  By Lemma \ref{grpfix}, $P_1$ contains a simultaneous zero of $N_{\he}(\Span_\C(d,n))$, hence in particular of $N_{C'(d)}(\C\cdot n)$. Note that by definition everything in $C'(d)$ centralises $d$. As $P_1$ consists of zeros of $d$, it will contain an irreducible component $P_2$ of the common zero locus of $d$ and $N_{C'(d)}(\C\cdot n)$. As a closed subscheme of $P_1$, $P_2$ is also projective. By the same argument, $P_2$ contains an projective irreducible component $P_3$ of the common zero locus of $d$ and $N^2_{C'(d)}(\C\cdot n)$. As in the proof of Lemma \ref{lemzer}, there exists a positive integer $k$ such that $N^k_{C'(d)}(\C\cdot n) = C'(d)$, hence we get a projective irreducible component $P_{k+1}$ of the common zero locus of $d$ and $C'(d)$. But again as in Lemma \ref{lemzer}, $C'(d)$, as well as $d$, is normalised by $C_{\he}(d)$. Hence inside $P_{k+1}$ there is a zero of $C_{\he}(d)$.
\end{proof}

\begin{lemma}\label{inj2}
Let $\Gs$ be a principally paired algebraic group. Assume that it acts on a connected smooth projective variety $X$, not necessarily regularly. The restriction $\C[\ZZ_\tot]^\Gs\to \C[\ZZ_\Gs]$ is injective, where $\ZZ_\tot$ and $\ZZ_\Gs$ are defined as before, as the zero schemes over $\geg$ and $\Ss$.
\end{lemma}

\begin{proof}
We have the sequence of restrictions $\C[\ZZ_\tot]^\Gs \to \C[\ZZ_\reg]^\Gs\to \C[\ZZ_\Gs]$. By Lemma \ref{inj1} we only need to prove that the first map is injective. Obviously the restriction $\C[\overline{\ZZ_\reg}] \to \C[\ZZ_\reg]$ is injective, where we take the closure of $\ZZ_\reg$ in $\ZZ_\tot$. We will prove that $\C[\ZZ_\tot]\to \C[\overline{\ZZ_\reg}]$ is injective and this will prove the theorem.

We employ Lemma \ref{projinj} for that. We have to prove that every point of $\ZZ_\tot$ is contained in a projective subvariety which intersects $\overline{\ZZ_\reg}$. Let $(v,p) \in \ZZ_\tot \subset \geg\times X$. Then $p$ is contained in the zero scheme of the vector field $V_v$, hence in some irreducible component $P$ thereof. It is a closed subscheme of $X$, hence it is projective. Then we have $\{v\}\times P \subset \ZZ_\tot$ as a projective closed subvariety. Let $v = d + n$ be the Jordan decomposition of $v$, as in Theorem \ref{defjord}. As $d$ and $n$ commute, they are contained in a Lie algebra $\bb$ of some Borel subgroup $\Bs\subset \Gs$. Let $\Ts$ be a maximal torus within $\Bs$ such that $d\in\ttt = \Lie(\Ts)$. Then from Lemma \ref{grpzer}, for $\Hs = \Bs$, we get that $P$ contains a simultaneous zero $x$ of $\ttt$. It is also a zero of $v$, hence we have $(\ttt+\C\cdot v)\times \{x\} \subset \ZZ_\tot$. Note that $\ttt$ contains a regular element, and as the regular elements within $\geg$ form an open subset, the regular elements of $\ttt+\C\cdot v$ form an open nonempty subset, hence they are dense. This means that $(\ttt+\C\cdot v)\times \{x\}\subset \overline{\ZZ_\reg}$, hence in particular $(v,x)\in\overline{\ZZ_\reg}$, and $(v,x)\in \{v\}\times P$, where $\{v\}\times P$ is a projective subvariety of $\ZZ_\tot$, therefore we are done.
\end{proof}

\begin{proof}[Proof of Theorem \ref{wholeg}]
The isomorphism follows from Lemmas \ref{surjtot} and \ref{inj2}.

For the grading, we just have to show that the defined action of $\Cs$ descends under the restriction $\C[\ZZ_\tot]^\Gs \to \C[\ZZ_\Gs]$ to the action defined in Theorem \ref{general}. Let $f$ be a $\Gs$-invariant function on $\ZZ_\tot$. Then for any $t\in\Cs$ the pullback $t^*f$ of $f$ by $t$ is defined by
$$t^*f(v,x) = f\left(\frac{1}{t^2} v,x\right).$$
As $f$ is $\Gs$-invariant, this means
$$t^*f(v,x) = f\left(\frac{1}{t^2} \Ad_{H^t}v,H^t x\right).$$
When we restrict to $\ZZ_\Gs$, the group $\Cs$ acts precisely by $\left(\frac{1}{t^2} \Ad_{H^t},H^t\right)$, by Theorem \ref{general}. Therefore the actions agree.
\end{proof}

\begin{example}
Assume that $\Gs$ is a reductive group acting on a partial flag variety $X = \Gs/\Ps$. Then the zero scheme is 
$$\tilde{\geg}_\Ps:=\{(x,\p^\prime)\in \geg\times\Gs/\Ps|x\in\p^\prime\},$$
which agrees with the partial Grothendieck--Springer resolution. Thus we get that as a $\C[\geg]^\Gs\cong H^*_\Gs$-module, the ring of invariant functions $\C[\tilde{\geg}_\Ps]^\Gs$ is equal to $H^*_\Gs(\Gs/\Ps) = H^*_\Ps = \C[\ttt]^{\Ws_\Ps}$.
\end{example}

\begin{figure}[ht!]
\begin{center}
 \subfloat{
  \includegraphics[width=7cm]{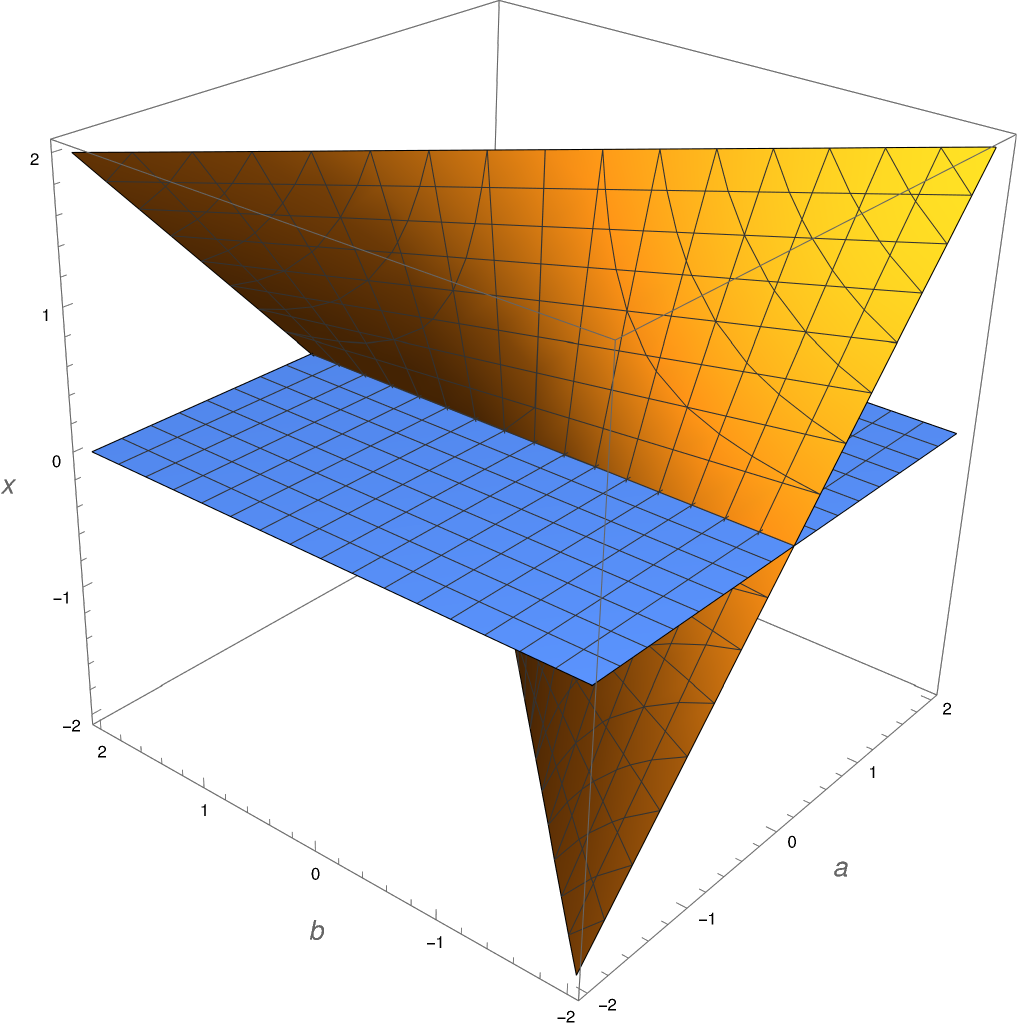}}
  \hfill
\subfloat{
  \includegraphics[width=7cm]{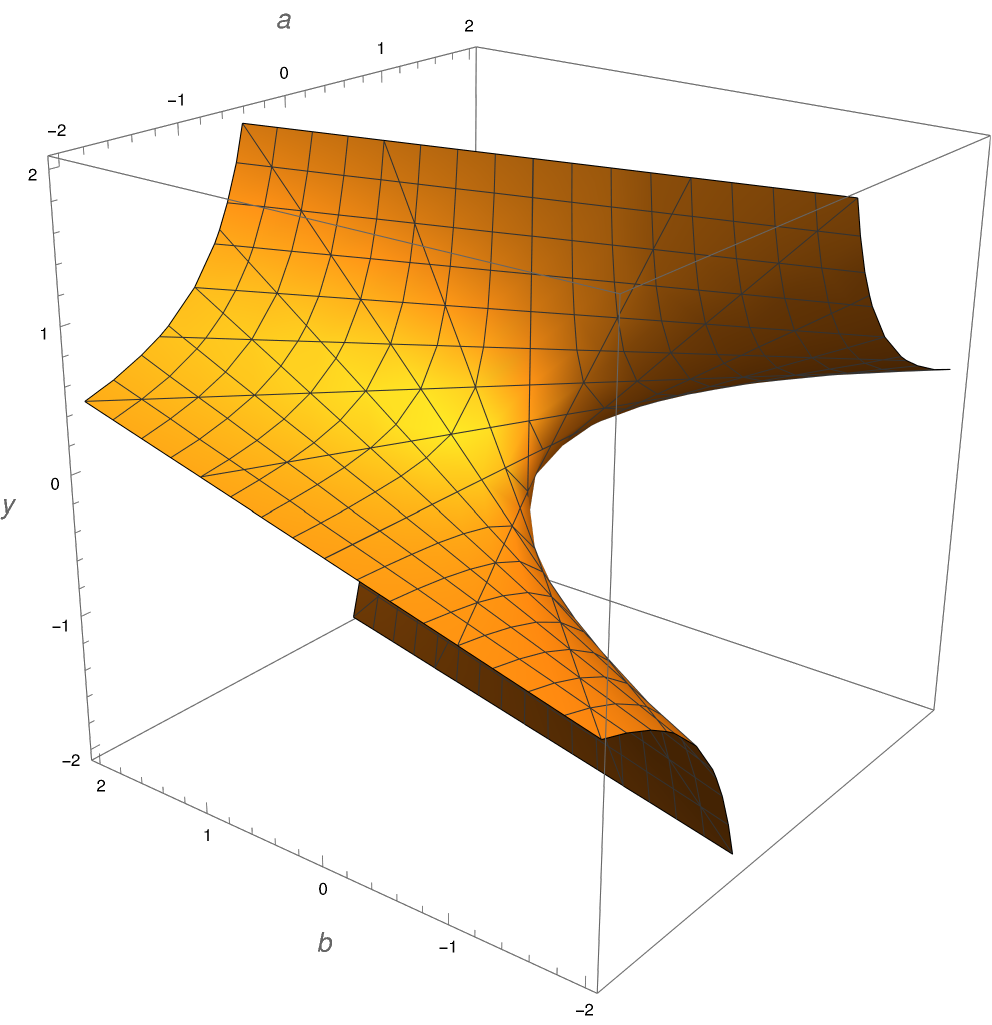}}
\end{center}
\caption[Total zero scheme for the action of $\Bs_2$ on $\PP^1$]{Affine parts of the total zero scheme for the action of $\Bs_2$ on $\PP^1$. The left part misses a line (over $b=0$), the right part misses the blue component.}
\label{totalp1}
\end{figure}
\begin{example}
There is one example that we are able to draw. It is the action of $\Bs_2$ on $\PP^1$ (see Example \ref{exsl22}). The total zero scheme is not affine, but we can cover it with two affine pieces, coming from affine cover of $\PP^1$. The first one will be the part contained in $\bb_2 \times \{[1:x]| x\in \C\}$ and the other will be the part contained in $\bb_2 \times \{[y:1] | y\in \C\}$. If we consider coordinates $(a,b)$ on $\bb_2$ that correspond to matrices
$$\begin{pmatrix}
a & b\\
0 & -a
\end{pmatrix} \in \bb_2,
$$
then the surface has the equations $2ax + bx^2 = 0$ in $(a,b,x)$-plane (the first piece) and $2ay+b = 0$ in $(a,b,y)$-plane (the second piece). The scheme has two irreducible components and the pieces are drawn in Figure \ref{totalp1}.

One sees that the $\Bs_2$-invariant functions on the blue part depend only on $a$, hence they form $\C[a]$. Analogously on the orange part the $\Bs_2$-invariant functions only depend on $a = -\frac{bx}{2}$, as for $a\neq 0$ any two points with the same $a$ are conjugate, and for $a = 0$ we get $b = 0$ or $x = 0$. The former line is a projective line on which an invariant function must attain the same value, and the latter lies in the blue part. This leaves us with two functions from $\C[a]$ with the same constant term. One easily sees that this ring is isomorphic to e.g. $\C[a,x]/x(x+2a)$.
\end{example}

\subsection{Equivariant cohomology of GKM spaces via total zero scheme}
\label{secgkm}

We suspect that the description of equivariant cohomology as the ring of regular functions on the total zero scheme might still hold in a larger generality. For example, one could presume that a sufficiently regular torus action might lead to such a description, even without embedding torus in a larger solvable group (as in Section \ref{solvsec}). Here we prove this equality for GKM spaces, whose equivariant cohomology we know.

\begin{theorem}\label{gkmthm}
Let a torus $\Ts \cong (\Cs)^r$ act on a smooth projective complex variety $X$ with finitely many zero and one-dimensional orbits. In other words, the $\Ts$-action makes $X$ a GKM space. Let $\ZZ=\ZZ_{\tot} \subset \ttt\times X$ be the reduced total zero scheme of this action (Definition \ref{totzer}). Then $\C[\ZZ] \cong H^*_\Ts(X)$ as graded algebras over $\C[\ttt]\simeq H^*_\Ts$:
$$ \begin{tikzcd}
  \C[\ZZ] \arrow{r}{\cong}  &
  H_\Ts^*(X;\C)   \\
   \C[\ttt] \arrow{r}{\cong} \arrow{u}&
  H^*_\Ts \arrow{u}.
\end{tikzcd}$$
\end{theorem}

Let us denote the $\Ts$-fixed points by $\zeta_1$, $\zeta_2$, \dots, $\zeta_s$ and the one-dimensional orbits by $E_1$, $E_2$, \dots, $E_{\ell}$. Recall that the closure of any $E_i$ is an embedding of $\PP^1$ and contains two fixed points $\zeta_{i_0}$ and $\zeta_{i_\infty}$, which for any $x\in E_i$ are equal to the limits $\lim_{t\to 0} t x$ and $\lim_{t\to\infty} tx$. The action of $\Ts$ on $E_i$ has a kernel of codimension $1$, which is uniquely determined by its Lie algebra $\kek_i$.
Then Theorem \ref{gkmres} says that the restriction $H_\Ts^*(X,\C)\to H_\Ts^*(X^T,\C) \cong \C[\ttt]^s$ is injective and its image is
$$
H = \left\{
(f_1,f_2,\dots,f_s)\in \C[\ttt]^s  \, \bigg| \, f_{i_0}|_{\kek_i} = f_{i_\infty}|_{\kek_i} \text{ for } j=1,2,\dots,\ell
\right\}.
$$

Using this description, we will proceed by finding an injective map $\rho:H^*_\Ts(X)\to \C[\ZZ]$ and an injective left inverse $r:\C[\ZZ]\to H \cong H^*_\Ts(X)$. We will use Lemma \ref{projinj} with $Y = \ZZ$ as defined above and $Z = \ttt \times X^\Ts$. Take any $(v,p)\in \ZZ$. The point $p$ lies in the zero scheme $\ZZ_v$ of the vector field on $X$ corresponding to $v$. As $\Ts$ is a commutative group, and hence acts trivially on its Lie algebra, it preserves zeros of $v\in\ttt$. Therefore $\{v\}\times \Ts\cdot p\subset \ZZ$ and $\{v\}\times \overline{\Ts\cdot p}$ is a closed projective subvariety of $\ZZ$. As $\Ts$ acts on it, by the Borel fixed point theorem it contains a fixed point of $\Ts$, hence intersects $Z$ nontrivially. Therefore this choice of $Y$ and $Z$ satisfies the conditions of the Lemma.

We know that there are finitely many distinct types of orbits of the $\Ts$-action on $X$.  This can be seen by embedding $X$ equivariantly in a projective space with a linear action of $\Ts$, see \cite[Theorem 7.3]{Dolgachev}. Therefore there exists a one-parameter subgroup $\{H^t\}_{t\in\Cs}\subset \Ts$ that is not contained in any proper centraliser. Then the fixed points of $H^t$ are automatically the fixed points of $\Ts$. Consider the Białynicki-Birula minus--decomposition, consisting of the cells
$$W_i^{-} = \{x\in X: \lim_{t\to\infty} H^t \cdot x = \zeta_i\}$$
for $\zeta_1$, $\zeta_2$, \dots, $\zeta_s$ being the fixed points of $\Ts$.

We first define the map $\rho:H^*_\Ts(X)\to \C[\ZZ]$. We will define it on closed points, using reducedness of $\ZZ$. Let $c\in H^*_\Ts(X)$. Assume that $(v,x)\in \ZZ$, i.e. the vector field $v$ is zero at $x$. We know that $x \in W_i^{-}$ for some $i\in\{1,2,\dots,s\}$. The restriction $c|_{\zeta_i}$ is an element of $H_\Ts^*(\pt) \cong \C[\ttt]$. Define then
$$\rho(c)(v,x) = c|_{\zeta_i} (v).$$
We first have to prove that this defines a regular function for each $c$.

\begin{lemma}
Let $\Ee$ be a $\Ts$-equivariant bundle on $X$. Then
$$\rho(c_k(\Ee))(v,x) = \Tr_{\Lambda^k(\Ee_x)}(\Lambda^k(v)).$$
In particular, $\rho(c_k(\Ee))$ is a regular function on $\ZZ$.
\end{lemma}

\begin{proof}
Let $c = c_k(\Ee)$.
Consider the curve $C = \overline{H^t \cdot x}$. In particular let $\zeta_i = \lim_{t\to\infty} H^t \cdot x\in C$. We then defined $\rho(c)(v,x) = c|_{\zeta_i}(v)$. But we know that this is equal to
$$c_k(\Ee)|_{\zeta_i}(v) = \Tr_{\Lambda^k(\Ee_{\zeta_i})}(\Lambda^k v).$$
However, as $\Ts$ is commutative, the action of any of its elements, in particular of $H^t$, on $X$ is $\Ts$-equivariant, therefore
$$\Tr_{\Lambda^k(\Ee_{x})}(\Lambda^k v) = \Tr_{\Lambda^k(\Ee_{H^t\cdot x})}(\Lambda^k v)$$
for any $t\in \Cs$. Therefore the equality stays true also in the limit, hence
$$\Tr_{\Lambda^k(\Ee_{x})}(\Lambda^k v) = \Tr_{\Lambda^k(\Ee_{\zeta_i})}(\Lambda^k v) = \rho(c)(v,x).$$
\end{proof}

\begin{proof}[Proof of Theorem \ref{gkmthm}]
We have defined the map $\rho$, we just have to prove that it is an isomorphism. For injectivity, note that $\ttt\times X^\Ts$ is contained in $\ZZ$. By definition, if $\rho(c)$ is zero on this subspace, then all localisations to $\Ts$-fixed points vanish. But by theorem \ref{gkmres} the localisation is injective, hence $c=0$.

The set $\ttt\times X^{\Ts}\subset \ZZ$ is closed and considering it as a reduced subvariety, by Lemma \ref{projinj} we get that the restriction map
$$r:\C[\ZZ] \to \C[\ttt\times X^{\Ts}] \cong \C[\ttt]^s$$
is injective. We need to prove that the image lies in $H$ and that $r\circ \rho$ is the localisation map $H^*_{\Ts}(X)\to H$. The latter comes directly from the definition, as $\rho(c)(v,\zeta_i) = c|_{\zeta_i}(v)$.

Now we need to prove that for any $E_i$ and $v\in\kek_i$ and $f\in \C[\ZZ]$ we have $f(v,\zeta_{i_0}) = f(v,\zeta_{i_\infty})$. Note that as the infinitesimal action of $\kek_i$ is trivial on $E_i$, we have $\kek_i \times \overline{E_i} \subset \ZZ$. This means that over each $v\in\kek_i$ there is a closed subset $\{v\}\times \overline{E_i}\subset \ZZ$. As the reduced subscheme structure makes this a projective variety ($\PP^1$, precisely), every global function on $\ZZ$ needs to be constant along this subvariety. As it contains $(v,\zeta_{i_0})$ and $(v,\zeta_{i_\infty})$, we get $f(v,\zeta_{i_0}) = f(v,\zeta_{i_\infty})$.
\end{proof}

\begin{remark}
Thus the ring of regular functions on the total scheme is isomorphic to the equivariant cohomology for regular actions of principally paired group on smooth projective varieties, by Theorem~\ref{wholeg}, as well as for GKM spaces by Theorem~\ref{gkmthm}. A similar result can be also proved e.g. for spherical varieties, see Theorem \ref{nonreg}.

In the above, we used the fact that the torus-fixed points are isolated, but we also needed the GKM cohomology result, i.e. Theorem \ref{gkmres}. This way we know that any function on the zero scheme will be a cohomology class, as it will determine an element that already lies in $H$. Note that for arbitrary torus actions, every 1-orbit defines a similar condition on the image of localisation, but the image of localisation will be in general strictly smaller than similarly defined $H$.

We can see that it is not enough to assume for the torus to act with isolated fixed points. Indeed, let us consider $X = \PP^2$, but we restrict the standard action of the two-dimensional torus to one-dimensional $\Cs$. Take e.g. the action $t\cdot [x:y:z] = [x:t^2y:t^4z]$. The only fixed points are $[1:0:0]$ and $[0:1:0]$ and $[0:0:1]$ and hence if we consider any nonzero $v\in \C \cong \Lie(\Cs)$, the associated vector field only has those three zeros. On the other hand, for $v = 0$ the zero scheme is the whole $\PP^2$. Therefore $\ZZ_\tot \subset \C\times X$ will consist of a vertical $\PP^2$ and three horizontal lines. The action of $t\in \Cs$ multiplies by $t^{-2}$ on each of those lines.

Any global function on $\ZZ_\tot$ determines polynomials $f_1$, $f_2$, $f_3$ on those lines. Then $\C[\ZZ_\tot] = \{(f_1,f_2,f_3)\in \C[x]^3 | f_1(0) = f_2(0) = f_3(0)\}$. There is an injective map $H^*_{\Cs}(\PP^2) \to \C[\ZZ_\tot]$, but it is not surjective. From Example \ref{exsl22} we have $H_{\Cs}^*(\PP^2) = \C[x,v]/\big(x(x+2v)(x+4v))$. Geometrically, we see the map $\ZZ_\tot\to \Spec H_{\Cs}^*(\PP^2)$ which contracts $\PP^2$ to the point. We see that $h_{\Cs}^2(\PP^2) = 2$, but $\C[\ZZ_\tot]^2 = \{(ax,bx,cx) | a,b,c\in \C\}$ is three-dimensional.

\end{remark}

\chapter{Further directions and open problems}\label{chapfur}

The previous chapter presented the results of \cite{HR}. The article opens a lot of interesting possibilities and motivates many questions, extending the original results. We discuss here the newer developments by the author, and sketch potential new directions and conjectures.

\section{General singularities}

In Proposition \ref{sing} we prove how to recover the spectrum of equivariant cohomology of possibly singular projective variety in a specific case. Namely, we require it to be embedded in a smooth variety with a regular action such that the restriction map on cohomology is surjective. For a completely general singular variety with a regular action, the situation might be much harder. There is in general no reason to expect that the variety will even only have even cohomology. However, even without the assumption on cohomology, we can still infer something about a subring of the equivariant cohomology ring. Namely, we need to restrict our attention to the subring generated by the equivariant Chern classes of equivariant vector bundles. This is in some sense the subring of those elements that come from equivariant geometry. In the last 30 years, similar subrings have been considered for moduli spaces, and are know as \emph{tautological rings}. They have been introduced by Mumford in \cite{Mumford} and it rose to importance with Kontsevich's work \cite{Kontsevich}. See also the expository article by Ravi Vakil \cite{Vakil}.

Let $\Hs$ be an algebraic group. For any $\Hs$-variety $X$, we define $\widetilde{H}^*_\Hs(X)$ as the subring of $H^*_\Hs(X)$ generated by the equivariant Chern classes of $\Hs$-linearised vector bundles. Notice that if $X$ is a smooth variety with regular $\Hs$-action, then by Lemma \ref{genprinc} we have $\widetilde{H^*}_\Hs(X) = H^*_\Hs(X)$.

\begin{theorem}\label{singful}
Assume that $X$ is a smooth projective variety with a regular action of a principally paired group $\Hs$. Let $Y\subset X$ be a closed $\Hs$-invariant subvariety, possibly singular. Let $\Ss$ be a Kostant section of $\Hs$. Consider the zero scheme $\ZZ^X \subset \Ss\times X$ of the total vector field restricted to $\Ss\times X$, as in Theorem \ref{general}. Let $\ZZ^Y = \ZZ^X\cap (\Ss\times Y)$ be the reduced intersection. Then the isomorphism $\rho:H^*_\Hs(X)\to \C[\ZZ^X]$ descends to an isomorphism $\widetilde{\rho}:\widetilde{H}^*_\Hs(Y) \to \C[\ZZ^Y]$, so that $\ZZ^Y\simeq \Spec \widetilde{H}^*_\Hs(Y)$. This makes the following diagram commutative.
\begin{equation}
\begin{tikzcd}
H^*_\Hs(X) \arrow[r, "\iota^*"] \arrow[dd,"\rho"]
& \widetilde{H}^*_\Hs(Y)  \arrow[dd,"\widetilde{\rho}"]
\\ \\
\C[\ZZ^X] \arrow[r]
& \C[\ZZ^Y]
\end{tikzcd}
\end{equation}
\end{theorem}

\begin{proof}
 As in the proof of Theorem \ref{semisimp}, the scheme $\ZZ^Y$ for a reductive or general principally paired group is the quotient of the scheme for its Borel by the Weyl group action. Similarly, if $\Ts$ is the maximal torus of $\Hs$, $\widetilde{H}^*_\Hs(X)$ is the Weyl-invariant part of $\widetilde{H}^*_\Ts(X)$, as in Lemma \ref{genred}. Hence we can assume that $\Hs$ is solvable. Then we actually have $H^*_\Hs(Y) = H^*_\Ts(Y)$ and $\widetilde{H}^*_\Hs(Y) = \widetilde{H}^*_\Ts(Y)$. Therefore we can in fact view $\widetilde{H}^*_\Hs(Y)$ as generated by $\Ts$-equivariant Chern classes in $H^*_\Ts(Y)$ of $\Hs$-equivariant vector bundles.
 
As $X$ is smooth, from Lemma \ref{genprinc} we have that $H^*_\Hs(X)$ is already generated by the Chern classes of $\Hs$-equivariant vector bundles. Therefore the restriction $H^*_\Hs(X) \to \widetilde{H}^*_\Hs(Y)$ is well defined. For any $c\in H^*_\Hs(X)$ the function $\rho(c_k^\Ts(\Ee))$ is defined on $(v,p)$ by localisation to a torus-fixed point in the same $\Hs$-orbit as $p$. Hence the values of $\iota^*(c)$ on the points of $\ZZ^Y$ are defined by localisation only to torus-fixed points in $Y$. Hence also $\widetilde{\rho}$ is well-defined. The diagram is then obviously commutative. We need to prove that $\widetilde{\rho}$ is an isomorphism. Surjectivity follows from surjectivity of the restriction $\C[\ZZ^X]\to \C[\ZZ^Y]$, as $\ZZ^Y$ is a closed subscheme of an affine scheme $\ZZ^X$.

We now only need to prove injectivity of $\widetilde{\rho}$. Note that we can push forward any $\Hs$-linearised vector bundle $\Ee$ on $Y$ to a coherent sheaf $\iota_*\Ee$ on $X$. Then for any $k$, the Chern class $c_i^\Ts(\iota_*\Ee) \in H^*_\Ts(X)$ is well-defined, as $X$ is smooth. Its pullback to $H^*_\Hs(X)$ is $c_i^\Ts(\Ee)$, as $\iota^*\iota_* = \id$ for $\iota$ being a closed embedding. As $\iota_*\Ee$ is trivial outside of $Y$, it localises trivially to any torus-fixed point not in $Y$. Hence we have shown that any element $c\in\widetilde{H}^*_\Gs(Y)$ has a lift to $\hat{c}\in H^*_\Hs(X)$ which localises trivially to torus-fixed points outside of $Y$. Assume then that $\widetilde{\rho}(c) = 0$. This means that $\hat{c}$ also localises trivially to all torus-fixed points in $Y$, as $\widetilde{\rho}(c)$ is computed by localisation to those points and we can recover them from $\widetilde{\rho}(c)$ on the regular semisimple locus. Therefore $\hat{c} = 0$, as the localisation on $X$ is injective by equivariant formality and \cite[Theorem 1.6.2]{GKM}. Hence $c=0$. This means that $\widetilde{\rho}$ is injective.
\end{proof}

\begin{example}
We extend Remark \ref{exdiscri}. We described there the discriminant variety $Y\subset X = \PP^3$. Consider the action of $\Gs = \SL_2$ on it. If we view the elements of $\C^4$ as the polynomials $ax^3 + bx^2y+ cxy^2 + dy^3$, then $Y$ is defined by the equation
$$27a^2d^2 + 4ac^3 + 4b^3d - b^2c^2 - 18abcd = 0.$$
It parametrises polynomials with multiple roots, or rather -- in homogeneous language -- polynomials that contain a square in the decomposition into linear terms. The variety consists of two $\SL_2$ orbits. The first one is smooth and dense in $Y$ and it consists of polynomials with two distinct factors:
$$U = \{f^2g | f,g\in \Span_\C[x,y]  f\neq g\}.$$
The second one is closed in $Y$ and it consists of cubes of linear factors:
$$D = \{f^3 | f\in\Span_\C[x,y] \}.$$
The variety $Y$ is singular along the latter. As $\dim Y = 2$ and $\dim D = 1$, the variety is not even normal. In fact, on the slice $b=0$ we see the classical cuspidal curve, as then, assuming $a=1$, the discriminant classically degenerates to $27d^2 + 4c^3$.

However, as we mentioned in Remark \ref{exdiscri}, $Y$ contains all the $\Ts$-fixed points of $X$ and so $\ZZ^X = \ZZ^Y$. Topologically $Y$ is homeomorphic to $\PP^1\times \PP^1$. The isomorphism $\PP^1\times\PP^1\to Y$ maps
$$(f,g) \mapsto f^2g,$$
where we view elements of $\PP^1$ as homogeneous linear polynomials. We have
$$H^*(\PP^1\times \PP^1) = \C[x,y]/(x^2,y^2)$$
with $\deg x = \deg y = 2$. Hence the $P_{H^*(Y)}(t) = (1+t^2)^2$. From equivariant formality we then have
$$P_{H^*_{\Gs}(Y)}(t) = P_{H^*(Y)}(t) \cdot P_{H^*_\Gs}(t) = \frac{(1+t^2)^2}{(1-t^4)} = 1 + 2(t^2 + t^4 + t^6 + \dots).$$
On the other hand, 
$$P_{\widetilde{H}^*_{\Gs}(Y)}(t) = P_{H^*_{\Gs}(X)}(t) = \frac{1+t^2+t^4+t^6}{(1-t^4)} = 1 + t^2 + 2(t^4 + t^6 + \dots).$$

Notice that unlike in Proposition \ref{sing}, where the restriction on equivariant cohomology was surjective, here the restriction is not surjective, but injective -- however, only in equivariant cohomology. As we can see from the Poincar\'{e} series, in $\widetilde{H}^*_{\Gs}(Y)$ there is one missing dimension in $H^2$. In fact, by the description of the isomorphism $\PP^1\times\PP^1\to Y$, one sees that $\widetilde{H}^2_\Gs(Y)$ is spanned by $2\hat{x} + \hat{y}$, where $\hat{x}$ and $\hat{y}$ are the generators of the two copies of $H^*_\Gs(\PP^1)$.
\end{example}

\section{Non-regular actions}

Most of the theorems we have proved so far on equivariant cohomology, e.g. Theorem \ref{semisimp}, Theorem \ref{sing}, Theorem \ref{singful}, relied on the assumption that the linear group $\Hs$ acts regularly, i.e. with a single zero of the regular nilpotent $e\in\he$. A particular exception is Theorem \ref{gkmthm}, where we work under the assumption a torus acts on $X$, making it a GKM space. Note that for a torus $\Ts$, the only regular nilpotent in $\ttt$ is $0$ -- and in fact, that is the only nilpotent. Therefore no torus action on a nontrivial variety can be regular, as the zero set of $0$ is the whole variety. However, as in Theorem \ref{gkmthm}, we can recover the equivariant cohomology as the ring of functions on the zero scheme also in some of those cases. This therefore motivates the following question: is there some joint set of assumptions, generalising the GKM and regular case, under which we can prove that the ring of functions on the zero scheme equals the equivariant cohomology ring? In this section we provide one particular weakening of the regularity assumption, where some of the results still hold. This covers many interesting examples, such as spherical varieties.

\begin{theorem}\label{nonreg}
Assume that a principally paired group $\Hs$ acts on a smooth projective variety $X$. Let $\Ss \subset \he$ be the Kostant section. Let $\ZZ \subset \Ss\times X$ be the zero scheme of the vector field restricted from the total vector field on $\he\times X$. Assume that $\dim \ZZ = \dim \Ss$. Then there is an isomorphism
$$ H^*_\Hs(X,\C) \to \C[\ZZ] $$
of graded $\C[\Ss]\simeq H^*_\Hs(\pt,\C)$-algebras. In other words, the following diagram is commutative.
$$
\begin{tikzcd}
H^*_\Hs(X,\C) \arrow[r, "\rho"] 
& \C[\ZZ]
\\ \\
H^*_\Hs(\pt,\C) \arrow[r, "\simeq"]\arrow[uu]
& \C[\Ss] \arrow[uu]
\end{tikzcd}
$$
The grading on $\C[\ZZ]$ is defined by the action of $\Cs$ given by
$$t\cdot (v,x) = \left(\frac{1}{t^2}\Ad_{H^t}(v), H^t x\right).$$
Moreover, $H^i(\ZZ,\OO_\ZZ) = 0$ for any $i>0$. 
\end{theorem}

\begin{remark}
Note first that any regular action satisfies the $\dim \ZZ = \dim \Ss$ condition. Indeed, by Lemma \ref{isoreg}, in that case all the fibers of the map $\ZZ \to \Ss$ are finite. In addition, all the spherical varieties, and in general all the varieties with finitely many orbits of $\Hs$, satisfy the condition. However, the special fibers of $\ZZ\to \Ss$ might have positive dimension. See the details in Example \ref{exsph}.
\end{remark}

\begin{proof}[Proof of Theorem \ref{nonreg}]
First, by the methods of Section \ref{secred}, we know that $\C[\ZZ]$ is isomorphic to the ring of $\Ws$-invariant functions on the analogous zero scheme for a solvable group. As $H^*_\Hs(X) = H^*_\Ts(X)^\Ws$, we can restrict to $\Hs$ being solvable itself. Therefore we assume $\Hs$ is solvable from now on. In particular, $\Ss = e +\ttt$, where $e$ is the chosen regular nilpotent, and $\ttt = \Lie(\Ts)$ for a maximal torus $\Ts$ containing $h$.

The scheme $\ZZ$ is defined as a zero scheme of the vector field $V_\Ss$, coming from restriction of the total vector field to $\Ss\times X$. The vector field $V_\Ss$ is vertical, i.e. at any point $(v,x)$ it is tangent to $\{v\}\times X$. Hence, in fact, it is a section of the \emph{vertical tangent bundle} $T_v$, i.e. the pullback of the tangent bundle of $X$ via the projection $\Ss\times X\to X$.

Note that $T_v$ is a vector bundle of rank equal to $n = \dim X$. Let $\Omega_v^p = \Lambda^p T_v^*$. Then Lemma \ref{thmkosz} and the condition $\dim \ZZ = \dim \Ss$ imply that the Koszul complex
$$0\to \Omega_v^n \xrightarrow{\iota_{V_{\Ss}}} \Omega_v^{n-1} \xrightarrow{\iota_{V_{\Ss}}} \dots \xrightarrow{\iota_{V_{\Ss}}}  \Omega_v^1  \xrightarrow{\iota_{V_{\Ss}}}  \Omega_v^0 \to 0$$
is a resolution of $\OO_\ZZ$. Therefore its hypercohomology equals the cohomology of $\OO_\ZZ$ on $\Ss\times X$. But note that
$$H^*(\Ss\times X,\OO_\ZZ) = H^*(\ZZ,\OO_\ZZ).$$
Indeed, $\OO_\ZZ$ on $\Ss\times X$ is the pushforward of $\OO_\ZZ$ on $\ZZ$, and pushforward along a closed embedding is exact.

Hence $H^*(\ZZ,\OO_\ZZ)$ is isomorphic to the hypercohomology of the Koszul complex. We denote the Koszul complex by $K^\bullet$, with cohomological grading, where the index ranges from $-n$ to $0$.
We know by Corollary \ref{corkosz} that there is a spectral sequence with the first page
$$
	E_1^{pq} = H^q(\Ss\times X,\Omega_v^{-p})
$$
convergent to $H^{p+q}(X,\OO_\ZZ)$. In $\Ss = e + \ttt$, the semisimple elements are dense. Indeed, the regular locus in $\ttt$ is a complement of hyperplanes, cf. Section \ref{secsolvreg}. By Corollary \ref{corre} for a regular element $v\in \ttt$, also $e+v$ is semisimple. The assumption $\dim \ZZ = \dim \Ss$ implies that the generic fibers of the projection $\ZZ\to \Ss$ are finite, and hence a generic regular semisimple element has finitely many zeros. As the $\Ts$-fixed points are zeros of semisimple elements, there are only finitely many of them.

In that case, the cohomology of $X$ is Tate, meaning $H^q(X,\Omega^p) = 0$ if $p\neq q$. Note that for any vector bundle $\Ee$ on $X$ we have
$$H^q(\Ss\times X,\pi_2^*(\Ee)) = H^q(X,\Ee) \otimes \C[\ttt].$$
The equality clearly holds for global sections and $\C[\ttt]$ is flat over $\C$. Therefore in the spectral sequence above, the only potentially nonzero entries are
$$E_1^{-p,p} = H^p(\Ss\times X,\Omega_v^{p}).$$

First, this means that $H^i(X,\OO_\ZZ)$ may only be nonzero if $i=0$. It remains to show that $H^0(X,\OO_\ZZ) = H^*_\Ts(X,\C)$. We first show that the Poincar\'{e} series are equal on both sides.

From the spectral sequence, there is a filtration $F_*$ on $H^0(\OO_\ZZ)$ such that $F_p/F_{p-1} = H^p(X,\Omega^p) \otimes \C[\ttt]$.  We follow the idea from the proof of Theorem \ref{thmgrad}. We find a $\Cs$-action on the Koszul complex.  For $t\in \Cs$, let $t_p:\Omega^p_v\to \Omega^p_v$ be the pullback of forms along the map $t:X\to X$.  On $\Ss\times X$ the torus $\Cs$ acts by
$$t\cdot (e+v,x) = (e + t^{-2} v, H^t\cdot x)$$
and this satisfies the property $t_* V_{\Ss} = t^2 V_{\Ss}$.
As in the proof of Theorem \ref{thmgrad}, the following is commutative.

$$
\begin{tikzcd}
0\arrow[r] &
\Omega^n_v \arrow[rr, "\iota_{V_{\Ss}}"] \arrow[d, "t^{2n} t^{-1}_n"]&&
\Omega^{n-1}_v \arrow[rr, "\iota_{V_{\Ss}}"] \arrow[d, "t^{2(n-1)} t^{-1}_{n-1}"] &&
\dots \arrow[r, "\iota_{V_{\Ss}}"] \arrow[d]&
\Omega^{1}_v \arrow[rr, "\iota_{V_{\Ss}}"] \arrow[d, "t^{2} t^{-1}_1"]  &&
\OO_{\Ss\times X} \arrow[r] \arrow[d,"(t^{-1})^*"] & 0\\
0\arrow[r] &
\Omega^n_v \arrow[rr, "\iota_{V_{\Ss}}"] &&
\Omega^{n-1}_v \arrow[rr, "\iota_{V_{\Ss}}"] &&
\dots \arrow[r, "\iota_{V_{\Ss}}"] &
\Omega^{1}_v \arrow[rr, "\iota_{V_{\Ss}}"] &&
\OO_{\Ss\times X} \arrow[r] & 0
\end{tikzcd}
$$

Hence we have lifted the action of $\Cs$ on $\OO_{\Ss\times X}$ to the action on the whole Koszul complex. On the level of the spectral sequence, $t\in\Cs$ acts on 
$H^p(\Ss\times X,\Omega_v^{p}) = H^p(X,\Omega^p) \otimes \C[\ttt]$ by $t^{2p} t^{-1}_{p}$. As multiplication by $t$ on $X$ is homotopic to the identity, it induces the identity on $H^p(X,\Omega^p)$. On $\C[\ttt]$, however, the action of $t$ is nontrivial, and the Poincar\'{e} series of $\C[\ttt]$ is $\frac{1}{(1-t^2)^r}$. As the spectral sequence is $\Cs$-equivariant, we have

$$P_{\C[\ZZ]}(t) = \frac{\sum \dim H^p(X,\Omega^p) t^{2p}}{(1-t^2)^r} = \frac{P_{H^*(X,\C)}(t)}{(1-t^2)^r} = 
P_{H^*_\Ts(X,\C)}(t),$$
where the last part follows from equivariant formality.

Now to finish the proof, it would be enough to construct an injective graded map $H^*_\Ts(X,\C) \to \C[\ZZ]$. We will show an injective graded map 
\begin{equation}\label{nonregmap}
\rho:H^*_\Ts(X,\C)\to \C[\ZZ]^{\red}
\end{equation}
from the equivariant cohomology to the reduction of the ring of functions on $\ZZ$. Note that a priori, the scheme $\ZZ$ could be non-reduced, and so could be $\C[\ZZ]$. However, $\C[\ZZ]^{\red}$ is also a graded ring, and it is a quotient of $\C[\ZZ]$, hence every term of $P_{\C[\ZZ]^{\red}}(t)$ is less than or equal to the corresponding term in $P_{\C[\ZZ]}(t)$. Hence an existence of an injective map \eqref{nonregmap} will also prove that $\C[\ZZ]$ is in fact a reduced ring, and hence provide an isomorphism $H^*_\Ts(X,\C)\to \C[\ZZ]$.

As the scheme $\ZZ$ is Noetherian, the ring $\C[\ZZ]^{\red}$ is the ring of set-theoretic functions from the closed points of $\ZZ$ to $\C$, which are given by regular functions in $\C[\ZZ]$. Let $c\in H^*_\Ts(X,\C)$. We need to define $\rho(c)$ as a function on the closed points of $\ZZ$. Let $(e+w,x)\in \ZZ$, so that the vector field $V_{e+w}$ vanishes at $x$. 

From Theorem \ref{jordan}, there exists $M\in \Hs$ such that $e+w = \Ad_M(w+n)$ with $[w,n] = 0$ and $n$ nilpotent. Then by Lemma \ref{lemad} $x = M y$ for some $y$, which is a zero of $w+n$. Let $P$ be an irreducible component of the zero scheme of $w+n$, which contains $y$. Then by Lemma \ref{grpzer} $P$ contains a fixed point $\zeta_i$ of $\Ts$. Both $P$ and $\zeta_i$ can be non-unique. However, let us make a choice for every point $(e+w,x)\in \ZZ$ and define
$$\rho(c)(e+w,x) = c|_{\zeta_i}(w).$$
Here $c|_{\zeta_i}$ is an element of $H^*_\Ts(\pt) = \C[\ttt]$ and hence a function on $\ttt$, which we apply to $w$. We will show that such defined $\rho(c)$ is in $\C[\ZZ]^{\red}$ -- it will follow that the definition does not depend on the choices. It will also be clear that the map $\rho$ is a homomorphism of rings, as restriction to the $\Ts$-fixed points is.

By Lemma \ref{chern1} and the remark following it, the algebra $H^*_\Ts(X)$ is generated by the Chern classes of $\Hs$-linearised vector bundles. Therefore it is enough to prove that for $c = c_k^\Ts(\Ee)$ the function $\rho(c)$ comes from a regular function. We will prove it by showing that
$$\rho(c)(e+w,x) = \Tr_{\Lambda^k \Ee_x}(\Lambda^k (e+\w)_x),$$
as the latter comes from a regular function on $\ZZ$.

The regularity of the right-hand side implies that the function is constant on any projective subvariety. Let us consider a fixed $(w,x)\in\ZZ$ and $M$ as above, so that $e+w = \Ad_M(w+n)$ with $[w,n] = 0$ and $n$ nilpotent, and $P$ an irreducible component of the zero scheme of $w+n$, which contains $y = M^{-1}x$, and $\zeta_i \in X^\Ts\cap P$. As $P$, and hence also $M\cdot P$, is projective, we have
$$\Tr_{\Lambda^k \Ee_x}(\Lambda^k (e+\w)_x) = 
\Tr_{\Lambda^k \Ee_{My}}(\Lambda^k (e+\w)_{My}) =
\Tr_{\Lambda^k \Ee_{M\zeta_i}}(\Lambda^k (e+\w)_{My})
$$
and as $\Ee$ is $\Hs$-linearised, this is equal to
$$\Tr_{\Lambda^k \Ee_{\zeta_i}}(\Lambda^k (e+\w)_{y}).$$
Further, as $e+\w$ is a Jordan decomposition, this equals
$$\Tr_{\Lambda^k \Ee_{\zeta_i}}(\Lambda^k (\w)_{y}).$$
By \eqref{cherntr} this is exactly $c|_{\zeta_i}(w)$.

Therefore we have defined the map $\rho$, we only have to show that it is injective. But for $w\in\ttt^{\reg}$ there exists $M\in\Hs$ such that $\Ad_M(w) = w+e$. Then for any $\zeta_i\in X^\Ts$ we have
$$\rho(c)(e+w,M^{-1}\zeta_i) = c|_{\zeta_i}(w).$$
As the regular elements are dense in $\ttt$, we see that from $\rho(c)$ we can recover the localisations of $c$ to all the $\Ts$-fixed points. By injectivity of localisation for equivariantly formal spaces, from that one can recover $c$. Hence the map $\rho$ is injective, which finishes the proof.
\end{proof}

\begin{example} \label{exsph}
Assume that $X$ is a smooth projective \emph{spherical} variety for a reductive group $\Gs$ \cite{BrionSph}. This means that $\Gs$ acts on $X$ and the Borel subgroup $\Bs\subset \Gs$ has an open dense orbit in $X$. If $\Gs$ is a torus, then $\Bs = \Gs$ and $X$ is simply a toric variety. In general, all the wonderful compactifications of $G/H$ for $H = G^\sigma$, where $\sigma:\Gs\to\Gs$ is an involution, are spherical \cite{deConcProc}. Classical examples include the \emph{variety of complete collineations} \cite{Thadd,VainsencherLinn,Tyrrell}, which compactifies $\PGL_n$, and the \emph{variety of complete quadrics} \cite{deConcQuad,Tyrrell,VainsencherQuad}, which compactifies $\SL_n/SO_n$. The renewed interest in spherical varieties in recent years \cite{sph,lect} comes from their connections to mirror symmetry.

All the spherical varieties satisfy the conditions of Theorem \ref{nonreg}. In fact, any smooth projective $\Gs$-variety with finitely many orbits of $\Gs$ will satisfy the conditions. Spherical $\Gs$-varieties have only finitely many $\Gs$ orbits by \cite[1.5]{BrionSph}. Now let us first consider the single orbit $\OO = \Gs/\Ks$ for $\Ks \subset \Gs$ a subgroup. Let us try to find $\ZZ^\reg\subset \geg^\reg\times \OO$, the zero scheme of the vector field defined by the action, restricted to the regular elements of $\geg$. By homogeneity, the dimensions of all the fibers over the elements of $\OO$ are the same. Hence $\dim \ZZ^\reg = \dim \OO + \dim \ZZ^\reg_1$, where $\ZZ^\reg_1$ is the fiber over $[1] = [\Ks]\in \Gs/\Ks$. But that fiber is simply equal to $\geg^\reg\cap \kek$. It is either empty, or of dimension equal to $\dim K$. Hence $\ZZ^\reg$ is either empty, or of dimension equal to the dimension of $\Gs$. The same will then be true if instead of a single orbit we consider a variety with finitely many orbits. Now, by the properties of the Kostant section, i.e. Corollary \ref{kostprincpair}, 
$$\dim \ZZ = \dim \ZZ^\reg - (\dim \Gs - \rk \Gs) = \rk \Gs,$$
hence the assumptions are satisfied.

However, the computations are very tedious, as the scheme $\ZZ$ is in general not affine. So far, we have not been able to compute any concrete ring of functions on $\ZZ$ for a nontrivial example, other than a projective space or a toric variety.
\end{example}

\begin{remark}
In case $\Gs = \Ts$ is a torus and the action of $\Ts$ is is faithful, the condition of the lemma can only be satisfied for $X$ being a toric variety. Indeed, in a torus we have $e = 0$, hence $\ZZ$ contains $\{0\}\times X$ as a subscheme. Therefore $\dim X \le \dim \Ts$. For a faithful action, this only holds when $\dim X = \dim \Ts$ and $X$ is toric. In that case, the variety is also a GKM space, so one can apply Theorem \ref{gkmthm}. That theorem holds for a larger class of spaces, however it does not say anything about higher cohomology $H^i(\ZZ,\OO_\ZZ)$. In fact, if $X$ is not a toric variety, then $H^i(\ZZ,\OO_\ZZ)$ might be nonzero even for $i>0$. For example, if $X = \Gr(4,2)$ is the Grassmannian of 2-planes in $\C^4$ and $\Ts \subset \SL_4$ is a maximal torus, of dimension 3, then $H^2(\ZZ,\OO_\ZZ) = \C$.
\end{remark}

One therefore notices that for a principally paired group $\Gs$ and a smooth projective $\Gs$-variety $X$, there are different levels of correspondence between the ring of functions and coherent cohomology on $\ZZ$ and the equivariant cohomology of $X$.

\begin{enumerate}
\item If the action of $\Gs$ is regular, then $\rho:H^*_\Gs(X)\to\C[\ZZ]$ is an isomorphism. Moreover, the variety $\ZZ$ is affine. Conversely, if $\ZZ$ is affine, then so is any fiber over an element of $\Ss$. In particular the zero scheme of $e\in\geg$ is affine. However, it is projective, as $X$ is projective, hence it is of dimension $0$. As $e$ generates an additive subgroup of $\Gs$, by Theorem \ref{horrfix} it has to be connected. Therefore it is just a single, potentially non-reduced, point. In other words, the action is regular.
\item If the action of $\Gs$ satisfies the condition of Theorem \ref{nonreg}, i.e. $\dim \ZZ = \rk \Gs$, then $H^*_\Gs(X)\to \C[\ZZ]$ is still an isomorphism, and the higher coherent cohomology of $\OO_\ZZ$ is trivial. However, $\ZZ$ does not need to be affine anymore.
\item There are some more situations where the action does not satisfy $\dim \ZZ = \rk \Gs$, but still $\C[\ZZ]\simeq H^*_\Gs(X)$. This covers in addition the case of a torus acting on a GKM space. It is not clear what are the exact conditions under which the isomorphism holds. Note that at least when the torus-fixed points are isolated, we can still define the map $\rho: H^*_\Gs(X) \to \C[\ZZ]^{\red}$.
\end{enumerate}

We then have the following diagram of implications. In the first column we have conditions on the geometry of $\Gs$-action, and in the other two the relations between $\ZZ$ and equivariant cohomology.

$$
\begin{tikzcd}
\text{Regular action} \arrow[d, Rightarrow, thick] \arrow[r, Leftrightarrow, thick] & \ZZ\simeq \Spec H^*_\Gs(X,\C) \arrow[d, Rightarrow, thick] \arrow[r, Leftrightarrow, thick] & \ZZ \text{affine}  \arrow[d, Rightarrow, thick] \\
\dim \ZZ = \rk \Gs \arrow[r,Rightarrow, thick] \arrow[d, Rightarrow, thick]
& H^i(\ZZ,\OO_{\ZZ}) = \begin{cases}
H^*_\Gs(X,\C) \text{ for } i = 0; \\
0 \text { otherwise}
\end{cases}  \arrow[r,Rightarrow, thick] \arrow[d, Rightarrow, thick] &
\begin{array}{c}
H^i(\ZZ,\OO_{\ZZ}) = 0\\
\text{ for } i>0
\end{array} \\
?   \arrow[r,Rightarrow, thick] & H^0(\ZZ,\OO_\ZZ) = H^*_\Gs(X,\C)
\end{tikzcd}
$$

It is not clear what should replace the question mark and a work is ongoing to determine under what assumptions this result holds. Apart from the above results, for many affine Bott--Samelson varieties L\"owit \cite{jakub} proves $H^*_G(X,\C)\cong \C[\ZZ_\tot]$.

We can however clearly state the following problem.

\begin{problem}
Determine whether the implications in the second row are equivalences. That is:
\begin{enumerate}
\item does the vanishing of $H^i(\ZZ,\OO_\ZZ)$ for $i>0$ imply $H^0(\ZZ,\OO_\ZZ) = H^*_\Gs(X,\C)$;
\item does the vanishing of $H^i(\ZZ,\OO_\ZZ)$ for $i>0$ together with $H^0(\ZZ,\OO_\ZZ) = H^*_\Gs(X,\C)$ imply that $\dim \ZZ = \rk \Gs$.
\end{enumerate}
\end{problem}

\section{Fixed point schemes and equivariant K-theory}

First question one might ask is about extending the results to other equivariant cohomology theories. The first non-trivial one, different from equivariant cohomology, is equivariant K-theory. As noted in Section \ref{seceqcoh}, the equivariant cohomology $H^*_\Gs(\pt,\C)$ of the point is equal to $\C[\ttt]^\Ws = \C[\geg]^\Gs$. On the other hand, by Section \ref{seceqk}, the equivariant algebraic K-theory $K^0_\Gs(\pt) \otimes \C$ of the point with complex coefficients is equal to the representation ring $R(\Gs)$ of $\Gs$. That equals $\C[\Ts]^\Ws = \C[\Gs]^\Gs$ \cite[Theorem 6.1.4]{ChGi}.

Hence, to change from equivariant cohomology to equivariant K-theory, we switch from the $\Gs$-invariant functions on the Lie algebra, to $\Gs$-invariant functions on the group. This suggests that one could recover K-theory by considering the action of the group instead of the Lie algebra. To this end, for an action of $\Gs$ on a scheme $X$, define the \emph{fixed point scheme} $\Fix_\Gs(X)$\footnote{This definition was provided by Jakub L\"{o}wit} by the following pullback diagram.

$$ \begin{tikzcd}
  \Fix_\Gs(X) \arrow{r}\arrow{d} &
  \Gs\times X\arrow{d}{\rho\times \pi_2}  \\
  X \arrow{r}{\Delta} &
  X\times X.
\end{tikzcd}$$

Here $X$ maps to $X\times X$ diagonally and the map $\Gs\times X$ maps $(g,x)$ to $(gx,x)$. Therefore, the scheme $\Fix_\Gs(X)$ parametrises the pairs $(g,x)\in \Gs\times X$, where $gx = x$. With this, Tam\'{a}s Hausel \cite{Tamas} has stated the following conjecture, based on Theorem \ref{wholeg}.

\begin{conjecture}\label{kthe}
 Assume that a principally paired group $\Gs$ acts on a smooth projective variety regularly. Let $\Gs$ act on $\Fix_\Gs(X)$ by $g\cdot (h,x) = (ghg^{-1},gx)$. Then the ring $\C[\Fix_\Gs(X)]^\Gs$ of $\Gs$-invariant functions on $\Fix_\Gs(X)$ is an algebra over $\C[\Gs]^\Gs\cong K^0_{\Gs}(\pt)\otimes \C$ isomorphic to the equivariant algebraic K-theory $K^0_{\Gs}(X)\otimes \C$.
 $$ \begin{tikzcd}
  \C[\Fix_\Gs(X)]^\Gs \arrow{r}{\cong} &
  K_\Gs^0(X)\otimes \C  \\
  \C[\Gs]^\Gs \arrow{r}{\cong} \arrow{u}&
  K_\Gs^0(\pt)\otimes \C \arrow{u}.
\end{tikzcd}$$
\end{conjecture}
 
One can then additionally ask three questions. First, whether we can also formulate and prove an analog of Theorem \ref{general}, where the total zero scheme is replaced with the zero scheme over the Kostant section. It seems an appropriate equivalent of the Kostant section is the Steinberg section \cite[4.15]{HumConj}, therefore we would like to prove also the analog of Theorem \ref{general}, with the Steinberg section replacing the Kostant section. One would expect that the potential proof, as in \cite{HR}, should start from considering the solvable case. Therefore one additional challenge is to find an appropriate notion of Steinberg section for solvable, or arbitrary principally paired groups, as we do in Section \ref{kostsecsec} for the Kostant section. If $\Gs$ a semisimple simply connected group, Holmes \cite{Holmes} has proved that the fixed point scheme for any partial flag variety $\Gs/\Ps$ over the Steinberg section is the spectrum of the equivariant K-theory.

Additionally, there exists the equivariant Chern character map, reviewed in Section \ref{seceqchern}, from equivariant K-theory to cohomology. This means that on the scheme level, we should get an analogous map in the other direction. Note that as the Chern character maps to the completion of cohomology, it will not give rise to an algebraic map. It has to be rather considered as a map from a formal scheme or a complex analytic map. In fact, we expect that the Chern character yields the $\Gs$-equivariant map

$$(v,x) \mapsto (\exp(v), x)$$

from $\ZZ_\tot$ to $\Fix_\Gs(X)$. Therefore along with proving Conjecture \ref{kthe}, we would like to show how the Chern character can be seen geometrically. Note that this may require taking total fixed point and zero schemes into account, as the exponential of the Kostant section is usually not in the Steinberg section. Alternatively, one might use a different section in the group.

Third, one could ask again how important the regularity assumption is. In fact, L\"{o}wit \cite{jakub} proves the claim for affine Bott--Samelson varieties.

Finally, we prove a K-theory analog of the GKM theorem for cohomology \ref{gkmthm}.

\begin{theorem}\label{gkmk}
Let a torus $\Ts \cong (\Cs)^r$ act on a smooth projective complex variety $X$ with finitely many zero and one-dimensional orbits. Let $\Fix'_\Ts(X)$ be the subscheme of $\Fix_\Ts(X)$, given by reduction. Then $\C[\Fix'_\Ts(X)] \cong K^0_\Ts(X)\otimes \C$ as algebras over $\C[\Ts]\simeq K^0_\Ts(\pt)\otimes \C$:
$$ \begin{tikzcd}
  \C[\Fix'_\Ts(X)] \arrow{r}{\cong}  &
  K^0_\Ts(X)\otimes \C   \\
   \C[\Ts] \arrow{r}{\cong} \arrow{u}&
  K^0_\Ts(\pt)\otimes \C \arrow{u}.
\end{tikzcd}$$
\end{theorem}

\begin{proof}
We proceed as in the proof of Theorem \ref{gkmthm}. We construct first a map
$$\rho: K^0_\Ts(X)\otimes \C \to \C[\Fix_\Ts(X)]$$
and then an injective left inverse.

We can define $\rho$ by its values, as $\Fix'_\Ts(X)$ is reduced. For any $\Ts$-linearised vector bundle $\Ee$ on $X$ let
$$\rho(\Ee)(t,p) = \Tr_{\Ee_p}(t).$$
Note that if $0\to \Ee \to \F \to \Geg \to 0$ is an exact sequence of $\Ts$-linearised vector bundles, then for any $(t,p)\in \Fix'_\Ts(X)$ we have $\Tr_{\F_p}(t) = \Tr_{\Ee_p}(t) + \Tr_{\Geg_p}(t)$. Therefore the function $\rho$ is well defined on $K^0_\Ts(X)$, and by definition it always gives a regular function on $\Fix_\Ts(X)$.

As in Theorem \ref{gkmthm}, let us denote the $\Ts$-fixed points by $\zeta_1$, $\zeta_2$, \dots, $\zeta_s$ and the one-dimensional orbits by $E_1$, $E_2$, \dots, $E_{\ell}$. The closure of any $E_i$ is an embedding of $\PP^1$ and contains two fixed points $\zeta_{i_0}$ and $\zeta_{i_\infty}$, which for any $x\in E_i$ are equal to the limits $\lim_{t\to 0} t x$ and $\lim_{t\to\infty} tx$. The action of $\Ts$ on $E_i$ has a kernel $K_i$ of codimension $1$. Then by \cite[Corollary A.5]{Rosu} the localisation $K_\Ts(X)\to K_\Ts(X^\Ts)$ is injective and its image is

$$
H = \left\{
(f_1,f_2,\dots,f_s)\in \C[\Ts]^s  \, \bigg| \, f_{i_0}|_{K_i} = f_{i_\infty}|_{K_i} \text{ for } j=1,2,\dots,\ell
\right\}.
$$

We will use Lemma \ref{surjtot} again. Take any $(t,p)\in \Fix'_\Ts(X)$. Then $p$ lies in the fixed point scheme of $t\in\Ts$. As $\Ts$ is commutative, any translate of $p$ is also a fixed point of $t$. Therefore we have $\{t\}\times \overline{T\cdot p}\subset \Fix'_\Ts(X)$ as a closed subvariety. Then by the Borel fixed point theorem it contains a fixed point of the whole torus. Hence the conditions of Lemma \ref{surjtot} are satisfied for the inclusion $\Ts\times X^\Ts \subset \Fix_\Ts(X)$, so the restriction of global functions $\C[\Fix'_\Ts(X)]\to\C[\Ts\times X^\Ts]$ is injective.

From this, we construct an injective left inverse of $\rho$. First, note that $\C[\Ts\times X^\Ts] = \C[\Ts]^s$. Then notice that the restriction $\tau:\C[\Fix'_\Ts(X)]\to\C[\Ts\times X^\Ts] = \C[\Ts]^s$ actually maps into $H$. Indeed, take $i\in\{1,2,\dots,\ell\}$. We want to prove that for any $f\in\C[\Fix'_\Ts(X)]$ the functions $\tau(f)|_{\Ts\times \zeta_{i_0}}$ and $\tau(f)|_{\Ts\times \zeta_{i_\infty}}$ are equal as functions on $\Ts$.

Indeed, in the fixed point scheme, for any $t\in K_i$ both points $(t,\zeta_{i_0})$ and $(t,\zeta_{i_\infty})$ lie in the same connected projective subvariety $t\times \overline{E_i}$. Therefore the values of $f$ on those points are the same.

We proved that $\tau$ is injective, we only need to prove that $\tau\circ\rho = \id$. Take a class $[\Ee]\in K^0_\Ts(X)$, where $\Ee$ is a $\Ts$-linearised vector bundle on $X$. Then $\tau\circ\rho([\Ee])$ is a function on $\Ts\times X^Ts$ which at the point $(t,\zeta_i)$ attains the value $\Tr_{\Ee_{\zeta_i}}(t)$. When we fix $\zeta_i$, this means that we restrict $\Ee$ to $\zeta_i$ to get a representation of $\Ts$, and then we check what the trace of $t$ on it is. But this is exactly how we define the isomorphism $K_0^*(\pt) \to \C[\Ts]$.
\end{proof}

\begin{remark}
The zeros of a torus action on a smooth variety are reduced, cf. Theorem \ref{fixred}. Therefore the author also expects the scheme $\Fix_\Ts(X)$ to actually be reduced in this situation, so that $\Fix_\Ts(X) = \Fix'_\Ts(X)$. However, the exact argument is missing.
\end{remark}

\pagestyle{plain}


%

\backmatter

\end{document}